\documentclass[11pt]{article}
\usepackage{amsmath,amssymb}
\usepackage{amsthm}
\usepackage{mathtools}
\usepackage[noend]{algorithmic}
\usepackage[ruled,vlined]{algorithm2e}
\usepackage{url}
\usepackage{makeidx}

\usepackage{graphicx,float,psfrag,epsfig,caption}
\usepackage[usenames,dvipsnames,svgnames,table]{xcolor}
\definecolor{darkgreen}{rgb}{0.0,0,0.9}
\usepackage[pagebackref,letterpaper=true,colorlinks=true,pdfpagemode=none,citecolor=OliveGreen,linkcolor=BrickRed,urlcolor=BrickRed,pdfstartview=FitH]{hyperref}
\usepackage{epstopdf}
\usepackage{color}
\usepackage{xr}
\usepackage{subfig}
\usepackage{caption}
\usepackage{graphicx}
\usepackage[utf8]{inputenc}
\usepackage{tikz}
\usepackage[english]{babel}
\usepackage{am}
\usepackage{dsfont}

\usepackage{scalerel,stackengine}

\usepackage[mathscr]{euscript}

\DeclareSymbolFont{rsfs}{U}{rsfs}{m}{n}
\DeclareSymbolFontAlphabet{\mathscrsfs}{rsfs}
\usepackage[mathscr]{euscript}

\DeclareSymbolFont{rsfs}{U}{rsfs}{m}{n}
\DeclareSymbolFontAlphabet{\mathscrsfs}{rsfs}

\numberwithin{equation}{section}

\newtheoremstyle{myexample} % name
    {\topsep}                    % Space above
    {\topsep}                    % Space below
    {\rm }                   % Body font
    {}                           % Indent amount
    {\bf }                   % Theorem head font
    {.}                          % Punctuation after theorem head
    {.5em}                       % Space after theorem head
    {}  % Theorem head spec (can be left empty, meaning normal)

\makeatletter
\newcommand{\grad}{\nabla}
\newcommand{\event}{\mathcal{E}}
\newcommand*{\rom}[1]{\expandafter\@slowromancap\romannumeral #1@}
\makeatother

\usetikzlibrary{arrows,shapes}
\newcommand{\be}{\begin{equation}}
\newcommand{\ee}{\end{equation}}
\newcommand{\ba}{\begin{aligned}}
\newcommand{\ea}{\end{aligned}}

\newcommand{\setsq}{\mathcal{S} }
\renewcommand{\d}{{\rm d}}

\newcommand{\ones}{\textbf{1}}
\newcommand{\ind}{\mathds{1}}

\newcommand{\one}{\textbf{1}}

\newcommand{\opt}{^*}
\newcommand{\owid}{^{*,{\mbox{\rm \tiny wide}}}}
\newcommand{\wtilde}{\widetilde}
\newcommand{\iid}{\stackrel{i.i.d}{\sim}}
\newcommand{\eps}{\varepsilon}
\newcommand{\la}{\lambda}

\newcommand{\margin}{\kappa}

\newcommand{\low}{\downarrow}
\newcommand{\upp}{\uparrow}

\def\ERM{{\rm ERM}}
\def\error{{\rm Err}}
\def\Pred{{\rm Err}}
\def\Bayes{{\rm Bayes}}
\def\erf{Q}
\def\cL{\mathcal{L}}
\def\cC{\mathcal{C}}

\def\cT{\mathcal{T}}

\def\hbtheta{\hat{\boldsymbol\theta}}
\def\hbu{\hat{\boldsymbol u}}

\def\sMM{\mbox{\tiny\rm MM}}
\def\sSM{\mbox{\tiny\rm SM}}
\def\sRF{\mbox{\tiny\rm RF}}

\def\snew{\mbox{\tiny\rm new}}
\def\siso{\mbox{\tiny\rm iso}}
\def\smiss{\mbox{\tiny\rm miss}}

\def\bSigma{{\boldsymbol \Sigma}}

\def\bLambda{{\boldsymbol \Lambda}}
\def\bTheta{{\boldsymbol \Theta}}
\def\bzero{{\boldsymbol 0}}
\def\bfone{{\boldsymbol 1}}
\def\by{{\boldsymbol y}}
\def\bx{{\boldsymbol x}}
\def\bu{{\boldsymbol u}}
\def\bxi{{\boldsymbol \xi}}
\def\bh{{\boldsymbol h}}
\def\bX{{\boldsymbol X}}

\def\blambda{{\boldsymbol \lambda}}
\def\btheta{{\boldsymbol \theta}}

\def\sT{{\sf T}}
\def\bg{{\boldsymbol g}}

\def\bw{{\boldsymbol w}}
\def\bv{{\boldsymbol v}}
\def\tbz{\tilde{\boldsymbol z}}
\def\bz{{\boldsymbol z}}
\def\bZ{{\boldsymbol Z}}
\def\bW{{\boldsymbol W}}
\def\bbeta{{\boldsymbol \beta}}

\def\bphi{{\boldsymbol \phi}}

\def\tX{\tilde{X}}
\def\tG{\tilde{G}}

\def\tx{\tilde{x}}
\def\tbx{\tilde{{\boldsymbol x}}}
\def\txi{\tilde{\xi}}

\def\Unif{{\sf Unif}}
\def\de{{\rm d}}
\def\id{{\boldsymbol I}}
\def\hy{\hat{y}}
\def\cbar{\bar{c}}
\def\kbar{\bar{\kappa}}

\def\<{\langle}
\def\>{\rangle}

\def\asL{\mathscrsfs{R}}
\def\asG{\mathscrsfs{R}_{\infty}}
\def\ed{\stackrel{{\rm d}}{=}}
\def\Ball{{\sf B}}
\def\set{{\mathcal S}}
\newcommand{\proj}{{\boldsymbol \Pi}}
\def\S{{\mathbb S}}

\def\lam{\lambda_{\textsf{M}}}

\newcommand{\pto}{\stackrel{p}{\to}}

\newcommand{\LL}{\mathcal{L}}

\newcommand{\lrho}{\lesssim_{\rho, f}}
\newcommand{\grho}{\gtrsim_{\rho, f}}

\newcommand{\hc}{\tilde{c}}
\newcommand{\las}{\lambda_{\textsf{s}}}

\newcommand{\lan}{\lambda_{\circ}}
\DeclareMathOperator{\plim}{p-lim}

\newcommand{\prob}{\P}

\begin{document}

\title{The generalization error of max-margin linear classifiers:\\ 
Benign overfitting and high dimensional asymptotics in the overparametrized regime}

\author{Andrea Montanari\thanks{Department of Electrical Engineering
    and Department of Statistics, Stanford University}, \;\;\;\;  Feng
  Ruan\thanks{Department of Statistics and Data Science, Northwestern University}, \;\;\;\; Youngtak Sohn\thanks{Department of Mathematics, Massachusetts Institute of Technology}, \;\;\;\; 
Jun Yan\thanks{Department of Statistics, Stanford University}}

\maketitle

\begin{abstract}
Modern machine learning classifiers often exhibit vanishing 
classification error on the training set. They achieve this by learning nonlinear
representations of the inputs that maps the data into linearly separable classes. 

Motivated by these phenomena, we revisit high-dimensional maximum margin classification
for linearly separable data. We consider a stylized setting in which 
data $(y_i,\bx_i)$, $i\le n$ are i.i.d.
with $\bx_i\sim\normal(\bzero,\bSigma)$  a
 $p$-dimensional Gaussian feature vector, and $y_i \in\{+1,-1\}$
a label whose distribution depends on a linear combination  of the covariates $\<\btheta_*,\bx_i\>$.
While the Gaussian model might appear extremely simplistic, universality arguments
can be used to show that the results derived in this setting also apply to the output of
certain nonlinear featurization maps. 

We consider the proportional asymptotics $n,p\to\infty$
with $p/n\to \psi$, and derive exact expressions for the limiting generalization error.
We use this theory to derive two  results of  independent interest:
$(i)$~Sufficient conditions on $(\bSigma,\btheta_*)$ for `benign overfitting' that parallel previously
derived conditions in the  case of linear regression; $(ii)$~An 
asymptotically exact expression for the generalization error when max-margin
classification is used in conjunction with feature vectors produced by random one-layer neural networks.
\end{abstract}

\tableofcontents

\section{Introduction}
\label{sec:Introduction}

\subsection{Background}

Modern machine learning models for classification, such as multi-layer neural networks,
are a composition of multiple nonlinear maps,
which produce increasingly simple representations of the data. 
A linear readout unit outputs the class label. In the case of binary classification, 
on input $\bz\in\reals^d$, such a model outputs
\begin{align}
\hy(\bz) = \sign\,  \<\btheta , \bphi(\bz;\bW)\>\, ,\label{eq:GenericModel}
\end{align}
where the featurization map $\bphi:\reals^d\to\reals^p$ encodes a nonlinear 
data representation, parametrized by weights $\bW$.
For instance, in the case of a multi-layer neural network,
 $\bphi(\bz;\bW) =  \sigma \circ \bW_1\circ \sigma\circ\cdots\sigma\circ\bW_L\bz$.
 
 In the practice of machine learning, it is often the case that these models
 achieve vanishing error on the training data and, despite this, they generalize well to unseen data.
Vanishing training error means that the representation $\bphi(\bz;\bW)$
 is able to map the data into linearly separable classes. There are two possible
mechanisms for this:
\begin{itemize}
\item[$(i)$] The parameters $\bW$ are also learnt from training data, and hence
$\bphi(\;\cdot\;;\bW)$  is a highly non-linear data-dependent map that
 makes the data separable. Notice that by allowing for a sufficiently rich class of 
 maps $\bphi(\;\cdot\;;\bW)$, linear separability can be achieved even with a low embedding 
 dimension  $p$. 
\item[$(ii)$] The map $\bphi(\;\cdot\;;\bW)$ is not learnt from the same training data and it is possibly entirely random. In this case, linear separability emerges mainly because of
the dimension blow up from $d$ to $p$.
\end{itemize}
While both mechanisms ---learning and dimensionality blow-up--- are relevant for fully
trained neural networks, this paper focuses on the second aspect. 
This is most important for nonlinear models in the neural tangent or lazy regime \cite{jacot2018neural}, but also for 
kernel methods \cite{hofmann2008kernel,wahba2002soft} and for their random features approximation 
\cite{neal1996priors,balcan2006kernels,rahimi2008random}.
Notice that in these applications, the interpretation of the features dimensions $p$
varies. For instance, in the case of neural nets in the lazy regime, $p$ is the overall number of parameter,
and not just the size of the last layer.
Given the current status
of mathematical technology, the dimensionality blow-up is more amenable to rigorous analysis and yet very challenging.
Indeed,  we will leave several mathematical questions unsolved and, despite the substantial follow-up work, many questions have been unsolved since the first appearance of this manuscript.

We are thus led to consider the set of linear classifiers with vanishing empirical error, namely:
\begin{align}
\ERM_0(\by,\bX) &:= \Big\{ \btheta\in \reals^p: \,\, \|\btheta\|_2=1,\;\;\min_{i\le n} y_i\<\btheta,\bx_i\>\geq 0\;\Big\}\, ,\label{eq:ERM1}\\
\bX & = \left[\begin{matrix}\;\text{---}\;\bx_1\;\text{---}\;\\
\vdots\\
\;\text{---}\;\bx_n\;\text{---}\;\end{matrix}\right]\, ,\;\;\;\; \bx_i = \bphi(\bx_i;\bW)\, , \label{eq:ERM2}
\end{align}
where $\bphi(\,\cdot\,;\bW)$ is a featurization map independent of the data.

A rich line of work supports the intuition that among all the possible classifiers 
with vanishing training error $\btheta\in \ERM_0(\by,\bX)$, the ones selected by  
optimization algorithms  used in practice have special `simplicity' properties
\cite{soudry2018implicit,gunasekar2018implicit,li2017algorithmic,gunasekar2018characterizing,arora2019implicit}. 
This phenomenon is commonly referred to as `implicit regularization.'

Of particular interest (and our focus in this paper) is the max-margin classifier:
\begin{align}
\hbtheta^{\sMM}(\by,\bX):=\arg\max \Big\{ \;\;\min_{i\le n} y_i\<\btheta,\bx_i\>:\;\;\; \|\btheta\|_2=1\;\;\Big\}\, .
\label{eq:MMdef}
\end{align}
Indeed, it was proven in \cite{soudry2018implicit} that gradient descent
(with respect to logistic loss) converges to $\hbtheta^{\sMM}(\by,\bX)$.
Namely, considering the gradient-descent iteration
\begin{align}
\hbtheta^{k+1} = \hbtheta^k-s_k \nabla \hat{L}_n(\hbtheta^k) \, ,\;\;\; \hat{L}_n(\btheta) \equiv \frac{1}{n}\sum_{i=1}^n\left\{-y_i\<\btheta,\bx_i\>+
\log\big(e^{\<\btheta,\bx_i\>}+e^{-\<\btheta,\bx_i\>} \big) 
\right\}\, .
\end{align}
we have $\hbtheta^k/\|\hbtheta^k\|\to \hbtheta^{\sMM}$ as $k\to\infty$. 

In this paper, we will  study the generalization error
of max-margin classification for i.i.d. separale data $(y_i,\bx_i)$, $i\le n$.
For this purpose, the form of the featurization map
$\bz_i\mapsto \bx_i = \bphi(\bz_i;\bW)$ only matters to the extent that it determines
the distribution of the feature vectors $\bx_i$ given the underlying distribution of the $\bz_i$.
We will consider a stylized model  whereby the features are Gaussian
with population covariance $\bSigma$: $\bx_i\sim\normal(0,\bSigma)$. At first sight, this might 
appear to be unrelated to the original problem. However, as further discussed below,
recent universality results \cite{hu2022universality, montanari2022universality}, as well as our companion paper \cite{montanari2023universality_MaxMargin}, indicate
that the characterization we obtain for Gaussian features applies to a class of featurization maps
provided we match the covariances $\bSigma = \E_{\bz}[\bphi(\bz;\bW)\bphi(\bz;\bW)^{\sT}]$.

Throughout this paper we will say that a model is \emph{overparametrized}
if the set of zero-error classifiers $\ERM_0(\by,\bX)$ is non-empty\footnote{As we will prove,
within the setting of the paper, this happens with high probability if $p/n>\psi^*$ 
for a certain threshold $\psi^*$ which we characterize.}.
Over the last couple of years, the generalization properties of overparametrized models have 
attracted considerable interest
(see also Section \ref{sec:Related}). 
A unified phenomenology has emerged from simulation studies with a number of  statistical models,
including kernel methods, random forests, and multilayer neural networks \cite{belkin2019reconciling}

In order to discuss this phenomenology, it is convenient to regard the prediction error as a function
of two quantities: the sample size $n$ and the number of parameters $p$. The classical
statistical theory assumes that $p$ is fixed and is related to the data distribution. 
$p$ can be either smaller or larger than $n$ depending on whether
the low-dimensional or high-dimensional regimes are considered, but typically larger $p$
is regarded as yielding a different, `harder', data distribution.
In contrast, we think here of the data distribution as fixed, and larger $p$ corresponds to different 
featurization maps.
When looking at the possibility of increasing $p$ in this way, two
 interesting statistical behaviors have been observed repeatedly:
\begin{enumerate}
\item \emph{Benign overfitting.} The excess error (difference between the prediction error and the Bayes error)
can vanish as $p,n$ get large, despite: $(i)$~the model is extremely overparametrized $p\gg n$;
 $(ii)$~the model is not regularized and in particular, it has vanishing training error. 
\item \emph{Optimality of overfitting.} For certain data distributions, the
 test error of overparametrized models  is smaller than the test error of
 underparametrized ones. In particular, the test  error is minimized when $p/n\gg 1$.
 \end{enumerate}
 
Rigorous confirmation of these phenomena were established in a number of papers 
\cite{belkin2019two,belkin2018overfitting,hastie2022surprises,bartlett2020benign,tsigler2020benign,
mei2019generalization,montanari2022interpolation}.
(See Section \ref{sec:Related} for further references.)
The bulk of these rigorous studies, and by far the most detailed picture was developed
 in the case of ridge regression and its ridge-less limit min-norm regression. 
 A number of models for the feature vectors $\bx_i$ were studied
 in this context: unstructured distributions with prescribed covariance
  \cite{hastie2022surprises,bartlett2020benign}, kernel methods \cite{liang2020just}
  random features models \cite{mei2022generalization}, and neural tangent
  features \cite{montanari2022interpolation}. However, all of these works rely on the 
  linear-algebraic
  structure of the ridge estimator, and leverage tools from random matrix theory
  to characterize its behavior.
  Moving beyond ridge regression is an important step that requires fundamentally 
  different mathematical ideas.
 
 At this point, it is legitimate to wonder whether max-margin classification warrants being
 revisited.
 After all, the machine learning community has devoted significant attention to the 
 analysis of max-margin classifiers. An incomplete
selection of references include \cite{bartlett1998sample,anthony2009neural,koltchinskii2002empirical,bartlett2002rademacher,kakade2009complexity,koltchinskii2011oracle}.
This line of work develops upper bounds on the generalization error (difference between test and training error)
 in terms of the complexity (e.g. the Radamacher complexity) of the underlying function class. In the case of maximum margin classification, this approach yields upper bounds 
that depend on the empirical margin or the empirical margin distribution.
 In this theory, the empirical margin concentrates around the population margin
 (or the population margin distribution).
 Intuitively, data are (approximately) separable because the signal-to-noise ratio is very strong.
 
 In contrast, we are interested in cases in which the signal-to-noise ratio is moderate
 and the population distribution is not linearly separable (not even approximately so). 
 In the regime studied here, separability is a high-dimensional phenomenon that arises because of 
 overparametrization. 
 Appendix \ref{sec:Special} illustrates this claim by providing concrete examples in which 
 classical margin-based bounds fail to capture the qualitative behavior of the generalization
 error.

\subsection{Overview of results}
\label{sec:Overview}

We assume the feature vectors $\bx_i$  to be independent draws from a $p$-dimensional centered Gaussian with covariance $\bSigma$,
and responses to be distributed according to
\begin{align}
\P\big( y_i = +1\big|\bx_i\big) &= 1-\P\big( y_i = -1\big|\bx_i\big)=f(\<\btheta_*,\bx_i\>)\, ,\label{eq:LabelProbability}\\
\bx_i&\sim\normal(\bzero,\bSigma)\, .\label{eq:Covariates}
\end{align}
We will assume throughout the proportional asymptotics 
\begin{equation}\label{eq:def:psi}
    n,p\to\infty\quad\textnormal{with}\quad \frac{p}{n}\to\psi\in(0,\infty),
\end{equation}
and determine the precise asymptotics
of the  test error. In what follows, we will index sequence of instances by $n\in \N$, and it will be understood that $p = p_n$.

In order for the limit to exist and be well defined, we need to make certain assumptions about the behavior of 
the covariance matrix $\bSigma = \bSigma_n$ and the `true' parameters vector $\btheta_*=\btheta_{*,n}$. 
These are -in a way-- analogous to the assumptions made in random matrix theory to derive
the asymptotics of the empirical spectral distribution.

Let 
$\bSigma_n= \sum_{i=1}^p \lambda_i\bv_i\bv_i^{\sT}$ be
the eigenvalue decomposition of $\bSigma$, with $\lambda_1\ge \lambda_2\ge \dots\ge \lambda_p$. 
Our first assumption requires that the eigenvalues of $\bSigma$ do not decay too rapidly. 
\begin{assumption}\label{assumption:Lambdas}
Denote $\lambda_{\max}(\bSigma_n)=\lambda_1(\bSigma_n)$. There exist constants $\lam, L>0,$ and $\eps>0$ such that
	\begin{equation}\label{eq:lambda:max}
	 \lambda_{\max}(\bSigma_n)\le \lam\,,
	\end{equation}
 and
 \begin{equation*}
 \frac{1}{p}\sum_{i=1}^{p} \left(\frac{1}{\la_i(\bSigma_n)}\right)^{1+\eps}\leq L.
 \end{equation*}
\end{assumption}

Our second assumption concerns the eigenvalue distribution of $\bSigma_n$ as well as the decomposition of $\btheta_{*,n}$ in the
basis of eigenvectors of $\bSigma_n$. 
\begin{assumption}\label{assumption:converge}
Let $\rho_n \equiv \<\btheta_{*,n},\bSigma_n\btheta_{*,n}\>^{1/2}$ and $\bar{w}_i \equiv  \sqrt{p\lambda_i} \<\bv_i,\btheta_{*,n}\>/\rho_n$.
Then the empirical distribution of $\{(\lambda_i,\bar{w}_i)\}_{1 \leq i \leq n}$ converges in Wasserstein-$2$ distance to a 
probability distribution $\mu$ on $\reals_{>0}\times \reals$:
\begin{align}
\frac{1}{p}\sum_{i=1}^p\delta_{(\lambda_i, \bar{w}_i)}\stackrel{W_2}{\Longrightarrow} \mu\, .
\label{eq:LambdaWconv}
\end{align}
In particular, $\int w^2 \mu(\de\lambda,\de w) = 1$, and we have that
\begin{equation}\label{eq:def:rho}
    \rho \equiv \lim_{n\to\infty} \rho_n = \int   (w^2/\lambda) \mu(\de\lambda,\de w).
\end{equation}
\end{assumption}
We refer the reader to Appendix \ref{app:Notation} for a reminder of the definition of the Wasserstein distance $W_2$.
 Here, we limit ourselves to mentioning that convergence 
in $W_2$ is equivalent to weak convergence plus the convergence of the second moment, see e.g. \cite{villani2008optimal}. 
In particular, the condition $\int (w^2/\lambda) \mu(\de\lambda,\de w) = 1/\rho^2$
implies $\lim_{n\to\infty}\|\btheta_{*,n}\|_2= 1$. Notice that this choice of normalization  implies no loss of generality: if 
$\lim_{n\to\infty}\|\btheta_{*,n}\|_2= c\neq 1$,   we can  rescale $\btheta_{*,n}$ (letting $\btheta^{\snew}_{*,n}=\btheta_{*,n}/c$) and the function
 $f$ (letting $f^{\snew}(t) = f(ct)$), as to satisfy the assumed normalization.

Finally, we state our assumptions on the function $f$.
\begin{assumption}
\label{assumption:non-degenerate-f}
Define $T = YG$ where 
\begin{equation}\label{eq:Y:G:dist}
    \P(Y = 1 \mid G)= 1-\P(Y = -1 \mid G)= f(\rho \cdot G),\quad G \sim \normal(0, 1). 
\end{equation}
We assume $f: \R \to [0,1]$ to be continuous, and it satisfies the following non-degeneracy condition: 
	\begin{equation*}
		\inf \Big\{x: \P(T<x)>0 \Big\} = -\infty
			~~\text{and}~~
		\sup \Big\{x: \P(T>x)>0   \Big\} = \infty\, .
	\end{equation*}
\end{assumption}
It is easy to check that the non-degeneracy condition is satisfied for most `reasonable' choices of $f$. In particular, it is sufficient that $f(x)\in (0,1)$
for all $x$.
\begin{remark}
At first sight, Assumption \ref{assumption:converge} is the strongest of our conditions.
Note however that the convergence of Eq.~\eqref{eq:LambdaWconv} always holds along subsequences
under some tightness condition (by Prokhorov's theorem). For instance, this is the case if we assume that $\sum_{i=1}^p|\lambda_i|^{2+\eps}\le C\, p$ and $\sum_{i=1}^p|\bar{w}_i|^{2+\eps}\le C\, p$ hold for some constants $C,\eps>0$.

Under tightness, we could always apply our theory to characterize each converging subsequence 
of instances.
\end{remark}

Under these assumptions, we establish the following results.
\begin{description}
\item[Asymptotic characterization of the maximum margin.] Define the maximum margin by
\begin{align}
\kappa_n(\by,\bX) \equiv \max\big\{ \min_{i\le n}y_i\<\btheta,\bx_i\>:\;\; \btheta\in\reals^p, \; \|\btheta\|_2 =1\big\}\, .
\end{align}
We prove that $\kappa_n(\by,\bX)\to \kappa\opt(\mu,\psi)$ in probability as $n\to\infty$ for some non-random 
asymptotic margin $\kappa\opt(\mu,\psi)$. We give an explicit characterization of the limiting value $\kappa\opt(\mu,\psi)$, stated in Section
\ref{sec:MainResults}. 

As a corollary, we derive the limiting value of the interpolation threshold, i.e. the minimum number of 
parameters per dimensions above which the data are linearly separable with a positive margin: $\psi\opt(\mu) \equiv \inf\{\psi\ge 0:\; \kappa\opt(\mu,\psi)>0\}$, and below which the data are non-separable.
(This generalizes the recent result of \cite{candes2018phase}.)
\item[Asymptotic characterization of prediction error.] Let the test error be defined by
\begin{align}\label{eq:def:pred:error}
\Pred_n(\by,\bX) \equiv \P\big(y^{\snew}\<\hbtheta^{\sMM}(\by,\bX),\bx^{\snew}\> \le 0\big)\, ,
\end{align}
where expectation is with respect to a fresh sample $(y^{\snew},\bx^{\snew})$ independent of the data $(\by,\bX)$.
We will sometimes refer to $\Pred_n(\by,\bX)$ as to the prediction error.
We prove that $\Pred_n(\by,\bX)\to \Pred\opt(\mu,\psi)$ in probability as $n\to\infty$ for a non-random limit $\Pred\opt(\mu,\psi)$, which we characterize explicitly,
 cf. Section \ref{sec:MainResults}. 
\item[Benign overfitting.] We use the asymptotic formula of the test error
$\Pred\opt(\mu,\psi)$ to characterize sequences $(\bSigma,\btheta_*)$ along which we
achieve benign overfitting. More precisely, for fixed $\eps>0$, we characterize those sequences along 
which $\Pred_n(\by,\bX)-\Bayes\le \eps$ when $n,p\to\infty$ with $p\asymp n$
(here $\Bayes$ denotes the Bayes error).
To the best of our knowledge, this is the first generalization of the  results of \cite{tsigler2020benign} 
beyond ridge regression.
\item[Random features models.]
We  apply our general theory to the random features model of
\cite{rahimi2008random}. This corresponds to the general setting of Eqs.~\eqref{eq:ERM1},
\eqref{eq:ERM2} with featurization map
\begin{align}
\bphi(\bz;\bW) := \big(\sigma(\bw_1^{\sT}\bz);\dots;\sigma(\bw_p^{\sT}\bz)\big)\, ,
\label{eq:FirstRF}
\end{align}
where $\bW= (\bw_i)_{i\le p}$ are i.i.d. random weights. In other words $\bphi(\bz;\bW)$
is the output of a one-layer random neural network with $p$ hidden neurons.
While the feature vectors $\bx_i=\bphi(\bz_i;\bW)$ are non-Gaussian,
universality results \cite{montanari2023universality_MaxMargin} will allow us to apply the Gaussian theory nevertheless.

We observe that the test  error decreases monotonically with the network width $p$
and is minimal in the limit of large overparametrization $p/n\gg 1$.
This confirms the general phenomenology described above
and provides the first exact asymptotics for random features models beyond simple ridge regression.
\item[Technical innovation.] Our analysis is based on Gordon's Gaussian comparison inequality \cite{Gordon88}
and, in particular, its application to convex-concave problems developed in  \cite{ThrampoulidisOyHa15}.
This approach allows us to replace the original optimization problem by a simpler one, which is 
nearly separable. By studying the asymptotics of this equivalent problem, it is possible to obtain 
a precise characterization of the original problem in terms of the solution of a set of nonlinear equations.

However, this asymptotic characterization holds only if the set of nonlinear equations admit a unique solution.
Proving uniqueness can be challenging, and is normally done on a case-by-case basis. Here, we develop a new technique 
to prove uniqueness. 
In extreme synthesis, we construct, in a natural way,  an infinite-dimensional convex problem
whose KKT conditions are equivalent to the same set of nonlinear equations. We exploit this underlying convex structure to prove uniqueness. We believe this technique is potentially applicable to a broad set of problems.
\end{description}

We will begin our exposition by applying the general theory to establish benign
overfitting in Section \ref{sec:Benign} and to study the random features model in Section
\ref{sec:RF_model}.
We will then survey related work in Section  \ref{sec:Related}, and state our general 
technical results in Section \ref{sec:MainResults}.
Section \ref{sec:ProofMain} outlines the proof of these results while deferring most of 
the technical work to the appendices.

\section{Benign overfitting and the role of overparametrization}
\label{sec:Benign}

In the context of binary classification, the Bayes error is defined as 
the minimum prediction error achieved by any predictor $\hy:\reals^d\to\{+1,-1\}$:
\begin{align}
\Bayes =  \inf_{\hy:\reals^p\to\{\pm 1\}} \P\big(\hy(\bx)\neq y\big)\, .
\end{align}
In this section, we characterize the sequences $(\bSigma_n,\btheta_n^*)$ for which 
the generalization 
error of the max-margin classifier gets arbitrarily close to the Bayes error. 
% \cite{bartlett2020benign, tsigler2020benign} and (iii) the projection of the 
% signal $\theta^*$ onto the eigenspaces of $\Sigma$ that corresponds to small eigenvalues has small magnitudes. 
Conversely, we show that overparameterization 
is also necessary in order for the maximum margin classifier to attain near Bayes error.

In order to contain the  technical overhead, we 
assume link function $f$ is monotonically increasing with $f(0) = 1/2$. Under these
assumptions, the Bayes classifier is simply linear and is given by
$\hy(\bx) = \sign(\< \btheta^*, \bx\>$ and the Bayes error is simply
\begin{equation}\label{eq:def:Bayes}
    \Bayes:=\P(YG\leq 0)
\end{equation}
where the law of $(Y, G)$ is defined in Eq.~\eqref{eq:Y:G:dist}. 

\begin{theorem}\label{thm:benign}
Let $\{(n,p_n,\bSigma_n,\btheta^*_n)\}_{n\ge 1}$ be a sequence 
of instances such that $p_n /n\to\psi\in (0,\infty)$ and
$(\bSigma_n,\btheta^*_n)$, $f$ satisfy Assumptions 
\ref{assumption:Lambdas}-\ref{assumption:non-degenerate-f}, additionally
 assume that the link function
$f$ is almost everywhere differentiable, monotonically increasing with 
$f(0) = 1/2$ and $f^\prime(0)>0$.

\begin{itemize}
\item (Necessity) There exists a constant $c>0$ depending only on $\rho$, $\lam$, $f$ such that 
with probability converging to one
	\begin{equation}
	\label{eqn:error:lower:bound}
		\error_n(\by,\bX)-\Bayes \geq c \cdot \frac{n}{p}\,.
	\end{equation}
\item (Sufficiency) There exists a constant $C > 0$ depending only on $\rho$, $\lam$, $f$ such that 
for any $\lambda > 0$, the following holds with probability converging to one:
\begin{equation}\label{eq:error:upper:bound}
\error_n(\by,\bX)-\Bayes\leq C\cdot\big(\mathcal{B}_n(\lambda)+\mathcal{V}_n(\lambda)\big)\, .
\end{equation}
Here, $\mathcal{B}(\lambda)$ and $\mathcal{V}(\lambda)$ are given by
\begin{equation}\label{eq:def:bias:var}
\begin{split}
    &\mathcal{B}_n(\lambda):= 
    \frac{1}{\<\btheta_n^*,\bSigma_n\btheta_n^*\>}
    \left\{
    \Big(\frac{\lambda r_1(\lambda)}{n}\Big)^2 \sum_{i:\lambda_i>\lambda}
    \frac{1}{\lambda_i}\<\bv_i,\btheta^*_n\>^2+ \sum_{i:\lambda_i\le \lambda}
    \lambda_i\<\bv_i,\btheta^*_n\>^2
    \right\}
    \, ,\\
    &\mathcal{V}_n(\lambda):= \frac{r_0(\lambda)}{n}+\frac{n}{\overline{r}(\lambda)}\, ,
\end{split}
\end{equation}
where we defined, for $q\ge 0$,
\begin{align}
r_q(\lambda) :=\sum_{i:\lambda_i\le \lambda}\Big(\frac{\lambda_i}{\lambda}\Big)^q \, ,
\;\;\;\;\;\;\; \overline{r}(\lambda) := \frac{r_1(\lambda)^2}{r_2(\lambda)}\, .
\end{align}
\end{itemize}
\end{theorem}
Roughly speaking, $\mathcal{B}_n(\lambda)$ and $\mathcal{V}_n(\lambda)$ correspond
to a bias and variance term, despite the fact that an exact bias-variance decomposition 
does not hold for classification error. The structure of these terms is very similar to
the one of the bounds holding for ridge(-less) regression 
\cite{bartlett2020benign, tsigler2020benign}.
In particular, the excess error is small if:
$(i)$~the model is sufficiently overparameterized  (i.e., $\psi = p/n$ is large);
$(ii)$~the eigenvalues of the population covariance  $\bSigma$ are slowly decaying; and 
$(iii)$~the projection of the 
signal $\btheta^*$ onto the the span of eigenvectors of  $\bSigma$ that corresponds to 
small eigenvalues has small magnitude. 

In addition,  Theorem \ref{thm:benign} shows that overparameterization 
is necessary component for  max-margin estimator to achieve near Bayes risk 
in the high-dimensional setting studied here.

The proof of Theorem \ref{thm:benign} proceeds by applying our general characterization of the limit of $\error_n(\by,\bX)$ in Section \ref{sec:MainResults}. For it's technicality, the proof is deferred to Section \ref{sec:proof:benign}.

Below we provide two concrete examples of sequences of instances along which
the max-margin classification is `$\eps$-consistent', where the notion of `$\eps$-consistent' 
means that that the excess risk can be made smaller than $\eps$ for any pre-assigned $\eps > 0$. 

\vspace{0.5cm}
\begin{example}
Let $\{(n,p_n,\bSigma_n,\btheta^*_n)\}_{n\ge 1}$ denote a sequence of instances 
where $p_n/n\to\psi\in (0,\infty)$. Here, the matrix
$\bSigma_n=\diag(\la_1,\dots ,\la_p)$ and the truth $\btheta_n^*$
form a bilevel structure, meaning that there is a subset of $k_n$ covariates that are 
much more powerful than the rest $p_n-k_n$ junk covariates in terms of prediction, which is similar to the setup studied in \cite{muthukumar2021classification}. More precisely, taking constants $\lam > \las$ (independent of $n$), we consider
\begin{equation*}
 \la_i =
    \begin{cases}
      \lam &~~~i\le k_n\\
      \las &~~~i >k_n
    \end{cases}
    ~~~~~~~~
     (\theta_{*, n})_i =
    \begin{cases}
       \frac{1}{\sqrt{k}} &~~~i\le k_n\\
       0 &~~~i > k_n,
    \end{cases}
\end{equation*}
where $k_n/p_n \to \phi$ for some $\phi \in (0, \infty)$.
Under the conditions of  Theorem \ref{thm:benign}, there exists $C_{\lam, f}>0$ depending only on 
$\lam, f$ such that the following holds with probability converging to one:
\begin{equation*}
\error_n(\by,\bX)-\Bayes\leq C_{\lam, f}\cdot\big(\overline{\mathcal{B}}_n+\overline{\mathcal{V}}_n\big)\,,
\end{equation*}
where $\overline{\mathcal{B}}_n$ and $\overline{\mathcal{V}}_n$ are given by
\begin{equation*}
    \overline{\mathcal{B}}_n:=\left(\frac{p_n}{n} \cdot \las\right)^2\, ,~~~~~~~~~~
    \overline{\mathcal{V}}_n:=\frac{k_n}{n}+\frac{n}{p_n-k_n}\, .
\end{equation*}
In particular, for any $\eps>0$, one can first pick $\psi$ large enough (say, $\psi^{-1} \ll \eps$), and then 
$\phi, \las$ small enough (say, $\psi \phi < \eps$ and $\psi \las \ll \eps$), 
such that the excess error $\error_n(\by,\bX)-\Bayes\leq \eps$ with probability converging to one. 
\end{example}

\begin{example}
Here, we consider a more involved example where the eigenvalues of the covariance decay to zero at a certain 
delicate rate $\lambda_i \asymp g(i)$ ($i \to \infty$) where 
$g(j) \equiv j^{-1} (\log j)^{-\alpha}$ for some $\alpha > 1$. This is motivated by 
a setting recently proposed in the literature where \emph{benign overfitting}---under the context of ridgeless 
regression---is observed~\cite{bartlett2020benign}. In this example, we show that the \emph{benign overfitting} 
continues showing up when we change the setting from regression to classification. 

As before, consider a sequence of instances $\{(n,p_n,\bSigma_n,\btheta^*_n)\}_{n\ge 1}$
where $p_n/n\to\psi\in (0,\infty)$. Fixing an absolute constant 
$\lam > 1$, and taking $k \equiv k_n$ with $k_n/p_n \to \phi$, we consider a sequence of pair of covariance
$\bSigma_n=\diag(\la_1,\dots ,\la_p)$ and the ground truth $\btheta_n^* = ((\theta_{*, n})_1, \ldots, (\theta_{*, n})_p)$ 
where
\begin{equation*}
 \la_i =
    \begin{cases}
       g(k) \cdot \lam  &~~~i\leq k\\
       g(i) &~~~i >k
    \end{cases}
    ~~~~~~~~
     (\theta_{*, n})_i =
    \begin{cases}
       \frac{1}{\sqrt{k g(k)}} &~~~i\leq k\\
       0 &~~~i > k
    \end{cases}.
\end{equation*}
Above $g(j) = j^{-1} (\log j)^{-\alpha}$. Rescaling the eigenvalues 
$\lambda_i \mapsto \lambda_i / g(k)$ and the coordinates $(\theta_{*, n})_i \mapsto (\theta_{*, n})_i \cdot \sqrt{g(k)}$
allows us to apply Theorem  \ref{thm:benign}, which yields an error bound for the 
max-margin classifier for this  setup (below the constant $C_{\lam, f} > 0$
depends only on $\lam, f$)
\begin{equation*}%\label{eq:error:upper:bound:1}
\error_n(\by,\bX)-\Bayes\leq C_{\lam, f}\cdot\big(\overline{\mathcal{B}}_n^{\;\prime} +\overline{\mathcal{V}}_n^{\;\prime} \big)\,
\end{equation*}
that holds with probability converging to one. Here the quantities
$\overline{\mathcal{B}}_n^{\;\prime}$ and $\overline{\mathcal{V}}_n^{\;\prime}$ are given by
\begin{equation*}
    \overline{\mathcal{B}}_n^{\;\prime} :=\left(\frac{k_n\log\big(p_n/k_n\big)}{n}\right)^2\, ,~~~~~~~~~~
    \overline{\mathcal{V}}_n^{\;\prime}:=\frac{k_n}{n}+\frac{n}{k_n\left(\log\big(p_n/k_n\big)\right)^2}\, .
\end{equation*}
In particular, for any $\eps>0$, one can first pick $\psi$ large enough and then $\phi$ small enough
such that the excess error $\error_n(\by,\bX)-\Bayes\le \eps$ holds with probability converging to one.
\end{example}

\begin{remark}
Let us emphasize that, while the bounds in Theorem \ref{thm:benign} and the above examples 
are of order one as $n,p\to\infty$ in the proportional asymptotics,
they reveal the dependence of the excess risk on $(\bSigma,\btheta_*)$ because the constant $C$ only depends on $\lam$, $f$. Hence, they allow to 
establish $\eps$-consistency results.
\end{remark}

\section{A random features model}
\label{sec:RF_model}

Random features methods originate in the work  of Neal \cite{neal1996priors}, Balcan, Blum, Vempala \cite{balcan2006kernels}, and of Rahimi, Recht 
\cite{rahimi2008random}. A sequence of recent papers \cite{jacot2018neural,du2018gradient,chizat2018note} suggests that in the so-called `lazy training' regime, 
the behavior of multilayer networks is well approximated by certain random features model, whereby the randomness is connected with the initialization of the training process. 

Under this model the feature vectors $\bx_i$ are obtained by mapping the covariates
$\bz_i$ through a nonlinear featurization map, see Eqs.~\eqref{eq:ERM2} and \eqref{eq:FirstRF}
and further explanation below. In particular, $\bx_i$ is non-Gaussian.
Our approach to the analysis of this model will be based on \emph{universality}.
Namely, the asymptotics of the margin and prediction error under the random feature models is 
the same as for a Gaussian model with matching second order statistics.

Universality results under random features models were proved for ridge regression 
in  \cite{mei2019generalization}, strongly convex empirical risk minimization in 
 \cite{hu2022universality} and nonconvex empirical risk minimization in \cite{montanari2022universality}.
(See also \cite{cheng2013spectrum,fan2019spectral} for related results in the context 
of random matrix theory.)
For technical reasons, these results do not apply to max-margin classification,
and we present an extension in a companion paper \cite{montanari2023universality_MaxMargin}.

\subsection{Classification using random features}
\label{sec:RFDef}
% binary classification using random features. 
We assume to be given data $\{(y_i,\bz_i)\}_{i\le n}$, whereby $y_i\in\{+1,-1\}$, 
$\bz_i\in\normal(0,\id_d)$ and  
\begin{align}
\P\big(y_i=+1\big| \bz_i\big) = h(\<\bbeta_*,\bz_i\>)\, ,\;\;\; \|\bbeta_*\|_2=1\, .
\end{align}
Let us emphasize that $\bbeta_*\in\reals^d$ is the coefficient vector with
respect to the original covariates $\bz_i$. This is different from the coefficient vector
$\btheta_*$ of Eq.~\eqref{eq:LabelProbability}.

In order to perform classification, we proceed as follows:
$(i)$ We generate features $\tx_{ij}= \sigma(\<\bw_j,\bz_i\>)$ where $\sigma:\reals\to\reals$ is a non-linear
function. Here $\bw_j$, $j\le p$ are $d$-dimensional vectors which we draw uniformly on the unit 
sphere $\S^{d-1}(1)$, $(\bw_j)_{j\le p}\sim\Unif(\S^{d-1}(1))$.
$(ii)$~We find a max-margin separating hyperplane for data $\{(y_i,\tbx_i)\}_{i\le n}$, where $\tbx_i = (x_{ij})_{j\le p}$.

Equivalently, letting $\bW\in\reals^{p\times d}$ be the matrix with rows $\bw_i$, $1\le i\le p$, 
we compute the max-margin classifier $\hy(\bz) = \sign\,  \<\hbtheta^{\sMM} , \bphi(\bz;\bW)\>$
where $\hbtheta^{\sMM}$ is given by Eq.~\eqref{eq:MMdef} with featurization map 
\eqref{eq:FirstRF}. We summarize relevant formulas below for the readers convenience:
\begin{align}
\hbtheta^{\sMM}(\by,\bZ) &= \arg\max\Big\{\min_{i\le n} y_i\<\btheta,\bphi(\bz_i;\bW)\>:\;\;\; \|\btheta\|_2=1\Big\}\, ,\\
\bphi(\bz;\bW) &:= \big(\sigma(\bw_1^{\sT}\bz);\dots;\sigma(\bw_p^{\sT}\bz)\big)\, .
\end{align}
This can be described as a two layers neural network, with random first-layer weights which 
are fixed to $\bW$ and non-optimized.
Second-layer weights are instead given by $\btheta\in\reals^p$ and chosen as to maximize the margin.

\subsection{Asymptotics via equivalent Gaussian model and universality}
\label{sec:GaussianModel}

Following \cite{mei2019generalization}-- we will now construct a Gaussian covariates model 
that is asymptotically equivalent to the above random features model (in the sense of having same margin
and prediction error)
 in the limit
\begin{align}
\label{def:psionepsitwo}
    p,n,d\to\infty\quad\text{with}\quad p/d\to\psi_1\quad\text{and}\quad n/d \to\psi_2.
\end{align}

In order to motivate our construction,  we decompose the activation function in $\cL^2(\reals,\nu_G)$ (the space of square-integrable functions, with respect to
 $\nu_G$ the standard Gaussian measure)
as follows
\begin{align}
\sigma(u) = \gamma_0+\gamma_1\, u+\gamma_*\sigma_{\perp}(u)\, .
\end{align}
Here the constants $\gamma_0,\gamma_1,\gamma_{*}$ are given by
\begin{align}
\label{def:gamma}
\gamma_0 = \E\{\sigma(G)\}, \gamma_1 = \E\{G\sigma(G)\}\quad\text{and}\quad \gamma_*^2 = \E\{\sigma(G)^2\}- \E\{G\sigma(G)\}^2 -\E\{\sigma(G)\}^2,
\end{align}
where the expectation is over $G\sim\normal(0,1)$. We can then rewrite the random 
features model of the previous section as follows
\begin{align}
\tx_{ij} & = \gamma_0+\gamma_1\<\bw_j,\bz_i\> +\gamma_*\txi_{ij}\, , \;\;\;\; \txi_{ij}=\sigma_{\perp}(\<\bw_j,\bz_i\>)\, ,\\
g_i &= \<\bbeta_*,\bz_i\>\, , \;\;\;\;\;\;\;
\P(y_i = +1|g_i)  = h(g_i)\, .
\end{align}
In what follows, to simplify calculation we will assume $\gamma_0=0$ (activations are centered). 
The generalization to $\gamma_0\neq 0$ is quite natural.\footnote{Namely the formula on the right-hand
    side of \eqref{eq:AsyErrorDef}
    for the asymptotic prediction error is to be replaced by
    $\P(Y(b\opt\nu\opt(\psi) G+ \sqrt{1-\nu\opt(\psi)^2} Z) \le 0)$ for a suitable offset $b$.} 

Note that the random variables $\txi_{ij}$ have zero mean and unit variance by construction. Further $\E_{\bz_i}\{\txi_{ij}\<\bw_j,\bz_i\>\} = 0$
since by construction $\E\{\sigma_{\perp}(G) G\}=0$. This suggest to replace the $\txi_{ij}$ by a collection of independent random variables:
\begin{align}
x_{ij} & = \gamma_1\<\bw_j,\bz_i\> +\gamma_*\xi_{ij}\, ,  \;\;\;\; \xi_{ij}\sim\normal(0,1)\, ,\\
g_i &= \<\bbeta_*,\bz_i\>\, ,\, , \;\;\;\;\;\;\;
\P(y_i = +1|g_i)  = f(g_i)\, ,
\end{align}
Here  $(\xi_{ij})_{i\le n,j\le p}$ are drawn independently of
 $\{\bw_i\}_{i\le p}$, $\{\bx_j\}_{j\le p}$. These equations define
the `noisy linear features model.' 

Under the noisy linear features model $\bx_i$ and $g_i$ are jointly Gaussian.
Hence we can rewrite the joint distribution of $(\bx_i,y_i)$ in the form of 
Eq.~\eqref{eq:LabelProbability}, \eqref{eq:Covariates}
$\bSigma_n$, $\btheta_{*,n}$, $f_n$:
\begin{align}
\bSigma_n&:= \gamma_1^2\bW\bW^{\sT}+\gamma_*^2\id_p\, ,\label{eq:SigmaRF}\\
\btheta_{*,n} &:= \alpha_n^{-1}\gamma_1\big( \gamma_1^2\bW\bW^{\sT}+\gamma_*^2\id_p\big)^{-1}\bW\bbeta_*\, ,\label{eq:ThetaStarRF}\\
f_n(x) &:= \E\{h(\alpha_n\, x+\tau_n G)\}\, ,\label{eq:fn_RF}\\
\alpha_n^2 & = \gamma_1^2 \bbeta_*^{\sT}\bW^{\sT}\big( \gamma_1^2\bW\bW^{\sT}+\gamma_*^2\id_p\big)^{-2}\bW\bbeta_*\, ,\\
\tau^2_n& = 1-\gamma_1^2 \bbeta_*^{\sT}\bW^{\sT}\big( \gamma_1^2\bW\bW^{\sT}+\gamma_*^2\id_p\big)^{-1}\bW\bbeta_*\, .\label{eq:TaunRF}
\end{align}

We next verify that the parameters $\bSigma_n, \btheta_{*,n}$, and link function $f_n(x)$
satisfy the conditions of our general theory, namely 
Assumptions \ref{assumption:Lambdas},\ref{assumption:converge}, \ref{assumption:non-degenerate-f}.
Assumption \ref{assumption:Lambdas} immediately follows since $\lambda_{\min}(\bSigma_n)\ge \gamma_*^2>0$, and (for any $c>0$)
$\lambda_{\max}(\bSigma_n)\le \gamma_1(1+\sqrt{p/d}+c)^2+ \gamma_*^2$, with probability at least $1-\exp(-\Theta(d))$, by standard
bounds on the eigenvalues of Wishart random matrices \cite{vershynin2018high}. 

Next, we check Assumption \ref{assumption:converge} and determine the limit probability measure $\mu$.
Fix numbers $\gamma_1$, $\gamma_*, \psi_1, \psi_2 >0$, and consider the 
following probability measure on $(0,\infty)$:
\begin{align}
\mu_{s}(\de x) &= \begin{cases}
(1-\psi_1^{-1})\delta_{0} +\psi_1^{-2}\nu_{1/\psi_1}(x/\psi_1)\de x&\;\;\mbox{if $\psi_1>1$,}\\
\nu_{\psi_1}(x)\de x&\;\;\mbox{if $\psi_1\in (0,1]$,}
\label{eq-MPlaw}
\end{cases}\\
\nu_{\lambda}(x)& = \frac{\sqrt{(\lambda_+-x)(x-\lambda_-)}}{2\pi \lambda x} \, \bfone_{x\in [\lambda_-,\lambda_+]}\, ,\\
\lambda_{\pm} & =  (1\pm\sqrt{\lambda})^2\, .
\end{align}
By Marchenko-Pastur's law, the empirical spectral distribution of $\bW\bW^{\sT}$ converges in $W_2$ to $\mu_s$ almost surely as $p,d\to\infty$ \cite{BaiSilverstein}.
Let $\tX\sim \mu_s$ independent of $G\sim\normal(0,1)$.
Using Eq.~\eqref{eq:ThetaStarRF}, we obtain that (recalling from Assumption  \ref{assumption:converge} that $\bar{w}_i
= \sqrt{p\lambda_i} \<\bv_i,\btheta_{*,n}\>/\rho_n$, $\rho_n = \<\btheta_{*,n},\bSigma_n\btheta_{*,n}\>^{1/2}$)
\begin{align}
\label{eq:RF:mu}
\frac{1}{p}\sum_{i=1}^p\delta_{(\lambda_i,\bar{w}_i)}\stackrel{W_2}{\Longrightarrow} \mu := {\sf Law}(X,W)\, ,
\end{align}
where
\begin{align}
\label{eq:FirstXW}
X = \gamma_1^2\tX+\gamma_*^2\, ,\;\;\;\;\;\;\;\;\;\;
W = \frac{\gamma_1\sqrt{\psi_1 \tX}\, G}{C_0 (\gamma_1^2\tX+\gamma_*^2)^{1/2}}\, ,\;\;\;\;\;\;\;\;\;\;
C_0 = \E\Big\{\frac{\gamma_1^2 \psi_1 \tX}{(\gamma_1^2\tX+\gamma_*^2)}\Big\} ^{1/2} \, .
\end{align}

Finally, we need to check Assumption \ref{assumption:non-degenerate-f}.
Using Eq.~\eqref{eq:TaunRF}, we obtain $\tau_n\to\tau$ as $n\to\infty$, where:
\begin{align}
\label{eqn:def:tau}
\tau^2 = 1-\psi_1\E\Big\{\frac{\gamma_1^2 \tX}{\gamma_1^2\tX+\gamma_*^2}\Big\}
\end{align}
Since $\tau^2>0$, it follows from Eq.~(\ref{eq:fn_RF})  that Assumption  \ref{assumption:non-degenerate-f}  holds.

In our companion paper \cite{montanari2023universality_MaxMargin}, we prove that the margin and
test error of the random features model are universal, namely they coincide
with the ones of the equivalent Gaussian model we defined in this section. As a consequence
of our main results presented in Section \ref{sec:MainResults}, we obtain the following sharp characterization
of the random features model. 
\begin{theorem}
\label{thm:RF:Gaussian}
Let $\kappa\opt_{\sRF,n}(\by,\bZ)$ and $\error\opt_{\sRF,n}(\by,\bZ)$ be the maximum margin and
 test error of the random features model of Section \ref{sec:RFDef}. 
Further, let $\kappa\opt(\mu,\psi)$ and  $\Pred\opt(\mu,\psi)$ be the theoretical predictions for 
the max margin and the test error given in Definition~\ref{def:KappaE} below.
Then, in the limit $p,n,d\to \infty$ with $p/d\to\psi_1$ and $n/d\to\psi_2$, we have
\begin{align}
\plim_{p,n,d\to\infty}\kappa\opt_{\sRF,n}(\by,\bZ) &= \kappa\opt(\mu,\psi)\, ,\\
\plim_{p,n,d\to\infty}\error\opt_{\sRF,n}(\by,\bZ) &= \Pred\opt(\mu,\psi)\, ,
\end{align}
where $\mu =  {\sf Law}(X,W)$ is defined by Eq.~\eqref{eq:FirstXW}, and $\psi:=\psi_1/\psi_2$.
\end{theorem}

\begin{remark}
  Independently of its relationship with the nonlinear random features model,
  the noisy linear features model is a valid statistical method, which is of independent interest.
  Given data $\{(y_i,\bz_i)\}_{i\le n}$ which are potentially non-separable, it embeds them in $p$ dimensions
  via the noisy linear map $\bz_i\mapsto \gamma_1\bW\bz_i+\gamma_*\bxi_i$: this map can be 
  implemented in practice. 
\end{remark}

\subsection{Numerical experiments}

\begin{figure}[t]
\phantom{A}\hspace{-0.5cm}
\includegraphics[width=0.54\textwidth]{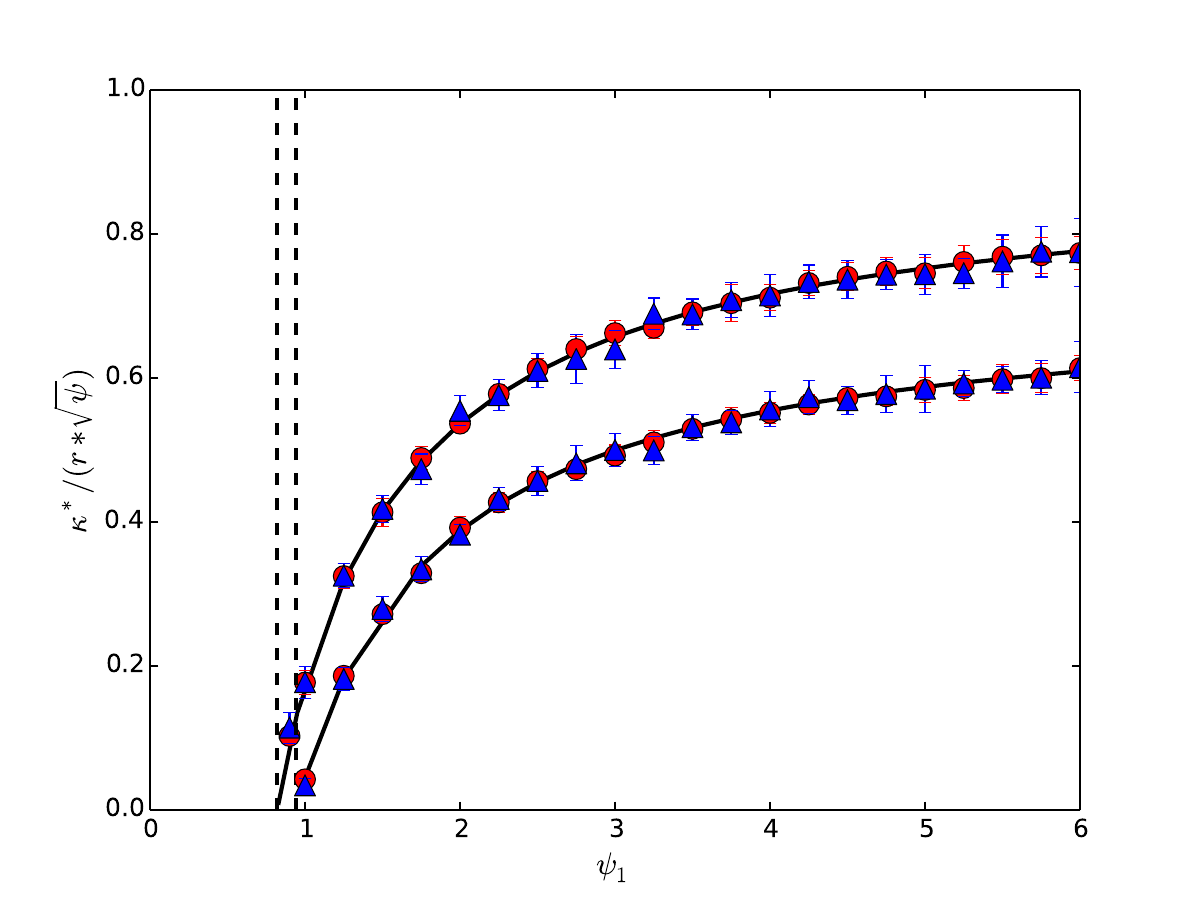}\hspace{-0.5cm}
\includegraphics[width=0.54\textwidth]{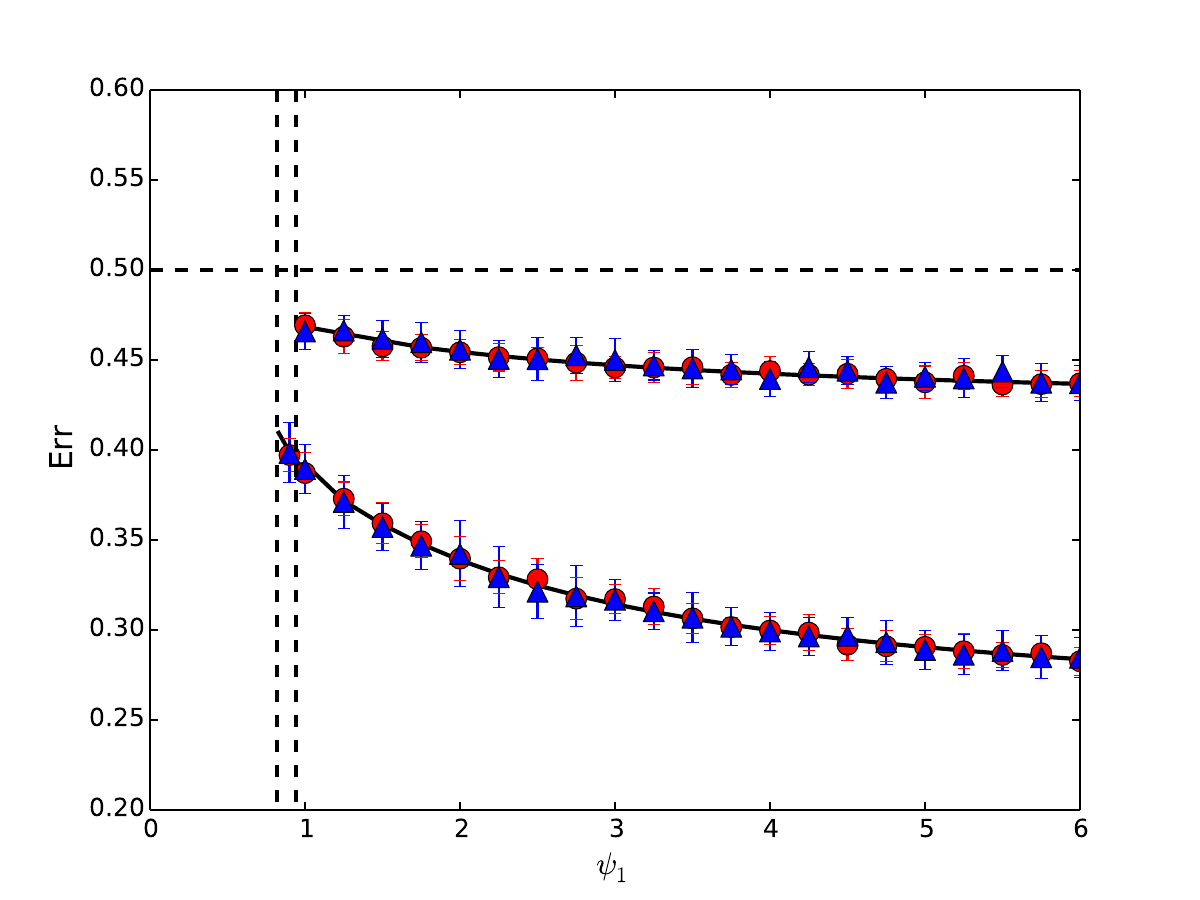}
\caption{Random features model, with ReLU activations. 
Left: maximum margin. Right: test error. Labels are generated using the logistic function 
$f(x) = (1+e^{-\beta x})^{-1}$ with $\beta =1, 4$ (bottom to top on the left and top to bottom on the right).
Here $\psi_2 = n/d= 2$, and red circles, blue triangles stand for $d = 400, 200$ respectively. Results were averaged over $20$ samples,
Dashed lines report the interpolation threshold, and continuous lines are the predicted test error, both within the Gaussian covariates model of Section 
\ref{sec:GaussianModel}.}\label{fig:random-features}
\end{figure}

\begin{figure}[t]
\phantom{A}\hspace{-0.5cm}
\includegraphics[width=0.54\textwidth]{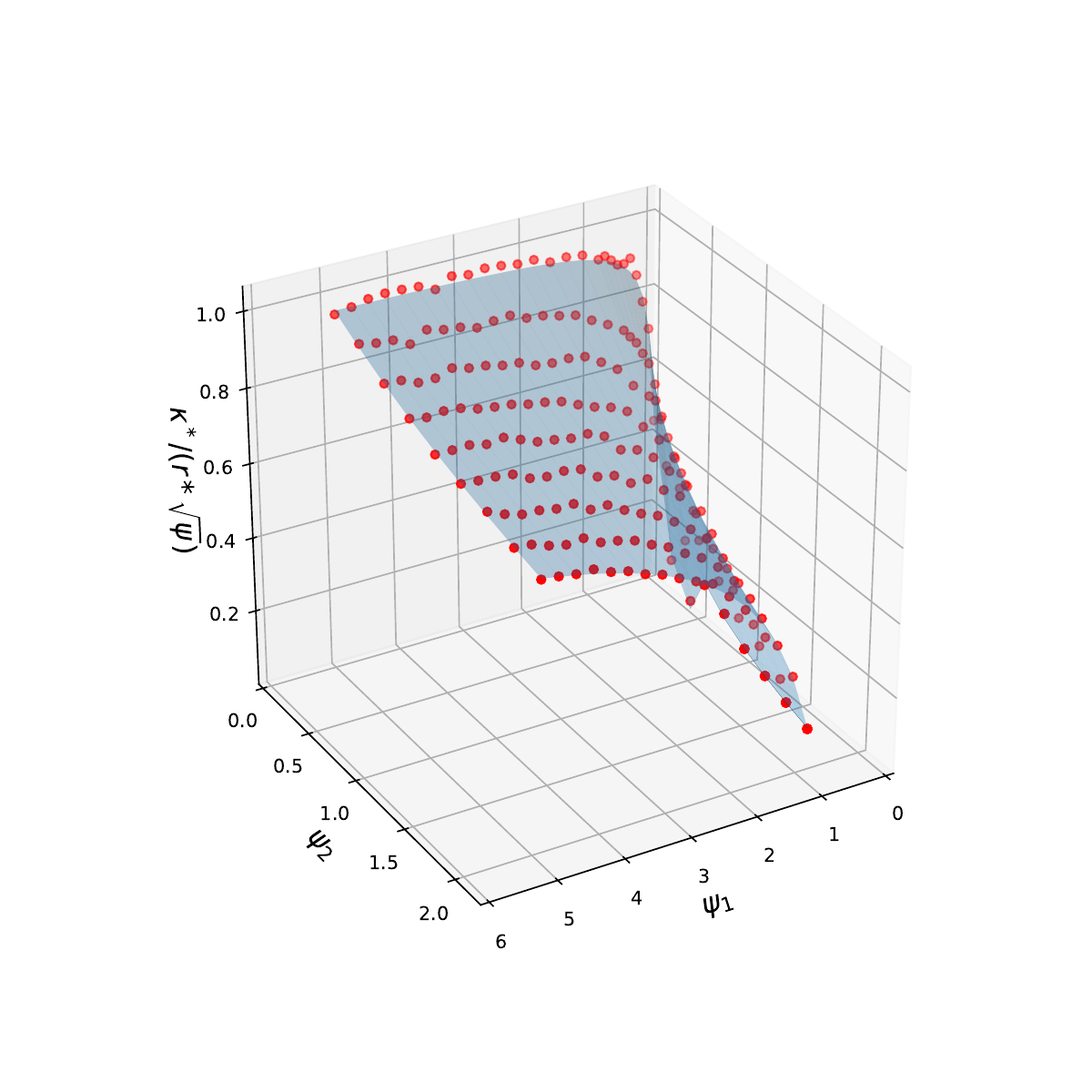}\hspace{-0.5cm}
\includegraphics[width=0.54\textwidth]{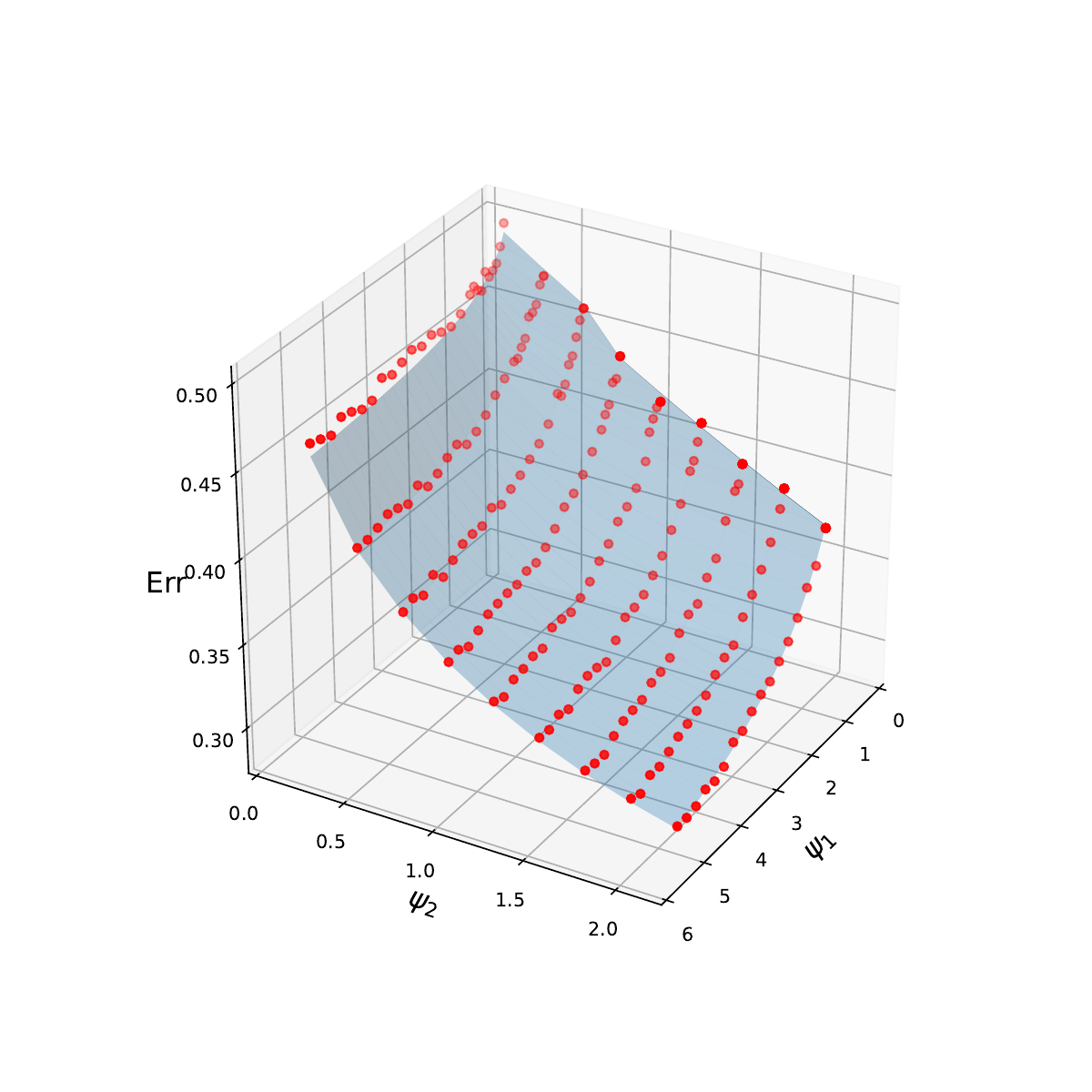}
\caption{Random features model, with ReLU activations. Recall the notation $\psi_1$ and $\psi_2$ from \eqref{def:psionepsitwo}.
Left: maximum margin. Right: test error. Labels are generated using the logistic function 
$f(x) = (1+e^{-\beta x})^{-1}$ with $\beta = 4$.
Here red circles stand for empirical values for $d = 400$, and results were averaged over $20$ samples.
Blue surfaces are the predicted values, both within the Gaussian covariates model of Section 
\ref{sec:GaussianModel}.}\label{fig:random-features-3d}
\end{figure}

\begin{figure}[t]
\phantom{A}\hspace{-0.5cm}
\includegraphics[width=0.54\textwidth]{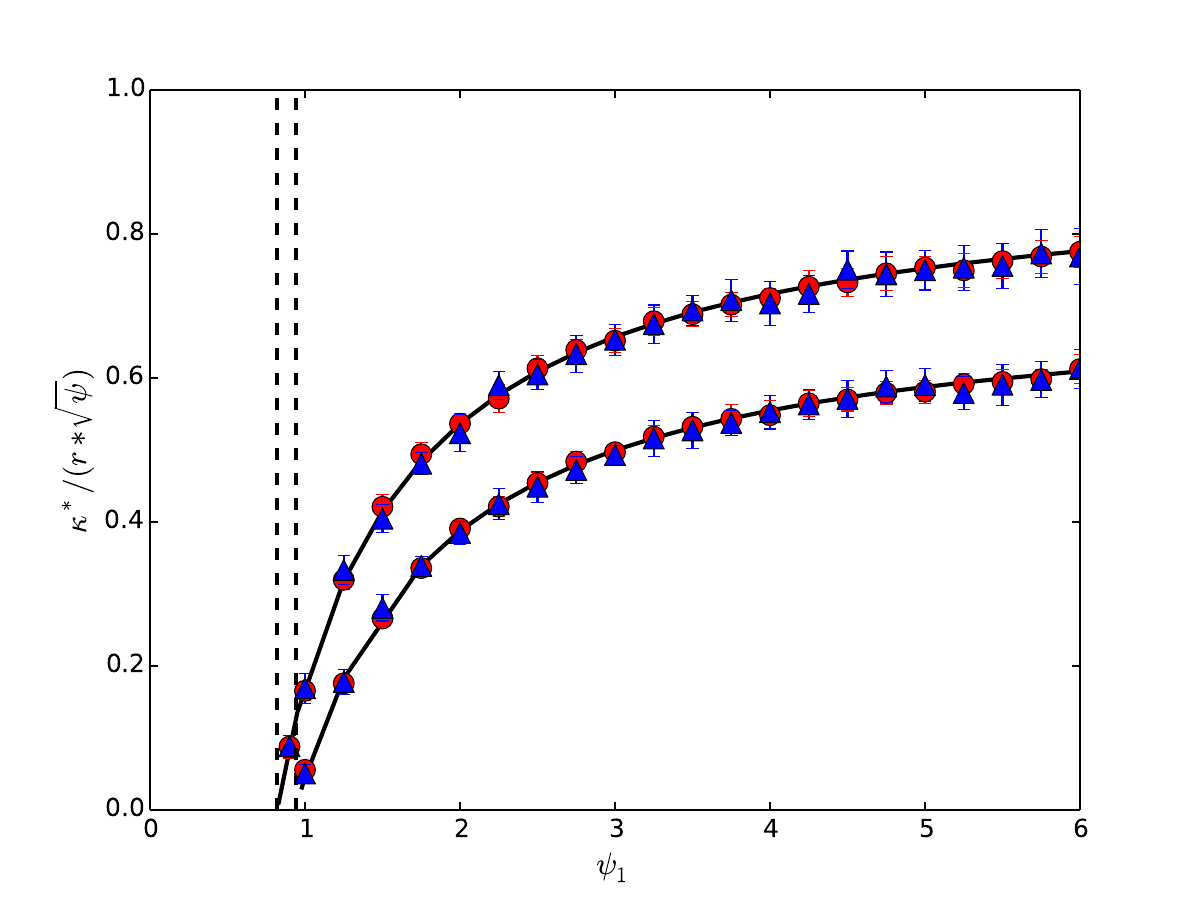}\hspace{-0.5cm}
\includegraphics[width=0.54\textwidth]{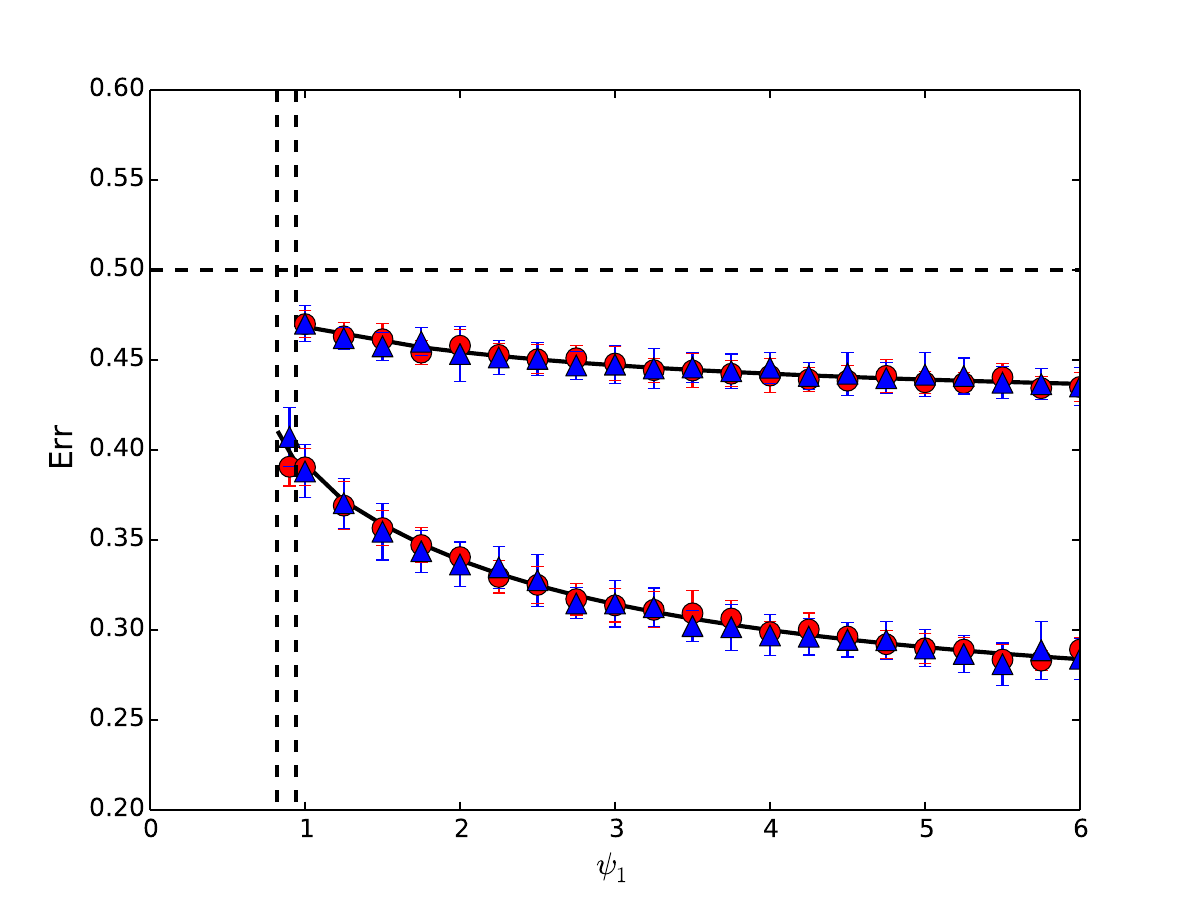}
\caption{Random features model: same setting as in Figure \ref{fig:random-features}, for a different activation function 
$\sigma_1(x)$ (see text). This activation is expected to have asymptotically the same maximum margin and test error as ReLU.}\label{fig:random-features-sigma2}
\end{figure}

In Figures \ref{fig:random-features}, \ref{fig:random-features-3d}  we report the results of 
numerical simulations within the random features model with ReLU
activations.
We compare the outcome of these simulation  with the analytical predictions 
of Theorem~\ref{thm:RF:Gaussian}.
 The agreement is excellent already at moderate values of $n,p,d$. 
 \begin{itemize}
 \item Vertical lines correspond to the analytical predictions for  
 the interpolation threshold $\psi\opt$. For $p/n\to \psi_1/\psi_2>\psi\opt$ 
 the data have (With high probability) a strictly positive margin.
 Indeed the margin appears to vanish linearly as $p/n$ approaches $\psi\opt$.
 \item The test error is monotonically decreasing with the overparametrization ratio 
$p/n = \psi_1/\psi_2$,
and its global minimum is achieved at large overparametrization $\psi_1/\psi_2\gg 1$.
\item  The margin is monotonically increasing in $\psi$ for  $\psi>\psi\opt$. 
\end{itemize}
At first sight, the last two observations might suggest that the decrease of 
test error can be explained by the increase of the margin using standard margin theory.
In order to understand whether this is the case, we consider for instance
 \cite[Theorem 26.14]{shalev2014understanding}, which implies, with our notations,
\begin{align}
\Pred_n(\by,\bX) \le \frac{4r\sqrt{\psi}}{\kappa\opt(\psi)} +o_n(1)\, ,\;\;\;\;\; r^2:=
 \int x\, \mu(\de x,\de w)\, . \label{eq:MarginBound}
\end{align}
Here $r$ is the typical  radius of the feature vectors, namely the asymptotic value of  
$r_n^2 = \E\|\bx_1\|^2/p$, (in the ReLU case, $r=1/\sqrt{2}$).
Even discarding the factor $4$ (which we do in Figure \ref{fig:IsoBound}), this upper bound has the wrong qualitative dependence
on $\psi$ and is never non-trivial in the present setting (never smaller than 1).
As it can be seen from the plots of the margin, this upper bound is 
is always larger than one in this application (even neglecting the factor $4$),
and therefore vacuous.
 
One particular prediction of our theory is that the test error and the margin should 
depend on the activation function only through the two coefficients $\gamma_1$
and $\gamma_*$. We check this numerically by repeating the same simulation  of Figure
 \ref{fig:random-features}, but using a different activation
function. Namely, we use activation $\sigma_2(x) = 
0.5 x_{+} + a_0 x_{+}^2 + a_1 x_{+} (1+x_{+})^{-1}$, where we choose $a_0$ and $a_1$ as to 
obtain the same values of $\gamma_1, \gamma_*$ as for ReLU. 
Figure \ref{fig:random-features-sigma2} reports the outcome of numerical simulations with 
activation $\sigma_2$.
As conjectured, the two sets of numerical results in Figures \ref{fig:random-features} 
and \ref{fig:random-features-sigma2} are hardly distinguishable.

\subsection{Wide network asymptotics}

\begin{figure}[t]
  \begin{center}
    \includegraphics[width=0.4\textwidth]{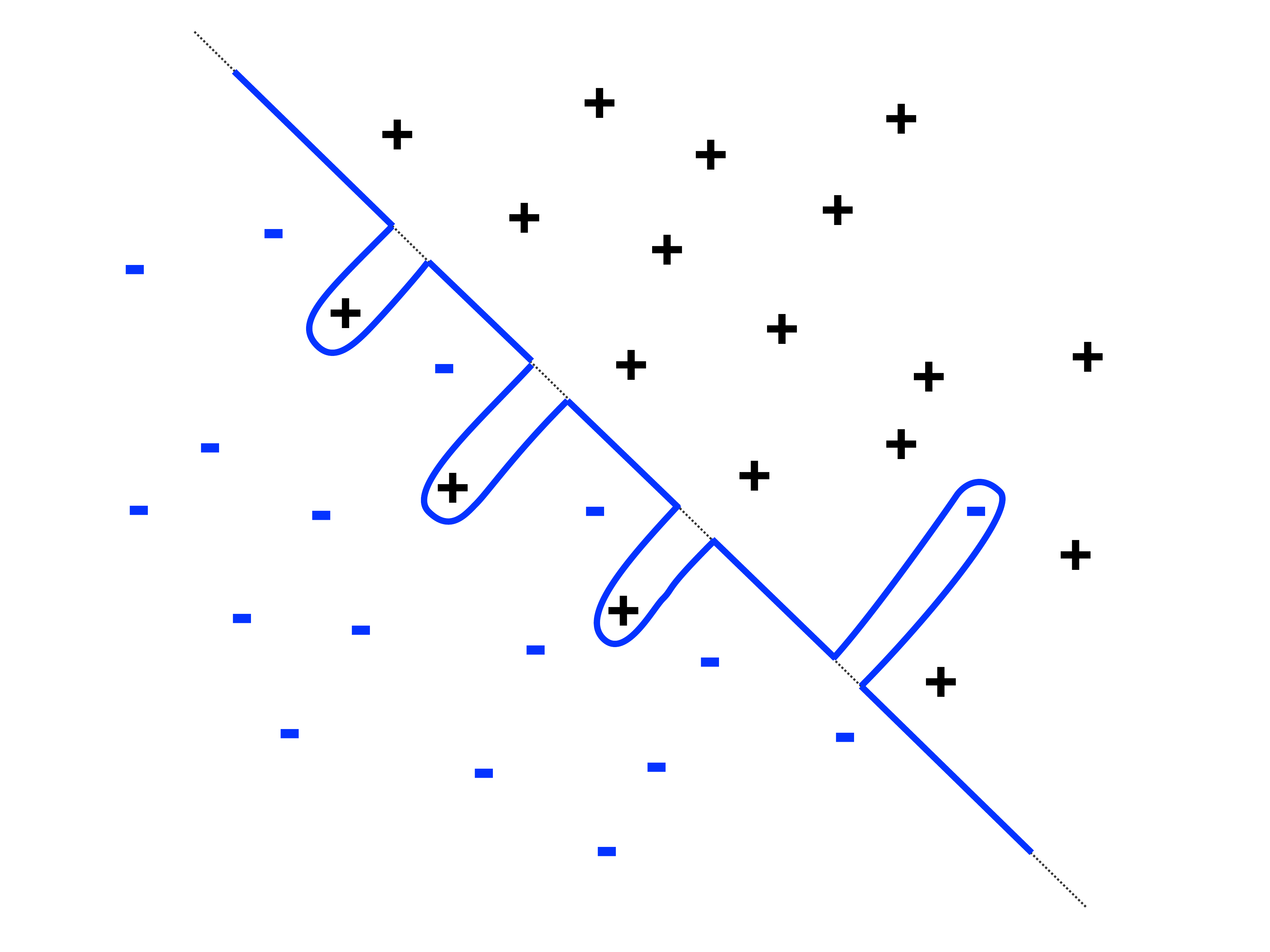}
    \end{center}
\caption{A cartoon of random features max margin classifiers. This is
  a pictorial representation of Proposition \ref{propo:Soft-Margin}.
  The random features classifier (continuous line) achieves
vanishing training error, but behaves as a linear classifier with a soft margin (dotted line).}\label{fig:cartoon-rf}
\end{figure}

Of  particular interest for the random features model is the wide network asymptotics
$\psi_1\to\infty$, at fixed $\psi_2$. This corresponds to a large number of neurons per dimension
 ($p/d$ large),
while the number of samples per dimension $n/d$ stays constant. It is important to bear
 in mind that these limits are taken \emph{after}
$p,n,d\to\infty$ with $p/d=\psi_1$, $n/d=\psi_2$: hence $p/d$ should be interpreted here as 
large  but of order one.

The next proposition characterizes this limit: its proof is deferred to 
Appendix \ref{sec:AsymptoticsRF}.
\begin{proposition}\label{prop:asymptoticRF}
 Let  $\kappa\opt(\mu_{\psi_1},\psi)$, $\Pred\opt(\mu_{\psi_1},\psi)$ be the asymptotic 
 maximum margin and classification error of the random features model
 defined in Section \ref{sec:RFDef}, cf. Theorem \ref{thm:RF:Gaussian}.
  For $\psi_2,\gamma_1,\gamma_{*}>0$, define $T_{\infty}(\,\cdot\,; \psi_2,\gamma_1,\gamma_{*}):\R_{>0}\to \R$ by 
    \begin{equation}
    \label{eqn:def:T-infty:0}
        T_{\infty}(\kbar; \psi_2,\gamma_1,\gamma_{*})= \min_{\substack{d_1^2+d_2^2 \leq 1,\\d_2\geq 0}} \bigg\{F_{\kbar}(\sqrt{\psi_2}\gamma_1 d_1,\sqrt{\psi_2}\gamma_1 d_2)-\gamma_1 d_2-\gamma_{*}\sqrt{1-d_1^2-d_2^2}\bigg\}\, .
      \end{equation}
      where $F_{\kappa}$ is defined in Eq.~\eqref{eqn:def-F-kappa} below.
      Denote the unique minimizer of this optimization  problem by $(d_1\opt,d_2\opt)=(d_1\opt(\kbar),d_2\opt(\kbar))$,
      and let
       \begin{align}
         \label{eq:kwideDef}
       \kbar\owid(\psi_2,\gamma_1,\gamma_{*}) \equiv \inf\big\{\kbar>0:\;\;T_{\infty}(\kbar; \psi_2,\gamma_1,\gamma_{*})=0\big\}\, .
       \end{align}
       Finally define $d_i\owid=d_i\opt(\kbar\owid)$ and
       \begin{align}\label{eq:PredErrorWide}
         &\Pred\owid(\psi_2,\gamma_1,\gamma_{*})\equiv \P\left(\nu\owid Y_0G+\sqrt{1-(\nu\owid)^{2}}Z \leq 0\right)\,,\\
         &\nu\owid=\nu\owid(\psi_2,\gamma_1,\gamma_{*})=\frac{d_1\owid}{\sqrt{(d_1\owid)^2+(d_2\owid)^2}},~~\text{where}~~d_i\owid=d_i\opt(\kbar\owid)~~\text{for}~~i=1,2\, ,
       \end{align}
       where expectation is with respect to $Z$ independent of $(Y_0, G)$, $\P(Y_0 = +1 \mid G) = h(G)$, $\P(Y_0 = -1 \mid G) = 1-h(G)$, $Z , G\iid \normal(0, 1)$.
        
       Then
       \begin{align}
         \lim_{\psi_1\to\infty}\frac{\kappa\opt(\mu_{\psi_1},\psi))}{\sqrt{\psi}} &= \kbar\owid(\psi_2,\gamma_1,\gamma_{*})\, ,
         \label{eq:Kwide}\\
         \lim_{\psi_1\to\infty}\Pred\opt(\mu_{\psi_1},\psi) & = \Pred\owid(\psi_2,\gamma_1,\gamma_{*})\, .  \label{eq:Predwide}
         \end{align}  
       \end{proposition}
        It turns out that Proposition
       \ref{prop:asymptoticRF} has a remarkably simple interpretation. In the wide, high-dimensional limit,
       the random features model behaves as a simple linear model in the covariates $\bz_i$, whereby instead of the
       maximum margin classifier of Eq.~\eqref{eq:MMdef}, we solve the following soft margin problem
       \begin{align}
         (\hbtheta^{\sSM},\hbu^{\sSM})\in\argmax_{\btheta\in\reals^d,\, \bu\in\reals^n}\left\{\min_{i\le n}
         \big[\gamma_1y_i\<\btheta,\bz_i\>+\gamma_*u_i\big]\; :\;\;\;
         \|\btheta\|_2^2+\frac{\|\bu\|_2^2}{d}= 1
         \right\}\, .\label{eq:SMClassifier}
       \end{align}
       We denote by $\kappa_n^{\sSM}(\by,\bZ)$ the corresponding soft margin, namely the
       value of this optimization problem. (Here $\bZ\in\reals^{n\times d}$ is the matrix whose $i$-th row is $\bz_i$.)
       \begin{proposition}\label{propo:Soft-Margin}
         Let $\kbar\owid(\psi_2,\gamma_1,\gamma_{*})$,
         $\Pred\owid(\psi_2,\gamma_1,\gamma_{*})$ be the asymptotic margin and classification
         error of the noisy linear features model,
         as defined in Proposition \ref{prop:asymptoticRF}, Eqs.~\eqref{eq:kwideDef},
         \eqref{eq:PredErrorWide}. Further, denote by $\kappa_n^{\sSM}(\by,\bZ)$
         the soft-margin, that is the optimum value of problem \eqref{eq:SMClassifier}, and 
         $\Pred_n^{\sSM}(\by,\bZ)$ the corresponding classification error. Consider the proportional
         asymptotics $n,d\to\infty$ with $n/d\to\psi_2\in(0,\infty)$. Then
         \begin{align}
           \lim_{n\to\infty}\kappa_n^{\sSM}(\by,\bZ) =\frac{\kbar\owid(\psi_2,\gamma_1,\gamma_{*})}{\sqrt{\psi_2}}\,,
           \;\;\;\;\;\;\;\;\;
               \lim_{n\to\infty}\Pred_n^{\sSM}(\by,\bZ) =\Pred\owid(\psi_2,\gamma_1,\gamma_{*})\, .
         \end{align}
       \end{proposition}
       The proof of Proposition \ref{propo:Soft-Margin} is deferred to Appendix \ref{sec:appendix:soft:margin}.
         \begin{remark}
           It is worth emphasizing that Proposition \ref{propo:Soft-Margin} establishes the equivalence of two
           classifiers that ---at first sight--- are very different. The first one is the original random features
           max-margin classifier which is linear in the $p$-dimensional lifted space, but non-linear in the underlying $d$-dimensional space, and has vanishing training error. The second one is the soft-margin classifier of
           Eq.~\eqref{eq:SMClassifier}. This is a linear classifier in $d$ dimensions.
  A pictorial representation of this result is given in Figure \ref{fig:cartoon-rf}.
  \end{remark}
      
\section{Further related work}
\label{sec:Related}

\paragraph{`Noisy' high-dimensional statistics.} Classical high-dimensional statistics 
\cite{buhlmann2011statistics}  studies the regime 
$n\lesssim p$  but under the assumption that the parameters' vector is highly structured.
For instance,  the sparsity $s_0$ is assumed satisfy $s_0\ll n/\log p$.
Concentration of measure is sufficient to prove consistency in such highly structured problems.

The present work contributes to a growing body of work focuses on a different `noisy high-dimensional regime'  
In this setting, the sample size is proportional to the number of parameters, and the estimation error 
(suitably rescaled) converges to a non-trivial limit \cite{montanari2018mean}. 
Asymptotically exact results have been obtained in a large array of problems
including sparse regression using $\ell_1$ penalization (Lasso) \cite{BayatiMontanariLASSO,miolane2018distribution}, 
general regularized linear regression \cite{donoho2013accurate,amelunxen2014living,ThrampoulidisOyHa15},
robust regression \cite{el2013robust,donoho2016high,el2018impact,thrampoulidis2018precise}, Bayesian estimation within generalized linear models
\cite{barbier2019optimal}, logistic regression \cite{candes2018phase,sur2019modern}, low rank matrix estimation 
\cite{deshpande2016asymptotic,lelarge2019fundamental,barbier2018rank}, and so on.
Several new mathematical techniques have been developed to address this regime: constructive methods based on message passing algorithms 
\cite{BM-MPCS-2011,bayati2015universality}; Gaussian comparison methods based on Gordon's inequality \cite{Gordon88,ThrampoulidisOyHa15};
interpolation techniques motivated from statistical physics \cite{barbier2019adaptive}. 

The study of this `noisy high-dimensional' regime has a long history in statistical physics. Non-rigorous methods 
from spin glass theory have been successfully used since the eighties in this context. 
We refer to \cite{engel2001statistical} for an overview of this early work, and to \cite{SpinGlass,MezardMontanari} for general introductions.
An early breakthrough in this line of work was the result by Elizabeth Gardner \cite{gardner1988space}, who computed the maximum margin 
$\kappa\opt(\psi)$ for the special case of isotropic features and purely random labels (i.e. $\bSigma_n = \id_p$ and $f(x) = 1/2$).

Here we follow the approach based on  Gordon's inequality \cite{Gordon88}, as formalized by Thrampoulidis, Oymak, Hassibi \cite{ThrampoulidisOyHa15}. Gordon's inequality was previously used to study various statistical learning problems such as compressed sensing \cite{stojnic2010l1, ChandrasekaranRecht} and Lasso \cite{stojnic2013lasso,oymak2013lasso}.
The closest results to ours in the earlier literature are the analysis of logistic regression by Sur and Cand\'es  \cite{candes2018phase,sur2019modern}, and the 
recent paper on regularized logistic regression by Salehi, Abbasi and Hassibi \cite{salehi2019impact}. Both of these analyses focus on the 
underparametrized regime, in which the maximum likelihood estimator is well defined and (with high probability)  unique. By contrast,
we focus on the overparametrized regime here. Further, while earlier work assumes isotropic covariates $\bx_i\sim\normal(0,\id_p)$,
we consider a general covariance structure, under Assumptions \ref{assumption:Lambdas} and \ref{assumption:converge}. This is crucial in order to be
able to capture the behavior of overparametrized random features models. From a technical point of view, our approach is related to the
one of  \cite{salehi2019impact}. Notice however that  \cite{salehi2019impact} does not prove uniqueness of the minimizer of Gordon's optimization
problem, while this is the main technical challenge that we address in our proof (in a more complicated setting, due to the general covariance).

\paragraph{Overparametrization and overfitting.} As discussed in the introduction, our work is
 connected to a substantial line of research that investigates the behavior of 
generalization error in overparametrized model that interpolate the data. 
This work was largely motivated
by the empirical observation that deep neural networks fit perfectly the data and yet generalize well 
\cite{zhang2016understanding,neyshabur2017exploring}.
It was noticed in \cite{belkin2018understand} that this behavior is significantly more general 
than neural networks, while \cite{belkin2019reconciling} pointed out 
that the classical U-shaped curve describing the behavior of generalization error 
as a function of number of parameters does not hold in general.
 Independently, \cite{geiger2019scaling} observed the 
same phenomenon in the context of multilayer networks, and connected it to phase transitions in physics. 

Mathematical results about generalization behavior in the overparametrized  regime have been obtained in several recent papers
\cite{belkin2018overfitting,belkin2019two,liang2020just,rakhlin2018consistency,muthukumar2019harmless,hastie2022surprises,bartlett2020benign,mei2019generalization}.
However, all of earlier work has focused on least squares (or ridge) regression with square loss. (Certain nearest-neighbor-like methods are also considered in
\cite{belkin2018overfitting}.)  The closest earlier
results in this literature are  \cite{hastie2022surprises,mei2019generalization} which 
use random matrix theory to characterize ridge regression in the proportional 
asymptotics $p,n\to\infty$ with $p/n=\psi$.
Moving beyond ridge regression requires abandoning the powerful tools of
 random matrix  theory and developing new mathematical tools.

\paragraph{Threshold for linear separability.} As a corollary of our theory,
we characterize a the threshold for linear separability $\psi\opt=\psi\opt(f)$.
For $p/n\to \psi >\psi\opt$ data are with high probability separable (with a margin
bounded away from $0$), while for $p/n\to \psi <\psi\opt$ they are not.
A classical result of Cover 
\cite{cover1965geometrical} yields $\psi\opt=1/2$ for the special case in which $y_i\sim\Unif(\{+1,-1\})$ independently of $\bx_i$,
provided the $(\bx_i)_{i\le n}$ are in generic positions.  This result was recently generalized by  
Cand\'es and Sur \cite{candes2018phase} for the more challenging setting in which
 $\P(y_i=+1|\bx_i)= (1+e^{-\<\btheta_*,\bx_i\>})^{-1}$ and $\bx_i$ is Gaussian. 
Our results yield a  generalization of the threshold obtained in \cite{candes2018phase},
but also characterize the margin in the overparametrized regime $\psi>\psi\opt$.
(The latter is significantly more challenging because the margin depends on all the entries of $(\bSigma,\btheta_*)$
while the separability does not.)

\paragraph{Independent and follow-up work.} The special case of
isotropic covariates $\bSigma = \id_p$ was treated independently from our work
in \cite{deng2022model}.

After the first version of this paper was posted online, 
our results were generalized in a number of significant ways by several authors.
 A few pointers to this literature are
\cite{goldt2020modelling,kini2020analytic,taheri2020fundamental,gerace2020generalisation,kini2021label,liang2022precise,javanmard2022precise}
(we limit ourselves to papers that derive sharp asymptotics in the proportional
regime).
The most closely related work is \cite{liang2022precise} that adapts the techniques of the present 
paper to analyze max $\ell_1$-margin (rather than max $\ell_2$ as we do here).

Non-asymptotic bounds on 
max-margin classification were proven in \cite{chatterji2021finite},
implying benign overfitting in certain settings in certain settings. 
Namely, the data distribution is a mixture $(\prob_+ +\prob_-)/2$
where $\prob_+(\bx_i\in\,\cdot\,)$ and 
$\prob_-(\bx_i\in\,\cdot\, )$ are assumed to be well separated, and
labels noise is independent of the features:
 $\prob_+(y_i=-1|\bx_i)= \prob_-(y_i=+1|\bx_i)= \eta$ independent of $\bx_i$.
This analysis was generalized to two-layer ReLU networks in 
\cite{frei2022benign,frei2023benign}.

Finally, the recent paper \cite{zhou2022non} (building on the earlier work \cite{koehler2021uniform})
obtains non-asymptotic bound on population loss in generalized linear models.
Their approach also uses Gaussian comparison inequalities as ours and imply 
certain benign overfitting guarantees. 

While limited to proportional asymptotics, our work is the first one characterizing 
covariances and parameters' sequences  $(\bSigma,\btheta_*)$ under which max-margin classification
displays benign overfitting in classification accuracy.

\section{Main results}
\label{sec:MainResults} 

Our main technical theorem characterizes the asymptotic value of the maximum margin $\kappa\opt(\mu,\psi)$ and the asymptotic generalization error 
of the max-margin classifier, to be denoted by $\error\opt(\mu,\psi)$. 

\subsection{Introducing the asymptotic predictions}

We start by defining our general analytical predictions  $\kappa\opt(\mu,\psi)$ and $\error\opt(\mu,\psi)$. Recall that $\rho\in\reals_{>0}$
and the probability measure $\mu$ on $\reals_{>0}\times\reals$ are
defined by Assumption \ref{assumption:converge}.
For any $\kappa\ge 0$, 
define $F_{\kappa}: \R \times \R_{\ge 0} \to \R_{\ge 0}$ by
\begin{equation}
\label{eqn:def-F-kappa}
F_{\kappa}(c_1, c_2) = \left(\E \left[(\kappa - c_1 YG - c_2 Z)_+^2\right]\right)^{1/2}
~~\text{where}~
	\begin{cases}
		Z \perp (Y, G) \\
		Z \sim \normal(0, 1), G \sim \normal(0, 1) \\
		\P(Y = +1 \mid G) = f(\rho \cdot G) \\
		\P(Y = -1 \mid G) = 1-f(\rho \cdot G)
	\end{cases}
\end{equation}
Let the random variables $X,W$ be such that $(X, W) \sim \mu$. Introduce the constants 
\begin{equation}
\zeta = \left(\E_{\mu} [X^{-1}W^2]\right)^{-1/2}~\text{and}~\omega = \left(\E_{\mu} [(1-\zeta^2 X^{-1})^2W^2]\right)^{1/2}. %\label{eq:ZetaDef}
\end{equation}
Define the functions $\psi_+: \R_{>0} \to \R$~~\text{and}~$\psi_-: \R_{>0} \to \R$ by
\begin{equation}
\begin{split}
\psi_+(\kappa) &= 
	\begin{cases}
		0~~&\text{if $\partial_1 F_{\kappa}(\zeta, 0) > 0$,} \\
			\partial_2^2 F_{\kappa}(\zeta, 0) - \omega^2 \partial_1^2 F_{\kappa}(\zeta, 0)~~~~~~~&\text{if otherwise,}
	\end{cases}
	\\ 
\psi_-(\kappa) &= 
	\begin{cases}
		0~~&\text{if $\partial_1 F_{\kappa}(-\zeta, 0) > 0$,} \\
			\partial_2^2 F_{\kappa}(-\zeta, 0) - \omega^2 \partial_1^2 F_{\kappa}(-\zeta, 0)~~&\text{if otherwise.}
	\end{cases}
\end{split}
\end{equation}
Finally, we define $\psi\opt(0)$ and $\psi^{\low}:\R_{>0}\to \R_{\ge 0}$ by 
\begin{align}
	\psi\opt(0) &= \min_{c\in \R} F_0^2 (c, 1)\, ,\label{eq:Psistar0}\\
\psi^{\low}(\kappa) &= \max\{\psi\opt(0), \psi_+(\kappa), \psi_-(\kappa)\}.\label{eq:PsilowDef}
\end{align}

The next proposition guarantees that the definition of $\kappa\opt(\mu, \psi)$, and $\error\opt(\mu, \psi)$ given below are meaningful.
Its proof is deferred to Appendix~\ref{sec:all-property-optimization}. In the following, we will often omit the argument $\mu$
from $\kappa\opt$, $\error\opt$, $\LL\opt$.
\begin{proposition}
\label{proposition:system-of-Eq-T}
\begin{enumerate}
\item[$(a)$]  For any $\psi > \psi^{\low}(\kappa)$, the following system of equations has unique solution $(c_1, c_2, s) \in \R \times \R_{>0} \times \R_{>0}$
(here expectation is taken with respect to $(X,W)\sim \mu$):
	\begin{equation}
	\label{eqn:system-of-equations-c-1-c-2-s-with-no-G}
		\begin{split}
		-c_1 &= \E_{\mu} \left[\frac{ \left(\partial_1 F_{\kappa}(c_1, c_2) -  
			c_1 c_2^{-1} \partial_2 F_{\kappa}(c_1, c_2)\right) W^2 X^{1/2}}
				{ c_2^{-1} \partial_2 F_{\kappa}(c_1, c_2) X^{1/2} + \psi^{1/2} s X^{-1/2}}\right] \, , \\
		c_1^2 + c_2^2 &=  \E_{\mu} \left[\frac{\psi X + 
			 \left(\partial_1 F_{\kappa}(c_1, c_2) -  c_1 c_2^{-1} \partial_2 F_{\kappa}(c_1, c_2)\right)^2 W^2 X}
				{( c_2^{-1} \partial_2 F_{\kappa}(c_1, c_2) X^{1/2} + \psi^{1/2} s X^{-1/2})^2}\right] \, ,\\
		1 &= \E_{\mu}
			\left[\frac{\psi + 
			 \left(\partial_1 F_{\kappa}(c_1, c_2) -  c_1 c_2^{-1} \partial_2 F_{\kappa}(c_1, c_2)\right)^2 W^2}
				{( c_2^{-1} \partial_2 F_{\kappa}(c_1, c_2) X^{1/2} + \psi^{1/2} s X^{-1/2})^2}\right]\, .
		\end{split}
	\end{equation}
\item[$(b)$] Define the function $T: (\psi, \kappa) \to \R$ (for any $\psi > \psi^{\low}(\kappa)$) by
	\begin{equation}
	\label{eqn:def-T}
		T(\psi, \kappa) =  \psi^{-1/2} \left(F_{\kappa}(c_1, c_2) - c_1 \partial_1 F_{\kappa}(c_1, c_2) -	
		c_2 \partial_2 F_{\kappa}(c_1, c_2) \right) - s\, ,
	\end{equation}
	where $(c_1(\psi, \kappa), c_2(\psi, \kappa), s(\psi, \kappa))$ is the unique solution of Eq
	\eqref{eqn:system-of-equations-c-1-c-2-s-with-no-G} in $\R \times \R_{>0} \times \R_{>0}$. 
	Then we have 
	\begin{enumerate}
	\item[$(i)$] $T(\cdot, \cdot), c_1(\cdot, \cdot), c_2(\cdot, \cdot), s(\cdot, \cdot)$ are continuous functions in the domain 
		$\left\{(\psi, \kappa): \psi > \psi^{\low}(\kappa)\right\}$. 
	\item[$(ii)$] For any $\kappa > 0$, the mapping $T(\, \cdot\, , \kappa)$ is strictly monotonically decreasing, and satisfies
	\begin{equation}
		\lim_{\psi \nearrow +\infty} T(\psi, \kappa) < 0 < \lim_{\psi \searrow \psi^{\low}(\kappa)} T(\psi, \kappa). \label{eq:Tpsi}
	\end{equation}
	\item[$(iii)$] For any $\psi > 0$, the mapping $T(\psi, \cdot)$ is strictly monotonically increasing, and satisfies
	\begin{equation}
		\lim_{\kappa \nearrow +\infty} T(\psi, \kappa) = \infty. 
	\end{equation}
	\end{enumerate}
\end{enumerate}
\end{proposition}

We are now in position  to define $\kappa\opt(\mu,\psi)$, $\error\opt(\mu,\psi)$ and $\mathcal{L}\opt(\mu, \psi)$. 
\begin{definition}\label{def:KappaE}
Recall the function $T$ defined at Eq.~\eqref{eqn:def-T}. For any $\psi\ge \psi\opt(0)$, we define the \emph{asymptotic max-margin}
as 
\begin{equation*}
\kappa\opt(\mu,\psi) = \inf\left\{\kappa \ge 0: T(\psi, \kappa) = 0\right\}. 
\end{equation*}
We further define the \emph{asymptotic generalization error} $\error\opt: (0, \infty) \to [0, 1]$ by
\begin{align}
\Pred\opt(\mu,\psi) &=  \P\left(\nu\opt(\psi) YG+ \sqrt{1-\nu\opt(\psi)^2} Z \le 0\right)\, ,\label{eq:AsyErrorDef}\\
  \nu\opt(\psi) & \equiv \frac{c_1\opt(\psi)}{\sqrt{(c_1\opt(\psi))^2 + (c_2\opt(\psi))^2}}\, ,
\end{align} 
where probability is over  $Z \perp (Y, G)$, with $G,Z \sim \normal(0, 1)$ and $
\P(Y= +1\mid G) = f(\rho \cdot G) =1-\P(Y= -1\mid G)$. Further $c_i\opt(\psi) = c_i(\psi, \kappa\opt(\psi))$, $i\in\{1,2\}$.
Lastly, for each $\psi >  \psi^{\low}(\kappa)$, we introduce the
random variable 
\begin{equation}
H_{\psi, \kappa}(G, X, W) = 
	-\frac{\psi^{1/2}G + \left(\partial_1 F_{\kappa}(c_1, c_2) -  c_1 c_2^{-1} \partial_2 F_{\kappa}(c_1, c_2)\right) W}
		{ c_2^{-1} \partial_2 F_{\kappa}(c_1, c_2) X^{1/2} + \psi^{1/2} s X^{-1/2}}.
\end{equation}
where $(c_1, c_2, s)$ is defined as in Proposition~\ref{proposition:system-of-Eq-T}. We use $\mathcal{L}_{\psi, \kappa}$
to denote the distribution of the random variable $(X, W, H_{\psi,
  \kappa}(G, X, W))$ when $(G, X, W) \sim \normal(0, 1) \otimes \mu$.
Define
%the 
%\emph{asymptotics} of the empirical distribution of $\{(\lambda_i, \bar{w}_i, \sqrt{p}\langle \hat{\btheta}_n^{{\sMM}}, \bv_i\rangle)\}_{i \in [p]}$ 
\begin{equation*}
\mathcal{L}\opt(\mu, \psi) = \mathcal{L}_{\psi, \kappa\opt(\psi)}.
\end{equation*} 
\end{definition}
Proposition~\ref{proposition:system-of-Eq-T} shows that the mapping $\psi \to \kappa\opt(\mu,\psi)$ is well-defined, 
strictly monotonically increasing, and satisfies $\lim_{\psi \to \infty} \kappa\opt(\mu,\psi) = \infty$. 
\subsection{Main statement}

Below we present the main mathematical result of this paper. Section \ref{sec:ProofMain} presents the
proof of this theorem, with most technical legwork deferred to the appendices.
% The next subsection briefly describe our proof technique.
%
\begin{theorem}\label{theorem:main}
Consider i.i.d. data $(\by,\bX)=\{(y_i,\bx_i)\}_{i\le n}$ where $\bx_i\sim\normal(\bzero,\bSigma_n)$ and $ \P\big( y_i = +1\big|\bx_i\big) =f(\<\btheta_{*,n},\bx_i\>)$.
Assume $n,p\to\infty$ with $p/n\to \psi\in(0,\infty)$, and  
$\btheta_{*,n},\bSigma_n,f$ satisfying Assumptions
\ref{assumption:Lambdas}, \ref{assumption:converge},
\ref{assumption:non-degenerate-f}. (In particular, $\rho, \mu$ are defined by Assumption \ref{assumption:converge}.)

Let $\psi\opt(0)$ be defined as per  Eq.~(\ref{eq:Psistar0}), and $\kappa\opt(\mu,\psi)$, $\error\opt(\mu,\psi)$ be determined as per
Definition \ref{def:KappaE}.
Then the following hold:
\begin{enumerate}
\item[$(a)$] With probability tending to one, the data are linearly separable if $\psi>\psi\opt(0)$ 
and are not linearly separable if $\psi < \psi\opt(0)$. 
 %with a margin which is bounded away from $0$
%if and only if $\psi>\psi\opt(0)$.
%
\item[$(b)$]  Let $\kappa_n(\by,\bX) \equiv\max_{\|\btheta\|_2 =1}\min_{i\le n}y_i\<\btheta,\bx_i\>$ be the maximum margin for data $(\by,\bX)$.
In the overparametrized regime $\psi>\psi\opt(0)$ we have, as $n \to \infty$,
\begin{align}\label{eq:maxmargin:converge}
\kappa_n(\by,\bX) \pto \kappa\opt(\mu,\psi)\, .
\end{align}
\item[$(c)$] Let $\Pred_n(\by,\bX) \equiv \P\big(y^{\snew}\<\hbtheta^{\sMM}(\by,\bX),\bx^{\snew}\> \le 0\big)$ be the prediction error of the maximum margin classifier.
In the overparametrized regime $\psi>\psi\opt(0)$ we have,  as $n \to \infty$,
\begin{align}\label{eq:error:converge}
\Pred_n(\by,\bX) \pto \Pred\opt(\mu,\psi)\, .
\end{align}
\item[$(d)$] Recall that $\{\lambda_i\}_{i \in [p]}$, $\{\bv_i\}_{i \in [p]}$ are the eigenvalues and eigenvectors of $\bSigma_n$ and 
$\bar{w}_i = \sqrt{p\lambda_i} \<\bv_i,\btheta_{*,n}\>/\rho_n$  for $\rho_n = \<\btheta_{*,n},\bSigma_n\btheta_{*,n}\>^{1/2}$ (see
Section~\ref{sec:Overview}).
Let $\hat{\LL}_n(\by, \bX)$ denote the empirical distribution induced by 
$\{(\lambda_i, \bar{w}_i, \sqrt{p}\langle \hat{\btheta}_n^{{\sMM}}, \bv_i\rangle)\}_{i \in [p]}$, i.e., 
$
\hat{\LL}_n(\by, \bX) = \frac{1}{p} \sum_{i=1}^p \delta_{(\lambda_i, \bar{w}_i, \sqrt{p}\langle \hat{\btheta}_n^{{\sMM}}, \bv_i\rangle)}.
$
In the overparametrized regime $\psi>\psi\opt(0)$ we have, as $n \to \infty$,
	\begin{equation}\label{eq:coord:dist:converge}
		W_2\left(\hat{\LL}_n(\by, \bX), \LL\opt(\mu, \psi)\right) \pto 0.
	\end{equation}

\end{enumerate}
\end{theorem}

\begin{remark}
Point $(a)$ in Theorem \ref{theorem:main} is a generalization of the recent result of \cite{candes2018phase}, which concerns the
case in which $f(x)$ is a logistic function. 

The main content of Theorem \ref{theorem:main} is in parts $(b)$, $(c)$ and $(d)$. To the best of our knowledge, the only case that had been characterized before
is the one of isotropic covariates and purely random labels (i.e.  $\bSigma_n = \id_{p(n)}$ and $f(x) = 1/2$). In this case the
asymptotic value of the maximum margin was first determined rigorously by  Shcherbina and Tirozzi \cite{shcherbina2003rigorous}, confirming
the non-rigorous result by Gardner \cite{gardner1988space}.
\end{remark}

\subsection{Proof technique}
\noindent{\bf Parts $(a)$, $(b)$.} Consider first the problem of determining the asymptotics of the maximum margin. Recall that $\bX\in\reals^{n\times p}$
denotes the matrix with rows $\bx_1,\dots, \bx_n$ and,
for any $\kappa > 0$, define the event  
\begin{equation*}
\event_{n, \psi, \margin} = \left\{ \exists \btheta \in \R^{p}, ~\ltwo{\btheta} \le 1, 
	\text{such that}~y_i \langle \bx_i, \btheta \rangle \ge \kappa~\text{for $i \in [n]$} \right\}.
\end{equation*}
In order to prove Theorem \ref{theorem:main}.$(b)$, we would like to determine for which pairs $(\psi,\kappa)$
we have $\P(\event_{n, \psi, \margin})\to 1$ and for which pairs instead $\P(\event_{n, \psi, \margin})\to 0$.

To this end, we define $\xi_{n, \psi, \kappa}$ by 
\begin{equation}
\label{eq:XiDef}
\xi_{n, \psi, \kappa} = \min_{\ltwo{\btheta} \le 1}~
		\max_{\ltwo{\blambda} \le 1, \by \odot \blambda \ge 0} \frac{1}{\sqrt{p}}\blambda^{\sT} (\kappa \by - \bX \btheta)\, .
\end{equation}
We then have:
\begin{equation}
\label{eqn:equiv-xi-event-want}
		\left\{\xi_{n, \psi, \kappa} > 0 \right\} \iff \event^c_{n, \psi, \margin}
			~~\text{and}~~
		\left\{\xi_{n, \psi, \kappa} = 0 \right\} \iff  \event_{n, \psi, \margin}\, .
\end{equation}
This equivalence follows immediately from the following identities
\begin{align*}
	\event_{n, \psi, \margin} &= \left\{ \exists \btheta \in \R^{p}, 
			\ltwo{\btheta} \le 1~\text{such that } \ltwo{(\kappa\one- (\by \odot \bX \btheta))_+}=0 \right\}\, ,\\
		\xi_{n, \psi, \kappa} &= \min_{\ltwo{\btheta} \le 1}~
			\frac{1}{\sqrt{p}}\ltwo{(\kappa\one- (\by \odot \bX \btheta))_+}\, .
\end{align*}
We are then reduced to study the typical value of the minimax problem \eqref{eq:XiDef}.
Notice that this problem is convex in $\btheta$, concave in $\blambda$, and linear in the Gaussian random matrix $\bX$.
We use Gordon's Gaussian comparison inequality \cite{Gordon88} (and in particular a refinement due to 
Thrampoulidis, Oymak, Hassibi \cite{ThrampoulidisOyHa15}) to study the asymptotics of 	$\xi_{n, \psi, \kappa}$.

The result on data separability (Theorem \ref{theorem:main}.$(a)$) essentially follows from the analysis 
of the maximum margin. First, a direct implication of Theorem \ref{theorem:main}.$(b)$ is the separability 
of data (with high probability) if $\psi > \psi^*(0)$. To show the other way 
around, we consider instead $\Xi_{n, \psi}$, whose definition replaces the constraint $\ltwo{\btheta} \le 1$
 in equation~\eqref{eq:XiDef} by   $\norm{\btheta}_{\bSigma_n}:=\langle \btheta, \bSigma_n \btheta \rangle^{1/2} = 1$ and substitutes $\kappa$ by $0$: 
\begin{equation}
\label{eqn:def-XI}
	\Xi_{n, \psi} = \min_{\norm{\btheta}_{\bSigma_n} = 1}~
		\max_{\ltwo{\blambda} \le 1, \by \odot \blambda \ge 0} -\frac{1}{\sqrt{p}}\blambda^{\sT}  \bX \btheta\,
		=  \min_{\norm{\btheta}_{\bSigma_n} = 1} \frac{1}{\sqrt{p}} \norm{(- \by \odot \bX \btheta)_+}_2.
\end{equation}
We can then apply the same technique as before to show the typical value of 
$\Xi_{n, \psi}$ is strictly positive, indicating non-separability of the data, under the situation when $\psi < \psi^*(0)$.

\vspace{0.25cm}

\noindent{\bf Part $(c)$.} 
Let $G, Z$ be independent $\normal(0, 1)$. Define for $r\in \R_{\ge 0}$ and $\nu \in [-1, 1]$ the error function: 
\begin{equation}\label{eq:def:Q}
	\erf(r, \nu) = \P\left(\nu YG+ \sqrt{1-\nu^2} Z \le 0\right)~~\text{where}~~
		\begin{cases}
			\P(Y= +1\mid G) = f(r \cdot G) \\
			\P(Y= -1\mid G) = 1 -f(r \cdot G)
		\end{cases}
\end{equation}
Note that in the expression \eqref{eq:def:pred:error}, $y^{\snew}$ depends on $\<\btheta_{*,n}, \bx^{\snew}\>$. Conditional on $(\by,\bX)$, $\<\btheta_{*,n}, \bx^{\snew}\>$ and $\<\hbtheta^{\sMM}(\by,\bX),\bx^{\snew}\>$ are jointly Gaussian with covariance $\langle \hbtheta^{\sMM}, \btheta_{*, n}\rangle_{\bSigma_n}:= \langle \hbtheta^{\sMM}, \bSigma_n \btheta_{*, n}\rangle$. Thus, it is straightforward to see that the generalization error of the max-margin classifier is given by
\begin{equation}
	\label{eqn:generalization-error-expression}
	\error_n(\by,\bX) = 
		\erf\left(\norm{\btheta_{*, n}}_{\bSigma_n}, 
			\langle \hbtheta^{\sMM}, \btheta_{*, n}\rangle_{\bSigma_n}/(\norm{\btheta_{*, n}}_{\bSigma_n} 
				\normsmall{\hat{\btheta}^{\sMM}}_{\bSigma_n})\right)\, .
\end{equation}
Comparing this expression to Theorem \ref{theorem:main}.$(c)$, and recalling that
$\norm{\btheta_{*, n}}_{\bSigma_n}:= \langle \btheta_{*,n}, \bSigma_n \btheta_{*,n}\rangle^{1/2}\to \rho$ by Assumption \ref{assumption:converge},  we see that 
it is sufficient to prove that for each $\psi > \psi\opt (0)$,
\begin{equation}
	 \lim_{n\to \infty} 
	\frac{\< \hbtheta^{\sMM}, \btheta_{*, n}\>_{\bSigma_n}}{\norm{\btheta_{*, n}}_{\bSigma_n} \normsmall{\hat{\btheta}^{\sMM}}_{\bSigma_n}}
	= \nu\opt(\psi) \,, \label{eq:LimitScalarProd}
\end{equation}
where $\nu^\star(\psi)$ is defined in \eqref{eq:AsyErrorDef}. To this end, we generalize the definition of $\xi_{n, \psi, \kappa}$ as follows. For any compact set
$\bTheta_p \subseteq \reals^p$, we define the quantity $\xi_{n, \psi, \kappa}(\bTheta_p)$ by
\begin{equation}\label{eq:xiDef}
\xi_{n, \psi, \kappa}(\bTheta_p)= \min_{\btheta \in \bTheta_p}~
		\max_{\ltwo{\blambda} \le 1, \by \odot \blambda \ge 0} \frac{1}{\sqrt{p}}\blambda^{\sT} (\kappa \by - \bX \btheta)\, .
\end{equation}
Notice that if $\xi_{n, \psi, \kappa}(\bTheta_p)>\xi_{n, \psi, \kappa}$ with high probability, then $\P(\hbtheta^{\sMM}_{n}\in \bTheta_p)\to 0$
as $n\to\infty$. 
In order to control the left hand side of Eq.~\eqref{eq:LimitScalarProd}, we consider sets of the form
\begin{align}
\bTheta_p = \left\{\btheta\in \reals^p :\; \|\btheta\|_2\le 1 \, ,\;\;\;\frac{\< \hbtheta^{\sMM}, \btheta_{*, n}\>_{\bSigma_n}}{
\norm{\btheta_{*, n}}_{\bSigma_n} \normsmall{\hat{\btheta}^{\sMM}}_{\bSigma_n}}\in J_n\right\}\, ,
\end{align}
for suitable sequences of compact sets $J_n\subseteq\reals$. 
Using Gordon's inequality to lower bound
$\xi_{n, \psi, \kappa}(\bTheta_p)$, we can guarantee that  Eq.~\eqref{eq:LimitScalarProd} holds. 

\section{Proofs}
\label{sec:ProofMain}

This section provides a complete outline of the proof of Theorem \ref{theorem:main}, deferring 
most technical steps to the appendices. 

Notice that the definition of the joint distribution of
$(y,\bx)$ and the statements in Theorem \ref{theorem:main} are independent of the choice of a basis on $\reals^p$.
We can therefore work in the basis in which $\bSigma$ is diagonal. This amounts to assuming that $\bSigma = \bLambda$
is diagonal. 

The proof proceeds through a sequence of steps to progressively simplify the quantities 
$\xi_{n, \psi, \kappa}$ $\xi_{n, \psi, \kappa}(\bTheta_p)$. We begin by setting $\xi_{n, \psi, \kappa}^{(0)}\equiv \xi_{n, \psi, \kappa}$ 
and $\xi^{(0)}_{n, \psi, \kappa}(\bTheta_p) \equiv \xi_{n, \psi, \kappa}(\bTheta_p)$. By Gordon's inequality and concentration, we will reduce  $\xi^{(0)}_{n, \psi, \kappa}$ to quantities $\xi^{(i)}_{n, \psi, \kappa}, i=1,2$, to be defined below.

\paragraph{Step 1: Reduction from $\xi_{n, \psi, \kappa}^{(0)}$ to $\xi_{n, \psi, \kappa}^{(1)}$ via Gordon's comparison inequality} 
We use Gordon's comparison 
inequality to reduce the original minimax of a complicated Gaussian process to that of a much simpler Gaussian process. 
We state Gordon's comparison inequality below for reader's convenience~\cite{Gordon88, ThrampoulidisOyHa15}.

\begin{theorem}[Theorem 3 from~\cite{ThrampoulidisOyHa15}]
\label{thm:Gordon-improve}
Let $\mathcal{C}_1 \subseteq \R^p$ and $\mathcal{C}_2 \subseteq \R^n$ be two compact sets and let 
$T: \mathcal{C}_1 \times \mathcal{C}_2 \to \R$ be a continuous function. Let 
$\bX = (X_{i, j})\iid \normal(0, 1) \in \R^{p\times n}$, $\bg \sim \normal(0, I_p)$ and 
$\bh \sim \normal(0, I_n)$ be independent vectors and matrices. Define, 
\begin{align*}
Q_1(\bX) &= \min_{\bw_1 \in \mathcal{C}_1}\max_{\bw_2 \in \mathcal{C}_2} \bw_1^{\sT} \bX \bw_2 + T(\bw_1, \bw_2) \\
Q_2(\bg, \bh) &= \min_{\bw_1 \in \mathcal{C}_1}\max_{\bw_2 \in \mathcal{C}_2} \ltwo{\bw_2} \bg^{\sT} \bw_1 + \ltwo{\bw_1} \bh^{\sT} \bw_2 + T(\bw_1, \bw_2)
\end{align*}
Then the following hold: 
\begin{enumerate}
\item For all $t\in \R$
	\begin{equation*}
	\P(Q_1(\bX) \le t) \le 2\, \P(Q_2(\bg, \bh) \le t). 
	\end{equation*}
\item Suppose $\mathcal{C}_1$ and $\mathcal{C}_2$ are both convex, and $T$ is convex concave in $(\bw_1, \bw_2)$. Then, for all $t\in \R$
	\begin{equation*}
	\P(Q_1(\bX) \ge t) \le 2\, \P(Q_2(\bg, \bh) \ge t). 
	\end{equation*}
\end{enumerate}
\end{theorem}
\noindent\noindent
Let $\bg\sim \normal(0, I_p)$, $\bh \sim \normal(0, I_n)$, $\bu \sim \normal(0, I_n)$ 
be independent Gaussian vectors and $\bw$ the unit vector in the direction of $\bLambda^{1/2}\btheta_*$, i.e.
$\bw = \bLambda^{1/2}\btheta_*/\|\bLambda^{1/2}\btheta_*\|_2$. 
Further, let $\by\in\{+1,-1\}^n$ be such that $y_i$ is conditional independent of $\bg,\bh$ and $(u_j)_{j\neq i}$ given $u_i$,
with $\P(y_i=+1|u_i) = f(\rho_n u_i)$ and $\rho_n=\|\bLambda^{1/2}\btheta_*\|_2$.
Define $\xi_{n, \psi, \kappa}^{(1)}$ and $\xi_{n, \psi, \kappa}^{(1)}(\bTheta_p)$ by letting
\begin{equation}
\label{eqn:def-xi-1}
 \xi_{n, \psi, \kappa}^{(1)}(\bTheta_{p}) \defeq  
 	\min_{\btheta \in \bTheta_{p}}
 	\max_{\ltwo{\blambda} \le 1, \blambda \odot \by \geq 0}~
		\frac{1}{\sqrt{p}}  \left(\blambda^{\sT} (\kappa \by- \langle \bLambda^{1/2} \bw, \btheta \rangle \bu - 
			\ltwobig{\proj_{\bw^{\perp}} \bLambda^{1/2} \btheta}\bh) + \ltwo{\blambda} \bg^{\sT}\proj_{\bw^{\perp}}\bLambda^{1/2} \btheta\right),
\end{equation}
and $\xi_{n, \psi, \kappa}^{(1)}=\xi_{n, \psi, \kappa}^{(1)}(\Ball^p(1))$.

We can apply Gordon's inequality (Theorem~\ref{thm:Gordon-improve}) to relate $\xi_{n, \psi, \kappa}^{(0)}$
to $\xi_{n, \psi, \kappa}^{(1)}$: the result is given in the next lemma, whose proof can be found in  
Appendix~\ref{sec:proof-lemma-xi-0-xi-1}.
\begin{lemma}
\label{lemma:xi-0-xi-1}
The following inequalities hold for any $t\in \R$, any compact  set $\bTheta_p \subseteq \reals^p$:
\begin{equation}
\begin{split}
&\P(\xi_{n, \psi, \kappa}^{(0)} \le t ) \le 2 \,\P(\xi_{n, \psi, \kappa}^{(1)}  \le t ) ~\text{and}~
\P(\xi_{n, \psi, \kappa}^{(0)}  \ge t ) \le 2\, \P(\xi_{n, \psi, \kappa}^{(1)}  \ge t ).  \\
&\P(\xi_{n, \psi, \kappa}^{(0)} (\bTheta_{p}) \le t ) \le 2 \P(\xi_{n, \psi, \kappa}^{(1)} (\bTheta_{p}) \le t )\, .
\end{split}
\end{equation}
Further, if $\bTheta_p$ is convex,  we have
 \begin{equation}
\P(\xi_{n, \psi, \kappa}^{(0)} (\bTheta_{p}) \ge t ) \le 2 \,\P(\xi_{n, \psi, \kappa}^{(1)} (\bTheta_{p}) \ge t )\, .
\end{equation}
\end{lemma}

\paragraph{Step 2: Reduction from $\xi_{n, \psi, \kappa}^{(1)}$ to $\xi_{n, \psi, \kappa}^{(2)}$.}
A simple calculation gives 
\begin{equation}
\label{eqn:def-xi-1-xi-1-bar}
\xi_{n, \psi, \kappa}^{(1)} = \left(\bar{\xi}_{n, \psi, \kappa}^{(1)}\right)_+
~\text{and}~
\xi_{n, \psi, \kappa}^{(1)}(\bTheta_p) = \left(\bar{\xi}_{n, \psi, \kappa}^{(1)}(\bTheta_p)\right)_+,
\end{equation}
where 
\begin{equation}
\label{eqn:def-xi-1-bar}
\begin{split}
 \bar{\xi}_{n, \psi, \kappa}^{(1)} \defeq \min_{\ltwo{\btheta} \le 1}
 	\frac{1}{\sqrt{p}}\left(\ltwobigg{\left(\kappa \one - \langle \bLambda^{1/2} \bw, \btheta \rangle (\by \odot \bu) 
		- \ltwobig{\proj_{\bw^{\perp}} \bLambda^{1/2} \btheta} (\by \odot \bh)\right)_+}  
		+ \bg^{\sT}\proj_{\bw^{\perp}}\bLambda^{1/2} \btheta\right) \\
 \bar{\xi}_{n, \psi, \kappa}^{(1)}(\bTheta_p) \defeq \min_{\btheta \in \bTheta_{p}}
 	\frac{1}{\sqrt{p}}\left(\ltwobigg{\left(\kappa \one - \langle \bLambda^{1/2} \bw, \btheta \rangle (\by \odot \bu) 
		- \ltwobig{\proj_{\bw^{\perp}} \bLambda^{1/2} \btheta} (\by \odot \bh)\right)_+}  
		+ \bg^{\sT}\proj_{\bw^{\perp}}\bLambda^{1/2} \btheta\right) 
\end{split}
\end{equation}
Recall the definition of $F_{\kappa}$. Now we define the quantity $\xi_{n, \psi, \kappa}^{(2)}$ and $\xi_{n, \psi, \kappa}^{(2)}$ by,
\begin{equation}
\label{eqn:def-xi-2}
\begin{split}
 \xi_{n, \psi, \kappa}^{(2)} &\defeq %\min_{\ltwo{\btheta} \le 1, \langle x, w\rangle_{\bLambda^{1/2}} = \nu}
 	%\E \left[\left( \kappa - \nu T - \ltwobig{\proj_{\bw^{\perp}} \bLambda^{1/2}x} Z]\right)_+\right]  + 
		%\frac{1}{\sqrt{p}} g^{\sT} \proj_{\bw^{\perp}}\bLambda^{1/2} x \\&
	\min_{\ltwo{\btheta} \le 1} 
		\psi^{-1/2} \cdot F_{\kappa} \left(\langle \bLambda^{1/2} \bw, \btheta \rangle, \ltwobig{\proj_{\bw^{\perp}} \bLambda^{1/2} \btheta} \right) +
			 \frac{1}{\sqrt{p}} \bg^{\sT} \proj_{\bw^{\perp}}\bLambda^{1/2} \btheta \nonumber  \\
 \xi_{n, \psi, \kappa}^{(2)}(\bTheta_{p}) &\defeq 
	\min_{\btheta \in \bTheta_{p}} ~~
		\psi^{-1/2} \cdot F_{\kappa} \left(\langle \bLambda^{1/2} \bw, \btheta \rangle, \ltwobig{\proj_{\bw^{\perp}} \bLambda^{1/2} \btheta} \right) +
			 \frac{1}{\sqrt{p}} \bg^{\sT} \proj_{\bw^{\perp}}\bLambda^{1/2} \btheta \nonumber
\end{split}
\end{equation}
The next lemma allows us to move from $\xi_{n, \psi, \kappa}^{(1)}$ to $(\xi_{n, \psi, \kappa}^{(2)})_+$. 
\begin{lemma}
\label{lemma:xi-1-xi-2}
The following convergence holds for any sequence of compact sets $\{\bTheta_p\}_{p\in \N}$ 
satisfying $\bTheta_p \subseteq \{\btheta \in \R^p: \ltwo{\btheta} \le 1\}$: 
\begin{equation}
\left|\xi_{n, \psi, \kappa}^{(1)}(\bTheta_p) -  \left(\xi_{n, \psi, \kappa}^{(2)}(\bTheta_p)\right)_+\right| \pto 0.
 \end{equation}
 In particular, we have 
 \begin{equation}
 \left| \xi_{n, \psi, \kappa}^{(1)} -  \left(\xi_{n, \psi, \kappa}^{(2)}\right)_+\right| \pto 0.
\end{equation}
\end{lemma}\noindent\noindent
One important benefit of this reduction is that both $\xi_{n, \psi, \kappa}^{(2)}$ and $\xi_{n, \psi, \kappa}^{(2)}(\bTheta_p)$ 
(for $\bTheta_p$ convex) are minima of convex optimization problems (in contrast, the optimization problems 
defining $\bar{\xi}_{n, \psi, \kappa}^{(1)}$ and $\bar{\xi}_{n, \psi, \kappa}^{(1)}(\bTheta_p)$ are not convex). Indeed, 
notice that the function $F_{\kappa}(c_1, c_2)$ is convex in $(c_1, c_2)$. We
collect all the useful properties of the function $F_{\kappa}(c_1, c_2)$ in the next lemma,  
whose proof is deferred to Appendix~\ref{sec: proof-lemma-F-convex-increasing}.
\begin{lemma}
\label{lemma:F-convex-increasing}
The following properties hold for $F_{\kappa}:\reals\times\reals\to\reals$:
\begin{itemize}
\item[$(a)$] If $\kappa > 0$, then the function $F_{\kappa}:\R \times \R \to \R$ is strictly convex.
\item[$(b)$] For any fixed $c_1 \in \R$, the function $c_2 \mapsto F_{\kappa}(c_1, c_2)$ is strictly increasing for $c_2\in\reals_{\ge 0}$. 
\item[$(c)$] The function $F_{\kappa}:\R \times \R\to \R$ is continuously differentiable.
\item[$(d)$] For any $c_1,c_2\in\reals$, the function $\kappa\mapsto F_{\kappa}(c_1,c_2)$ is strictly increasing for $\kappa\in\reals$.
\end{itemize}
\end{lemma}

\paragraph{Step 3: Analysis of $\xi_{n, \psi, \kappa}^{(2)}$} 
Characterizing the limit of $\xi_{n, \psi, \kappa}^{(2)}$ is the technically most challenging part. Our approach 
is to find a new representation of $\xi_{n, \psi, \kappa}^{(2)}$ that allows one to easily guess
its asymptotic behavior. Let us define $\bar{\bw} = \sqrt{p} \bw$. Recall 
\begin{align}
\xi_{n, \psi, \kappa}^{(2)}
	= \min_{\ltwo{\btheta} \le 1} ~~
		\psi^{-1/2} \cdot F_{\kappa} \left(\langle \btheta, \bLambda^{1/2}\bw\rangle, 
			\ltwobig{\proj_{\bw^{\perp}} \bLambda^{1/2} \btheta} \right) 
				+ \frac{1}{\sqrt{p}} \bg^{\sT} \proj_{\bw^{\perp}} \bLambda^{1/2} \btheta \label{eqn:xi-2-expansion} 
\end{align}
Thus we have (note we rescale $\btheta$ by $\sqrt{p}$)
\begin{align}
	\xi_{n, \psi, \kappa}^{(2)} =\min_{\frac{1}{p}\sum_{i=1}^p \theta_i^2 \le 1} ~~
		\psi^{-1/2} \cdot F_{\kappa} \left(\frac{1}{p}\sum_{i=1}^p \theta_i \lambda_i^{1/2} \bar{w}_i, 
			\left(\frac{1}{p} \sum_{i=1}^p (\proj_{\bw^{\perp}}(\bLambda^{1/2} \btheta))_i^2\right)^{1/2} \right) 
				+ \frac{1}{p}\sum_{i=1}^p \left(\proj_{\bw^{\perp}} \bg\right)_i \lambda_i^{1/2} \theta_i.
				\label{eqn:xi-2-expansion-two}
\end{align}
Let $\Q_n$ be the empirical distribution of the coordinates of $(\bg, \blambda, \bar{\bw})$, i.e.
the probability measure on $\reals^3$ defined by
\begin{equation}
    \Q_n =\frac{1}{p}\sum_{i=1}^{p}\delta_{(g_i,\lambda_i,\bar{w}_i)}.
\end{equation}
Let $\cL^2(\Q_n) = \cL^2(\Q_n,\reals^3)$ be the space of functions $h:\reals^3\to\reals$, $(g,\lambda,w)\mapsto h(g,\lambda,w)$ that are square integrable 
with respect to $\Q_n$. Notice that the $n$ points that form $\Q_n$ are almost surely distinct, and therefore we can identify this space with the
space of vectors $\btheta\in\reals^p$. We also define three random variables in the same space by $G(g,\lambda,w) = g$, 
$X(g,\lambda,w) = \lambda$, $W(g,\lambda,w) = w$. Denote $\langle \cdot,\cdot\rangle_{\Q_n}$ (resp. $\norm{\cdot}_{\Q_n}$) denote the inner product (resp.norm) in  $\cL^2(\Q_n)$. With these definitions, we can rewrite the expression in Eq.~\eqref{eqn:xi-2-expansion-two} as
\begin{equation}
\label{eqn:xi-2-new-rep}
\xi_{n, \psi, \kappa}^{(2)} =
	\min_{h \in \cL^2(\Q_n): \norm{h}_{\Q_n} \le 1} 
		\psi^{-1/2} \cdot F_{\kappa} \left( \langle X^{1/2} h, W \rangle_{\Q_n}, 
			\normbig{\proj_{W^{\perp}}(X^{1/2} h)}_{\Q_n} \right)  + \langle X^{1/2} \proj_{W^{\perp}}(G), h\rangle_{\Q_n}.
\end{equation}
Now we define $\Q_{\infty} \defeq \normal(0, 1) \otimes \mu$. By Assumption \ref{assumption:converge}, the following convergence holds almost surely
\begin{equation}
\label{eqn:Q-n-converge-to-Q-infty}
\Q_n \stackrel{W_2}{\Longrightarrow} \Q_{\infty}.
\end{equation}
Motivated by the representation in Eq.~\eqref{eqn:xi-2-new-rep} and the convergence in Eq.~\eqref{eqn:Q-n-converge-to-Q-infty}, 
we define $\xi_{\psi, \kappa}$ by 
\begin{equation}
\label{eqn:def-xi-infty-psi-kappa}
\xi_{\psi, \kappa} \defeq 
	\min_{h = h(g, X, W) \in \cL^2(\Q_{\infty}): \norm{h}_{\Q_{\infty}} \le 1} 
		\psi^{-1/2} \cdot F_{\kappa} \left( \langle X^{1/2} h, W \rangle_{\Q_\infty}, 
			\normbig{\proj_{W^{\perp}}(X^{1/2} h)}_{\Q_\infty} \right)  + \langle X^{1/2} \proj_{W^{\perp}}(G), h\rangle_{\Q_\infty}.
\end{equation}
In other words, in defining $\xi_{\psi, \kappa}$, we replace $\Q_n$ on the right-hand side of Eq.~\eqref{eqn:xi-2-new-rep} by its limit 
$\Q_{\infty}$. Proposition~\ref{proposition:xi-n-2-xi-infty} below characterizes both the asymptotic behavior of the optimal value 
$\xi^{(2)}_{n, \psi, \kappa}$ and the optimal solution $\hat{\btheta}_{n, \psi, \kappa}^{(2)}$ of the problem defined in 
Eq.~\eqref{eqn:xi-2-expansion} (note that $\xi^{(2)}_{n, \psi, \kappa}$ and $\hat{\btheta}_{n, \psi, \kappa}^{(2)}$ are random).
The proof of Proposition~\ref{proposition:xi-n-2-xi-infty} can be found  in Appendix~\ref{sec:all-property-optimization}.
\begin{proposition}
\label{proposition:xi-n-2-xi-infty}
\begin{enumerate}
\item[$(a)$] If $\psi \le \psi^{\low}(\kappa)$, then almost surely 
	\begin{equation*}
	\liminf_{n\to \infty, p/n\to \psi} \xi_{n, \psi, \kappa}^{(2)} > 0.
	\end{equation*} 
\item[$(b)$] If $\psi > \psi^{\low}(\kappa)$, then 
	\begin{itemize}
	\item the minimum value $\xi_{n, \psi, \kappa}^{(2)}$ satisfies the almost sure convergence
		\begin{equation}
			\label{eqn:goal-finite-infinite}
		\lim_{n\to \infty, p/n\to \psi} \xi_{n, \psi, \kappa}^{(2)} = \xi_{\psi, \kappa} = T(\psi, \kappa).
		\end{equation}
	\item the minimum $\hat{\btheta}_{n, \psi, \kappa}^{(2)}$ (of the problem defined in Eq.~\eqref{eqn:xi-2-expansion})
		is uniquely defined. It satisfies the almost sure convergence 
		\begin{equation}
		\label{eqn:lim-btheta_n_2_c1c2}
		\begin{split}
			\lim_{n\to \infty, p/n\to \psi} \langle \hat{\btheta}_{n, \psi, \kappa}^{(2)}, \bLambda^{1/2}\bw\rangle &= c_1(\psi, \kappa), \\
				%~~\text{and}~~
			\lim_{n\to \infty, p/n\to \psi} \normbig{\hat{\btheta}_{n, \psi, \kappa}^{(2)}}_{\bLambda^{1/2}} &= 
				\left(c_1^2(\psi, \kappa) + c_2^2(\psi, \kappa)\right)^{1/2}
		\end{split}
		\end{equation}
		where $(c_1(\psi, \kappa), c_2(\psi, \kappa))$ is defined as in Proposition~\ref{proposition:system-of-Eq-T}.  
		Further, denote $\mathcal{L}_{n, \psi, \kappa}^{(2)}$ to be the empirical distribution of	 
		$\{(\lambda_i, \bar{w}_i, \sqrt{p}\hat{\btheta}_{n, \psi, \kappa; i}^{(2)})\}_{i \in [p]}$ where 
		$\{\hat{\btheta}_{n, \psi, \kappa; i}^{(2)}\}_{i \in [p]}$ are the coordinates of $\btheta_{n, \psi, \kappa}^{(2)}$. 
		Then $\mathcal{L}_{n, \psi, \kappa}^{(2)}$ satisfies the almost sure convergence  
		\begin{equation}
			\label{eqn:W-2-converge}
			\lim_{n\to \infty, p/n\to \psi} 
				W_2 \left(\mathcal{L}_{n, \psi, \kappa}^{(2)}, \mathcal{L}_{\psi, \kappa}\right) = 0,
		\end{equation}
		where $L_{\psi, \kappa}$ is defined in Definition~\ref{def:KappaE}. 
	\end{itemize}
\end{enumerate}
\end{proposition}

Let us emphasize that the almost sure convergence from $\lim_{n\to \infty}\xi_{n, \psi, \kappa}^{(2)} = \xi_{\psi, \kappa}$ (i.e., 
Eq.~\eqref{eqn:goal-finite-infinite}) is not an immediate consequence of the convergence 
$\Q_n \stackrel{W_2}{\Longrightarrow} \Q_{\infty}$  (Eq.~\eqref{eqn:Q-n-converge-to-Q-infty}).
Indeed, the optimization problem defining $\xi_{n, \psi, \kappa}^{(2)}$ has dimension $p$ increasing with $n$,
while  the problem defining $\xi_{\psi, \kappa}$ is infinite-dimensional (cf. Eq.~\eqref{eqn:xi-2-new-rep} 
and Eq.~\eqref{eqn:def-xi-infty-psi-kappa}). As a consequence, elementary arguments from empirical process theory do not apply: 
we refer to Appendix~\ref{sec:all-property-optimization} for details. 

As an immediate consequence of Proposition~\ref{proposition:system-of-Eq-T} and Proposition~\ref{proposition:xi-n-2-xi-infty},  
we obtain that, 
\begin{equation*}
\begin{split}
\liminf_{n \to \infty, p/n\to \psi} \xi_{n, \psi, \kappa}^{(2)} > 0~~~\text{for $\kappa > \kappa\opt(\psi)$}. \\
\lim_{n \to \infty, p/n\to \psi} \xi_{n, \psi, \kappa}^{(2)} \le 0~~~\text{for $\kappa \le \kappa\opt(\psi)$}. 
\end{split}
\end{equation*}

Together with Lemma~\ref{lemma:xi-0-xi-1} and Lemma~\ref{lemma:xi-1-xi-2}, we can pass the above result to $\xi_{n, \psi, \kappa}$:
\begin{itemize}
\item For $\kappa > \kappa\opt(\psi)$ 
	\begin{equation}
	\label{eqn:summary-xi-pos}
		\lim_{n \to \infty, p/n\to \psi} \P\left(\xi_{n, \psi, \kappa} > 0\right) = 1.
	\end{equation}
\item For $\kappa < \kappa\opt(\psi)$ 
	\begin{equation}
	\label{eqn:summary-xi-neg}
		\lim_{n \to \infty, p/n\to \psi} \P\left(\xi_{n, \psi, \kappa} = 0\right) = 1.
	\end{equation}
\item For $\kappa = \kappa\opt(\psi)$, we have for any $\eps > 0$, 
	\begin{equation}
	\label{eqn:summary-xi-0}
		\lim_{n \to \infty, p/n\to \psi} \P\left(\xi_{n, \psi, \kappa}\in [0, \eps]\right) = 1.
	\end{equation}
\end{itemize}
This characterizes the asymptotics of $\xi_{n,\psi,\kappa}$. We proceed analogously to characterize the behavior
of $\xi_{n, \psi, \kappa}(\bTheta_p)$ and therefore determine the high-dimensional limit of
$\langle \hat{\btheta}_n^{\sMM}, \btheta_{*, n}\rangle_{\bSigma_n}/(\norm{\btheta_{*, n}}_{\bSigma_n} \cdot \normsmall{\hat{\btheta}_n^{\sMM}}_{\bSigma_n})$.
The main result of this analysis is presented in the next proposition, whose proof is given in appendix~\ref{sec:proof-proposition:nu-n-nu-infty}.
\begin{proposition}
\label{proposition:nu-n-nu-infty}
Let $\psi > \psi\opt(0)$. For the max-margin linear classifier $\bx \to \sign(\langle \hat{\btheta}_n^{\sMM}, \bx \rangle)$, we have 
\begin{equation*}
	\frac{\< \hat{\btheta}^{\sMM}, \btheta_{*, n}\>_{\bSigma_n}}{\norm{\btheta_{*, n}}_{\bSigma_n} \cdot \normsmall{\hat{\btheta}^{\sMM}}_{\bSigma_n}}
		 \pto \nu\opt(\psi).
\end{equation*}
\end{proposition}

As a direct generalization of Proposition~\ref{proposition:nu-n-nu-infty}, Proposition~\ref{proposition:ECD-maxmargin} below provides asymptotics 
of the empirical distribution of $\{(\lambda_i, \bar{w}_i, \sqrt{p} \hat{\btheta}_{n, i}^{\sMM})\}_{i \in [p]}$. 
We defer the proof to appendix~\ref{sec:proof-proposition:ECD-maxmargin}.

\begin{proposition}
\label{proposition:ECD-maxmargin}
Let $\psi > \psi\opt(0)$. Recall that  $\hat{\LL}_n$ is the empirical distribution induced by\\
$\{(\lambda_i, \bar{w}_i, \sqrt{p}\langle \hat{\btheta}_n^{{\sMM}}, \bv_i\rangle)\}_{i \in [p]}$.
Then  $\hat{\LL}_n$ 
converges in probability to $\LL_{\psi, \kappa\opt(\psi)}$ in $W_2$ distance: 
	\begin{equation*}
		W_2(\hat{\LL}_n, \LL_{\psi, \kappa\opt(\psi)}) \pto 0.
	\end{equation*}
\end{proposition}

Finally,  we discuss the proportional asymptotics of $\Xi_{n, \psi}$ (defined in equation~\eqref{eqn:def-XI}).
Using the same proof techniques as above, we show the following result. The proof is in Appendix~\ref{sec:data-separability}.
\begin{lemma}
\label{lemma:proportional-asymptotics-of-Xi}
For $\psi < \psi^*(0)$, we have $\lim_{n \to \infty, p/n \to \psi} \P(\Xi_{n, \psi} > 0) = 1$.
\end{lemma}

The proof of Theorem~\ref{theorem:main} now follows: 
Parts $(a)$, $(b)$ follow immediately from Eq.~\eqref{eqn:equiv-xi-event-want}, 
	Eq.~\eqref{eqn:summary-xi-pos}, Eq.~\eqref{eqn:summary-xi-neg} and 
	Lemma~\ref{lemma:proportional-asymptotics-of-Xi}.
Part $(c)$ follows from Eq.~\eqref{eqn:generalization-error-expression} and 
	Proposition~\ref{proposition:nu-n-nu-infty}. Part $(d)$ follows from 
Proposition~\ref{proposition:ECD-maxmargin}. The proof of Theorem~\ref{thm:benign} proceeds by bounding $\error^\star(\mu,\psi)$ in Theorem \ref{theorem:main} in a delicate manner: for it's technicality, the proof is deferred to Section ~\ref{sec:proof:benign}. The proof of Theorem ~\ref{thm:RF:Gaussian} follows from Theorem ~\ref{theorem:main} and the universality of the random features model and the equivalent Gaussian model (cf. Section \ref{sec:GaussianModel}) proven in \cite{montanari2023universality_MaxMargin}. The proof of Propositions~\ref{prop:asymptoticRF} and  \ref{propo:Soft-Margin} are deferred to Sections~ \ref{sec:AsymptoticsRF} and \ref{sec:appendix:soft:margin} for their technicality.

\section*{Acknowledgements}

This work was partially supported by grants NSF CCF-1714305, IIS-1741162, and ONR
N00014-18-1-2729. We also acknoweldge support
 NSF through award DMS-2031883, the Simons Foundation
through Award 814639 for the Collaboration on the Theoretical Foundations of Deep Learning.

\appendix

\section{Notations}
\label{app:Notation}

We typically use lower case letters to denote scalars (e.g. $a,b,c,\dots\in\reals$), boldface lower case to denote vectors
(e.g. $\bu,\bv,\bw,\dots\in\reals^d$), and boldface upper case to denote matrices (e.g. $\bX, \bZ,\dots\in\reals^{d_1\times d_2}$).
The standard scalar product of two vectors $\bu,\bv\in\reals^d$ will be denoted by $\<\bu,\bv\>=\bu^{\sT}\bv = \sum_{i=1}^d u_iv_i$.
The corresponding norm is $\|\bv\|_2 = \<\bv,\bv\>^{1/2}$.  We will define other norms and scalar products within the text.

We occasionally use the notation $[a\pm b] \equiv[a-b,a+b] \equiv\{x\in\reals:\; a-b\le x\le a+b\}$ for intervals on the real line.

Given two probability measures $\mu_1, \mu_2$ on $\reals^d$, their Wasserstein distance $W_2$ is defined as
\begin{align}
W_2(\mu_1,\mu_2) \equiv \left\{\inf_{\gamma\in \cC(\mu_1,\mu_2)} \int
  \|\bx_1-\bx_2\|_2^2\gamma(\de\bx_1,\de\bx_2)\right\}^{1/2}\, ,
\end{align}
where the infimum is taken over the set of couplings $\cC(\mu_1,\mu_2)$ of $\mu_1$, $\mu_2$. 

Throughout the paper, we are interested in the limit $n,p\to\infty$, with $p/n\to \psi\in(0,\infty)$. We do not write this explicitly
each time, and often only write $n\to\infty$ (as in, for instance, $\lim_{n\to\infty}$). It is understood that $p=p_n$ is such that
$p_n/n\to \psi$.

\section{Properties of the asymptotic optimization problem}
\label{sec:all-property-optimization}

In this appendix we derive some important properties of the asymptotic optimization problem that 
determines the asymptotic maximum margin and prediction error. This has two formulations: the one 
given in Proposition \ref{proposition:system-of-Eq-T}, in terms of the three parameters $(c_1,c_2,s)$ and
the infinite-dimensional optimization in Eq.~\eqref{eqn:def-xi-infty-psi-kappa}. 

We begin by recalling some definitions, and introducing new ones in the next subsection. We will then establish  
some useful properties of the function $F$ in Section \ref{sec: proof-lemma-F-convex-increasing}, and 
of the asymptotic optimization problem \eqref{eqn:def-xi-infty-psi-kappa} in Sections \ref{sec:TechnicalLemmas} to
\ref{sec:ProofEqT}.

\subsection{Definitions}

Given a probability distribution $\P$ on $\reals^m$, we write $\cL^2(\P) =\cL^2(\reals^m,\P)$ for the Hilbert space of square integrable functions
$h:\reals^m\to\reals$, with scalar product
\begin{equation*}
\langle h_1, h_2 \rangle_{\P} = \E_{\P}\big\{h_1(\bZ) h_2(\bZ)\big\} = \int h_1(\bz)h_2(\bz) \, \P(\de\bz), 
\end{equation*}
and corresponding norm  $\norm{h}_{\P} = \langle h, h\rangle_{\P}^{1/2}$. (As usual, measurable functions are 
considered modulo the equivalence relation $h_1\sim h_1\Leftrightarrow \P(h_1\neq h_1) =0$.)
 
We use $W^{\perp}(\P)$ to denote the subspace of $\cL^2(\P)$ orthogonal to the random variable $W\in \cL^2(\P)$:
\begin{equation*}
W^{\perp}(\P) = \left\{ h \in \cL^2(\P): \langle h, W\rangle_{\P} = \E_{\P}[h W] = 0\right\}.
\end{equation*}
We denote by $\proj_{W^{\perp}, \P}$  the orthogonal projection operator onto the orthogonal complement $W^{\perp}(\P)$, i.e., for any 
$h \in \cL^2(\P)$, we define 
\begin{equation}
\proj_{W^{\perp}, \P}(h) = h - \frac{\langle h, W\rangle_{\P}}{\norm{W}_{\P}^2} W \, .
\end{equation}
Notice that the projector $\proj_{W^{\perp}, \P}$ depends on $\P$ because the scalar product $\langle h, W\rangle_{\P}$
and the norm $\norm{W}_{\P}$ do. However, we typically will drop this dependency as it is clear from the context.

In all of our applications, we  will actually consider $m=3$, denote by $(g,x,w)$ the coordinates in $\reals^3$,
and by $G(g,x,w) =g$, $X(g,x,w)=x$, $W(g,x,w) = w$ the corresponding random variables. 
We will be particularly interested in two cases: 
\begin{itemize}
\item[$(i)$] $\P=\Q_{\infty} \defeq \normal(0, 1) \otimes \mu$, with $\mu$ as per Assumption \ref{assumption:converge}.
\item[$(ii)$] $\P = \Q_n = p^{-1}\sum_{i=1}^{p}\delta_{(g_i,\lambda_i,\bar{w}_i)}$, the empirical distribution defined in Section \ref{sec:ProofMain}.
\end{itemize}

Define $\asL_{\psi, \kappa, \P}: \cL^2(\P)\to \R$ by
\begin{equation}
\asL_{\psi, \kappa, \P}(h) =  \psi^{-1/2} \cdot F_{\kappa}\left(\<h, X^{1/2}W \>_{\P},\normbig{\proj_{W^{\perp}}(X^{1/2}h)}_{\P}\right)+
	\< h, X^{1/2} \proj_{W^{\perp}}(G) \>_{\P},\label{eq:asLDef}
\end{equation}
where $F_{\kappa}(c_1,c_2)$ is defined as per Eq.~\eqref{eq:FkDef}.
We consider the optimization problem: 
\begin{equation}
\label{eqn:infinite-dim-prob}
\begin{split}
	\minimize~~ &\asL_{\psi, \kappa, \P}(h) \, ,\\
	\subjectto~~  &\norm{h}_{\P} \le 1\, ,
\end{split}
\end{equation}
and denote its minimum value by $\asL_{\psi, \kappa, \P}\opt$, i.e., 
\begin{equation}
\label{eqn:def-L-star}
\asL_{\psi, \kappa, \P}\opt = \min\Big\{ \asL_{\psi, \kappa, \P}(h) \mid \norm{h}_{\P} \le 1\Big\}\, .
\end{equation}

\subsection{Properties of the function $F$:
Proof of Lemma~\ref{lemma:F-convex-increasing}}
\label{sec: proof-lemma-F-convex-increasing}

\begin{proof-of-lemma}[\ref{lemma:F-convex-increasing}]
Recall the definition of $F$
%:
\begin{equation}
F_{\kappa}(c_1, c_2) = \left(\E \left[(\kappa - c_1 YG - c_2 Z)_+^2\right]\right)^{1/2} \, ,
\end{equation}
with expectation taken w.r.t $Z \perp (Y, G)$, $Z , G\sim \normal(0, 1)$
$\P(Y = +1 \mid G) = f(G) =1-\P(Y = -1 \mid G)$.

\vspace{0.25cm}

\noindent\emph{(a)  $F_{\kappa}:\R \times \R \to \R$  is strictly convex.}
 Note for any random variables $W_1$ and $W_2$: 
\begin{equation}
\label{eqn:inequality-basic-one}
\left(\E [W_1^2]\right)^{1/2} + \left(\E [W_2^2]\right)^{1/2}
	\ge \left(\E [(W_1 + W_2)^2]\right)^{1/2}
\end{equation}
with equality if and only if $b_1 W_1 = b_2 W_2$ for some nonzero pair $(b_1, b_2)$. 
Therefore, for any $(c_1, c_2), (c_1^\prime, c_2^\prime) \in \R \times \R_{\ge 0}$, 
\begin{align}
F_{\kappa}(c_1, c_2) + F_{\kappa}(c_1^\prime, c_2^\prime) 
	&=  \left[\left(\E [\margin - c_1 T - c_2 Z]_+^2\right)^{1/2} + 
		\left(\E [\margin - c_1^{\prime} T - c_2^{\prime} Z]_+^2\right)^{1/2}\right] \nonumber  \\
	&\stackrel{(i)}{\ge}  \left[\left(\E\left[ \big((\margin - c_1 T - c_2 Z)_+ 
		+  (\margin - c_1^{\prime} T - c_2^{\prime} Z)_+\big)^2\right]\right)^{1/2}\right] \nonumber \\
	&\stackrel{(ii)}{\ge} 2  \left(\E\left[ \left(\margin -\half\left(c_1 + c_1^\prime\right) T -
		\half\left(c_2 + c_2^\prime\right)  Z\right)_+^2\right] \right)^{1/2}  \nonumber \\
	&= 2 F_{\kappa}\left(\frac{c_1+c_1^\prime}{2}, \frac{c_2 + c_2^\prime}{2}\right). 
	\label{eqn:F-convex}
\end{align}
where $(i)$ follows from inequality~\eqref{eqn:inequality-basic-one} and $(ii)$ follows from
convexity of $x\mapsto x_+$. Equation~\eqref{eqn:F-convex} gives the convexity of 
$F_{\kappa}$. Assumption~\ref{assumption:non-degenerate-f} implies that, when 
$(c_1, c_2) \neq (c_1^\prime, c_2^\prime)$ and $\kappa \neq 0$, 
\begin{equation}
	\P\left(b (\margin - c_1 T - c_2 Z)_+ \neq b^\prime (\margin - c_1^\prime T - c_2^\prime Z)_+\right) > 0.
\end{equation}
for any nonzero pair $(b, b^\prime)$. Hence, inequality $(i)$ holds strictly for any 
$(c_1, c_2) \neq (c_1^\prime, c_2^\prime)$. This proves that the function $(c_1, c_2) \to F_{\kappa}(c_1, c_2)$ 
is strictly convex.

\vspace{0.25cm}

\noindent\emph{(b) The function $c_2 \mapsto F_{\kappa}(c_1, c_2)$ is strictly increasing for $c_2\ge 0$.} 
Denote $Z_1, Z_2, Z_3$ to  be mutually independent $\normal(0, 1)$ random variables. Note for any $c_2^{(1)} \ge c_2^{(2)} \ge 0$, 
there exist  $c_3^{(2)} \ge $ such that
\begin{equation}
c_2^{(1)} Z_1 \ed c_2^{(2)} Z_2 + c_2^{(3)} Z_3\, .
\end{equation}
Thus, when $c_2^{(1)} \ge c_2^{(2)} \ge 0$, we have that 
\begin{align*}
F_{\kappa}(c_1, c_2^{(1)}) &=   \left(\E [\margin - c_1 T - c_2^{(1)} Z_1]_+^2\right)^{1/2} 
	=  \left(\E [\margin - c_1 T - c_2^{(2)} Z_2 - c_2^{(3)} Z_3]_+^2\right)^{1/2} \nonumber \\
	&\stackrel{(i)}{\ge}  \left(\E [\margin - c_1 T - c_2^{(2)} Z_2 - c_2^{(3)} \E Z_3]_+^2\right)^{1/2} 
	=   \left(\E [\margin - c_1 T - c_2^{(2)} Z_2]_+^2\right)^{1/2}  = F_{\kappa}(c_1, c_2^{(2)})
	%	\label{eqn:F-increasing}
\end{align*}
where $(i)$ holds due to Jensen's inequality. Note that $(i)$ becomes a strict inequality whenever 
$c_2^{(1)} > c_2^{(2)}$. 

\vspace{0.25cm}

\noindent\emph{(c) $F_{\kappa}:\R \times \R\to \R$ is continuously differentiable.}
This follows by an application of dominated convergence, by using two facts: 
$(i)$ the mapping $x \to (x)_+^2$ is continuously differentiable, and 
$(ii)$ $F_{\kappa}(c_1, c_2) > 0$ for any $(c_1, c_2) \in \R \times \R_{\ge 0}$ (indeed $(\kappa-c_1YG-c_2Z)_+^2\ge 0$ 
and by Assumption~\ref{assumption:non-degenerate-f} the inequality is strict with positive probability).

\vspace{0.25cm}

\noindent\emph{(d)  The function $\kappa\mapsto F_{\kappa}(c_1,c_2)$ is strictly increasing.}
Let $\kappa_2>\kappa_1$. Then $F_{\kappa_2}(c_1,c_2)^2-F_{\kappa_1}(c_1,c_2)^2 = \E\{g(Y,G,Z)\}$,
where $g(Y,G,Z) = (\kappa_2-c_1YG-c_2Z)_+^2-(\kappa_1-c_1YG-c_2Z)_+^2$ is non-negative, and strictly positive with positive probability
(again by Assumption~\ref{assumption:non-degenerate-f}).
\end{proof-of-lemma}

\begin{lemma}
\label{lemma:alpha-upper-bound}
Suppose $(c_1, c_2) \in \R \times \R_{\ge 0}$ satisfies the condition  
\begin{equation*}
	\partial_1 F_{\kappa}(c_1, c_2) = 0. 
\end{equation*}
Then, we have the estimate: 
\begin{equation*}
\partial_2 F_{\kappa} (c_{1}, c_{2}) \le \min_{c \in \R} F_0(c, 1). 
\end{equation*}
\end{lemma}
\begin{proof}
Since $F_{\kappa}$ is convex by Lemma~\ref{lemma:F-convex-increasing}, we have for all $c\in \R, t\in \R_{\ge 0}$,
\begin{equation*}
F_{\kappa}(ct, t) \ge F_{\kappa}(c_1, c_2) + \partial_2 F_{\kappa}(c_1, c_2)(t-c_2).
\end{equation*}
This shows in particular that for any $c \in \R$
\begin{equation*}
F_0(c, 1) = \lim_{t\to \infty} \frac{F_{\kappa}(ct, t)}{t} \ge \partial_2 F_{\kappa}(c_1, c_2).
\end{equation*}
Taking minimum over $c$ on both sides gives the desired claim. 
\end{proof}

\begin{lemma}
\label{lemma:continuity-partial-f}
For any $\kappa > 0$, the limit $\lim_{c_2 \to 0^+} c_2^{-1} F(c_1, c_2)$ exists for all $c_1$ and the function $G_\kappa$ defined below 
is continuous on $\R \times \R_{\ge 0}$: 
\begin{equation*}
	G_\kappa(c_1, c_2) \defeq
		\begin{cases}
		c_2^{-1} F_\kappa(c_1, c_2)~~&\text{if $c_2 > 0$}. \\
		 \lim_{c_2 \to 0^+} c_2^{-1} F_\kappa(c_1, c_2)~~&\text{if $c_2 = 0$}.
		 \end{cases}
\end{equation*}
\end{lemma}

\begin{proof}
For all $\kappa > 0$, $F_\kappa(c_1, c_2) > 0$ for all $c_1, c_2$ (where we use Assumption~\ref{assumption:non-degenerate-f}). 
This gives us
\begin{equation*}
	c_2^{-1} F_\kappa (c_1, c_2) = -\frac{\E[(\kappa-c_1 YG - c_2 Z)_+ Z]}{c_2F_\kappa(c_1, c_2)}.
\end{equation*}
As $(c_1, c_2) \mapsto F_\kappa(c_1, c_2)$ is continuous, it suffices to show that 
the mapping $(c_1, c_2) \mapsto H(c_1, c_2)$ defined below 
\begin{equation*}
	H(c_1, c_2) \defeq \frac{1}{c_2} \E[(\kappa-c_1 YG - c_2 Z)_+ Z]
\end{equation*}
has $\lim_{c_2 \mapsto 0^+} H(c_1, c_2)$ and  can be continuously extended to $\R \times \R_{\ge 0}$. 

This can be easily done by using dominated convergence theorem. The key to the proof is to notice that 
\begin{equation*}
	H(c_1, c_2) = \E \left[ \frac{1}{c_2} \Big((\kappa-c_1 YG - c_2 Z)_+ - (\kappa-c_1 YG)_+\Big) Z\right]. 
\end{equation*}
where we've used that $Z \perp (Y, G)$ and $\E[Z] = 0$.  Now that if we denote 
\begin{equation*}
	J(Y,G,Z; c_1,c_2) =  \frac{1}{c_2} \Big((\kappa-c_1 YG - c_2 Z)_+ - (\kappa-c_1 YG)_+\Big),
\end{equation*}
then we have $|J(Y,G,Z; c_1,c_2)| \le |Z|$ for all $c_1, c_2$, with the limit 
\begin{equation*}
	\lim_{c_2 \to 0^+} J(Y,G,Z; c_1,c_2)
		= - \left(Z \mathbf{1}_{\kappa-c_1 YG > 0} + (Z)_+ \mathbf{1}_{\kappa-c_1 YG = 0}\right)
\end{equation*}
Dominated convergence theorem immediately yields the existence of the limit $\lim_{c_2 \mapsto 0^+} H(c_1, c_2)$, 
and as a consequence, the mapping $(c_1, c_2) \mapsto H(c_1, c_2)$ can be extended continuously to $\R \times \R_{\ge 0}$.
\end{proof}

\subsection{Properties of $\asL_{\psi,\kappa,\P}$}
\label{sec:TechnicalLemmas}

In this section we state three lemmas establishing several properties of the variational problem \eqref{eqn:def-L-star}.
We will  prove these properties in the next subsections.

\begin{lemma}
\label{lemma:technical-1}
Assume that $\E_{\P}[W^2] = 1$ and $\P(X \in (0, x_{\max}]) = 1$  for some $x_{\max} < \infty$.

Then  the function $\asL_{\psi,\kappa,P}:\cL^2(\P)\to\reals$ is lower semicontinuous (in the weak topology) and strictly convex.

As a consequence, the minimum of the optimization problem~\eqref{eqn:infinite-dim-prob} is achieved at a unique function
	$h\opt \in \cL^2(\P)$. (Uniqueness holds in the sense that, 
	any other minimizer $\tilde{h}\opt$ must satisfy $\P(\tilde{h}\opt \neq h\opt) = 0$.)
\end{lemma}

\begin{lemma}\label{lemma:technical-2}
Under the assumptions of Lemma \ref{lemma:technical-1}, 
define   $\zeta(\P)$, $\eta>0$ by %\am{The definition of $\zeta$ here is the inverse of the one in Eq.~\eqref{eq:ZetaDef}.
%Which one is correct? }
%
\begin{align}
\label{eqn:def-zeta}
\zeta(\P) &= \normbig{X^{-1/2} W}_{\P}^{-1}\, ,\\
\eta &= \frac{1}{6 x_{\max}^{1/2}} \cdot 
	\min\left\{\min_{|c_1| \le x_{\max}^{1/2}, 0\le c_2 \le x_{\max}^{1/2}}\{F_{\kappa}(c_1, c_2) - F_0(c_1, c_2)\},
	\min_{|c_1| \le x_{\max}^{1/2}} F_{\kappa}(c_1, 0) \right\}.
\end{align}
 (Note $\eta$ is independent of $\P$.) 

Further, the call that $\psi\opt(0) \equiv \min_{c} F_0(c,1)^2$, and 
assume one of the three conditions below to be satisfied: 
\begin{itemize}
\item[{\sf A1.}] $\psi\opt(0)^{1/2} \ge \psi^{1/2}\norm{\proj_{W^{\perp}}(G)}_{\P} - \eta$.
\item[{\sf A2.}] 
	$\partial_1 F_{\kappa}(\zeta, 0) \le \eta$
	and 
	$\partial_2 F_{\kappa}(\zeta , 0) \ge 
				\normbig{\partial_1 F_{\kappa}(\zeta , 0)(1-\zeta^2 X^{-1}) W + \psi^{1/2} \proj_{W^{\perp}}(G)}_{\P} -  \eta$.
\item[{\sf A3.}] $\partial_1 F_{\kappa}(-\zeta, 0) \ge -\eta$ and
	$\partial_2 F_{\kappa}(-\zeta , 0) \ge 
				\normbig{\partial_1 F_{\kappa}(-\zeta , 0)(1-\zeta^2 X^{-1}) W + \psi^{1/2} \proj_{W^{\perp}}(G)}_{\P} -  \eta$.
\end{itemize}
Then we have 
\begin{equation}
\asL_{\psi, \kappa, \P}\opt \ge \psi^{-1/2} x_{\max}^{1/2} \eta > 0. 
\end{equation}
\end{lemma}

\begin{lemma}\label{lemma:technical-3}
Under the assumptions of Lemma \ref{lemma:technical-2}, further 
assume all of the three conditions below are satisfied: 
\begin{itemize}
\item[{\sf B1.}] $\psi\opt(0)^{1/2} < \psi^{1/2}\normbig{\proj_{W^{\perp}}(G)}_{\P}$.
\item[{\sf B2.}] Either $\partial_1 F_{\kappa}(\zeta, 0) > 0$ or
	$\partial_2 F_{\kappa}(\zeta , 0) < 
				\normbig{\partial_1 F_{\kappa}(\zeta , 0)(1-\zeta^2 X^{-1}) W + \psi^{1/2} \proj_{W^{\perp}}(G)}_{\P} $.
\item[{\sf B3.}] Either $\partial_1 F_{\kappa}(-\zeta, 0) < 0$ or
	$\partial_2 F_{\kappa}(-\zeta , 0) < 
				\normbig{\partial_1 F_{\kappa}(-\zeta , 0)(1-\zeta^2 X^{-1}) W + \psi^{1/2} \proj_{W^{\perp}}(G)}_{\P}$.
\end{itemize}
Then the following hold
\begin{itemize}
	\item[$(a)$] The system of equations below 
		\begin{equation}
		\label{eqn:system-of-equations-P}
		\begin{split}
		-c_1 &= \E_{\P} \left[\frac{(\psi^{1/2} \proj_{W^{\perp}}(G) +  \left(\partial_1 F_{\kappa}(c_1, c_2) -  
			c_1 c_2^{-1} \partial_2 F_{\kappa}(c_1, c_2)\right) W)WX^{1/2}}
				{c_2^{-1} \partial_2 F_{\kappa}(c_1, c_2) X^{1/2} + \psi^{1/2} s X^{-1/2}}\right]  \\
		c_1^2 + c_2^2 &=  \E_{\P} \left[\frac{(\psi^{1/2} \proj_{W^{\perp}}(G) +  
			 \left(\partial_1 F_{\kappa}(c_1, c_2) -  c_1 c_2^{-1} \partial_2 F_{\kappa}(c_1, c_2)\right) W)^2 X}
				{(c_2^{-1} \partial_2 F_{\kappa}(c_1, c_2) X^{1/2} + \psi^{1/2} s X^{-1/2})^2}\right] \\
		1 &= \E_{\P}
			\left[\frac{\left(\psi^{1/2}\proj_{W^{\perp}}(G) + 
				\left(\partial_1 F_{\kappa}(c_1, c_2) -  c_1 c_2^{-1} \partial_2 F_{\kappa}(c_1, c_2)\right) W\right)^2}
				{(c_2^{-1} \partial_2 F_{\kappa}(c_1, c_2) X^{1/2} + \psi^{1/2} s X^{-1/2})^2}\right].
		\end{split}
		\end{equation}
		admits a unique solution in $\R \times \R_{>0} \times \R_{>0}$. 
	\item[$(b)$] The unique minimizer $h\opt$ of the optimization problem~\eqref{eqn:infinite-dim-prob} must satisfy 
		\begin{equation}
		\label{eqn:representation-of-h}
		h\opt = -\frac{\psi^{1/2}\proj_{W^{\perp}}(G) + \left(\partial_1 F_{\kappa}(c_1, c_2) -  c_1 c_2^{-1} \partial_2 F_{\kappa}(c_1, c_2)\right) W}
			{ c_2^{-1} \partial_2 F_{\kappa}(c_1, c_2) X^{1/2} + \psi^{1/2} s X^{-1/2}},
		\end{equation}
		where $(c_1, c_2, s)$
		denotes the unique solution of Eq.~\eqref{eqn:system-of-equations-P} in $\R \times \R_{>0} \times \R_{>0}$. 
		Moreover, %$(c_1, c_2, s)$, $h\opt$ satisfy
		\begin{equation}
			c_1 = \langle X^{1/2}h\opt, W\rangle~~\text{and}~~c_2 = \normbig{\proj_{W^{\perp}}(X^{1/2} h\opt)}.
		\end{equation}
	\item[$(c)$] We have
		\begin{equation}
		\label{eqn:new-representation}
			\asL_{\psi, \kappa, \P}\opt =  
				\psi^{-1/2} \Big(F_{\kappa}(c_1, c_2) - c_1 \partial_1 F_{\kappa}(c_1, c_2)- c_2\partial_2 F_{\kappa}(c_1, c_2)\Big)-s
		\end{equation}
		where $(c_1, c_2, s)$
		denotes the unique solution of Eq.~\eqref{eqn:system-of-equations-P} in $\R \times \R_{>0} \times \R_{>0}$.
	\item[$(d)$] The unique solution of Eq.~\eqref{eqn:system-of-equations-P}, $(c_1, c_2, s) \in \R\times \R_{>0}\times \R_{>0}$, satisfies the bound
		\begin{equation}
			|c_1| \le x_{\max}^{1/2}~~\text{and}~~0 < c_2 \le x_{\max}^{1/2}
		\end{equation} 
%for some $\Delta=\Delta_{\psi,\kappa}(\P)$ that is strictly positive, and jointly continuous in $(\psi,\kappa)$ and in $\P$ 
%with respect to the $W_2$ topology. 
%\am{Is this correct? Do we need any other properties of $\Delta$? It is not useful to give the explicit characterization of
%$\Delta$ unless we really need it.}
%
\end{itemize}
\end{lemma}

\subsubsection{Proof of Lemma \ref{lemma:technical-1}}

Throughout this proof, we keep $\psi,\kappa,\P$ fixed, and hence we drop them from the the arguments of $\asL$ to
simplify notations (hence writing $\asL(h) = \asL_{\psi,\kappa,\P}(h)$). Further, we will drop the subscripts $\P$ from 
$\<h_1,h_2\>_{\P}$ and $\|h\|_{\P}$.

We begin by noticing that $\asL: \cL^2(\P)\to \R$ is lower semicontinuous with respect to the weak-$*$ topology
(which coincide with the weak topology since $\cL^2(\P)$ is an Hilbert space). 
Indeed note that: $(i)$ The mappings $h\mapsto \<h, X^{1/2}W \>$ and $h\mapsto \< h, X^{1/2} \proj_{W^{\perp}}(G) \>$ are continuous;
$(ii)$ The mapping $h \to \norm{\proj_{W^{\perp}}(X^{1/2}h)}$ is lower semicontinuous 
$(iii)$ $(c_1,c_2)\mapsto F_{\kappa}(c_1,c_2)$ is continuous (by Lemma~\ref{lemma:F-convex-increasing}.$(c)$), 
and monotone increasing in $c_2$ (by Lemma~\ref{lemma:F-convex-increasing}.$(b)$).

 Together, $(i)$, $(ii)$, $(iii)$ imply the lower semicontinuity of $\asL$.
Since the constraint set $\{h: \norm{h} \le 1\}$ is sequentially compact w.r.t the weak-$*$ topology by the 
Banach-Alaoglu Theorem, this immediately implies that the minimum of the optimization 
problem~\eqref{eqn:infinite-dim-prob} is achieved by some $h\opt \in \cL^2(\P)$.

In order to prove uniqueness of the minimizer, we show that  $\asL: \cL^2(\P)\to \R$ is strictly convex
	\begin{equation}
	\label{eqn:strict-convex-meaning}
		\half \left(L(h_0) + L(h_1)\right) > L\left(\half(h_0 + h_1)\right)~~\text{for any $h_0, h_1$ such that 
				$P(h_0 \neq h_1) > 0$}. 
	\end{equation}
	Pick $h_0, h_1 \in \cL^2(\P)$ such that $\P(h_0 \neq h_1) > 0$. Denote $h_{1/2} = \half(h_0 + h_1)$. Notice that  
	\begin{equation*}
	\begin{split}
	&\half(\asL(h_0) + \asL(h_1)) - \asL(h_{1/2}) \\
	= &~\psi^{-1/2} \cdot \Bigg\{\half F_{\kappa}\left(\langle h_0,X^{1/2}W \rangle,\normbig{\proj_{W^{\perp}}(X^{1/2}h_0)}\right) + 
				\half F_{\kappa}\left(\langle h_1, X^{1/2} W \rangle,\normbig{\proj_{W^{\perp}}(X^{1/2}h_1)}\right)  \\
	&~~~~~~~~~~~~~~~~~		- F_{\kappa}\left(\langle h_{1/2}, X^{1/2} W \rangle ,\normbig{\proj_{W^{\perp}}(X^{1/2}h_{1/2})}\right)\Bigg\} \\
	\stackrel{(i)}{\ge} &~ \psi^{-1/2}
		 \Bigg\{F_{\kappa}\left(\langle h_{1/2},X^{1/2}W \rangle, \half\left(\normbig{\proj_{W^{\perp}}(X^{1/2}h_0)} +
\normbig{\proj_{W^{\perp}}(X^{1/2}h_1)}\right)\right)  \\
	&~~~~~~~~~~~~~~~~~- F_{\kappa}\left(\langle h_{1/2},X^{1/2} W \rangle ,\normbig{\proj_{W^{\perp}}(X^{1/2}h_{1/2})}\right)\Bigg\} \\
	\stackrel{(ii)}{\ge} &~0.
	\end{split}
	\end{equation*}
	where $(i)$ follows since $F_{\kappa}$ is convex by Lemma~\ref{lemma:F-convex-increasing}.$(a)$ and $(ii)$ follows 
	since $F_{\kappa}$ is increasing with respect to its second argument by Lemma~\ref{lemma:F-convex-increasing}.$(b)$. 

Next  we prove that one of the inequalities $(i)$ and $(ii)$ must be strict when $\P(h_0 \neq h_1) > 0$.
To see this, suppose both inequalities $(i)$ and $(ii)$ become equalities  for some $h_0, h_1$. By Lemma~\ref{lemma:F-convex-increasing}.$(a)$, 
	we know that $F_{\kappa}$ is strictly convex, and strictly increasing w.r.t its second argument. Thus, if
	both inequalities $(i)$ and $(ii)$ become equalities, $h_0, h_1$ and $h_{1/2} = \half(h_0 + h_1)$ 
	need to satisfy 
	\begin{equation}
	\label{eqn:strict-one}
	\begin{split}
		\langle h_0,X^{1/2}W \rangle &= \langle h_1,X^{1/2}W \rangle \, ,\\
		\normbig{\proj_{W^{\perp}}(X^{1/2}h_0)} &= \normbig{\proj_{W^{\perp}}(X^{1/2}h_1)}\, , \\
		\normbig{\proj_{W^{\perp}}(X^{1/2}h_{1/2})} &= 
			\half\left(\normbig{\proj_{W^{\perp}}(X^{1/2}h_0)} + \normbig{\proj_{W^{\perp}}(X^{1/2}h_1)}\right).
	\end{split}
	\end{equation}
	Now,  the first equality of Eq.~\eqref{eqn:strict-one} is equivalent to
	\begin{equation}
	\label{eqn:strict-two}
		\langle X^{1/2}(h_0 - h_1), W\rangle = 0, 
	\end{equation}
	and the last two equalities are equivalent to
	\begin{equation}
	\label{eqn:strict-three}
	\proj_{W^{\perp}}(X^{1/2} (h_0 - h_1)) = 0~~\P-a.s.
	\end{equation}
	Thus, if both inequalities $(i)$ and $(ii)$ become equalities,
	it must happen that 
	\begin{equation*}
	X^{1/2} (h_0 - h_1) = 0~~\P-a.s.
	\end{equation*}
	which implies $\P(h_0 \neq h_1)=0$  since we assumed $\P(X > 0) = 1$. This completes the proof that $\asL: \cL^2(\P)\to \R$ is 
	strictly convex in the sense of Eq.~\eqref{eqn:strict-convex-meaning}.

Strict convexity implies immediately uniqueness of the minimizer of $\asL$. Given two minimizers $h\opt$ and
$\tilde{h}\opt$, we must have $\P(\tilde{h}\opt \neq h\opt) = 0$, because otherwise $h_{1/2} = (h\opt+\tilde{h}\opt)/2$ would achieve a 
strictly smaller cost.

\subsubsection{Proof of Lemma \ref{lemma:technical-2}}

Throughout this proof, we keep $\P$ fixed, and hence we drop it from the the arguments of $\asL$ to
simplify notations (hence writing $\asL_{\psi,\kappa}(h) = \asL_{\psi,\kappa,\P}(h)$), and from 
$\<h_1,h_2\>_{\P}$ and $\|h\|_{\P}$.

We organize the proof in three parts depending on which of the three conditions {\sf A1}, {\sf A2} or {\sf A3} holds.

\vspace{0.25cm}

\noindent\emph{Condition {\sf A1} holds:}
	\begin{equation}
	\label{eqn:first-assumption}
		\psi\opt(0)^{1/2} \ge \psi^{1/2} \norm{\proj_{W^{\perp}}(G)} - \eta\, . 
	\end{equation}
	Define the constant $c_{\kappa}$ by
	\begin{equation}
	\label{eqn:def-c-kappa}
		c_{\kappa} = \min_{|c_1| \le x_{\max}^{1/2}, 0 \le c_2 \le x_{\max}^{1/2}} 
				\left\{F_{\kappa}(c_1, c_2) - F_0(c_1, c_2)\right\}. 
	\end{equation}  
Note that $c_{\kappa}>0$ strictly since $\kappa\mapsto F_{\kappa}(c_1,c_2)$ is strictly increasing.

	Since for any $h$ such that $\norm{h} \le 1$, we have 
	\begin{equation}
	\label{eqn:basic-estimate}
		\langle h, X^{1/2} W\rangle \le \normbig{X^{1/2}W} \norm{h} \le x_{\max}^{1/2}
			~~\text{and}~~
		\normbig{\proj_{W^{\perp}}(X^{1/2}h)} \le \normbig{X^{1/2} h} \le x_{\max}^{1/2}, 
	\end{equation}
	the definition of $c_{\kappa}$ at Eq.~\eqref{eqn:basic-estimate} implies for any 
	$h$ satisfying $\norm{h} \le 1$,
	\begin{equation}
		\label{eqn:L-strict-kappa}
	\asL_{\psi, \kappa}(h) \ge \asL_{\psi, 0}(h) + \psi^{-1/2}c_{\kappa}.
	\end{equation} 
        Now we note that, by Cauchy-Schwartz,
        \begin{equation}
        \label{eqn:weird-bound-exp}
        		\big|\langle X^{1/2} h, \proj_{W^{\perp}} (g) \rangle \big| 
			= \big|\langle \proj_{W^{\perp}}(X^{1/2} h),  \proj_{W^{\perp}}(G) \rangle \big| 
				\le \normbig{\proj_{W^{\perp}}(X^{1/2}h)}\norm{\proj_{W^{\perp}}(G)}
	%		= \langle h, \proj_{W^{\perp}}(G)\rangle + \langle g, W\rangle \langle h, X^{1/2} W\rangle
       	% 			= \langle \proj_{W^{\perp}}(h), g\rangle + \langle g, W\rangle \langle h, X^{1/2} W\rangle.
        \end{equation}
        Thus we have that for any $h \in \cL^2(\P)$ satisfying $\norm{h} \le 1$, 
        \begin{equation*}
        \begin{split}
        		\asL_{\psi, \kappa}(h) &\stackrel{(i)}{\ge} \asL_{\psi, 0}(h) + \psi^{-1/2}c_{\kappa}
       		 	= \psi^{-1/2} \left(F_0\left(\langle h,X^{1/2}W \rangle,\normbig{\proj_{W^{\perp}}(X^{1/2}h)}\right) + c_{\kappa}\right)
        				+ \langle X^{1/2} h, \proj_{W^{\perp}} (G)\rangle \\
       	 	&\stackrel{(ii)}{\ge} \left(\psi^{-1/2} \cdot \min_{c \in \R} F_0(c, 1)\right) \cdot \normbig{\proj_{W^{\perp}}(X^{1/2}h)} + \psi^{-1/2}c_{\kappa}
        				- \norm{\proj_{W^{\perp}}(G)}\normbig{\proj_{W^{\perp}}(X^{1/2}h)} \\ %- x_{\max}\left|\langle g, W\rangle\right| \\
		&= \left((\psi\opt(0)/\psi)^{1/2} - \norm{\proj_{W^{\perp}}(G)}\right) \normbig{\proj_{W^{\perp}}(X^{1/2}h)} + \psi^{-1/2} c_{\kappa} \\
		&\stackrel{(iii)}{\ge} \psi^{-1/2}c_{\kappa} - \psi^{-1/2} x_{\max}^{1/2}\eta \stackrel{(iv)}{\ge} \psi^{-1/2}x_{\max}^{1/2} \eta.
        \end{split}
        \end{equation*}
        where in $(i)$, we use Eq.~\eqref{eqn:L-strict-kappa}; in $(ii)$, we use the bound in Eq.~\eqref{eqn:weird-bound-exp}
        and the fact that $F_0(c_1, c_2) = c_2 F_0(c_1/c_2, 1) \ge c_2 \min_{c\in \R} F_0(c, 1)$ for any $(c_1, c_2) \in \R\times \R_+$;
         in $(iii)$, we use the assumption in Eq.~\eqref{eqn:first-assumption} and the fact that $\normbig{\proj_{W^{\perp}}(X^{1/2}h)} \le \norm{X^{1/2}h} 
        \le x_{\max}^{1/2}$ for all $h$ satisfying $\norm{h} \le 1$; $(iv)$ follows from the definition of $c_{\kappa}$ and $\eta$. 
        This proves that 
        \begin{equation*}
       		 \asL_{\psi, \kappa}\opt = \min_{h: \norm{h} \le 1} \asL_{\psi, \kappa}(h) \ge \psi^{-1/2}x_{\max}^{1/2}\eta > 0.
        \end{equation*}

\vspace{0.25cm}

\noindent\emph{Condition {\sf A2} holds:}
	\begin{equation}
	\label{eqn:priori-two-case-one}
	\partial_1 F_{\kappa}(\zeta, 0) \le \eta
	~~\text{and}~~
	\partial_2 F_{\kappa}(\zeta , 0) \ge 
				\normbig{\partial_1 F_{\kappa}(\zeta , 0)(1-\zeta^2 X^{-1}) W + \psi^{1/2} \proj_{W^{\perp}}(G)} - \eta.
	\end{equation}
	To start with, by Lemma~\ref{lemma:F-convex-increasing}, $F_{\kappa}$ is convex. Hence, for any $h$, 
	\begin{equation}
	\label{eqn:priori-two-convex}
	\begin{split}
	\asL_{\psi, \kappa}(h) &= \psi^{-1/2} \cdot F_{\kappa}(\langle h, X^{1/2} W\rangle, \normbig{\proj_{W^{\perp}}(X^{1/2}h)}) 
		+  \langle \proj_{W^{\perp}}(G), X^{1/2} h\rangle\\
	&\ge \psi^{-1/2} \left(F_{\kappa}(\zeta, 0) +
		\partial_1 F_{\kappa}(\zeta , 0)\left(\langle h, X^{1/2} W\rangle -\zeta \right) + 
			\partial_2 F_{\kappa}(\zeta , 0) \normbig{\proj_{W^{\perp}}(X^{1/2}h)} \right) \\
		&~~~~~~~~~~~~+ \langle \proj_{W^{\perp}}(G),  \proj_{W^{\perp}} (X^{1/2} h)\rangle. 
	\end{split}
	\end{equation}
	By Cauchy-Schwartz inequality, the inequality below holds for any $h$ such that $\norm{h} \le 1$: 
	\begin{equation}
	\label{eqn:weird-CS}
	\zeta =  \zeta^2\normbig{X^{-1/2}W} \ge \zeta^2 \langle h, X^{-1/2}W \rangle = 
		\zeta^2 \langle X^{1/2} h, X^{-1} W\rangle.
	\end{equation}
	As a consequence, we obtain for any $h$ such that $\norm{h} \le 1$, 
	\begin{equation}
	\label{eqn:key-bound-weird-CS}
	\begin{split}
	\langle h, X^{1/2} W\rangle -\zeta &\le \langle X^{1/2}h, W \rangle -  \langle X^{1/2} h, \zeta^2 X^{-1} W\rangle
		= \langle X^{1/2} h, (1-\zeta^2 X^{-1}) W\rangle  \\
%		&= \langle \proj_{W^{\perp}}(X^{1/2}h), (1-\zeta^2 X^{-1}) W\rangle.
	\end{split}
	\end{equation}
	As $\langle W, (1-\zeta^2 X^{-1}) W\rangle = 0$ (since $\|W\|=1$), we obtain
	$(1-\zeta^2 X^{-1}) W = \proj_{W^{\perp}}\left((1-\zeta^2 X^{-1}) W\right)$. Hence, 
	\begin{equation}
	\label{eqn:key-bound-weird-CS-final}
		\langle h, X^{1/2} W\rangle -\zeta \le \langle X^{1/2}h, \proj_{W^{\perp}}((1-\zeta^2 X^{-1}) W)\rangle 
			=  \langle \proj_{W^{\perp}}(X^{1/2} h), (1-\zeta^2 X^{-1}) W\rangle
	\end{equation}
	Therefore, we have for all $h$ such that $\norm{h} \le 1$,
	\begin{align}
	&\partial_1 F_{\kappa}(\zeta , 0)\left(\langle h, X^{1/2} W\rangle -\zeta \right) \nonumber \\
	=&~ \left(\partial_1 F_{\kappa}(\zeta , 0) - \eta \right)\left(\langle h, X^{1/2} W\rangle -\zeta \right) 
		+ \eta \left(\langle h, X^{1/2} W\rangle -\zeta \right) \nonumber \\
	\stackrel{(i)}{\ge}&~ \left(\partial_1 F_{\kappa}(\zeta , 0) - \eta \right) \langle \proj_{W^{\perp}}(X^{1/2} h), (1-\zeta^2 X^{-1}) W\rangle
		+ \eta \left(\langle h, X^{1/2} W\rangle -\zeta \right) \nonumber  \\
	{=}&~\partial_1 F_{\kappa}(\zeta , 0)\langle \proj_{W^{\perp}}(X^{1/2} h), (1-\zeta^2 X^{-1}) W\rangle
		+ \eta  \left(\langle h, X^{1/2} W\rangle -\zeta - \langle X^{1/2} h, (1-\zeta^2 X^{-1}) W\rangle\right) \nonumber  \\
	\stackrel{(ii)}{\ge}&~\partial_1 F_{\kappa}(\zeta , 0)\langle \proj_{W^{\perp}}(X^{1/2} h), (1-\zeta^2 X^{-1}) W\rangle
					- 4 \eta x_{\max}^{1/2}.
	\label{eqn:main-term-bound}
	\end{align}
	where in $(i)$ we use Eq.~\eqref{eqn:key-bound-weird-CS-final} and the assumption 
	$\partial_1 F_{\kappa}(\zeta , 0) \le \eta$; %in $(ii)$, we note the identity in Eq.~\eqref{eqn:werid-term-orthogonal-to-W}; 
         in $(ii)$, we use the bounds below that hold for all $h$ with $\norm{h} \le 1$:
	\begin{equation*}
	\begin{split}
	&\left|\langle h, X^{1/2} W\rangle\right| \le \normbig{X^{1/2}W} \le x_{\max}^{1/2}, 
		~~~~~~
	\zeta \le (x_{\max}^{-1/2} \norm{W})^{-1} = x_{\max}^{1/2}\, ,
	 \\
	&\langle X^{1/2} h, (1-\zeta^2 X^{-1}) W\rangle \le \normbig{X^{1/2}(1-\zeta^2 X^{-1}) W} 
	\le \normbig{X^{1/2}W} + \zeta^2 \normbig{X^{-1/2}W} = x_{\max}^{1/2} + \zeta \le 2x_{\max}^{1/2}. 
	\end{split}
	\end{equation*}
	Substituting Eq.~\eqref{eqn:main-term-bound} into Eq.~\eqref{eqn:priori-two-convex}, we have for all 
	$h$ satisfying $\norm{h} \le 1$, 
%	Recall Eq.~\eqref{eqn:weird-bound-exp}, we have %and Cauchy-Schwartz inequality, we have 
%	\begin{equation}
%	\label{eqn:weird-bound-one}
%	\langle h, g \rangle_{X^{1/2}} \ge \langle \proj_{W^{\perp}}(X^{1/2}h), g\rangle - 
%		\left|\langle g, W\rangle\right| \normbig{hX^{1/2}} 
%		\ge \langle \proj_{W^{\perp}}(X^{1/2}h), g\rangle - 
%		x_{\max}\left|\langle g, W\rangle\right|.
%	\end{equation}
	\begin{equation*}
	\begin{split}
	 \asL_{\psi, \kappa}(h) &\ge \psi^{-1/2} \cdot \left(  F_{\kappa}(\zeta, 0) +
		\partial_2 F_{\kappa}(\zeta , 0) \normbig{\proj_{W^{\perp}}(X^{1/2}h)} \right) - 4 \psi^{-1/2}x_{\max}^{1/2}\eta  \\
		&~~~~~~
		+ \langle \psi^{-1/2}\partial_1 F_{\kappa}(\zeta , 0)(1-\zeta^2 X^{-1}) W + \proj_{W^{\perp}}(G), \proj_{W^{\perp}}(X^{1/2}h) \rangle
				\\
		&\stackrel{(i)}{\ge} \psi^{-1/2} F_{\kappa}(\zeta, 0) 	-  4 \psi^{-1/2}x_{\max}^{1/2}\eta \\
		&~~~~~~	+ \psi^{-1/2}\left(\partial_2 F_{\kappa}(\zeta , 0) -
				\normbig{\partial_1 F_{\kappa}(\zeta , 0)(1-\zeta^2 X^{-1}) W + \psi^{1/2}\proj_{W^{\perp}}(G)}\right) \normbig{\proj_{W^{\perp}}(X^{1/2}h)} \\
		&\stackrel{(ii)}{\ge} \psi^{-1/2} F_{\kappa}(\zeta, 0) - 5 \psi^{-1/2}x_{\max}^{1/2}\eta \ge \psi^{-1/2}x_{\max}^{1/2}\eta.
	\end{split}
	\end{equation*}
%\am{Modified what was here before: The previous point $(i)$ did not make sense to me}
	where, in $(i)$, we use the Cauchy-Schwartz inequality and in $(ii)$, we use the assumption \eqref{eqn:priori-two-case-one}
	and the bound $\normbig{\proj_{W^{\perp}}(X^{1/2}h)}\le x_{\max}^{1/2}$ that holds whenever $\norm{h} \le 1$. 
	As a consequence, this proves that 
	\begin{equation*}
       		 \asL_{\psi, \kappa}\opt = \min_{h: \norm{h} \le 1} \asL_{\psi, \kappa}(h) \ge \psi^{-1/2}x_{\max}^{1/2}\eta.
        \end{equation*}

\vspace{0.25cm}

\noindent\emph{Condition {\sf A3} holds:}
	\begin{equation*}
	%\label{eqn:priori-two-case-two}
	\partial_1 F_{\kappa}(-\zeta, 0) \ge -\eta
	~~\text{and}~~
	\partial_2 F_{\kappa}(-\zeta , 0) \ge 
				\normbig{\partial_1 F_{\kappa}(-\zeta , 0)(1-\zeta^2 X^{-1}) W + \psi^{1/2}\proj_{W^{\perp}}(G)} - \eta. 
	\end{equation*}
In this case,  the inequality $\asL_{\psi, \kappa}\opt \ge  \psi^{-1/2}x_{\max}^{1/2}\eta > 0$ follows from essentially the same 
argument as in the previous point. We omit the details.

\subsubsection{Proof of Lemma \ref{lemma:technical-3}}

Throughout the proof, we will drop $\P$ from  subscripts in order to lighten the notations.

By Lemma~\ref{lemma:F-convex-increasing}, the function $h\mapsto \asL_{\psi, \kappa, \P}(h)$ is strictly convex. %, by Lemma \ref{lemma:technical-1}.
Hence, the unique minimizer of problem  \eqref{eqn:def-L-star} is determined by the Karush---Kuhn---Tucker (KKT) conditions.
Namely, $h$ is the minimum of problem~\eqref{eqn:def-L-star} if and only if, 
for some scalar $s$ and some measurable function $Z = Z(g, x, w)$, the following hold
\begin{equation}
\label{eqn:KKT-opt-general}
\begin{split}
& X^{1/2}\proj_{W^{\perp}}(G) 
	+ \psi^{-1/2} X^{1/2} \Big( \partial_1F_{\kappa}(\langle h, X^{1/2} W\rangle, \normbig{\proj_{W^{\perp}}(X^{1/2}h)}) W+ \\
&~~~~~~~~~~~~~~~~~~~~~	
	+ \partial_2 F_{\kappa}(\langle h, X^{1/2} W\rangle, \normbig{\proj_{W^{\perp}}(X^{1/2}h)}) \proj_{W^{\perp}}(Z) \Big) + sh = 0.\\
&~Z = \begin{cases}
			\normbig{\proj_{W^{\perp}}(X^{1/2}h)}^{-1} \cdot \proj_{W^{\perp}}(X^{1/2}h)  ~&\text{if $\normbig{\proj_{W^{\perp}}(X^{1/2}h)} > 0$} \\
			Z^\prime(g, x, w) ~~\text{where $\norm{Z^\prime} \le 1$} ~&\text{if $\normbig{\proj_{W^{\perp}}(X^{1/2}h)}= 0$}. 
		\end{cases}  \\
& s(\norm{h} - 1) \ge 0, ~s \ge 0, ~\norm{h} \le 1.
\end{split}
\end{equation}
%
%\am{Modified the last line. I think there was a mistake}
For completeness, we provide a derivation of the KKT conditions  in Appendix~\ref{sec:proof-KKT-condition}. 

We claim that the KKT conditions \eqref{eqn:KKT-opt-general} imply that
any minimizer $h$ and its associated dual variable $s$ must satisfy 
\begin{equation}
\label{eqn:structure-KKT}
s > 0~~\text{and}~~\normbig{\proj_{W^{\perp}}(X^{1/2}h)}  \neq 0. 
\end{equation}

To show this, \emph{first assume by contradiction $s=0$}. Denote 
\begin{equation*}
c_1 = \langle h, X^{1/2} W\rangle,~\text{and}~c_2 = \normbig{\proj_{W^{\perp}}(X^{1/2}h)}.
\end{equation*}
Since $X > 0$ by assumption, Eq.~\eqref{eqn:KKT-opt-general} now implies
\begin{equation}
\label{eqn:s-equals-zero}
    \proj_{W^{\perp}}(G) + \psi^{-1/2} \cdot \Big( \partial_1F_{\kappa}\left(c_1, c_2\right)W+
    		\partial_2 F_{\kappa}\left(c_1, c_2 \right) \proj_{W^{\perp}}(Z) \Big)= 0.
\end{equation}
By taking inner products with $W$ on both sides of Eq.~\eqref{eqn:s-equals-zero}, and using $\|W\|=1$, we get that
\begin{equation}
\label{eqn:partialone-equals-zero}
    \psi^{-1/2} \cdot \partial_1F_{\kappa}\left(c_1, c_2\right)= 0. 
\end{equation}
Plugging Eq.~\eqref{eqn:partialone-equals-zero} into Eq.~\eqref{eqn:s-equals-zero}, we obtain the identity: 
\begin{equation}
\label{eqn:s-equals-zero-partialtwo}
\psi^{-1/2} \cdot 
	\partial_2F_{\kappa}\left(c_1, c_2\right)  \proj_{W^{\perp}}(Z)=- \proj_{W^{\perp}}(G).
\end{equation}
Note that $\norm{\proj_{W^{\perp}}(Z)} \le \norm{Z} \le 1$. By taking norm 
on both sides of Eq.~\eqref{eqn:s-equals-zero-partialtwo}, we get the bound: 
\begin{equation}
\label{eqn:s-equals-zero-alpha-lower}
    \psi^{-1/2} \cdot \partial_2 F_{\kappa}\left(c_1, c_2\right) \ge \norm{\proj_{W^{\perp}}(G)}.
\end{equation}
Now, we recall Lemma~\ref{lemma:alpha-upper-bound}. 
By Lemma~\ref{lemma:alpha-upper-bound}, 
Eq.~\eqref{eqn:partialone-equals-zero} and 
Eq.~\eqref{eqn:s-equals-zero-alpha-lower} imply that
\begin{equation}
(\psi\opt(0)/\psi)^{1/2}
	= \psi^{-1/2} \cdot \min_{c\in \R} F_0(c, 1) \ge \norm{\proj_{W^{\perp}}(G)}.
\end{equation}
This contradicts our  assumption on $\psi$. We therefore conclude that $s>0$.

Next, again by contradiction, assume $s>0$ but $\normbig{\proj_{W^{\perp}}(X^{1/2}h)} =0$. 
Then $X^{1/2} h = c W$, or 
\begin{equation}
	h = c X^{-1/2} W~~\text{for some $c \in \R$}.
\end{equation} 
Note that $\norm{h} = 1$ since $s > 0$ and the KKT condition, we obtain $|c| = \norm{X^{-1/2}W}^{-1} = \zeta$. 
Now, we divide our discussion into two cases, based on the value of $c = + \zeta$ and $c = - \zeta$:
\begin{enumerate}
\item $c = \zeta$. Multiplying Eq.~\eqref{eqn:KKT-opt-general} by $X^{-1/2}$, we reach the identity: 
	\begin{equation}
	\label{eqn:case-2-KKT}
  		 \proj_{W^{\perp}}(G) + \psi^{-1/2} \big( \partial_1F_{\kappa}(\zeta, 0) W+\partial_2 F_{\kappa}(\zeta, 0) \proj_{W^{\perp}}(Z) \big) 
    			+ s \zeta X^{-1} W = 0.
	\end{equation}
	Taking inner products with $W$ on both sides of Eq.~\eqref{eqn:case-2-KKT}, we obtain (recall: $\zeta = \norm{X^{-1/2}W}^{-1}$)
	\begin{equation}
	\label{eqn:case-2-s-one}
		\psi^{-1/2}\partial_1 F_{\kappa}(\zeta, 0)
			+ s \zeta^{-1}  = 0.
	\end{equation}
	Now we can eliminate the variable $s$ from Eq.~\eqref{eqn:case-2-KKT} and Eq.~\eqref{eqn:case-2-s-one} and get 
	\begin{equation}
	\label{eqn:case-2-s-two}
	\begin{split}
		&\psi^{1/2}\proj_{W^{\perp}}(G) +\partial_1F_{\kappa}(\zeta, 0) W+\partial_2 F_{\kappa}(\zeta, 0) 
	   	\proj_{W^{\perp}}(Z) 
		= \zeta^2 \partial_1 F_{\kappa}(\zeta, 0) X^{-1}W.
	\end{split}
	\end{equation}
	Simple algebraic manipulation of Eq.~\eqref{eqn:case-2-s-two} yields 
	\begin{equation}
		\label{eqn:case-2-KKT-2-one}
		-\partial_2 F_{\kappa}(\zeta, 0) \proj_{W^{\perp}} (Z) 
			= \partial_1 F_{\kappa}(\zeta, 0) (1- \zeta^2 X^{-1} ) W + \psi^{1/2}\proj_{W^{\perp}}(G).
	\end{equation}
	Now that $\norm{\proj_{W^{\perp}} (Z)} \le \norm{Z} \le 1$. By taking norm on both sides of 
	Eq.~\eqref{eqn:case-2-KKT-2-one}, we get
	\begin{equation}
	\partial_2 F_{\kappa}(\zeta, 0) \ge 	
		\normbig{\partial_1 F_{\kappa}(\zeta, 0) (1- \zeta^2 X^{-1} ) W + \psi^{1/2} \proj_{W^{\perp}}(G)}.
	\end{equation}
	Moreover, since $s > 0$, Eq.~\eqref{eqn:case-2-s-one} yields $\partial_1 F_{\kappa}(\zeta, 0) < 0$.
	
	Summarizing, we see that the case where $s>0$, $\normbig{\proj_{W^{\perp}}(X^{1/2}h)} = 0$ can happen, only if 
	\begin{equation*}
	\partial_1 F_{\kappa}(\zeta, 0) \le 0~\text{and}~
	\partial_2 F_{\kappa}(\zeta, 0) \ge 	
		\normbig{\partial_1 F_{\kappa}(\zeta, 0) (1- \zeta^2 X^{-1} ) W + \psi^{1/2}\proj_{W^{\perp}}(G)}.
	\end{equation*}	
	which contradicts the assumed condition on $\psi, \zeta, \Delta$.
\item $c = -\zeta$. Similar to the previous case, one can show that, this can happen only if 
	\begin{equation*}
	\partial_1 F_{\kappa}(-\zeta, 0) \ge 0~\text{and}~
	\partial_2 F_{\kappa}(-\zeta, 0) \ge 	
		\normbig{\partial_1 F_{\kappa}(-\zeta, 0) (1- \zeta^2 X^{-1} ) W +  \psi^{1/2}\proj_{W^{\perp}}(G)}.
	\end{equation*}
	which contradicts the assumed condition on $\psi, \zeta, \Delta$.
\end{enumerate}
Summarizing the above discussion, we have shown the desired result in Eq.~\eqref{eqn:structure-KKT}.

Using the fact that $s>0$ and $\normbig{\proj_{W^{\perp}}(X^{1/2}h)}>0$, we can simplify the KKT condition~\eqref{eqn:KKT-opt-general}. 
Denote
\begin{equation}
\label{eqn:def-c-1-c-2}
c_1 = \langle h, X^{1/2} W\rangle~~\text{and}~~c_2 = \normbig{\proj_{W^{\perp}}(X^{1/2}h)}.
\end{equation}
The KKT condition (i.e., Eq.~\eqref{eqn:KKT-opt-general}) can be equivalently written as: 
\begin{equation*}
%\label{eqn:KKT-opt-general}
\begin{split}
& X^{1/2}\proj_{W^{\perp}}(G) + \psi^{-1/2} X^{1/2} \big( \partial_1F_{\kappa}(c_1, c_2) W+
	\partial_2 F_{\kappa}(c_1, c_2) \proj_{W^{\perp}}(Z) \big) + sh = 0\\
& \norm{h} = 1, ~Z = c_2^{-1} \cdot \proj_{W^{\perp}}(X^{1/2}h)  
\end{split}
\end{equation*}
Note:
$
\proj_{W^{\perp}}(Z) = c_2^{-1} \proj_{W^{\perp}}(X^{1/2}h) = c_2^{-1}\left(X^{1/2} h - \langle X^{1/2}h, W\rangle W\right) 
	= c_2^{-1}\left(X^{1/2} h - c_1 W\right). 
$
The KKT condition can be equivalently represented as 
\begin{equation}
\label{eqn:KKT-opt-general-equiv}
\proj_{W^{\perp}}(G) + \psi^{-1/2}\left(\partial_1 F_{\kappa}(c_1, c_2) W + c_2^{-1}\partial_2 F_{\kappa}(c_1, c_2)(X^{1/2} h - c_1 W)\right)
 + sX^{-1/2} h = 0~\text{and}~\norm{h} = 1.
\end{equation}
The first equation now immediately implies
\begin{equation}
h = -\frac{\psi^{1/2}\proj_{W^{\perp}}(G) + \left(\partial_1 F_{\kappa}(c_1, c_2) -  c_1 c_2^{-1} \partial_2 F_{\kappa}(c_1, c_2)\right) W}
	{ c_2^{-1} \partial_2 F_{\kappa}(c_1, c_2) X^{1/2} + \psi^{1/2} s X^{-1/2}}.
\end{equation}
Plug in the above expression of $h$ into the three equations below (cf. Eq.~\eqref{eqn:def-c-1-c-2} and Eq.~\eqref{eqn:KKT-opt-general-equiv}): 
\begin{equation*}
c_1 = \E_{\P}[WX^{1/2} h], ~~c_1^2 + c_2^2 = \E_{\P}[Xh^2]~~\text{and}~~1 = \E_{\P}[h^2].
\end{equation*}
We derive the following system of equations:  
\begin{equation}
\label{eqn:three-equations}
\begin{split}
-c_1 &= \E_{\P} \left[\frac{( \psi^{1/2}  \proj_{W^{\perp}}(G) +\left(\partial_1 F_{\kappa}(c_1, c_2) -  
	c_1 c_2^{-1} \partial_2 F_{\kappa}(c_1, c_2)\right) W)WX^{1/2}}
		{c_2^{-1} \partial_2 F_{\kappa}(c_1, c_2) X^{1/2} + \psi^{1/2}  s X^{-1/2}}\right]  \\
c_1^2 + c_2^2 &= 
	 \E_{\P} \left[\frac{(\psi^{1/2}  \proj_{W^{\perp}}(G) + \left(\partial_1 F_{\kappa}(c_1, c_2) -  c_1 c_2^{-1} \partial_2 F_{\kappa}(c_1, c_2)\right) W)^2 X}
	{(c_2^{-1} \partial_2 F_{\kappa}(c_1, c_2) X^{1/2} + \psi^{1/2}  s X^{-1/2})^2}\right] \\
1 &= \E_{\P}
	\left[\frac{\left(\psi^{1/2}  \proj_{W^{\perp}}(G) + \left(\partial_1 F_{\kappa}(c_1, c_2) -  c_1 c_2^{-1} \partial_2 F_{\kappa}(c_1, c_2)\right) W\right)^2}
	{(c_2^{-1} \partial_2 F_{\kappa}(c_1, c_2) X^{1/2} + \psi^{1/2}  s X^{-1/2})^2}\right]
\end{split}
\end{equation}
Recall that the minimum $h\opt$ is unique. Thus the value $c_1 = \langle h, X^{1/2} W\rangle$, $c_2 = \normbig{\proj_{W^{\perp}}(X^{1/2}h)}$ 
and hence the value $s$ that satisfy the KKT condition Eq.~\eqref{eqn:KKT-opt-general} is unique. 
Since the KKT condition, i.e., Eq.~\eqref{eqn:KKT-opt-general} is equivalent to Eq.~\eqref{eqn:three-equations}, this 
implies the existence and uniqueness of $(c_1, c_2, s)$ that satisfy the Eq.~\eqref{eqn:three-equations}. Moreover, 
we have that solution $(c_1, c_2)$ satisfies 
\begin{equation*}
|c_1| = |\langle h, X^{1/2} W\rangle| \le \normbig{X^{1/2}h} \le x_{\max}^{1/2}
~~\text{and}~~
c_2 = \normbig{\proj_{W^{\perp}}(X^{1/2}h)} \le \normbig{X^{1/2}h} \le x_{\max}^{1/2}.
\end{equation*}
Now, by taking inner products with $X^{1/2} h$ on both sides of the first equation of Eq.~\eqref{eqn:KKT-opt-general-equiv}, 
we get 
\begin{equation}
\label{eqn:KKT-opt-general-cor-one}
\langle \proj_{W^{\perp}}(G) , X^{1/2} h\rangle + \psi^{-1/2} \left(c_1 \partial_1 F_{\kappa}(c_1, c_2)
	+ c_2 F_{\kappa}(c_1, c_2) \right) + s = 0,
\end{equation}
since
$
\langle X^{1/2}h, (X^{1/2} h - c_1 W) \rangle = \normbig{\proj_{W^{\perp}}(X^{1/2}h)}^2$ and $
	\langle X^{-1/2}h, X^{1/2}h \rangle = \norm{h}^2 = 1$. Now
Eq.~\eqref{eqn:KKT-opt-general-cor-one} leads to the following characterization of $\asL_{\psi, \kappa}\opt$: 
\begin{equation}
\asL_{\psi, \kappa}\opt 
	= \psi^{-1/2}\left(F_{\kappa}(c_1, c_2) - c_1 \partial_1 F_{\kappa}(c_1, c_2)
	- c_2 \partial_2 F_{\kappa}(c_1, c_2) \right) - s,
\end{equation}
where $(c_1, c_2, s)$ on the right-hand side is the unique solution of Eq.~\eqref{eqn:three-equations}.

\subsection{Consequences for $\P=\Q_n$ and $\P=\Q_{\infty}$}
\label{sec:Corollaries}

The technical lemmas in  Section \ref{sec:TechnicalLemmas} can be directly applied to   $\P=\Q_n$ and $\P=\Q_{\infty}$,
yielding some important consequences. Here $\psi^{\low}(\kappa)$ is defined as in the statement of Proposition \ref{proposition:system-of-Eq-T}, see 
Eq.~\eqref{eq:PsilowDef}.
\begin{corollary}\label{coro:Psilow}
If $\psi\le\psi^{\low}(\kappa)$, then, almost surely, 
	\begin{equation}
		\asL\opt_{\psi, \kappa, \Q_{\infty}} > 0~~\text{and}~~
			\liminf_{n} \asL\opt_{\psi, \kappa, \Q_n} > 0.
	\end{equation}
\end{corollary}
\begin{proof}
Define the quantities %\am{This notation is inconsistent with main text, see Eq.~\eqref{eq:ZetaDef}. Please fix.}
\begin{equation}
\zeta_n =  \zeta(\Q_{n}) = \normbig{X^{-1/2}W}_{\Q_{n}}^{-1}
	~~\text{and}~~
\zeta_{\infty} = \zeta(\Q_{\infty}) = \normbig{X^{-1/2}W}_{\Q_{\infty}}^{-1}.\label{eq:ZetaDef-App}
\end{equation}
Then, by definition of $\psi^{\low}(\kappa)$, one of the following conditions  hold:
\begin{enumerate}
\item $\psi\opt(0)^{1/2} \ge \psi^{1/2} = \psi^{1/2} \cdot \norm{\proj_{W^{\perp}}(G)}_{\Q_{\infty}}$ 
\item $\partial_1 F_{\kappa}(\zeta_{\infty}, 0) \le 0$ and 
	$\partial_2 F_{\kappa}(\zeta_{\infty} , 0) \ge 
				\normbig{\partial_1 F_{\kappa}(\zeta_{\infty} , 0)(1-\zeta_{\infty}^2 X^{-1}) W + \psi^{1/2}\proj_{W^{\perp}}(G)}_{\Q_{\infty}}$.
\item $\partial_1 F_{\kappa}(-\zeta_{\infty}, 0) \ge 0$ and 
	$\partial_2 F_{\kappa}(-\zeta_{\infty}, 0) \ge 
				\normbig{\partial_1 F_{\kappa}(-\zeta_{\infty}, 0)(1-\zeta_{\infty}^2 X^{-1}) W + \psi^{1/2}\proj_{W^{\perp}}(G)}_{\Q_{\infty}}$.
\end{enumerate}
Therefore, the assumptions of Lemma~\ref{lemma:technical-2} are satisfied for $\P = \Q_{\infty}$. 
Further recall that, by Eq.~\eqref{eqn:Q-n-converge-to-Q-infty},
\begin{equation}
W_2 \left(\Q_n, \Q_{\infty}\right) \to 0,
\end{equation}
and therefore
\begin{equation}
\label{eqn:convergence-W-2-cor}
\begin{split}
\norm{\proj_{W^{\perp}}(G)}_{\Q_n} &\to \norm{\proj_{W^{\perp}}(G)}_{\Q_{\infty}}  \, ,\\
\zeta_n = \zeta(\Q_n) &\to \zeta(\Q_{\infty}) = \zeta_{\infty} \, ,\\
\normbig{\partial_1 F_{\kappa}(\zeta_n , 0)(1-\zeta_n^2 X^{-1}) W + \psi^{1/2}\proj_{W^{\perp}}(G)}_{\Q_{n}}
&\to 
\normbig{\partial_1 F_{\kappa}(\zeta_{\infty} , 0)(1-\zeta_{\infty}^2 X^{-1}) W + \psi^{1/2}\proj_{W^{\perp}}(G)}_{\Q_{\infty}}\, .
\end{split}
\end{equation}
Therefore, the three conditions stated in Lemma~\ref{lemma:technical-2} are also satisfied when 
$\P = \Q_n$ for sufficiently large $n$. The resut follows by applying the Lemma~\ref{lemma:technical-2}.
\end{proof}

\begin{corollary}\label{coro:PsiLarge}
If $\psi>\psi^{\low}(\kappa)$, then, almost surely:
\begin{enumerate}
\item[$(a)$] For both $\P = \Q_{\infty}$ and $\P = \Q_n$ (and $n$ sufficiently large), the system of equations \eqref{eqn:three-equations}
	has unique solutions $(c_{1, \psi, \kappa}(\P), c_{2, \psi, \kappa}(\P), s_{\psi, \kappa}(\P)) \in \R \times \R_{>0} \times \R_{>0}$.
\item[$(b)$] For $\P = \Q_{\infty}$,  the system of equations \eqref{eqn:three-equations} is equivalent to the system 
\eqref{eqn:system-of-equations-c-1-c-2-s-with-no-G}, and therefore we can identify
	\begin{equation*}
	c_{1, \psi, \kappa}(\Q_{\infty}) = c_1(\psi, \kappa),~ 
	c_{2, \psi, \kappa}(\Q_{\infty}) = c_2(\psi, \kappa),~
	s_{\psi, \kappa}(\Q_{\infty}) = s(\psi, \kappa)\, ,
	\end{equation*}
	where $(c_1(\psi, \kappa), c_2(\psi, \kappa), s(\psi, \kappa))$ is the unique solution of Eq.~\eqref{eqn:system-of-equations-c-1-c-2-s-with-no-G} in $\R \times \R_{>0} \times \R_{>0}$.
\item[$(c)$] For all $n$ sufficiently large, the minimizer  $\hat{\btheta}_{n, \psi, \kappa}^{(2)}$ 
of the problem  \eqref{eqn:xi-2-expansion} satisfies
	\begin{equation*}
		\sqrt{p}\hat{\theta}_{n, \psi, \kappa; i}^{(2)} = 
			-\frac{\psi^{1/2}(\proj_{\bw^{\perp}}(\bg))_i + \left(\partial_1 F_{\kappa}
				(c_{1, n}, c_{2, n}) -  c_{1, n} c_{2, n}^{-1} \partial_2 F_{\kappa}(c_{1, n}, c_{2, n})\right) w_i}
			{ c_{2, n}^{-1} \partial_2 F_{\kappa}(c_{1, n}, c_{2, n}) \lambda_i^{1/2} + \psi^{1/2} s_n \lambda_i^{-1/2}},
	\end{equation*}
	where $\hat{\theta}_{n, \psi, \kappa; i}^{(2)}$ on the LHS denotes the $i$-th coordinate of $\hat{\btheta}_{n, \psi, \kappa}^{(2)}$, 
	and $(c_{1, n}, c_{2, n}, s_n)$ on the RHS denotes $(c_{1, \psi, \kappa}(\Q_n), c_{2, \psi, \kappa}(\Q_n), s_{\psi, \kappa}(\Q_n))$. 		
	Moreover, $\hat{\btheta}_{n, \psi, \kappa}^{(2)}$ satisfies 
	\begin{equation*}
		\langle \hat{\btheta}_{n, \psi, \kappa}^{(2)}, \bLambda_n^{1/2}\bw\rangle = c_{1, \psi, \kappa}(\Q_n)
			~\text{and}~
		\normbig{\proj_{W^{\perp}} \bLambda_n^{1/2} \hat{\btheta}_{n, \psi, \kappa}^{(2)}} = c_{2, \psi, \kappa}(\Q_n).
	\end{equation*}
\item[$(d)$] The following representation holds  for both $\P = \Q_{\infty}$ and $\P = \Q_n$  (and $n$ sufficiently large)
	\begin{equation}
	\begin{split}
		&\asL_{\psi, \kappa, \P}\opt =  
			\psi^{-1/2} \Big(F_{\kappa}(c_{1, \psi, \kappa}(\P), c_{2, \psi, \kappa}(\P)) - c_{1, \psi, \kappa}(\P)  \cdot 
				\partial_1 F_{\kappa}(c_{1, \psi, \kappa}(\P), c_{2, \psi, \kappa}(\P))-  \\
		&~~~~~~~~~~~~~~~~~~~~~~~~~~~~~~~~~~~
			c_{2, \psi, \kappa}(\P) \cdot\partial_2 F_{\kappa}(c_{1, \psi, \kappa}(\P), c_{2, \psi, \kappa}(\P))\Big)-s_{\psi, \kappa}(\P)\, .
	\end{split}
	\end{equation}
\item[$(e)$] The following bounds hold for both $\P = \Q_{\infty}$ and $\P = \Q_n$ (and $n$ sufficiently large)
\begin{align}
	\label{eqn:c-1-c-2-bound}
	|c_{1, \psi, \kappa}(\P)| &\le x_{\max}^{1/2}, ~0 \le c_{2, \psi, \kappa}(\P)  \le x_{\max}^{1/2},  ~0 \le s_{\psi, \kappa}(\P) \le C.%\label{eqn:s-bound}
	\end{align}
Above $C$ is a constant that depends only on $\Q_\infty, \kappa, x_{\max}, \psi$ and not on $n$.
%where $\Delta_{\psi,\kappa}(\Q_{\infty})>0$  is jointly continuous with respect to $\psi,\kappa$ and with respect to $\Q_{\infty}$ (in the $W_2$ topology).
\end{enumerate}
\end{corollary}
\begin{proof}
By definition of $\psi^{\low}(\kappa)$, all the following conditions are satisfied (with $\zeta_{\infty}$ given by Eq.~\eqref{eq:ZetaDef-App})
\begin{enumerate}
\item $\psi\opt(0)^{1/2} < \psi^{1/2} = \psi^{1/2} \cdot \norm{\proj_{W^{\perp}}(G)}_{\Q_{\infty}}$ 
\item Either $\partial_1 F_{\kappa}(\zeta_{\infty}, 0) > 0$ or
	$\partial_2 F_{\kappa}(\zeta_{\infty} , 0) <
				\normbig{\partial_1 F_{\kappa}(\zeta_{\infty} , 0)(1-\zeta_{\infty}^2 X^{-1}) W + \psi^{1/2}\proj_{W^{\perp}}(G)}_{\Q_{\infty}}$.
\item Either $\partial_1 F_{\kappa}(-\zeta_{\infty}, 0) < 0$ or
	$\partial_2 F_{\kappa}(-\zeta_{\infty}, 0) < 
				\normbig{\partial_1 F_{\kappa}(-\zeta_{\infty}, 0)(1-\zeta_{\infty}^2 X^{-1}) W + \psi^{1/2}\proj_{W^{\perp}}(G)}_{\Q_{\infty}}$.
\end{enumerate}
Hence, the assumptions of Lemma~\ref{lemma:technical-3} are satisfied for $\P = \Q_\infty$. 
Since, by Eq.~\eqref{eqn:Q-n-converge-to-Q-infty}, $W_2 (\Q_n, \Q_{\infty}) \to 0$, 
the assumptions of Lemma~\ref{lemma:technical-3} are also satisfied for $\P = \Q_n$ for all sufficiently large  $n$.

Then the claims $(a)$-$(d)$ immediately follow by applying Lemma~\ref{lemma:technical-3}. Also, Lemma~\ref{lemma:technical-3}
implies the first two bounds of Eq.~\eqref{eqn:c-1-c-2-bound}, i.e., $|c_{1, \psi, \kappa}(\P)| \le x_{\max}^{1/2}$
and $c_{2, \psi, \kappa}(\P)  \le x_{\max}^{1/2}$.

It remains to prove the last part of Eq.~\eqref{eqn:c-1-c-2-bound}: for some constant $C$, $s_{\psi, \kappa}(\P) \le C$ is satisfied 
 for $\P = \Q_\infty$ and $\P = \Q_n$ for all sufficiently large  $n$. To do so, 
let us define the functions $V(c_1, c_2, s)$ and $V^{\upp}(c_1, c_2, s)$ by
\begin{equation*}
\begin{split}
V(c_1, c_2, s) &\defeq \frac{\left[\psi^{1/2}\proj_{W^{\perp}}(G) +
		\Big(\partial_1 F_{\kappa}(c_1, c_2) -  c_1 c_2^{-1} \partial_2 F_{\kappa}(c_1, c_2)\Big) W\right]^2}
			{(c_2^{-1} \partial_2 F_{\kappa}(c_1, c_2)
                          X^{1/2} + \psi^{1/2} s X^{-1/2})^2} -1\, ,\\
V^{\upp}(c_1, c_2, s) &\defeq \psi^{-1} s^{-2} X\left[\psi^{1/2} \proj_{W^{\perp}}(G) +  \Big(\partial_1 F_{\kappa}(c_1, c_2) -  
		c_1 c_2^{-1} \partial_2 F_{\kappa}(c_1, c_2)\Big) W\right]^2 - 1.
\end{split}
\end{equation*}
Recall that $(c_1, c_2, s) = (c_{1, \psi, \kappa}(\P), c_{2, \psi, \kappa}(\P), s_{\psi, \kappa}(\P))$ satisfies the system of 
equations \eqref{eqn:system-of-equations-c-1-c-2-s-with-no-G} for either $\P = \Q_{\infty}$ or $\P = \Q_n$ (and $n$
is sufficiently large), whence
\begin{equation}
\E_{\P} \left[V(c_1(\P), c_2(\P), s(\P))\right] = 0\, .
\end{equation}
Lemma~\ref{lemma:F-convex-increasing} implies $\partial_2 F_{\kappa}(c_1, c_2) \ge 0$ for all $(c_1, c_2) \in \R\times \R_{\ge 0}$. 
Thus $V(c_1, c_2, s) \le V^{\upp}(c_1, c_2, s)$ for all $(c_1, c_2, s) \in \R\times \R_{\ge 0} \times \R_{\ge 0}$. Hence, 
for either $\P = \Q_{\infty}$ or $\P = \Q_n$ (and $n$ is sufficiently large)
\begin{equation*}
\E_{\P}  \left[V^{\upp}(c_1(\P), c_2(\P), s(\P))\right] \ge 0, 
\end{equation*}
which by an algebraic manipulation is equivalent to (writing for simplicity $c_1=c_1(\P)$, $c_2=c_2(\P)$, $s=s(\P)$)
\begin{equation}
\label{eqn:V3-equiv}
	s^2 \le \E_{\P} \left\{X\left[\proj_{W^{\perp}}(G) + \psi^{-1/2}  \left(\partial_1 F_{\kappa}(c_1, c_2) -  
		c_1 \cdot c_2^{-1}\partial_2 F_{\kappa}(c_1, c_2)\right) W\right]^2\right\}\, .
\end{equation}
According to Lemma~\ref{lemma:continuity-partial-f}, we can set $\overline{C} = \max_{|c_1| \le x_{\max}^{1/2}, 
c_2 \in [0, x_{\max}^{1/2}]} c_2^{-1} |\partial_2 F_{\kappa}(c_1, c_2)| +|\partial_1 F_\kappa(c_1, c_2)| < \infty$.
As we have $X \le x_{\max}$ by assumption, and $|c_1(\P)| \le x_{\max}^{1/2}$ and $|c_2(\P)| \le x_{\max}^{1/2}$ 
as previously shown, we obtain
\begin{equation*}
%\label{eqn:V3-eqiuv-1}
\begin{split}
s^2 &\le  x_{\max}^2 \E_{\P} \left[\left(\left|\proj_{W^{\perp}}(G)\right| + \psi^{-1/2}  \overline{C} |W|\right)^2\right] 
		\le 2x_{\max}^2 \left[\E_{\P}[\proj_{W^{\perp}}(G)^2] + \psi^{-1}\overline{C}^2\right].
\end{split}
\end{equation*}
where we use the fact that $\E_{\P}[W^2] = 1$. As $W_2(\Q_n,\Q_{\infty})\to 0$
 (Eq.~\eqref{eqn:Q-n-converge-to-Q-infty}), we obtain the limit
\begin{equation}
\label{eqn:proj-W-g-converge}
\lim_{n\to \infty} \E_{\Q_n}[( \proj_{W^{\perp}}(G))^2] = \E_{\Q_{\infty}} [(\proj_{W^{\perp}}(G))^2] = 1. 
\end{equation}	
which implies $s_{\psi, \kappa}(\P) \le 2x_{\max}^2 (2+\psi^{-1}\overline{C}^2)$
for $\P = \Q_\infty$ and for $\P = \Q_n$ for all sufficiently large $n$.
\end{proof}

\subsection{Proof of Proposition \ref{proposition:system-of-Eq-T}}
\label{sec:ProofEqT}
\indent\indent
Point $(a)$ follows immediately from Corollary
\ref{coro:PsiLarge}.$(a)$.  Indeed, %as we already pointed out, 
the system of equations \eqref{eqn:three-equations}
coincides with the system \eqref{eqn:system-of-equations-c-1-c-2-s-with-no-G} for $\P = \Q_{\infty}$.

By Corollary \ref{coro:PsiLarge}.$(d)$, we also have for $\psi > \psi^{\low}(\kappa)$,
\begin{equation}
\label{eqn:T-to-asl}
	T(\psi, \kappa) = \asL\opt_{\psi, \kappa, \Q_{\infty}}\, .
\end{equation}

First of all, we claim that $(\psi, \kappa) \to T(\psi, \kappa)$ is continuous and strictly increasing with respect to $\kappa$, and strictly decreasing with respect to $\psi$.

In order to prove this claim, recall that, by Eq.~\eqref{eqn:def-L-star},
\begin{align}
\asL_{\psi, \kappa, \P}\opt &= \min\Big\{ \asL_{\psi, \kappa, \P}(h) \mid \norm{h}_{\P} \le 1\Big\}\, ,\\
\asL_{\psi, \kappa, \P}(h) &=  \psi^{-1/2} \cdot F_{\kappa}\left(\<h, X^{1/2}W \>_{\P},\normbig{\proj_{W^{\perp}}(X^{1/2}h)}_{\P}\right)+
	\< h, X^{1/2} \proj_{W^{\perp}}(G) \>_{\P}.
\end{align}
Notice that: $(i)$~for any $c_1,c_2\in\reals$, $\kappa \to F_{\kappa}(c_1,c_2)$ is continuous and strictly increasing;
$(ii)$~as a consequence, for any fixed $h$, $\asL_{\psi, \kappa, \P}(h)$ is strictly increasing with respect to $\kappa$ and decreasing with respect to $\psi$;
$(iii)$~the minimum $\asL\opt_{\psi, \kappa, \Q_{\infty}}$ in the above optimization problem
is achieved  by some $h_{\psi, \kappa}\opt \in \cL^2(\Q_{\infty})$ 
with $\normsmall{h_{\psi, \kappa}\opt}_{\Q_{\infty}} \le 1$ (see Lemma \ref{lemma:technical-1}). 
The claim that $(\psi, \kappa) \to T(\psi, \kappa)$ is continuous and strictly increasing with respect to $\kappa$ then follows by a standard argument.
This proves the continuity of $T$ in point $(b.i)$ and the monotonicity properties in points $(b.ii)$ and $(b.iii)$.

 Next, we claim that 
\begin{equation*}
		\lim_{\psi \to \infty} T(\psi, \kappa) = \lim_{\psi \to \infty} \asL\opt_{\psi, \kappa, \Q_{\infty}} < 0\, .
\end{equation*}
Indeed, let $\bar{h} = X^{-1/2}G/\normsmall{X^{-1/2}G}_{\Q_{\infty}}$. Then $\norm{\bar{h}}_{\Q_{\infty}} \le 1$ and
$\langle \bar{h}, X^{1/2} \proj_{W^{\perp}}(G)\rangle_{\Q_{\infty}} = 1/\normsmall{X^{-1/2}G}_{\Q_\infty} > 0$.
By definition, $\asL\opt_{\psi, \kappa, \Q_{\infty}} \le L_{\psi, \kappa, \Q_{\infty}}(\bar{h})$. This implies
	\begin{equation*}
		\asL\opt_{\psi, \kappa, \Q_{\infty}}  \le \psi^{-1/2} \cdot \max_{|c_1| \le x_{\max}^{1/2}, c_2 \in [0, x_{\max}^{1/2}]}
			F_{\kappa}(c_1, c_2) - 1/\normsmall{X^{-1/2}G}_{\Q_\infty}.
	\end{equation*}
	Now the desired result follows by taking $\psi \to \infty$. This proves the first bound in point $(b.ii)$. The second bound in 
	point $(b.ii)$ follows because
	\begin{equation*}
		\lim_{\psi \searrow \psi^{\low}(\kappa)} T(\psi, \kappa) \stackrel{(i)}{=} 
			\lim_{\psi \searrow \psi^{\low}(\kappa)} \asL\opt_{\psi, \kappa, \Q_{\infty}}
			\ge \asL\opt_{\psi^{\low}(\kappa), \kappa, \Q_{\infty}} \stackrel{(ii)}{>} 0.
	\end{equation*}
	where $(i)$ is due to Eq~\eqref{eqn:T-to-asl} and (ii) holds because of Corollary \ref{coro:Psilow}.
	
%	\am{What about the second bound in \eqref{eq:Tpsi}??}

 Third, we show that 
	\begin{equation*}
		\lim_{\kappa \to \infty} T(\psi, \kappa) = \lim_{\kappa\to \infty} \asL\opt_{\psi, \kappa, \Q_{\infty}} = \infty.
	\end{equation*}
	This is due to the following bound on $\asL\opt_{\psi, \kappa, \Q_{\infty}}$: 
	\begin{equation*}
	\asL\opt_{\psi, \kappa, \Q_{\infty}} \ge \psi^{-1/2} \cdot \min_{|c_1| \le x_{\min}^{1/2}, c_2 \in [0, x_{\min}^{1/2}]}
		F_{\kappa}(c_1, c_2) - x_{\max}^{1/2}\, ,
	\end{equation*}
	and the fact that $\lim_{\kappa\to \infty} \min_{|c_1| \le x_{\max}^{1/2}, c_2 \in [0, x_{\max}^{1/2}]}
		F_{\kappa}(c_1, c_2) = \infty$. This concludes the proof of point $(b.iii)$.

Last, we show that $c_1(\cdot, \cdot), c_2(\cdot, \cdot), s(\cdot, \cdot)$ are continuous function on the 
	domain $\left\{(\psi, \kappa): \psi > \psi^{\low}(\kappa)\right\}$. Pick any point $(\psi_0, \kappa_0)$ such 
	that $\psi_0 > \psi^{\low}(\kappa_0)$. Let $\{\psi_l\}_{l\in \N}$ and $\{\kappa_l\}_{l \in \N}$ be two 
	sequences such that $\psi_l \to \psi_0$ and $\kappa_l \to \kappa_0$. It suffices to show that 
	\begin{equation}
	\label{eqn:c-1-c-2-s-convergence}
		\lim_{l\to \infty} (c_1(\psi_l, \kappa_l), c_2(\psi_l, \kappa_l), s(\psi_l, \kappa_l))
			= (c_1(\psi_0, \kappa_0), c_2(\psi_0, \kappa_0), s(\psi_0, \kappa_0)).
	\end{equation}
	Corollary \ref{coro:PsiLarge} implies that for some constant $C$ independent of $n$, the following holds for all $l = 0$ 
	and for all sufficiently large $l \in \N$:  
	\begin{equation}
	\begin{split}
	c_1(\psi_l, \kappa_l) &= c_{1, \psi_l, \kappa_l}(\Q_{\infty}) \in [-x_{\max}^{1/2}, x_{\max}^{1/2}], \\~
	c_2(\psi_l, \kappa_l) &= c_{2, \psi_l, \kappa_l}(\Q_{\infty}) \in [0, x_{\max}^{1/2}], \\~
	s(\psi_l, \kappa_l) &= s_{\psi_l, \kappa_l}(\Q_{\infty})~~ \in [0,C].
	\end{split}
	\end{equation}
	Write  $\mathcal{S}\defeq [-M, M]  \times [0,  M] \times [0, M]$ for $M = \max\{C, x_{\max}^{1/2}\}$. Then 
	$\left(c_1(\psi_l, \kappa_l), c_2(\psi_l, \kappa_l), s(\psi_l, \kappa_l)\right) \in \mathcal{S}$ for all large enough $l$. Below 
	we show that any limit point of 
	$\{(c_1(\psi_l, \kappa_l), c_2(\psi_l, \kappa_l), s(\psi_l, \kappa_l))\}_{l\in \N}$ must be 
	$(c_1(\psi_0, \kappa_0), c_2(\psi_0, \kappa_0), s(\psi_0, \kappa_0))$. 
	To do so, take any limit point of $(c_1(\psi_l, \kappa_l), c_2(\psi_l, \kappa_l), s(\psi_l, \kappa_l))$, and denote it by
	$(\tilde{c}_1\opt, \tilde{c}_2\opt, \tilde{s}\opt)$. It is clear that $(\tilde{c}_1\opt, \tilde{c}_2\opt, \tilde{s}\opt) \in \mathcal{S}$, 
	and moreover, $(\tilde{c}_1\opt, \tilde{c}_2\opt, \tilde{s}\opt)$ must satisfy the system of 
	equations \eqref{eqn:three-equations} for $\psi = \psi_0, \kappa = \kappa_0, \P = \Q_{\infty}$. 
	The next lemma shows that both $\tilde{c}_2\opt$ and $\tilde{s}^*$ must be non-zero. 
	\begin{lemma}
	\label{lemma:nonzero-c-2-t}
	Assume $\psi > \psi\opt(0)$. Suppose that $(c_1, c_2, s) \in \R \times \R_{\ge 0} \times \R_{\ge 0}$ satisfies the system of equations 
	\eqref{eqn:three-equations} (with $\P=\Q_{\infty}$)\footnote{Note that the R.H.S. of Eq.~\eqref{eqn:three-equations} 
	can be properly defined for all $(c_1, c_2, s) \in \R \times \R_{\ge 0} \times \R_{\ge 0}$; see Lemma~\ref{lemma:continuity-partial-f}).}. 
	Then we must have that $c_2 \neq 0$ and $s \neq 0$.
	\end{lemma}
	We therefore know that $(\tilde{c}_1\opt, \tilde{c}_2\opt, \tilde{s}\opt)$ satisfies both $\tilde{c}_2\opt > 0$ and $\tilde{s}\opt> 0$ 
	and the system of equations 
      \eqref{eqn:three-equations}  corresponding to $\psi = \psi_0, \kappa = \kappa_0, \P = \Q_{\infty}$. Now since that solution 
	is known to be unique, by Corollary \ref{coro:PsiLarge}.$(a)$,
 we conclude  that $(\tilde{c}_1\opt, \tilde{c}_2\opt, \tilde{s}\opt) = (c_1(\psi_0, \kappa_0), c_2(\psi_0, \kappa_0), s(\psi_0, \kappa_0))$.
	This proves the convergence statement \eqref{eqn:c-1-c-2-s-convergence}, and concludes the proof of point $(b.i)$.

\subsubsection{Proof of Lemma \ref{lemma:nonzero-c-2-t}}
\label{sec:lemma-nonzero-c-2-t}
First assume by contradiction that $c_2 = 0$. Write $G(c_1, 0) = \lim_{c_2 \to 0^+} c_2^{-1} \partial_2 F_{\kappa}(c_1, c_2)$. 
Set  $h = h(G, X, W)$ to be 
\begin{equation}
\label{eqn:real-form-of-h}
h = -\frac{\psi^{1/2}\proj_{W^{\perp}}(G) + \left(\partial_1 F_{\kappa}(c_1, c_2) -  c_1 G(c_1, 0)\right) W}
	{ G(c_1, 0) X^{1/2} + \psi^{1/2} s X^{-1/2}}.
\end{equation}
The condition that  $(c_1, 0, s)$ is a solution of  the system of equations \eqref{eqn:three-equations} 
is then equivalent to 
\begin{equation}
	\E_{\Q_{\infty}}[hWX^{1/2}] = -c_1,~~\E_{\Q_{\infty}}[h^2 X] = c_1^2,~~\E_{\Q_{\infty}}[h^2] = 1.
\end{equation}
As $\E_{\Q_{\infty}}[W^2] = 1$, we obtain that $\E_{\Q_\infty} (hX^{1/2} - c_1W)^2 =0$. This shows that 
$h =h(G, X, W) = c_1 X^{-1/2}W$ holds almost surely.  A comparison of this form with Eq.~\eqref{eqn:real-form-of-h} 
(noticing that $G \sim \normal(0, 1)$ under $\Q_\infty$) yields contradiction. As a result, we have shown 
that $c_2 \neq 0$ for any solution of  the system of equations \eqref{eqn:three-equations}.

Next assume by contraction that $s = 0$. Say that $(c_1, c_2, 0)$ is a solution of 
Eq.~\eqref{eqn:three-equations} (with $\P=\Q_{\infty}$) for some $c_1$ and some $c_2 \neq 0$. Denote $h = h(G, X, W)$ to be 
the function 
\begin{equation}
\label{eqn:def-h-s=0}
	h = \frac{c_2 \proj_{W^{\perp}}(G) + \psi^{1/2}(c_2 \partial_1 F_{\kappa}(c_1, c_2)
			- c_1 \partial_2 F_{\kappa}(c_1, c_2))W }{\psi^{1/2}\partial_2 F_{\kappa}(c_1, c_2) X^{1/2} }.
\end{equation}
The condition that  $(c_1, c_2, 0)$ is a solution of  the system of equations \eqref{eqn:three-equations} 
is then equivalent to 
\begin{equation}
\label{eqn:equiv-t-c-2=0}
	\E_{\Q_{\infty}}[hWX^{1/2}] = -c_1,~~\E_{\Q_{\infty}}[h^2 X] = c_1^2 + c_2^2,~~\E_{\Q_{\infty}}[h^2] = 1.
\end{equation}
Since $\E_{\Q_{\infty}}[W^2] = 1$ and $\E_{\Q_{\infty}} [\proj_{W^{\perp}}(G) W] = 0$, the first of these equations is equivalent to 
\begin{equation*}
	\psi^{1/2} c_2 \partial_1 F_{\kappa}(c_1, c_2) = 0,
\end{equation*}
and since $c_2 \neq 0$, we conclude that 
\begin{equation}
\label{eqn:equiv-t=0-one}
	 \partial_1 F_{\kappa}(c_1, c_2) = 0.
\end{equation}
Using Eq.~\eqref{eqn:equiv-t=0-one}, and the fact that $W$ is independent of $G$ under $\Q_{\infty}$, the second equation of \eqref{eqn:equiv-t-c-2=0} 
is equivalent to 
\begin{equation*}
	c_2^2 ((\psi^{1/2}\partial_2 F_{\kappa}(c_1, c_2))^2 - 1)  = 0.
\end{equation*}
Again, since $c_2 \neq 0$ and $\partial_2 F_{\kappa}(c_1, c_2) > 0$ by Lemma~\ref{lemma:F-convex-increasing}, 
we conclude that 
\begin{equation}
	\psi^{1/2}\partial_2 F_{\kappa}(c_1, c_2) = 1.
\end{equation}
Now Lemma~\ref{lemma:alpha-upper-bound} implies that 
\begin{equation}
(\psi/\psi\opt(0))^{1/2}
	= \psi^{1/2} \cdot \min_{c\in \R} F_0(c, 1) \ge 1,
\end{equation}
which contradicts our assumed condition on $\psi$. This shows $s \neq 0$.

\newcommand{\asto}{\stackrel{{\rm a.s.}}{\to}}

\section{Analysis of the Gordon's optimization problem: Proof of Proposition~\ref{proposition:xi-n-2-xi-infty}}
This section builds upon the notation and results in Section~\ref{sec:all-property-optimization}. As a 
kind suggestion, the reader needs to go over all the main results in 
Section~\ref{sec:all-property-optimization} before reading the rest of the section.  

Notice that $\xi_{n, \psi, \kappa}^{(2)} = \asL\opt_{\psi, \kappa, \Q_n}$.  Therefore, 
point $(a)$ follows by  Corollary \ref{coro:Psilow} and we can assume hereafter $\psi>\psi^\low(\kappa)$.

We claim that the following holds for any $\psi>\psi^\low(\kappa)$:
        \begin{equation}
        \label{eqn:converge-c-1-c-2-s}
        \lim_{n \to \infty, p/n\to \psi}
        	(c_{1, \psi, \kappa}(\Q_n), c_{2, \psi, \kappa}(\Q_n), s_{\psi, \kappa}(\Q_n)) 
		= (c_{1, \psi, \kappa}(\Q_{\infty}), c_{2, \psi, \kappa}(\Q_{\infty}), s_{\psi, \kappa}(\Q_{\infty})).
        \end{equation}
Before proving this claim, let us show that it implies point $(b)$:
        \begin{itemize}
        \item Equation \eqref{eqn:goal-finite-infinite} follows from the Corollary \ref{coro:PsiLarge}.$(d)$ and Eq.~\eqref{eqn:converge-c-1-c-2-s}.
        \item The first limit in Eq.~\eqref{eqn:lim-btheta_n_2_c1c2} follows from
        \begin{equation*}
        		\lim_{n \to \infty} \langle \xi_{n, \psi, \kappa}^{(2)}, \bLambda^{1/2}\bw \rangle
			\stackrel{(a)}{=} \lim_{n\to \infty} c_{1, \psi, \kappa}(\Q_n)
				\stackrel{(b)}{=} c_{1, \psi, \kappa}(\Q_{\infty})  
			= c_1(\psi, \kappa)\, ,
	\end{equation*}
        where $(a)$ is a consequence of Corollary \ref{coro:PsiLarge}.$(c)$ and $(b)$ follows from Eq.~\eqref{eqn:converge-c-1-c-2-s}.
	\item The second limit in Eq.~\eqref{eqn:lim-btheta_n_2_c1c2} follows from the same argument
	\begin{equation*}
		\lim_{n \to \infty} \normbig{\hat{\btheta}_{n, \psi, \kappa}^{(2)}}_{\bLambda_n}
			= \left(c_{1, \psi, \kappa}(\Q_\infty)^2 + c_{2, \psi, \kappa}(\Q_\infty)^2\right)^{1/2}
			= \left(c_1(\psi, \kappa)^2 + c_2(\psi, \kappa)^2\right)^{1/2}.
        \end{equation*}
        \item Equation \eqref{eqn:W-2-converge} can be derived from Corollary \ref{coro:PsiLarge}.$(c)$ and Eq.~\eqref{eqn:converge-c-1-c-2-s}. 
        	Let us define 
	\begin{equation*}
	\begin{split}
	H_{n, \psi, \kappa} (G, X, W) &= -\frac{\psi^{1/2}(\proj_{\bW^{\perp}, \Q_n}(G)) + \left(\partial_1 F_{\kappa}
				(c_{1, n}, c_{2, n}) -  c_{1, n} c_{2, n}^{-1} \partial_2 F_{\kappa}(c_{1, n}, c_{2, n})\right) W}
			{ c_{2, n}^{-1} \partial_2 F_{\kappa}(c_{1, n}, c_{2, n}) X^{1/2} + \psi^{1/2} s_n X^{-1/2}}, \\
	H_{\infty, \psi, \kappa} (G, X, W) &= -\frac{\psi^{1/2}G + \left(\partial_1 F_{\kappa}
				(c_{1, \infty}, c_{2, \infty}) -  c_{1, \infty} c_{2, \infty}^{-1} \partial_2 F_{\kappa}(c_{1, \infty}, c_{2, \infty})\right) W}
			{ c_{2, \infty}^{-1} \partial_2 F_{\kappa}(c_{1, \infty}, c_{2, \infty}) X^{1/2} + \psi^{1/2} s_\infty X^{-1/2}},	
	\end{split}
	\end{equation*}
	where $(c_{1, n}, c_{2, n}, s_n) = (c_{1, \psi, \kappa}(\Q_n), c_{2, \psi, \kappa}(\Q_n), s_{\psi, \kappa}(\Q_n))$ and 
	$(c_{1, \infty}, c_{2, \infty}, s_\infty) = (c_{1, \psi, \kappa}(\Q_\infty), c_{2, \psi, \kappa}(\Q_\infty), s_{\psi, \kappa}(\Q_\infty))$. 
	Corollary \ref{coro:PsiLarge}.$(c)$ shows that the empirical distribution induced by 
	$\{(\lambda_i, \bar{w}_i, \sqrt{p}\hat{\btheta}_{n, \psi, \kappa; i}^{(2)})\}_{i \in [p]}$
	is just the same as the distribution of $(X_n, W_n, H_{n, \psi, \kappa} (G_n, X_n, W_n))$ where 
	$(G_n, X_n, W_n) \sim \Q_n$\footnote{When talking about the distribution of $(X_n, W_n, H_{n, \psi, \kappa} (G_n, X_n, W_n))$, the 
	numbers 
	$(c_{1, n}, c_{2, n}, c_{3, n})$ in the definition of $H_{n, \psi, \kappa} (G_n, X_n, W_n)$ are viewed as deterministic, 
	and only $(G_n, X_n, W_n) \sim \Q_n$ is random.}. Moreover, 
	the definition of $\mathcal{L}_{\psi, \kappa}$ (Definition~\ref{def:KappaE}) says that $\mathcal{L}_{\psi, \kappa}$ is the 
	same as the distribution of $(G_\infty, X_\infty, H_{\infty, \psi, \kappa} (G_\infty, X_\infty, W_\infty))$ where 
	$(G_\infty, X_\infty, W_\infty) \sim \Q_\infty$. 
	Thus our target Equation \eqref{eqn:W-2-converge} is equivalent to 
	\begin{equation*}
		W_2 \left((X_n, W_n, H_{n, \psi, \kappa}(G_n, X_n, W_n)), (X_\infty, W_\infty, H_{\infty, \psi, \kappa}(G_\infty, X_\infty, W_\infty))\right) 
			\to 0.
	\end{equation*}
	By triangle inequality, it suffices to prove the convergence 
	\begin{enumerate}[(a)]
	\item $W_2
		\left((X_n, W_n, H_{\infty, \psi, \kappa}(G_n, X_n, W_n)), (X_\infty, W_\infty, H_{\infty, \psi, \kappa}(G_\infty, X_\infty, W_\infty))\right) \to 0$.
	\item $W_2 \left((X_n, W_n, H_{n, \psi, \kappa}(G_n, X_n, W_n)), (X_n, W_n, H_{\infty, \psi, \kappa}(G_n, X_n, W_n))\right) \to 0$.
	\end{enumerate}
	To prove point $(a)$, we note that 
	$W_2 \left(f(G_n, X_n, W_n), f(G_\infty, X_\infty, W_\infty)\right) \to 0$ for any continuous function $f: \R^3 \to \R^3$ satisfying 
	$\sup_{(G, X, W)\in\R^3} \norm{f(G, X, W)}/{\norm{(G, W, X)}_2} < \infty$. In particular, the continuous function $H_{\infty, \psi, \kappa}(G, X, W)$ 
	satisfies  $\sup_{(G, X, W)\in \R^3} H_{\infty, \psi, \kappa}(G, X, W)/{\norm{(G, W, X)}_2} < \infty$. This is easy to show: 
	the crucial part is to notice that the denominator of $H_\infty(G, X, W)$ satisfies 
	\begin{equation*}
	c_{2, \infty}^{-1} \partial_2 F_{\kappa}(c_{1, \infty}, c_{2, \infty}) X^{1/2} + \psi^{1/2} s_\infty X^{-1/2} 
	 \ge \psi^{1/2} s_\infty X^{-1/2}  \ge  c > 0.
	\end{equation*}
	Above we've used the fact that $s_\infty > 0$ and $X < C$ almost surely. This proves point $(a)$. \\
	To show point $(b)$, it suffices to prove $W_2(H_{n, \psi, \kappa}(G_n, X_n, W_n), H_{\infty, \psi, \kappa}(G_n, X_n, W_n)) \to 0$.
	In fact, we prove a strengthened result:  
	\begin{equation}
	\label{eqn:strengthen-b}
	\E_{(G, X, W) \sim \Q_n} \left[(H_{n, \psi, \kappa}(G, X, W) - H_{\infty, \psi, \kappa}(G, X, W))^2\right] \asto 0.
	\end{equation}
	We can view both $H_{n, \psi, \kappa}(G, X, W)$ and $H_{\infty, \psi, \kappa}(G, X, W)$ as linear functions of 
	$G, W, \langle G, W\rangle_{\Q_n} G$, i.e., 
	\begin{equation*}
	\begin{split}
	H_{n, \psi, \kappa}(G, X, W) &= a_1(X; c_{1, n}, c_{2, n}, s_n) G + a_2(X; c_{1, n}, c_{2, n}, s_n) W + 
		a_3(X; c_{1, n}, c_{2, n}, s_n) \langle G, W\rangle_{\Q_n} G. \\
	H_{\infty, \psi, \kappa}(G, X, W) &= a_1(X; c_{1, \infty}, c_{2, \infty}, s_\infty) G + a_2(X; c_{1, \infty}, c_{2, \infty}, s_\infty) W.
	\end{split}
	\end{equation*}	
	Recall that $0<c < X < C < \infty$ by our assumption. Define for $j \in \{1, 2, 3\}$,  
	\begin{equation*}
	\eps_{j, n} = 
		\begin{cases}
		\sup_{X \in [c, C]} \left|a_j(X; c_{1, n}, c_{2, n}, s_n) - a_j(X; c_{1, \infty}, c_{2, \infty}, s_\infty)\right| &\text{if $j = 1, 2$}. \\
		\sup_{X \in [c, C]} \left|a_j(X; c_{1, n}, c_{2, n}, s_n) \right|~~&\text{if $j = 3$}.
		\end{cases}
	\end{equation*} 
	It is easy to show that (the reader can check the details by himself)
	\begin{itemize}
	\item For some $M < \infty$, $|\eps_{j, n}| \le M$ for all $n \in \N$. This is mostly due to Corollary \ref{coro:PsiLarge}.
		$(e)$---we know for some $C < \infty$, $|c_{1, n}| \le C$, $|c_{2, n}| \le C$,  $|c_{1, \infty}| \le C$, $|c_{2, \infty}| \le C$
		holds for all $n \in \N$.
	\item Almost surely, $\lim_{n \to \infty}|\eps_{j, n}| \to 0$ for $j \in \{1, 2\}$. This is mostly due to the 
		convergence in Eq.~\eqref{eqn:converge-c-1-c-2-s}--- we know that 
		$(c_{1, n}, c_{2, n}, s_n) \to (c_{1, \infty}, c_{2, \infty}, s_{\infty})$. Also the denominator of $H_n(G, X, W)$ is 
		uniformly lower bounded for all sufficiently large $n$. 
	\end{itemize}
	
	Now, the desired Eq.~\eqref{eqn:strengthen-b} follows since by Cauchy Schwartz inequality
	\begin{equation*}
	\begin{split}
		&\E_{(G, X, W) \sim \Q_n} \left[(H_{n, \psi, \kappa}(G, X, W) - H_{\infty, \psi, \kappa}(G, X, W))^2\right]  \\
		&\le 3 \left(\eps_{1, n}^2\E_{Q_n}[G^2] + \eps_{2, n}^2\E_{Q_n}[W^2] + 
				\eps_{3, n}^2|\langle G, W_n\rangle_{\Q_n}|^2 \E_{\Q_n}[G^2]\right) \asto 0,
	\end{split}
	\end{equation*}
	where in the last step, we used the fact that $\E_{\Q_n}[G^2] \to \E_{\Q_\infty}[G^2]  =1$, $\E_{Q_n}[W^2] = 1$
	and  $|\langle G, W_n\rangle_{\Q_n}| \to |\langle G, W_n\rangle_{\Q_\infty}| = 0$ due to the convergence 
	$\Q_n \stackrel{W_2}{\rightarrow} \Q_\infty$.
        \end{itemize}

We are now left with the task of proving  Eq.~\eqref{eqn:converge-c-1-c-2-s}.
We fix $\psi, \kappa$ in the rest of the proof. For notational convenience, we drop $\psi, \kappa$ from the arguments in 
what follows. Define the functions
        \begin{equation}
        \begin{split}
        V_{1,n}(c_1, c_2, s) &= \frac{(c_2 \psi^{1/2}\proj_{W^{\perp},\Q_n}(G) + \left(c_2  \partial_1 F_{\kappa}(c_1, c_2) -  
        			c_1 \partial_2 F_{\kappa}(c_1, c_2)\right) W)WX^{1/2}}
        				{\partial_2 F_{\kappa}(c_1, c_2) X^{1/2} + c_2 \psi^{1/2}  s X^{-1/2}} + c_1\, , \\
        V_{2,n}(c_1, c_2, s) &= \frac{(c_2 \psi^{1/2}  \proj_{W^{\perp}.\Q_n}(G) +  \left(c_2  \partial_1 F_{\kappa}(c_1, c_2) -  
        			c_1 \partial_2 F_{\kappa}(c_1, c_2)\right) W)^2 X}
        				{(\partial_2 F_{\kappa}(c_1, c_2) X^{1/2} + c_2 \psi^{1/2}  s X^{-1/2})^2} - (c_1^2 + c_2^2)\, ,\\
        V_{3,n}(c_1, c_2, s) &= \frac{(c_2\psi^{1/2}  \proj_{W^{\perp},\Q_n} (G) +  \left(c_2  \partial_1 F_{\kappa}(c_1, c_2) -  
        			c_1 \partial_2 F_{\kappa}(c_1, c_2)\right) W)^2}
        				{( \partial_2 F_{\kappa}(c_1, c_2) X^{1/2} + c_2 \psi^{1/2} s X^{-1/2})^2} - 1.
        \end{split}	
        \end{equation}
Notice that these functions depend on $n$ because $\proj_{W^{\perp},\Q_n}$ does.
        We introduce the shorthands 
        \begin{equation*}
        (c_{1, n}, c_{2, n}, s_n) = (c_1(\Q_n), c_2(\Q_n), s(\Q_n))
        		~~\text{and}~~
	(c_{1, \infty}, c_{2, \infty}, s_{\infty}) = (c_1(\Q_{\infty}), c_2(\Q_{\infty}), s(\Q_{\infty}))
        \end{equation*}
        For sufficiently large $n$, $\left(c_{1, n}, c_{2, n}, s_{n}\right)$ is the solution of the system of equations
        \begin{equation}
        \E_{\Q_n} [V_{1,n}(c_1, c_2, s)] = 0, ~\E_{\Q_n} [V_{2,n}(c_1, c_2, s)] = 0, ~\E_{\Q_n} [V_{3,n}(c_1, c_2, s)] = 0.
        \end{equation}
      Corollary \ref{coro:PsiLarge} implies the existence of $M > 0$ such that  
	\begin{equation}
	\begin{split}
	\limsup_{n\to \infty} |c_{1, n}| < M,~~0 \le \liminf_{n\to \infty} c_{2, n} \le  \limsup_{n\to \infty} c_{2, n} < M,~~ \limsup_{n\to \infty} s_n < M.
	\end{split}
	\end{equation}
	Define the compact set $\setsq \defeq [-M, M]  \times [0,  M] \times [0, M]$. The next lemma
	establishes for each $i = 1, 2, 3$ the uniform convergene result of $\E_{\Q_n} [V_{i,n}(c_1, c_2, s)]$ to $\E_{\Q_{\infty}}[V_{i,\infty}(c_1, c_2, s)]$
	on the compact set $\setsq$. To avoid interrupting the flow, we defer its proof to the next subsection.
  \begin{lemma}
        \label{lemma:uniform-convergence-c-1-c-2-t}
         For $i = 1, 2, 3$, we have almost surely
        \begin{equation}
        \limsup_{n\to \infty} \sup_{(c_1, c_2, s) \in \setsq} 
        	\left|\E_{\Q_n} [V_{i,n}(c_1, c_2, s)] - \E_{\Q_{\infty}}[V_{i,\infty}(c_1, c_2, s)]\right| = 0
        \end{equation}
        \end{lemma}
        Now, we are ready to show the desired convergence result in Eq.~\eqref{eqn:converge-c-1-c-2-s}. We prove that any limit point of 
        $(c_{1, n}, c_{2, n}, s_n)$ must be $(c_{1, \infty}, c_{2, \infty}, s_{\infty})$. To 
        do this, first take any limit point of $(c_{1, n}, c_{2, n}, s_n)$,  and denote it to be $(\tilde{c}_1\opt, \tilde{c}_2\opt, \tilde{s}\opt)$. 
        Since by definition we have for $i=1,2,3$, 
        \begin{equation*}
        \E_{\Q_n}[V_{i,n}(c_{1, n}, c_{2, n}, s_n)] = 0,
        \end{equation*}
        the triangle inequality immediately implies that, for $i=1,2,3$,
        \begin{equation}
        \begin{split}
        |\E_{\Q_{\infty}} [V_{i,\infty}(\tilde{c}_1\opt, \tilde{c}_2\opt, \tilde{s}\opt)]| 
        		&\le \left|\E_{\Q_{\infty}}[V_{i,\infty}(\tilde{c}_1\opt, \tilde{c}_2\opt, \tilde{s}\opt) - \E_{\Q_{\infty}}[V_{i,\infty}(c_{1, n}, c_{2, n}, s_n)]\right| \\
        		&~~~~~~~~~
			+ | \E_{\Q_{\infty}}[V_{i,\infty}(c_{1, n}, c_{2, n}, s_n) -  \E_{\Q_{n}}[V_{i,n}(c_{1, n}, c_{2, n}, s_n)|
	\end{split}
        \end{equation}
        for all $n \in \N$. Now, by definition of $\setsq$, $(c_{1, n}, c_{2, n}, s_n) \in \setsq$ for large enough $n$. Hence, 
        \begin{equation}
        \begin{split}
        |\E_{\Q_{\infty}} [V_{i,\infty}(\tilde{c}_1\opt, \tilde{c}_2\opt, \tilde{s}\opt)]| 
        		&\le \limsup_{n}\left|\E_{\Q_{\infty}}[V_{i,\infty}(\tilde{c}_1\opt, \tilde{c}_2\opt, \tilde{s}\opt) - \E_{\Q_{\infty}}[V_{i,\infty}(c_{1, n}, c_{2, n}, s_n)]\right| \\
        		&~~~~
			+\limsup_{n}\sup_{(c_1, c_2, s) \in \setsq} | \E_{\Q_{\infty}}[V_{i,\infty}(c_{1}, c_{2}, s) -  \E_{\Q_{n}}[V_{i,n}(c_{1}, c_{2}, s)| = 0,
	\end{split}
        \end{equation}
        where the last identity uses the fact that the mapping $(c_1, c_2, s) \to \E_{\Q_{\infty}}[V_{i,\infty}(c_1, c_2, s)]$ is continuous on 
        $\setsq$, and the uniform convergence result by Lemma~\ref{lemma:uniform-convergence-c-1-c-2-t}. This shows that any 
        limit point $(\tilde{c}_1\opt, \tilde{c}_2\opt, \tilde{s}\opt) \in \setsq$ must satisfy the system of equations below
        \begin{equation}
        \label{eqn:infty-V-i-sytem-of-equations}
         \E_{\Q_{\infty}} [V_{1,\infty}(c_1, c_2, s)] = 0, ~\E_{\Q_{\infty}} [V_{2,\infty}(c_1, c_2, s)] = 0, ~\E_{\Q_{\infty}} [V_{3,\infty}(c_1, c_2, s)] = 0.
        \end{equation}
        Now we recall Lemma~\ref{lemma:nonzero-c-2-t}. 
	By Lemma~\ref{lemma:nonzero-c-2-t}, we know $(\tilde{c}_1\opt, \tilde{c}_2\opt, \tilde{s}\opt)$ must satisfy
	both $\tilde{c}_2\opt > 0$ and $\tilde{s}\opt> 0$ and the system of equations \eqref{eqn:infty-V-i-sytem-of-equations}. 
	Now since the solution of the system of equations \eqref{eqn:infty-V-i-sytem-of-equations} is unique in $\R \times \R_{>0} \times \R_{>0}$ (by Corollary \ref{coro:PsiLarge}), 
	this shows that $(\tilde{c}_1\opt, \tilde{c}_2\opt, \tilde{s}\opt) = (c_{1, \infty}, c_{2, \infty}, s_{\infty})$. Thus, we proved that
	any limit point of $\{(c_{1, n}, c_{2, n}, s_n)\}_{n\in \N}$ must be $(c_{1, \infty}, c_{2, \infty}, s_{\infty})$. This implies 
	the desired convergence \eqref{eqn:converge-c-1-c-2-s}.

\subsection{Proof of Lemma~\ref{lemma:uniform-convergence-c-1-c-2-t}}
\label{sec:lemma-uniform-convergence-c-1-c-2-t}

By Eq.~\eqref{eqn:Q-n-converge-to-Q-infty}, we know that  $\Q_n \stackrel{W_2}{\Longrightarrow} \Q_{\infty}$. % almost surely.
Therefore, it is sufficient to prove that $\Q_n \stackrel{W_2}{\Longrightarrow} \Q_{\infty}$ implies
\begin{equation*}
    	\limsup_{n\to \infty} \sup_{(c_1, c_2, s) \in \setsq} 
        		\left|\E_{\Q_n} [V_{i,n}(c_1, c_2, s)] - \E_{\Q_{\infty}}[V_{i,\infty}(c_1, c_2, s)]\right| = 0.
\end{equation*}
Let us denote for $i \in \{1, 2, 3\}$
\begin{equation*}
\bar{V}_{i,n}(c_1, c_2, s) = \E_{\Q_n} [V_{i,n}(c_1, c_2, s)] ~~\text{and}~~
\bar{V}_{i,\infty}(c_1, c_2, s) = \E_{\Q_\infty} [V_{i,\infty}(c_1, c_2, s)].
\end{equation*}
Note that $\bar{V}_{i, n}$ and $\bar{V}_{i,\infty}$ are continuous.
By Arzel\`{a}-Ascoli and Dini's theorem, it suffices to show that 
\begin{enumerate}[(a)]
\item For each $i\in \{1, 2, 3\}$, we have for any fixed $(c_1, c_2, s)\in \setsq$
	\begin{equation*}
	\lim_{n\to \infty} \left|\bar{V}_{i,n}(c_1, c_2, s)-\bar{V}_{i,\infty}(c_1, c_2, s)\right| = 0.
	\end{equation*}
\item For $i = 1, 2$, the functions $\{\bar{V}_{n, 1}(c_1, c_2, s)\}_{n \in \N}$ is equicontinuous on $\setsq$, i.e.,
	for any $\eps > 0$, there exists some $\delta > 0$, such that for any 
	$(c_1, c_2, s), (c_1^\prime, c_2^\prime, s^\prime) \in \setsq$ satisfying 
	$\ltwo{(c_1, c_2, s) - (c_1^\prime, c_2^\prime, s^\prime)} < \delta$,
	\begin{equation*}
		\sup_{n\in \N} \left|\bar{V}_{i,n}(c_1, c_2, s)-\bar{V}_{i,n}(c_1^\prime, c_2^\prime, s^\prime)\right| \le \eps. 
	\end{equation*}
\item For $i = 3$, the functions $\{\bar{V}_{n, 3}(c_1, c_2, s)\}_{n \in \N}$ is monotonically decreasing with 
	the parameter $s$, and moreover, for any given $s$, it is also equicontinuous w.r.t $(c_1, c_2)$ in the sense that
	for any $\eps > 0$, there exists some $\delta > 0$, such that for any 
	$(c_1, c_2, s), (c_1^\prime, c_2^\prime, s) \in \setsq$ satisfying 
	$\ltwo{(c_1, c_2) - (c_1^\prime, c_2^\prime)} < \delta$,
	\begin{equation*}
		\sup_{n\in \N} \left|\bar{V}_{3,n}(c_1, c_2, s)-\bar{V}_{3,n}(c_1^\prime, c_2^\prime, s)\right| \le \eps. 
	\end{equation*}
\end{enumerate}
Below we prove the above points. We start with the notation.

\newcommand{\up}{\uparrow}
\newcommand{\down}{\downarrow}
\newcommand{\lowz}{{\rm low}}
\newcommand{\upz}{{\rm up}}
\newcommand{\errorr}{\mathcal{E}}
\paragraph{Notation}
For notational simplicity, we introduce for $* \in \{n, \infty\}$
\begin{equation}
\label{eqn:V-up-down}
        \begin{split}
        V_{1,*}^{\up}(c_1, c_2, s) &= (\psi^{1/2}\proj_{W^{\perp},\Q_*}(G) + \left(\partial_1 F_{\kappa}(c_1, c_2) -  
        			c_1 c_2^{-1} \partial_2 F_{\kappa}(c_1, c_2)\right) W)W\, , \\ 
	V_{2,*}^{\up}(c_1, c_2, s) &= (\psi^{1/2}  \proj_{W^{\perp}, \Q_*}(G) +  \left(\partial_1 F_{\kappa}(c_1, c_2) -  
        			c_1 c_2^{-1}\partial_2 F_{\kappa}(c_1, c_2)\right) W)^2 \, , \\
	V_{3,*}^{\up}(c_1, c_2, s) &= (\psi^{1/2}  \proj_{W^{\perp}, \Q_*}(G) +  \left(\partial_1 F_{\kappa}(c_1, c_2) -  
        			c_1 c_2^{-1} \partial_2 F_{\kappa}(c_1, c_2)\right) W)^2 X^{-1}\, , \\
	V_{1}^{\down}(c_1, c_2, s) & = c_2^{-1}\partial_2 F_{\kappa}(c_1, c_2)  + \psi^{1/2}  s X^{-1}\, , \\
	V_2^{\down}(c_1, c_2, s) &= (c_2^{-1} \partial_2 F_{\kappa}(c_1, c_2)  + \psi^{1/2}  s X^{-1})^2\, , \\
        V_3^{\down}(c_1, c_2, s) &= (c_2^{-1} \partial_2 F_{\kappa}(c_1, c_2)  + \psi^{1/2} s X^{-1})^2\, ,
        \end{split}	
\end{equation}
so that this definition allows us to easily express (for $* \in \{n, \infty\}$)
\begin{equation}
\label{eqn:connection-V-n-V-up-V-down}
\begin{aligned}
	V_{1,*}(c_1, c_2, s) &= \frac{V_{*,1}^{\up}(c_1, c_2, s)}{V_1^{\down}(c_1, c_2, s)} + c_1\, \\ %\equiv U_{1,*}(c_1, c_2, s) +c_1\, , \\
	V_{2, *}(c_1, c_2, s) &= \frac{V_{*,2}^{\up}(c_1, c_2, s)}{V_2^{\down}(c_1, c_2, s)} - (c_1^2 + c_2^2)\, \\ %\equiv U_{2,*}(c_1, c_2, s) - (c_1^2 + c_2^2) ,\\
	V_{3,*}(c_1, c_2, s) &= \frac{V_{*,3}^{\up}(c_1, c_2, s)}{V_3^{\down}(c_1, c_2, s)} - 1 % \equiv U_{3,*}(c_1, c_2, s) - 1\, .
\end{aligned}
\end{equation}

\paragraph{Proof of Point $(a)$} 
 Recall that $\Q_n \stackrel{W_2}{\Longrightarrow} \Q_{\infty}$ implies 
\begin{equation} 
\label{eqn:W-2-imply-key-convergence}
\lim_{n\to \infty}\left|\E_{\Q_n}[f(G, X, W)] - \E_{\Q_{\infty}}[f(G, X, W)]\right| = 0\, ,
\end{equation}
for any continuous function $f: \R^3 \to \R$ satisfying 
$\sup_{(G, X, W)\in\R^3} \frac{f(G, X, W)}{(G^2 + W^2)/X} < \infty$.

We will use this fact to prove point $(a)$. Our starting point is the following decomposition:  
\begin{equation}
\begin{split}
	\bar{V}_{i,n}(c_1, c_2, s) &- \bar{V}_{i,\infty}(c_1, c_2, s)
		= \E_{\Q_n}\left[\frac{V_{i,n}^{\up}(c_1, c_2, s)}{V_i^{\down}(c_1, c_2, s)}\right] - 
			\E_{\Q_\infty}\left[\frac{V_{i,\infty}^{\up}(c_1, c_2, s)}{V_i^{\down}(c_1, c_2, s)}\right] \\
		&=  \underbrace{\left(\E_{\Q_n}\left[\frac{V_{i,\infty}^{\up}(c_1, c_2, s)}{V_i^{\down}(c_1, c_2, s)}\right] - 
				\E_{\Q_\infty}\left[\frac{V_{i,\infty}^{\up}(c_1, c_2, s)}{V_i^{\down}(c_1, c_2, s)}\right]\right)}_{\error_{n, 1, i}}
		+ \underbrace{\E_{\Q_n}\left[\frac{(V_{i,n}^{\up}(c_1, c_2, s) - V_{i,\infty}^{\up}(c_1, c_2, s))}{V_i^{\down}(c_1, c_2, s)}\right]}_{\error_{n, 2, i}} 
			.
\end{split}
\end{equation}
To prove point $(a)$, it suffices to prove for any $(c_1, c_2, s) \in \setsq$: 
\begin{equation}
\label{eqn:desired-goal-of-a-uniform-convergence}
\lim_{n\to \infty}\error_{n, 1, i}(c_1, c_2, s) = 0
~~\text{and}~~
\lim_{n\to \infty}\error_{n, 2, i}(c_1, c_2, s) = 0.
\end{equation}

We begin with the first limit in Eq.~\eqref{eqn:desired-goal-of-a-uniform-convergence}.
The idea is to apply the convergence statement in Eq~\eqref{eqn:W-2-imply-key-convergence}. We claim that 
\begin{equation}
{\rm ess}\sup_{(G, X, W)\in\R^3} \frac{V_{i,\infty}^{\up}(c_1, c_2, s)}{V_i^{\down}(c_1, c_2, s) (\ltwo{(G, W)}^2/X)} < \infty
\end{equation}
(where the essential $\sup$ holds both under  $\Q_n$ and under $\Q_{\infty}$.)
This follows by the below two observations. 
\begin{itemize}
\item There exists $c_{\lowz} = c_{\lowz}(\setsq, \psi, \kappa) > 0$ such that for any $i \in \{1, 2, 3\}$,
	\begin{equation}
	\label{eqn:V-i-down-lower-bound}
		\min_{(c_1, c_2, s) \in \setsq}V_{i}^{\down}(c_1, c_2, s) > c_{\lowz}.
	\end{equation}
	The reason is: $(i)$ for some constant $c^\prime > 0$, $c_2^{-1}\partial_2 F_{\kappa}(c_1, c_2) > c^\prime$ holds
	 for all $(c_1, c_2, s) \in \setsq$
	(Lemma~\ref{lemma:continuity-partial-f} and Lemma~\ref{lemma:F-convex-increasing})
	and $(ii)$ $V_i^{\down}(c_1,c_2, s) > (c_2^{-1}\partial_2 F_{\kappa}(c_1, c_2))^{j_i}$ 
	where $j_1 = 1$, $j_2 = j_3 = 2$.
	\item There exists  $C_{\upz} = C_{\upz}(\setsq, \psi, \kappa) > 0$ such that for any $i \in \{1, 2, 3\}$,
	\begin{equation}
	\label{eqn:uniform-bound-V-infty-up}
		\sup_{(c_1, c_2, s) \in \setsq}\sup_{(G, X, W)\in \R^3} \frac{V_{i,\infty}^{\up}(c_1, c_2, s)}{\ltwo{(G, W)}^2/X} \le C_{\upz} .
	\end{equation}
	For convenience of the reader, we write explicitly $V_{i,\infty}^{\up}(c_1, c_2, s)$: 
	\begin{equation}
	\begin{split}
	\label{eqn:V-infty-up}
		V_{1,\infty}^{\up}(c_1, c_2, s) &= (\psi^{1/2}G + \left(\partial_1 F_{\kappa}(c_1, c_2) -  
        			c_1 c_2^{-1} \partial_2 F_{\kappa}(c_1, c_2)\right) W)W \, ,\\ 
		V_{2, \infty}^{\up}(c_1, c_2, s) &= (\psi^{1/2}G+  \left(\partial_1 F_{\kappa}(c_1, c_2) -  
        			c_1 c_2^{-1} \partial_2 F_{\kappa}(c_1, c_2)\right) W)^2  \, ,\\
		V_{3, \infty}^{\up}(c_1, c_2, s) &= (\psi^{1/2} G+  \left( \partial_1 F_{\kappa}(c_1, c_2) -  
        			c_1 c_2^{-1} \partial_2 F_{\kappa}(c_1, c_2)\right) W)^2 X^{-1}\, .
	\end{split}
	\end{equation}
         Equation~\eqref{eqn:uniform-bound-V-infty-up} holds because $(i)$ $\partial_1 F_{\kappa}(c_1, c_2)$ 
	and $c_2^{-1}\partial_2 F_{\kappa}(c_1, c_2)$ are well-defined continuous functions of
	$(c_1, c_2)$ on $\R \times \R_{\ge 0}$ (by Lemma~\ref{lemma:continuity-partial-f} and 
	Lemma~\ref{lemma:F-convex-increasing}) and thus are uniformly bounded on any compact set
	$(ii)$ by Assumption~\ref{assumption:Lambdas}, there exists a constant $C < \infty$ such that, almost surely, $X < C$.
	%Therefore, there exists a quadratic polynomial of $(G, W)$ such that $|V_{i,\infty}^{\up}(c_1, c_2, s)|$ is bounded 
	%by that quadratic polynomial for any $(c_1, c_2, s)$.

\end{itemize}
Next we establish the second limit in Eq.~\eqref{eqn:desired-goal-of-a-uniform-convergence}.
In light of Eq.~\eqref{eqn:V-i-down-lower-bound}, it suffices to show for $(c_1, c_2, s) \in \setsq$, 
	\begin{equation}
	\label{eqn:E-converge-V-n-V-infty-up}
		\lim_{n \to \infty} \E_{\Q_n} \left[|V_{i,n}^{\up}(c_1, c_2, s) - V^{\up}_{i,\infty}(c_1, c_2, s)|\right] = 0.
	\end{equation}
%	For future use, let's prove a strengthened version of Eq~\eqref{eqn:E-converge-V-n-V-infty-up}, i.e.
%	\begin{equation}
%	\label{eqn:E-converge-V-n-V-infty-up-uniform}
%		\lim_{n \to \infty} \sup_{(c_1, c_2, s) \in \setsq}
%			\E_{\Q_n} \left[|V_{i,n}^{\up}(c_1, c_2, s) - V^{\up}_{i,\infty}(c_1, c_2, s)|\right] = 0.
%	\end{equation}
%	The proof is relatively straightforward yet involves somewhat tedious calculations. The key to note 
%	is that only difference between $V_{i,n}^{\up}(c_1, c_2, s)$ and $V^{\up}_{i,\infty}(c_1, c_2, s)$ 
%	in their very definition is due to difference between $\proj_{W^{\perp},\Q_n}(G) = G - \langle G, W\rangle_{\Q_n} W$ 
%	and $\proj_{W^{\perp},\Q_\infty}(G) = G$. Below we make the calculations explicit. 
	Compute the difference between $V_{i,n}^{\up}(c_1, c_2, s)$ and $V^{\up}_{i,\infty}(c_1, c_2, s)$. 
	Write $a_n = \langle G, W\rangle_{\Q_n}$. We reach 
	\begin{equation}
	\label{eqn:explicit-form-Delta}
	\begin{split}
	V_{1, \infty}^{\up}(c_1, c_2, s) - V^{\up}_{1, n}(c_1, c_2, s) &
		= a_n \psi^{1/2} W^2 \\
	V_{2, \infty}^{\up}(c_1, c_2, s) - V^{\up}_{2, n}(c_1, c_2, s) &= - \psi a_n^2 W^2
		- 2\psi^{1/2}a_n f(c_1, c_2, \kappa)W^2  \\
	% 2a_n f_1(c_2, \psi) f_2(c_1, c_2, \psi, \kappa) W^2 
	%+f_1^2(c_2, \psi)(2a_n  GW-a_n^2 W^2)\\
	V_{3, \infty}^{\up}(c_1, c_2, s) - V^{\up}_{3, n}(c_1, c_2, s) &=  - \psi a_n^2 W^2 X^{-1}
		- 2\psi^{1/2}a_n f(c_1, c_2, \kappa)W^2 X^{-1}.
	\end{split}
	\end{equation}	
	where $f$ is a continuous function depending only on $c_1, c_2, \kappa$ and independent of $G, W, X$: 
	\begin{equation*}
		f(c_1, c_2, \kappa) = \partial_1 F_\kappa(c_1, c_2) - c_1 c_2^{-1} \partial_2 F_\kappa(c_1, c_2). 
	\end{equation*}
	Now that Eq~\eqref{eqn:E-converge-V-n-V-infty-up} follows since $(i)$ we have the convergence 
	$a_n \to \E_{Q_\infty} [GW] = 0$ (since $\Q_n \stackrel{W_2}{\Longrightarrow} \Q_{\infty}$), 
	$(ii)$ the function $f$ is uniformly bounded on $\setsq$ (by Lemma~\ref{lemma:continuity-partial-f} and 
	Lemma~\ref{lemma:F-convex-increasing})
	and $(iii)$ each individual term involving $(G, W, X)$ on the RHS of 
	Eq~\eqref{eqn:explicit-form-Delta} i.e., $W^2$, $W^2/X$ satisfies the property that 
 	$\E_{\Q_n}[|W^2|]$ and 
	$\E_{\Q_n}[|W^2/X|]$ are uniformly bounded over $n \in \N$.

\paragraph{Proof of Point $(b)$ and Point $(c)$} Let $i \in \{1, 2, 3\}$.
%Recall Eq.~\eqref{eqn:connection-V-n-V-up-V-down}, and define $\bar{U}_{i,n}(c_1,c_2,s)= \E_{\Q_n}U_{i,n}(c_1,c_2,s)$. 
%notice that it is sufficient to prove equicontinuity  of $\bar{U}_{i,n}$. 
Pick any $(c_1, c_2, s), (c_1^\prime, c_2^\prime, s^\prime) \in \setsq$.
\begin{equation}
\label{eqn:V-n-diff}
\begin{split}
&\bar{V}_{i,n}(c_1, c_2, s) - \bar{V}_{i,n}(c_1^\prime, c_2^\prime, s^\prime)
		= \E_{\Q_n}\left[\frac{V_{i,n}^{\up}(c_1, c_2, s)}{V_i^{\down}(c_1, c_2, s)}\right] - 
			\E_{\Q_n}\left[\frac{V_{i,n}^{\up}(c_1^\prime, c_2^\prime, s^\prime)}{V_i^{\down}(c_1^\prime, c_2^\prime, s^\prime)}\right] \\
		&=  \E_{\Q_n}\left[\frac{V_{i,n}^{\up}(c_1, c_2, s) - V_{i, n}^{\up} (c_1', c_2', s')}{V_i^{\down}(c_1, c_2, s)}\right] 
			+ \E_{\Q_n} \left[V_{i, n}^{\up} (c_1', c_2', s')\left(\frac{1}{V_i^{\down}(c_1, c_2, s)}- \frac{1}{V_i^{\down}(c_1', c_2', s')}\right)\right]
		%&= \E_{\Q_n}\left[\frac{V_{i,n}^{\up}(c_1, c_2, s)V_i^{\down}(c_1^\prime, c_2^\prime, s^\prime)
		%	- V_{i,n}^{\up}(c_1^\prime, c_2^\prime, s^\prime)V_i^{\down}(c_1, c_2, s)}{V_i^{\down}(c_1, c_2, s)V_i^{\down}(c_1^\prime, c_2^\prime, s^\prime)}\right].
\end{split}
\end{equation}
The key to the proof is to establish the following statements. In fact, Point (b) follows from Eq.~\eqref{eqn:V-i-down-lower-bound}, 
Claim~\ref{claim:one} and Claim \ref{claim:two}. Point (c) follows from Eq.~\eqref{eqn:V-i-down-lower-bound}, Claim~\ref{claim:one}
and Claim~\ref{claim:three}. 

\begin{claim}
\label{claim:one}
Let $i \in \{1, 2, 3\}$. For any $\eps > 0$, there exists 
$\delta > 0$ such that when
$\ltwo{(c_1, c_2, s) - (c_1^\prime, c_2^\prime, s^\prime)} < \delta$,
\begin{equation}
\label{eqn:equiv-goal-of-equicontinuity-one}
\begin{split}
	\sup_{n\in \N} \E_{\Q_n}\left[\left|V_{i,n}^{\up}(c_1, c_2, s)
			- V_{i,n}^{\up}(c_1^\prime, c_2^\prime, s^\prime) \right|\right] &\le \eps 
\end{split}
\end{equation}
\end{claim}

\begin{claim}
\label{claim:two}
Let $i \in \{1, 2\}$. For any $\eps > 0$, there exists 
$\delta > 0$ such that when
$\ltwo{(c_1, c_2, s) - (c_1^\prime, c_2^\prime, s^\prime)} < \delta$,
\begin{equation}
\label{eqn:equiv-goal-of-equicontinuity-two}
	\sup_{n\in \N} \E_{\Q_n}\left[\left|V_{i,n}^{\up}(c_1', c_2', s') 
		\left(\frac{1}{V_i^{\down}(c_1, c_2, s)}- \frac{1}{V_i^{\down}(c_1', c_2', s')}\right)\right|\right] \le \eps	 
\end{equation}
\end{claim}

\begin{claim}
\label{claim:three}
Let $i = 3$ and $s = s'$. For any $\eps > 0$, there exists 
$\delta > 0$ such that when
$\ltwo{(c_1, c_2) - (c_1^\prime, c_2^\prime)} < \delta$,
\begin{equation}
\label{eqn:equiv-goal-of-equicontinuity-three}
	\sup_{n\in \N} \E_{\Q_n}\left[\left|V_{3,n}^{\up}(c_1', c_2', s) 
		\left(\frac{1}{V_3^{\down}(c_1, c_2, s)}- \frac{1}{V_3^{\down}(c_1', c_2', s)}\right)\right|\right] \le \eps	 
\end{equation}
\end{claim}

Below we give the deferred proof of Claim~\ref{claim:one} to Claim~\ref{claim:three}. 

We start by proving Claim~\ref{claim:one}.  
An inspection of $V_{i,n}^{\up}(c_1, c_2, s)$ (recall Eq.~\eqref{eqn:V-up-down}) shows that $V_{i,n}^{\up}(c_1, c_2, s)$ 
takes the form of 
\begin{equation}
\label{eqn:representation-of-V-n-i-f-g}
	V_{i,n}^{\up}(c_1, c_2, s) = \sum_{j \le J} f_{j, i}(c_1, c_2, s, \psi, \kappa) \cdot g_{j, i}\left(\proj_{W^{\perp}, \Q_n}(G), W\right) h_i(X)
\end{equation}
where each $f_{j, i}$ is a continuous function of $(c_1, c_2, s, \psi, \kappa)$ and each $g_{j, i}$ is a quadratic polynomial of 
$(\proj_{W^{\perp}, \Q_n}(G), W)$ independent of $c_1, c_2, s, \psi, \kappa$ and $h_i(X) = 1$ for $i = 1, 2$ and $h_i(X) = X^{-1}$
for $i = 3$. Consequently, we obtain 
\begin{equation}
\label{eqn:cs-inequality-V-n-i}
	\begin{split}
		&\sup_{n \in \N} \E_{\Q_n} \left|V_{i,n}^{\up}(c_1, c_2, s) - V_{i,n}^{\up}(c_1^\prime, c_2^\prime, s^\prime)\right| \\
		&\le  \sum_{j \le J} \left|f_{j, i}(c_1, c_2, s, \psi, \kappa) - f_{j, i}(c_1^\prime, c_2^\prime, s^\prime, \psi, \kappa)\right|
			 \cdot \sup_{n \in \N} \E_{\Q_n}\left[\left|g_{j, i}\left(\proj_{W^{\perp}, \Q_n}(G), W\right) h_i(X)\right|\right]
	\end{split}
\end{equation}
The function $f_{j, i}(c_1, c_2, s, \psi, \kappa)$ is continuous and hence equicontinuous on any compact set. Note 
\begin{equation}
	\label{eqn:exp-g-C-up}
		\sup_{n \in \N}  \E_{\Q_n} \left[\left|g_{j, i}(\proj_{W^{\perp}, \Q_n}(G), W)h_i(X)\right|\right] \le C_{\upz}< \infty. 
\end{equation}
This is true since (i) $g_{j, i}$ is quadratic so that  $\left|g_{j, i}(\proj_{W^{\perp}, \Q_n}(G), W)\right| \le C_{j,i}' (G^2 + W^2)$ for some constant 
$C_{j, i}' > 0$, and (ii) $\E_{\Q_n}[|(G^2 + W^2)|]$ and $\E_{\Q_n}[|(G^2 + W^2)/X|]$ are  uniformly bounded over $n \in \N$.

Next we prove Claim~\ref{claim:two}. Note first there exists a constant $C_{\upz} < \infty$ such that when $i = 1, 2$: 
\begin{equation}
\label{eqn:uniform-bound-V-n-i}
	\sup_{n \in \N} \sup_{(c_1, c_2, s) \in \setsq}  \left|V_{i,n}^{\up}(c_1, c_2, s)\right| \le C_{\upz} (G^2 + W^2)< \infty. 
\end{equation}
This follows from the expression in equation~\eqref{eqn:representation-of-V-n-i-f-g}. Next, we show the key estimate for $i = 1, 2$
\begin{equation}
\label{eqn:uniform-bound-V-n-i-annoying}
	\left|\frac{1}{V_i^{\down}(c_1, c_2, s)}- \frac{1}{V_i^{\down}(c_1', c_2', s')}\right| \le \Delta((c_1, c_2, s), (c_1', c_2', s')) \cdot (1+X^{-1}).
\end{equation}
for some function $\Delta((c_1, c_2, s), (c_1', c_2', s'))$ which satisfies 
\begin{equation*}
	\lim_{\delta \to 0} \sup_{\ltwo{(c_1, c_2, s) - (c_1^\prime, c_2^\prime, s^\prime)} < \delta}\Delta((c_1, c_2, s), (c_1', c_2', s')) \to 0
\end{equation*}
To see this, the key is that  $V_1^{\down}(c_1, c_2, s) \ge c_{\lowz} > 0$ holds uniformly over $(c_1, c_2, s) \in \setsq$ 
(Eq.~\eqref{eqn:V-i-down-lower-bound}). A simple calculation shows that the result holds for $i = 1$. For $i = 2$, we reduce the 
case to $i = 1$ by leveraging the fact that $V_2^{\down}(c_1, c_2, s) = (V_1^{\down}(c_1, c_2, s))^2$. As 
$\E_{\Q_n}[|(G^2 + W^2)|]$ and $\E_{\Q_n}[|(G^2 + W^2)/X|]$ are  uniformly bounded over $n \in \N$, Claim~\ref{claim:two} follows 
from the estimate in Eq.~\eqref{eqn:uniform-bound-V-n-i} and Eq.~\eqref{eqn:uniform-bound-V-n-i-annoying}. 

Finally, we prove Claim~\ref{claim:three}. Note first there exists a constant $C_{\upz} < \infty$ such that when $i = 3$: 
\begin{equation}
\label{eqn:uniform-bound-V-n-i-hard}
	\sup_{n \in \N} \sup_{(c_1, c_2, s) \in \setsq}  \left|V_{3,n}^{\up}(c_1, c_2, s)\right| \le C_{\upz} (G^2 + W^2)/X< \infty. 
\end{equation}
This follows from the expression in equation~\eqref{eqn:representation-of-V-n-i-f-g}. Next, 
we can show when $i = 3$ and $s = s'$
\begin{equation}
\label{eqn:uniform-bound-V-n-i-annoying-hard}
	\left|\frac{1}{V_3^{\down}(c_1, c_2, s)}- \frac{1}{V_3^{\down}(c_1', c_2', s)}\right| \le \Delta((c_1, c_2, s), (c_1', c_2', s)).
\end{equation}
for some function $\Delta((c_1, c_2, s), (c_1', c_2', s))$ satisfying
\begin{equation*}
	\lim_{\delta \to 0} \sup_{\ltwo{(c_1, c_2) - (c_1^\prime, c_2^\prime)} < \delta, |s| \le M}\Delta((c_1, c_2, s), (c_1', c_2', s)) \to 0
\end{equation*}
To see this, we leverage the fact that (i) $V_3^{\down}(c_1, c_2, s) = (V_1^{\down}(c_1, c_2, s))^2$ and 
(ii) $V_1^{\down}(c_1, c_2, s)$ is uniformly lower bounded over $(c_1, c_2, s) \in \setsq$ 
(Eq.~\eqref{eqn:V-i-down-lower-bound}), and (iii) $\lim_{\delta \to 0} \sup_{\ltwo{(c_1, c_2) - (c_1^\prime, c_2^\prime)} < \delta, |s| \le M}
|V_1^{\down}(c_1, c_2, s) - V_1^{\down}(c_1', c_2', s)| = 0$ holds by a direct evaluation.
As $\E_{\Q_n}[|(G^2 + W^2)/X|]$ is uniformly bounded over $n \in \N$, Claim~\ref{claim:three} now
follows from Eq.~\eqref{eqn:uniform-bound-V-n-i-hard} and Eq.~\eqref{eqn:uniform-bound-V-n-i-annoying-hard}.

\section{Reduction to the Gordon's optimization problem}

\subsection{Proof of Lemma~\ref{lemma:xi-0-xi-1}}
\label{sec:proof-lemma-xi-0-xi-1}

Recall that $\bX$ has i.i.d. rows $\bx_i\sim \normal(\bzero,\bLambda)$, with $\bLambda$ a diagonal matrix. 
We rewrite $\bX = \bar{\bX}\bLambda^{1/2}$, where $(\bar{X}_{ij})_{i\le n, j\le p}\sim\normal(0,1)$. 
Therefore, Eq.~\eqref{eq:xiDef} yields
\begin{equation}
\xi^{(0)}_{n, \psi, \kappa}(\bTheta_p)= \min_{\ltwo{\btheta} \le 1, \btheta \in \bTheta_p}~
		\max_{\ltwo{\blambda} \le 1, \by \odot \blambda \ge 0} \frac{1}{\sqrt{p}}\blambda^{\sT} (\kappa \by - \bar{\bX}\bLambda^{1/2} \btheta)\, .\label{eq:xi-0-proof}
\end{equation}
We need to be cautious when we apply Theorem~\ref{thm:Gordon-improve} to $\xi^{(0)}_{n, \psi, \kappa}$:
$\by$ is not independent of the Gaussian random matrix $\bar{\bX}$ . To circumvent this technical 
difficulty, recall that $\bw=\bar{\btheta}_{*}/\|\bar{\btheta}_*\|_2$ where $\bar{\btheta}_* = \bLambda^{1/2}\btheta_*$($\bw$ can be chosen arbitrarily if $\bar{\btheta}_*=0$).
We decompose $\bar{\bX}$ into orthogonal components as follows:
\begin{equation}
\bar{\bX} = \bu  \bw^{\sT} + \bX \proj_{\bw^{\perp}}~~\text{where $\bu = \bX\bw \sim \normal(0, I_n)$}.
\end{equation}
(recall the unit vector $\bw$ that parallels $\bar{\btheta}_{*}$) Since $\bar{\bX}$ is isotropic  Gaussian, 
$\bar{\bX}\proj_{\bw^{\perp}}$ is independent of $(\bu, \by)$. Substituting in Eq.~\eqref{eq:xi-0-proof}, we get
\begin{equation}
\xi_{n, \psi, \kappa}^{(0)}(\bTheta_p) 
		= \min_{\ltwo{\btheta} \le 1, , \btheta \in \bTheta_p}~
		\max_{\ltwo{\blambda} \le 1, \by \odot \blambda \ge 0} \frac{1}{\sqrt{p}}\blambda^{\sT} (\kappa \by - 
			\langle \bLambda^{1/2}\bw, \btheta \rangle \bu 
			-\bar{\bX}\proj_{\bw^{\perp}} \bLambda^{1/2} \btheta).
\end{equation}
Consider, to be definite, the case $\bTheta_p = \Ball^p(1)$ (corresponding to $\xi_{n, \psi, \kappa}^{(0)}= \xi_{n, \psi, \kappa}^{(0)} (\Ball^p(1))$). 
By conditioning on $(\bu, \by)$, we can apply Theorem~\ref{thm:Gordon-improve} to get for
any $t\in \R$: 
\begin{equation*}
%\begin{split}
\P(\xi_{n, \psi, \kappa}^{(0)}  \le t \mid \bu, \by) \le 2 \, \P(\xi_{n, \psi, \kappa}^{(1)} \le t \mid \bu, \by) \;\;\;\; \text{and}\;\;\;\;
\P(\xi_{n, \psi, \kappa}^{(0)}  \ge t \mid \bu, \by) \le 2 \, \P(\xi_{n, \psi, \kappa}^{(1)}  \ge t \mid \bu, \by)\, .
%\end{split}
\end{equation*}
Taking expectation over $\bu, \by$ on both sides of the equation gives for any $t\in \R$, 
\begin{equation}
\label{eqn:xi-0-xi-1-equiv}
%\begin{split}
\P(\xi_{n, \psi, \kappa}^{(0)} \le t ) \le 2 \, \P(\xi_{n, \psi, \kappa}^{(1)}  \le t ) \;\;\;\; \text{and}\;\;\;\; 
\P(\xi_{n, \psi, \kappa}^{(0)}  \ge t ) \le 2 \, \P(\xi_{n, \psi, \kappa}^{(1)}  \ge t )\, .
%\end{split}
\end{equation}
The claim for $\xi^{(0)}_{n,\psi,\kappa}(\bTheta_p)$, $\xi^{(1)}_{n,\psi,\kappa}(\bTheta_p)$ follows by the same argument.

\newcommand{\lowzz}{{\rm low}}
\subsection{Proof of Lemma~\ref{lemma:xi-1-xi-2}}
\label{sec:lemma:xi-1-xi-2}

Let us introduce the notation 
\begin{equation*}
\begin{split}
f_{n, \psi, \kappa}^{(1)} (\btheta) &= \frac{1}{\sqrt{p}}\ltwobigg{\left(\kappa \one - \langle \bLambda^{1/2} \bw, \btheta \rangle (\by \odot \bu) 
		- \ltwobig{\proj_{\bw^{\perp}} \bLambda^{1/2} \btheta} (\by \odot \bh)\right)_+}  
		+ \frac{1}{\sqrt{p}} \bg^{\sT}\proj_{\bw^{\perp}}\bLambda^{1/2} \btheta, \\
f_{n, \psi, \kappa}^{(2)}(\btheta) &= 
		\psi^{-1/2} \cdot F_{\kappa} \left(\langle \bLambda^{1/2} \bw, \btheta \rangle, \ltwobig{\proj_{\bw^{\perp}} \bLambda^{1/2} \btheta} \right) +
			 \frac{1}{\sqrt{p}} \bg^{\sT} \proj_{\bw^{\perp}}\bLambda^{1/2} \btheta.
\end{split}
\end{equation*}
We claim that
\begin{equation}
\label{eqn:uniform-convergence-f-1-f-2}
\sup_{\btheta: \ltwo{\btheta} \le 1} \left|f_{n, \psi, \kappa}^{(1)} (\btheta)  - f_{n, \psi, \kappa}^{(2)} (\btheta) \right| \pto 0.
\end{equation}
Let us show that this claim implies the statement of Lemma~\ref{lemma:xi-1-xi-2}. By 
definition, Eq.~\eqref{eqn:def-xi-1-xi-1-bar}, Eq.~\eqref{eqn:def-xi-1-bar} and Eq.~\eqref{eqn:def-xi-2} 
gives for any compact $\bTheta_p$, 
\begin{equation}
\label{eqn:new-represent-xi-1-xi-2-f-Theta-p}
\xi_{n, \psi, \kappa}^{(1)}(\bTheta_p) = \left(\bar{\xi}_{n, \psi, \kappa}^{(1)}(\bTheta_p)\right)_+ = 
 	\min_{\btheta \in \bTheta_p} \left(f_{n, \psi, \kappa}^{(1)} (\btheta)\right)_+ 
~~\text{and}~~
\xi_{n, \psi, \kappa}^{(2)}(\bTheta_p) = 
 	\min_{\btheta \in \bTheta_p} \left(f_{n, \psi, \kappa}^{(2)} (\btheta)\right)_+. 
\end{equation}
Note that the mapping $x \to (x)_+$ is Lipschitz. Thereby, we have when $\bTheta_p \subseteq \{\btheta: \ltwo{\btheta} \le 1\}$,  
\begin{equation}
\left|\xi_{n, \psi, \kappa}^{(1)}(\bTheta_p) - \xi_{n, \psi, \kappa}^{(2)}(\bTheta_p)\right|
	\le \sup_{\btheta: \ltwo{\btheta} \le 1} 
	\left|\left(f_{n, \psi, \kappa}^{(1)} (\btheta)\right)_+  - \left(f_{n, \psi, \kappa}^{(2)} (\btheta)\right)_+ \right| 
	\le \sup_{\btheta: \ltwo{\btheta} \le 1} 
	\left|f_{n, \psi, \kappa}^{(1)} (\btheta)  - f_{n, \psi, \kappa}^{(2)} (\btheta) \right|.
\end{equation}
With the uniform convergence result at Eq.~\eqref{eqn:uniform-convergence-f-1-f-2}, this immediately implies the 
desired Lemma~\ref{lemma:xi-1-xi-2}.  

In the rest of the proof, we prove the claim \eqref{eqn:uniform-convergence-f-1-f-2}. 
We introduce the two functions 
\begin{equation}
\label{eqn:introduce-aux-g}
g_{n, \psi, \kappa}^{(1)}(\nu,q)=\frac{1}{\sqrt{p}}\ltwo{\left(\kappa \one - \nu (\by \odot \bu) 
		- q (\by \odot \bh)\right)_+}  
~~\text{and}~~
g_{\psi, \kappa}^{(2)}(\nu,q)= \psi^{-1/2} \cdot F_{\kappa}(\nu, q),
\end{equation}
and we denote for each $\btheta\in \R^p$, 
\begin{equation*}
 \nu(\btheta)=\langle \bLambda^{1/2} \bw, \btheta \rangle
 ~~\text{and}~~
  q(\btheta)=\ltwobig{\proj_{\bw^{\perp}} \bLambda^{1/2} \btheta}.
\end{equation*}
By definition, we have that for all $\btheta \in \R^p$, 
\begin{equation}
\label{eqn:uniform-convergence-from-f-to-g}
%\sup_{\btheta: \ltwo{\btheta} \le 1} 
\left|f_{n, \psi, \kappa}^{(1)} (\btheta)  - f_{n, \psi, \kappa}^{(2)} (\btheta) \right| 
= 
%\sup_{\btheta: \ltwo{\btheta} \le 1} 
	\left|g_{n, \psi, \kappa}^{(1)}\left(\nu(\btheta),q(\btheta)\right) - g_{\psi, \kappa}^{(2)}\left(\nu(\btheta),q(\btheta)\right)\right|
\end{equation}
Note $\ltwo{w}\le 1$ and  $\opnorm{\bLambda} \le C$ by Assumption~\ref{assumption:Lambdas}. 
Thus, for all $\btheta \in \{\btheta: \ltwo{\btheta} \le 1\}$, 
\begin{equation}
\label{eqn:nu-q-bound}
|\nu(\btheta)| \le \ltwobig{\bLambda^{1/2}\bw} \ltwo{\btheta} \le C^{1/2}
~~\text{and}~~
q(\btheta) \le \ltwobig{\bLambda^{1/2}\btheta} \le C^{1/2}.
\end{equation}
Therefore, Eq.~\eqref{eqn:uniform-convergence-from-f-to-g} and Eq.~\eqref{eqn:nu-q-bound} immediately 
implies that
\begin{equation}
\label{eqn:pass-f-to-g}
\sup_{\btheta: \ltwo{\btheta} \le 1} 
	\left|f_{n, \psi, \kappa}^{(1)} (\btheta)  - f_{n, \psi, \kappa}^{(2)} (\btheta) \right| 
\le \sup_{|\nu| \le C^{1/2}, q \le C^{1/2}}
	\left|g_{n, \psi, \kappa}^{(1)}\left(\nu, q\right) - g_{\psi, \kappa}^{(2)}\left(\nu, q\right)\right|
\end{equation}
By Eq.~\eqref{eqn:pass-f-to-g}, we can establish the desired uniform convergence 
result \eqref{eqn:uniform-convergence-f-1-f-2}, by proving
\begin{equation}
\label{eqn:goal-uniform-converge-g}
\sup_{|\nu| \le C^{1/2}, q \le C^{1/2}}
	\left|g_{n, \psi, \kappa}^{(1)}\left(\nu, q\right) - g_{\psi, \kappa}^{(2)}\left(\nu, q\right)\right| \pto 0. 
\end{equation}
The proof of Eq.~\eqref{eqn:goal-uniform-converge-g} is based on standard uniform convergence argument 
from empirical process theory. To start with, let us introduce the i.i.d random variables $\{Z_i(\nu, q)\}_{i=1}^n$ by
\begin{equation}
\label{eqn:def-Z}
Z_i (\nu, q) = (\kappa-\nu y_i u_i - q y_i h_i)_+.
\end{equation}
so we can have by definition, 
\begin{equation*}
\left(g_{n, \psi, \kappa}^{(1)}\left(\nu, q\right)\right)^2 = \frac{1}{p}\sum_{i=1}^n Z_i(\nu, q)^2 = \psi^{-1} 
	\cdot \frac{1}{n}\sum_{i=1}^n Z_i(\nu, q)^2.
\end{equation*}
Now, it is natural to introduce the quantity $\bar{g}_{n, \psi, \kappa}^{(1)}\left(\nu, q\right)$ such that
\begin{equation}
\left(\bar{g}_{n, \psi, \kappa}^{(1)}\left(\nu, q\right)\right)^2 = \psi^{-1} \cdot \E\left[Z(\nu, q)^2\right]
	= \psi^{-1} \cdot \E \left[(\kappa-\nu y u - q y h)_+^2\right],
\end{equation}
where $Z(\nu, q) = (\kappa-\nu y u - q y h)_+$ and the expectation on the RHS is taken w.r.t the 
random variables $(y, u, h)$, whose joint distribution is specified by  
\begin{equation}
\label{eqn:def-yuh}
(h, u) \perp y,~~h, u \sim \normal(0, 1),~~
\P(y = 1\mid u) = f(\|\btheta_{*,n}\|_{\bSigma} \,\cdot u)
\end{equation}
Recall the definition of $g_{\psi, \kappa}^{(2)}\left(\nu, q\right)$ in Eq.~\eqref{eqn:introduce-aux-g}. We 
know that 
\begin{equation}
\left(g_{\psi, \kappa}^{(2)}\left(\nu, q\right)\right)^2 = \psi^{-1} \cdot F_{\kappa}^2(\nu, q)
	= \psi^{-1} \cdot \E [ (\kappa-\nu Y U - q Y H)^2_+]
\end{equation}
where the expectation on the RHS is taken w.r.t the random variables $(Y, U, H)$, whose joint distribution is specified by  
\begin{equation}
\label{eqn:def-YUH}
(H, U) \perp Y,~~H, U \sim \normal(0, 1),~~\P(Y = 1\mid U) = f(\rho \cdot U).
\end{equation}
We can prove the desired Eq.~\eqref{eqn:goal-uniform-converge-g} by showing that 
\begin{itemize}
\item The uniform convergence from $g_{n, \psi, \kappa}^{(1)}(\nu, q)$ to 
	$ \bar{g}_{n, \psi, \kappa}^{(1)}\left(\nu, q\right)$:
	\begin{equation}
	\label{eqn:conv-g1}
\sup_{|\nu| \le C^{1/2}, 0 \le q \le C^{1/2}}
	\left|g_{n, \psi, \kappa}^{(1)}\left(\nu, q\right) - \bar{g}_{n, \psi, \kappa}^{(1)}\left(\nu, q\right)\right| \asto 0
	\end{equation}
\item The uniform convergence from $\bar{g}_{n, \psi, \kappa}^{(1)}(\nu, q)$ to $
	\bar{g}_{n, \psi, \kappa}^{(2)}\left(\nu, q\right)$:
	\begin{equation}
	\label{eqn:conv-g1bar}
\sup_{|\nu| \le C^{1/2}, 0 \le q \le C^{1/2}}
	\left|\bar{g}_{n, \psi, \kappa}^{(1)}\left(\nu, q\right) - g_{\psi, \kappa}^{(2)}\left(\nu, q\right)\right| \asto 0		
	\end{equation}
\end{itemize}
\newcommand{\setzz}{\mathcal{S}} In the rest of the proof, for notational simplicity, we introduce the compact set 
	$\setzz_C \subseteq \R^2$
\begin{equation*}
\setzz_C = \left\{(\nu, q): |\nu| \le C^{1/2}, 0 \le q \le C^{1/2}\right\}.
\end{equation*}
Now, we establish the below three important facts. 
\begin{enumerate}[(a)]
\item There exists some constant $c_0 > 0$ independent of $n, \nu, q$, such that for all $t > 0$, and $(\nu, q) \in \setzz_C$
\begin{equation}
\label{eqn:concentration-g1}
\P \left( \left|\frac{1}{n}\sum_{i=1}^n Z_i(\nu, q)^2- \E [Z(\nu, q)^2] \right| > t \right) 
	\le 2\exp \left(-c_0 n\min\{t, t^2\}\right),
\end{equation}
where $Z(\nu, q) = (\kappa-\nu y u - q y h)_+$ and $y,u,h$ is distributed according to Eq.~\eqref{eqn:def-yuh}. Indeed, 
it is not hard to show for some $M < \infty$ independent of $\nu, q, n$, we have for all $(\nu, q) \in \setzz_C$, the 
i.i.d random variables $\{Z_{i}(\nu,q) \}_{1\leq i \leq n}$ are subgaussian with parameter at most $M$, and thereby 
$\{Z_{i}(\nu,q)^2 \}_{1\leq i \leq n}$ are subexponential with parameter at most $M^2$. Thus, the 
desired concentration result at Eq.~\eqref{eqn:concentration-g1} follows by the standard Bernstein inequality 
\cite[Thm 2.8]{Vershynin18}. 

\item There exists some numerical constant $c_1 > 0$ so that with probability at least $1-\exp(-c_1 n)$, 
	the mapping $(\nu,q) \mapsto g_{n,\psi,\kappa}^{(1)}(\nu,q)$ is $3\psi^{-1/2}$-Lipschitz continuous. In fact, 
	for any pairs $(\nu, q)$ and $(\nu^\prime, q^\prime)$, by triangle inequality (and recall the definition of 
	$g_{n,\psi,\kappa}^{(1)}$ at Eq.~\eqref{eqn:introduce-aux-g})
	\begin{equation*}
	\begin{split}
		\left| g_{n,\psi,\kappa}^{(1)}(\nu,q) -g_{n,\psi,\kappa}^{(1)}(\nu^\prime,q^\prime) \right| 
			&\leq 
		\frac{1}{\sqrt{p}} \left(|\nu - \nu^\prime|\ltwo{(\by\odot \bu)_+}+|q-q^\prime|\ltwo{(\by\odot \bh)_{+}}\right)  \\
			&\le 
		\frac{1}{\sqrt{p}} \left(\ltwo{\bu} + \ltwo{\bh}\right)  \ltwo{(\nu-\nu^\prime, q-q^\prime)}.
	\end{split}
	\end{equation*}
	Now, we recall $\bu \sim \normal(0, I_n)$ and $\bh \sim \normal(0, I_n)$, and thus with high probability 
	$\ltwo{\bu} + \ltwo{\bh} \le 3\sqrt{n}$.

\item The function $(\nu,q) \mapsto \bar{g}_{n,\psi,\kappa}^{(1)}(\nu,q)$ is $3\psi^{-1/2}$-Lipschitz continuous. 
	Recall $Z(\nu, q) = (\kappa-\nu y u - q y h)_+$. Define $\Delta(\nu, q) = -\nu yu - q yh$. By 
	the elementary inequality $(a+b)_+ \le a_+ + b_+$, we have 
	\begin{equation}
		Z(\nu, q)_+ \le Z(\nu^\prime, q^\prime)_+ + \Delta(\nu - \nu^\prime, q - q^\prime)_+. 
	\end{equation}
	Therefore, Minkowski's inequality implies that for any pairs $(\nu, q)$ and $(\nu^\prime, q^\prime)$, 
	\begin{equation}
	\label{eqn:minkovski-to-Z}
		\left(\E [Z(\nu, q)_+]^2\right)^{1/2} \le \left(\E [Z(\nu^\prime, q^\prime)_+^2]\right)^{1/2}
			+ \E [\left(\Delta(\nu - \nu^\prime, q - q^\prime)_+^2]\right)^{1/2},
	\end{equation}
	where we are taking expectation over $y,u,h$ whose distribution is specified by Eq.~\eqref{eqn:def-yuh}. 
	Now that 
	\begin{equation}
	\E [\Delta(\nu, q)_+^2] \le \E \left|\Delta(\nu, q)\right|^2] \le 2\left(\nu^2 \E[|yu|^2] + q^2\E[|yh|^2]\right)
		= 2(\nu^2 + q^2).
	\end{equation}
	Thereby, Eq.~\eqref{eqn:minkovski-to-Z} implies for any pairs $(\nu, q)$ and $(\nu^\prime, q^\prime)$, 
	\begin{equation}
		\left|\left(\E [Z(\nu, q)_+]^2\right)^{1/2} - \left(\E [Z(\nu^\prime, q^\prime)_+]^2\right)^{1/2} \right|
			\le 2\ltwo{(\nu-\nu^\prime, q-q^\prime)}.
	\end{equation}
	This proves the result as $\bar{g}_{n, \psi, \kappa}^{(1)}\left(\nu, q\right) = \psi^{-1/2} \cdot \left(\E [Z(\nu, q)_+]^2\right)^{1/2}$.
\end{enumerate}

\paragraph{Proof of \eqref{eqn:conv-g1}} 

Now, we can prove Eq.~\eqref{eqn:conv-g1} via standard covering argument. Take the minimal $\eps_n = n^{-1/4}$ covering 
of the set $\setzz_C$, and we denote them by $\{(\nu_i, q_i)\}_{i \in T_n}$. Standard volume argument shows that 
$T_n \le 2C^2 \eps_n^{-2}$. Denote the event $\event_n$ to be 
\begin{equation}
\label{eqn:def-targetevent}
\event_n = 
	\left\{\max_{i \in [T_n]} \left| g_{n,\psi,\kappa}^{(1)}(\nu_i, q_i)- \bar{g}_{n,\psi,\kappa}^{(1)}(\nu_i, q_i) \right| \le \eps_n \right\}
		\cap 
	\left\{(\nu,q) \mapsto g_{n,\psi,\kappa}^{(1)}(\nu,q)~\text{is $3\psi^{-1/2}$-Lipschitz}\right\}.
\end{equation}
It is clear that on the event $\event_n$, we have 
\begin{equation}
\label{eqn:g-n-bar-g-n-sup-bound}
\sup_{(\nu, q)\in \setzz_C} \left| g_{n,\psi,\kappa}^{(1)}(\nu, q)- \bar{g}_{n,\psi,\kappa}^{(1)}(\nu, q)  \right|
	\le (3\psi^{-1/2} + 1) \eps_n. 
\end{equation}
Point $(a), (b)$ and the union bound above shows that $\P(\event_n) \ge 1 - (T_n + 1) \exp(-c_2 n\eps_n^2)$ for some numerical 
constant $c_2 > 0$. By Borel-Cantelli lemma, (almost surely) there exists some $N \in \N$ such that $\event_n$ 
happens for all $n > N$. Thus, for some $N \in \N$, we have Eq.~\eqref{eqn:g-n-bar-g-n-sup-bound} holds for all $n > N$, 
and since $\eps_n \to 0$, this shows the desired almost sure convergence: 
\begin{equation*}
\lim_{n\to \infty}\sup_{(\nu, q)\in \setzz_C} \left| g_{n,\psi,\kappa}^{(1)}(\nu, q)- \bar{g}_{n,\psi,\kappa}^{(1)}(\nu, q)  \right| = 0.
\end{equation*} 

\paragraph{Proof of \eqref{eqn:conv-g1bar}} Note that $\|\btheta_{*,n}\|_{\bSigma} \to \rho$. By dominated 
convergence theorem, we have for $(\nu, q) \in \setzz$
	\begin{equation*}
		\lim_{n \to \infty}
			\left|\bar{g}_{n, \psi, \kappa}^{(1)}\left(\nu, q\right) - g_{\psi, \kappa}^{(2)}\left(\nu, q\right)\right|  = 0.
	\end{equation*}
Point $(c)$ above implies that the class of functions $\Big\{\bar{g}_{n, \psi, \kappa}^{(1)}\left(\nu, q\right)\Big\}_{n\in \N}$ is 
equicontinuous. Thus, Arzel\`{a}-Ascoli theorem implies the desired uniform convergence:
\begin{equation}
\lim_{n \to \infty} \sup_{(\nu, q) \in \setzz}
			\left|\bar{g}_{n, \psi, \kappa}^{(1)}\left(\nu, q\right) - g_{\psi, \kappa}^{(2)}\left(\nu, q\right)\right|  = 0.
\end{equation}

\newcommand{\ball}{\mathcal{B}}
\newcommand{\gap}{{\rm gap}}

\section{Asymptotics of the prediction error: Proof of Proposition~\ref{proposition:nu-n-nu-infty}}
\label{sec:proof-proposition:nu-n-nu-infty}

We introduce the notation: %(recall that $c_i\opt(\psi) = c_i(\psi, \kappa\opt(\psi))$ for $i \in \{1, 2\}$): 
\begin{equation}
\label{eqn:def-kappa-n-nu-n-nu-opt}
\hat{\kappa}_n = \min_{i \in [n]} y_i \langle \hat{\btheta}_n^{\sMM}, \bx_i\rangle,\;\;\;\;
\hat{\nu}_n = \frac{\langle \hat{\btheta}_n^{\sMM}, \btheta_{*, n}\rangle_{\bSigma_n}}
	{\normbig{\hat{\btheta}_n^{\sMM}}_{\bSigma_n}\norm{\btheta_{*, n}}_{\bSigma_n}}, \;\;\;\;
%\nu\opt(\psi) = \frac{c_1\opt(\psi)}{\sqrt{(c_1\opt(\psi))^2 +(c_2\opt(\psi))^2}}
\end{equation}
We define auxiliary functions 
$\xi_{n, \alpha, \kappa}^{(i)}(\nu)$, $\bar{\xi}_{n, \alpha, \kappa}^{(i)}(\nu_1, \nu_2)$ as follows:
\begin{itemize}
\item We set for $i \in \{0, 1, 2\}$, $\nu \in [-1, 1]$, 
	\begin{equation}
		\xi_{n, \psi, \kappa}^{(i)}(\nu) = \xi_{n, \psi, \kappa}^{(i)}(\bTheta_p(\nu)),
	\end{equation}
	where 
	\begin{equation}
	\bTheta_p(\nu) 
		= \left\{\btheta \in \R^p: \ltwo{\btheta} \le 1,
			~\frac{\langle \btheta_{*, n}, \btheta\rangle_{\bSigma_n}}{\norm{\btheta_{*, n}}_{\bSigma_n}\norm{\btheta}_{\bSigma_n}}
			=\nu \right\}.
	\end{equation}
	So in particular, we have 
	\begin{equation}
	\begin{split}
	\label{eqn:def-xi-nu-illustration}
	\xi_{n, \psi, \kappa}(\nu) &=  \min_{\ltwo{\btheta} \le 1,  
		\langle \btheta, \btheta_{*, n}\rangle_{\bSigma_n}/(\norm{\btheta_{*, n}}_{\bSigma_n}\norm{\btheta}_{\bSigma_n}) = \nu} 
			\frac{1}{\sqrt{p}}\ltwo{(\kappa \one - \by \odot \bX\btheta)_+} \, ,\\
	 \xi_{n, \psi, \kappa}^{(2)}(\nu) &=  \min_{\ltwo{\btheta} \le 1,  \langle \btheta, \bLambda^{1/2} \bw \rangle/\norm{\btheta}_{\bLambda} = \nu} 
		\psi^{-1/2} \cdot F_{\kappa} \left(\langle \bLambda^{1/2}\bw, \btheta\rangle, \ltwobig{\proj_{\bw^{\perp}} \bLambda^{1/2} \btheta} \right) +
			 \frac{1}{\sqrt{p}} \bg^{\sT} \proj_{\bw^{\perp}}\bLambda^{1/2} \btheta \, .
	\end{split}
	\end{equation}
	Note by definition   
	\begin{equation}
	\label{eqn:def-xi-nu-global-illustration}
	\begin{split}
	\xi_{n, \psi, \kappa} &= \min_{\nu\in [-1, 1]} \xi_{n, \psi, \kappa}(\nu)
		= \min_{\ltwo{\btheta} \le 1}~ \frac{1}{\sqrt{p}}\ltwo{(\kappa \one - \by \odot \bX\btheta)_+} \\
	 \xi_{n, \psi, \kappa}^{(2)} &=  \min_{\nu\in [-1, 1]} \xi_{n, \psi, \kappa}^{(2)}(\nu) = 
	 	\min_{\ltwo{\btheta} \le 1} ~
		\psi^{-1/2} \cdot F_{\kappa} \left(\langle \bLambda^{1/2}\bw, \btheta\rangle, \ltwobig{\proj_{\bw^{\perp}} \bLambda^{1/2} \btheta} \right) +
			 \frac{1}{\sqrt{p}} \bg^{\sT} \proj_{\bw^{\perp}}\bLambda^{1/2} \btheta
	\end{split}
	\end{equation}
\item We set for any $i \in \{0, 1, 2\}$, $\nu_1, \nu_2 \in [-1, 1]$, 
	\begin{equation}
	\label{def:bar-xi-nu-1-nu-2}
		\bar{\xi}_{n, \alpha, \kappa}^{(i)}(\nu_1, \nu_2)
			= \min \left\{\min_{\nu \le \nu_1} \xi^{(i)}_{n, \psi, \kappa}(\nu), \min_{\nu \ge \nu_2} \xi^{(i)}_{n, \psi, \kappa}(\nu)\right\}
	%	~~\text{and}~~
	%	\xi_{n, \psi, \kappa}^{(i), \ge}(\nu_0) = \min_{\nu \ge \nu_0} \xi_{n, \psi, \kappa}^{(i)}(\nu).
	\end{equation}
%	One important observation is that
%	\begin{equation}
%		\xi_{n, \psi, \kappa}^{(i), \le}(\nu) = \xi_{n, \psi, \kappa}^{(i)}(\bTheta_p^{\le}(\nu))
%		~~\text{and}~~
%		\xi_{n, \psi, \kappa}^{(i), \ge}(\nu) = \xi_{n, \psi, \kappa}^{(i)}(\bTheta_p^{\ge}(\nu))
%	\end{equation}
%	where 
%	\begin{equation*}
%	\bTheta_p^{\le}(\nu) 
%		= \left\{\btheta: \frac{\langle \btheta_{*, n}, \btheta\rangle_{\bSigma_n}}{\norm{\btheta_{*, n}}_{\bSigma_n}\norm{\btheta}_{\bSigma_n}}
%			\le \nu \right\}
%		~~\text{and}~~
%	\bTheta_p^{\ge}(\nu) 
%		= \left\{\btheta: \frac{\langle \btheta_{*, n}, \btheta\rangle_{\bSigma_n}}{\norm{\btheta_{*, n}}_{\bSigma_n}\norm{\btheta}_{\bSigma_n}}
%			\ge \nu \right\}.
%	\end{equation*}
\end{itemize}
The above definitions  imply
\begin{equation}
\label{eqn:nu-kappa-formula}
\hat{\kappa}_n = \sup_{\kappa} \left\{\xi_{n, \psi, \kappa} = 0\right\},\;\;\;\;\;
\hat{\nu}_{n} \in \left\{\nu: \xi_{n, \psi, \hat{\kappa}_n}(\nu) = 0\right\}. 
%\hat{\alpha}_{n, 2} = \inf_{\alpha_2 \ge 0} \left\{\xi_{n, \psi, \hat{\kappa}_n, 2}(\alpha_2) = 0\right\}, 
\end{equation}
Now, we prove our desired goal of the proposition, i.e.,
\begin{equation}
\label{eqn:goal-of-proposition-generalization}
\hat{\nu}_n \pto \nu\opt(\psi). 
\end{equation}
For any $\eps > 0$, denote $\nu_+\opt(\psi;\eps) = \nu\opt(\psi) + \eps$ and $\nu_-\opt(\psi;\eps) = \nu\opt(\psi) - \eps$. 
By Eq.~\eqref{eqn:nu-kappa-formula}, it suffices to show that
\begin{equation}
\label{eqn:claim-prop-generalization}
	\lim_{n\to \infty, p/n\to \psi}	\P\left(\bar{\xi}_{n, \psi, \hat{\kappa}_n}
		(\nu_{-}\opt(\psi; \eps), \nu_+\opt(\psi; \eps)) > 0\right) = 1~~\text{for any $\eps > 0$}.
\end{equation}
The rest of the proof will establish Eq.~\eqref{eqn:claim-prop-generalization}. 
Define for any $\nu_0, \nu_1 \in [-1, 1]$,
\begin{equation}
\label{eqn:def-gap-function}
\gap(\nu_0, \nu_1) = \min_{\substack{
	c_{1, 0}/\sqrt{c_{1, 0}^2 + c_{2, 0}^2} = \nu_0,~ 
	c_{1, 1}/\sqrt{c_{1, 1}^2 + c_{2, 1}^2} = \nu_1, \\
	\max\left\{c_{1, 0}^2 + c_{2, 0}^2,~c_{1, 1}^2 + c_{2, 1}^2\right\} \le x_{\max}\\
	{c_{1, 1/2} = (c_{1, 0}+c_{1, 1})/2,~ c_{2, 1/2} = (c_{2, 0}+c_{2, 1})/2}}} 
	\left\{ \half \Big(F_{\kappa}(c_{1, 0}, c_{2, 0}) + 
		F_{\kappa}(c_{1, 1}, c_{2, 1}) \Big) - F_{\kappa}(c_{1, 1/2}, c_{2, 1/2}) \right\}.
\end{equation}
By Lemma~\ref{lemma:F-convex-increasing}, $F_{\kappa}$ is strictly convex and continuous, and thus 
we know that $(i)$ $(\nu_0, \nu_1) \to \gap(\nu_0, \nu_1)$ is lower-semicontinuous and 
$(ii)$ $\gap(\nu_0, \nu_1) > 0$ when $\nu_0 \neq \nu_1$. The crucial observation is 
Lemma~\ref{lemma:xi-F-kappa-nu} below, whose proof we defer into Section~\ref{sec:proof-lemma-xi-F-kappa-beta}.

\begin{lemma}
\label{lemma:xi-F-kappa-nu}
For any $\nu_0, \nu_1 \in [-1, 1]$, we have  
\begin{equation}
\label{eqn:xi-2-F-kappa-lower-bound-nu}
\half \Big(\xi^{(2)}_{n, \psi, \kappa}(\nu_0) + \xi^{(2)}_{n, \psi, \kappa}(\nu_1)\Big) - \xi^{(2)}_{n, \psi, \kappa}
	\ge \psi^{-1/2} \cdot \gap(\nu_0, \nu_1).
\end{equation}
\end{lemma} 
\noindent\noindent
Now we are ready to prove Eq.~\eqref{eqn:claim-prop-generalization}. Recall our notation that 
$\hat{\btheta}_{n, \psi, \kappa}^{(2)} \in \R^{p}$ is the optimal solution of the optimization problem defining 
$\xi_{n, \psi, \kappa}^{(2)}$ (see Eq.~\eqref{eqn:def-xi-nu-global-illustration}). Let us denote 
$\hat{\nu}_{n, \psi, \kappa}^{(2)}$ to be 
\begin{equation*}
\hat{\nu}_{n, \psi, \kappa}^{(2)} = \langle \xi_{n, \psi, \kappa}^{(2)}, \bLambda^{1/2} \bw \rangle/\normbig{\xi_{n, \psi, \kappa}^{(2)}}_{\bLambda}.
\end{equation*}
By definition, it is clear for all $n, \psi, \kappa$, 
\begin{equation}
\label{eqn:def-stupid-global-min}
\xi_{n, \psi, \kappa}^{(2)} = \xi_{n, \psi, \kappa}^{(2)} \left(\hat{\nu}_{n, \psi, \kappa}^{(2)}\right).  
\end{equation}
Now, we use Lemma~\ref{lemma:xi-F-kappa-nu}. Plugging $\nu_0 = \hat{\nu}_{n, \psi, \kappa}^{(2)}$, 
$\kappa = \kappa\opt(\psi)$ into Eq.~\eqref{eqn:xi-2-F-kappa-lower-bound-nu}, and using
Eq.~\eqref{eqn:def-stupid-global-min}, we get 
\begin{equation}
\label{eqn:special-case-xi-gap-nu}
\xi_{n, \psi, \kappa\opt(\psi)}^{(2)}(\nu_1) - \xi_{n, \psi, \kappa\opt(\psi)}^{(2)} 
	\ge 2\psi^{-1/2} \cdot \gap(\hat{\nu}_{n, \psi, \kappa\opt(\psi)}^{(2)}, \nu_1).
\end{equation}
holds for all $\nu_1 \in [-1, 1]$. Now, we define for any $\nu_0, \nu_1, \nu_2 \in [-1, 1]$, $\kappa > 0$
\begin{equation}
\label{def:bar-gap-nu-0-nu-1-nu-2}
\overline{\gap}(\nu_0, \nu_1, \nu_2) = \min \left\{\min_{\nu \le \nu_1} \gap(\nu_0, \nu),~
	\min_{\nu \ge \nu_2} \gap_{\kappa}(\nu_0, \nu)\right\}. 
\end{equation}
Note that $(i)$ $(\nu_0, \nu_1, \nu_2) \to \overline{\gap}_{\kappa}(\nu_0, \nu_1, \nu_2)$ is lower-semicontinuous 
(because $\gap_{\kappa}(\cdot, \cdot)$ is lower semicontinuous), and $(ii)$ for any $\nu_1 < \nu_0 < \nu_2$
$\overline{\gap}_{\kappa}(\nu_0, \nu_1, \nu_2) > 0$ (because $\gap(\kappa)(\nu_0, \nu) > 0$ for any $\nu \neq \nu_0$). 
Now, by Eq.~\eqref{eqn:special-case-xi-gap-nu}, Eq.~\eqref{def:bar-xi-nu-1-nu-2} and Eq.~\eqref{def:bar-gap-nu-0-nu-1-nu-2}, 
we know for all $\nu_1, \nu_2 \in [-1, 1]$, 
\begin{equation}
\label{eqn:special-case-xi-gap-nu-1-nu-2}
\bar{\xi}_{n, \psi, \kappa\opt(\psi)}^{(2)}(\nu_1, \nu_2) - \xi_{n, \psi, \kappa\opt(\psi)}^{(2)} 
	\ge 2\psi^{-1/2} \cdot \overline{\gap}(\hat{\nu}_{n, \psi, \kappa\opt(\psi)}^{(2)}, \nu_1, \nu_2).
\end{equation}
Recall Proposition~\ref{proposition:xi-n-2-xi-infty}. We have almost surely (recall that $c_i\opt(\psi) = c_i(\psi, \kappa\opt(\psi))$ for $i \in \{1, 2\}$), 
\begin{equation}
\label{eqn:almost-sure-converge-xi-2-nu-2}
\lim_{n\to \infty, p/n\to \psi} \xi^{(2)}_{n, \psi, \kappa\opt(\psi)} = T(\psi, \kappa\opt(\psi)) = 0
~~\text{and}~~
\lim_{n\to \infty, p/n\to \psi} \hat{\nu}_{n, \psi, \kappa\opt(\psi)}^{(2)} = 
	\frac{c_1\opt(\psi)}{\sqrt{(c_1\opt(\psi))^2 +(c_2\opt(\psi))^2}} = \nu\opt(\psi).
\end{equation}
Therefore, for any fixed $\nu_1, \nu_2$, 
by Eq.~\eqref{eqn:special-case-xi-gap-nu-1-nu-2} and Eq.~\eqref{eqn:almost-sure-converge-xi-2-nu-2}, 
we have almost surely
\begin{equation}
\label{eqn:xi-2-nu-1-nu-2-eps-out}
\liminf_{n\to \infty, p/n\to \psi} \bar{\xi}_{n, \psi, \kappa\opt(\psi)}^{(2)}(\nu_1, \nu_2) 
	\ge 2\psi^{-1/2} \cdot \overline{\gap}(\nu\opt(\psi), \nu_1, \nu_2).
\end{equation} 
Now, for any fixed $\eps > 0$, we define $\eta(\psi; \eps) > 0$ by 
\begin{equation}
\eta(\psi; \eps) = 2 \psi^{-1} \cdot \overline{\gap}(\nu\opt(\psi), \nu_-\opt(\psi; \eps), \nu_+\opt(\psi; \eps)).
\end{equation}
then Eq.~\eqref{eqn:xi-2-nu-1-nu-2-eps-out} implies in particular that 
\begin{equation}
\label{eqn:xi-2-pos-eps-out}
\liminf_{n\to \infty, p/n\to \psi} \P\left(\bar{\xi}_{n, \psi, \kappa\opt(\psi)}^{(2)}(\nu_-\opt(\psi; \eps), \nu_+\opt(\psi; \eps)) 
	\ge \eta(\psi; \eps)\right) = 1.
\end{equation}
Notice that $\bar{\xi}_{n, \psi, \kappa\opt(\psi)}^{(2)}(\nu_1, \nu_2) = 
\bar{\xi}_{n, \psi, \kappa\opt(\psi)}^{(2)}(\bTheta_p(\nu_1, \nu_2))$ where the set 
$\bTheta_p(\nu_1, \nu_2)$ is defined by 
\begin{equation}
\bTheta_p(\nu_1, \nu_2) = 
\left\{\btheta:  \ltwo{\btheta} \le 1, 
	\frac{\langle \btheta_{*, n}, \btheta\rangle_{\bSigma_n}}{\norm{\btheta_{*, n}}_{\bSigma_n}\norm{\btheta}_{\bSigma_n}}
			\le \nu_1 \right\}
		\bigcup 
\left\{\btheta:  \ltwo{\btheta} \le 1, 
	\frac{\langle \btheta_{*, n}, \btheta\rangle_{\bSigma_n}}{\norm{\btheta_{*, n}}_{\bSigma_n}\norm{\btheta}_{\bSigma_n}}
			\ge \nu_2 \right\}.
\end{equation} 
Thus, Eq.~\eqref{eqn:xi-2-pos-eps-out}, Lemma~\ref{lemma:xi-0-xi-1} and Lemma~\ref{lemma:xi-1-xi-2} imply that 
\begin{equation}
\label{eqn:xi-n-psi-kappa-opt-psi-pos}
\lim_{n\to \infty} \P\left(\xi_{n, \psi, \kappa\opt(\psi)}(\nu_1, \nu_2) > \eta(\psi; \eps) \right) = 1.
\end{equation}
Finally, we notice that 
\begin{itemize}
\item The function $\kappa \to \xi_{n, \psi, \kappa}(\nu_1, \nu_2)$ is $((\psi)^{-1/2})$-Lipschitz for all $n, \psi, \nu_1, \nu_2$.
	This is due to $(i)$ the mapping $\kappa\to \frac{1}{\sqrt{p}}\ltwo{(\kappa \ones - y \odot X\theta)_+}$ is $((\psi)^{-1/2})$ 
	Lipschitz, and $(ii)$ the variational characterization below: 
	\begin{equation*}
		\xi_{n, \psi, \kappa}(\nu_1, \nu_2) = \min_{\btheta: \norm{\btheta} \le 1, \btheta \in \bTheta_p(\nu_1, \nu_2)}
			\frac{1}{\sqrt{p}}	\ltwo{(\kappa \ones - y \odot X\theta)_+}
	\end{equation*}
\item the convergence  $\hat{\kappa}_n \pto \kappa\opt(\psi)$ which is implied by 
	Eq.~\eqref{eqn:nu-kappa-formula}, Eq.~\eqref{eqn:summary-xi-pos} and Eq.~\eqref{eqn:summary-xi-neg}.
\end{itemize}
Using the above two facts, and the high probability bound at Eq.~\eqref{eqn:xi-n-psi-kappa-opt-psi-pos}, we get
\begin{equation}
\lim_{n\to \infty} \P\left(\xi_{n, \psi, \hat{\kappa}_n}(\nu_1, \nu_2) > \eta(\psi; \eps) \right) = 1.
\end{equation}
This gives the desired claim at Eq.~\eqref{eqn:claim-prop-generalization}, and thus the proposition.

\subsection{Proof of Lemma~\ref{lemma:xi-F-kappa-nu}}
\label{sec:proof-lemma-xi-F-kappa-beta}
Let $\btheta_0 \in \R^p$, $\btheta_1 \in \R^p$ be such that 
$\langle \bLambda^{1/2}\bw, \btheta_0\rangle/\norm{\btheta_0}_{\bLambda}= \nu_0$,
$\langle \bLambda^{1/2}\bw, \btheta_1\rangle/\norm{\btheta_1}_{\bLambda}= \nu_1$ and
\begin{equation}
\label{eqn:xi-F-kappa-one}
\begin{split}
\xi^{(2)}_{n, \psi, \kappa}(\nu_0) = \psi^{-1/2} \cdot F_{\kappa}
	\left(\langle \bLambda^{1/2}\bw, \btheta_0\rangle,~\ltwobig{\proj_{\bw^{\perp}}\bLambda^{1/2} \btheta_0}\right) + \frac{1}{\sqrt{p}}
		\bg^{\sT} \proj_{\bw^{\perp}} \Lambda^{1/2}\btheta_0 \\
\xi^{(2)}_{n, \psi, \kappa}(\nu_1) = \psi^{-1/2} \cdot F_{\kappa}
	\left(\langle \bLambda^{1/2}\bw, \btheta_1\rangle,~\ltwobig{\proj_{\bw^{\perp}}\bLambda^{1/2} \btheta_1}\right) + \frac{1}{\sqrt{p}}
		\bg^{\sT} \proj_{\bw^{\perp}} \Lambda^{1/2}\btheta_1.
\end{split}
\end{equation}
Denote $\btheta_{1/2} = \half(\btheta_0 + \btheta_1)$. By definition of $\xi^{(2)}_{n, \psi, \kappa}$, we know that 
\begin{equation}
\label{eqn:xi-F-kappa-two}
\begin{split}
\xi^{(2)}_{n, \psi, \kappa} &\le  \psi^{-1/2} \cdot F_{\kappa}
	\left(\langle \bLambda^{1/2}\bw, \btheta_{1/2}\rangle, \ltwobig{\proj_{\bw^{\perp}}\bLambda^{1/2} \btheta_{1/2}}\right) + \frac{1}{\sqrt{p}}
		\bg^{\sT} \proj_{\bw^{\perp}} \Lambda^{1/2}\btheta_{1/2} \\
	&\le \psi^{-1/2} \cdot F_{\kappa}
		\left(\langle \bLambda^{1/2}\bw, \btheta_{1/2}\rangle,~\half\left( \ltwobig{\proj_{\bw^{\perp}}\bLambda^{1/2} \btheta_{0}} +  
			\ltwobig{\proj_{\bw^{\perp}}\bLambda^{1/2} \btheta_{1}} \right)\right) + \frac{1}{\sqrt{p}}
		\bg^{\sT} \proj_{\bw^{\perp}} \Lambda^{1/2}\btheta_{1/2}
\end{split}
\end{equation}
where the second line follows since $F_{\kappa}$ is increasing w.r.t its second argument (see Lemma 
\ref{lemma:F-convex-increasing}). Denote
\begin{equation*}
c_{1, 0} = \langle \bLambda^{1/2}\bw, \btheta_0\rangle,~ 
c_{1, 1} = \langle \bLambda^{1/2}\bw, \btheta_1\rangle,~ 
c_{2, 0} = \ltwobig{\proj_{\bw^{\perp}}\bLambda^{1/2} \btheta_0},~
c_{2, 1} = \ltwobig{\proj_{\bw^{\perp}}\bLambda^{1/2} \btheta_1}.
\end{equation*} 
Let  $c_{1, 1/2} = \half(c_{1, 0} + c_{1, 1})$ and $c_{2, 1/2} = \half(c_{2, 0} + c_{2, 1})$.
From Eq.~\eqref{eqn:xi-F-kappa-one} and Eq.~\eqref{eqn:xi-F-kappa-two}, we have 
\begin{equation*}
\half \left(\xi^{(2)}_{n, \psi, \kappa}(\nu_0) + \xi^{(2)}_{n, \psi, \kappa}(\nu_1)\right)- \xi^{(2)}_{n, \psi, \kappa}
\ge \psi^{-1/2} \cdot  \half \Big(F_{\kappa}(c_{1, 0}, c_{2, 0}) + 
		F_{\kappa}(c_{1, 1}, c_{2, 1}) \Big) - F_{\kappa}(c_{1, 1/2}, c_{2, 1/2}).
\end{equation*}
Note that $\nu_0 = c_{1, 0}/\sqrt{c_{1, 0}^2 + c_{2, 0}^2}$ and $\nu_1 = c_{1, 1}/\sqrt{c_{1, 1}^2 + c_{2, 1}^2}$.
This gives the desired Lemma~\ref{lemma:xi-F-kappa-nu}.

\section{Asymptotics of the Coordinate Distribution of $\sqrt{p}\hat{\btheta}_n^{\sMM}$: Proof of Proposition~\ref{proposition:ECD-maxmargin}}
\label{sec:proof-proposition:ECD-maxmargin}

\subsection{Notation}
Recall $\hat{\kappa}_n$ denotes the margin of the maximum-linear classifier, i.e., 
$\hat{\kappa}_n = \min_{i \in [n]} y_i \langle \hat{\btheta}_n^{{\sMM}}, \bx_i\rangle$. 
Equivalently, we have that $\xi_{n, \psi, \hat{\kappa}_n} = 0$ and 
\begin{equation}
\label{eqn:nu-kappa-connection}
\hat{\kappa}_n = \sup_{\kappa} \left\{\xi_{n, \psi, \kappa} = 0\right\}.
\end{equation}
For each $\btheta \in \R^p$ with $\ltwo{\btheta} \le 1$, we denote $\LL_p(\btheta)$ to be the following empirical distribution 
\begin{equation*}
\LL_p(\btheta) = \frac{1}{p} \sum_{i=1}^p \delta_{(\lambda_i, \bar{w}_i, \sqrt{p}\btheta_i)}.
\end{equation*}
For each $\eps > 0$, define $\bTheta_p(\eps)$ to be the Wasserstein ball around the distribution $\LL_{\psi, \kappa^*(\psi)}$: 
\begin{equation*}
\bTheta_p(\eps) = \left\{\btheta\in \R^p: \ltwo{\btheta} \le 1, W_2(\LL_p(\btheta), \LL_{\psi, \kappa^*(\psi)}) \le \eps\right\}.
\end{equation*}
Denote $\bTheta_p^c(\eps)$ to be its complement w.r.t the unit ball $\{\btheta \in \R^p: \ltwo{\btheta} \le 1\}$: 
\begin{equation*}
\bTheta_p^c(\eps) = \left\{\btheta\in \R^p: \ltwo{\btheta} \le 1, W_2(\LL_p(\btheta), \LL_{\psi, \kappa^*(\psi)})  > \eps\right\}.
\end{equation*}
Using the above notation, the goal of the section is to establish 
\begin{equation*}
W_2(\LL_p({\hat{\btheta}_n^{\sMM}}), \LL_{\psi, \kappa\opt(\psi)}) \pto 0.
\end{equation*}
In other words, the goal is to establish for any $\eps > 0$, 
\begin{equation*}
  \lim_{n \to \infty} \P\left(\hat{\btheta}_n^{{\sMM}} \in \bTheta_p(\eps)\right) = 1.
\end{equation*}

\subsection{Main content of the analysis} 
Recall that $\xi_{n, \psi, \hat{\kappa}_n} = 0$. Thus, by definition, $\xi_{n, \psi, \hat{\kappa}_n}(\bTheta_p^c(\eps)) > 0$ 
implies that $\hat{\btheta}_n^{{\sMM}} \in \bTheta_p(\eps)$. Therefore, it suffices to prove for any $\eps > 0$, 
\begin{equation}
\label{eqn:goal-one-asprediction}
    \lim_{n \to \infty} \P\left(\xi_{n,\psi,\hat{\kappa}_n}\left(\bTheta_{p}^c(\eps)\right) > 0\right)=1.
\end{equation}
The proof is based on the following three steps.  
\begin{enumerate}
\item In the first step, we prove that, for any $\eps, \Delta > 0$, 
\begin{equation}
\label{eqn:step-one-wball}
\limsup_{n\to \infty, p/n \to \psi} \left[\P\left(\xi_{n,\psi,\hat{\kappa}_n}\left(\bTheta_{p}^c(\eps)\right) \le \Delta\right) - 2 \P\left(
	\xi^{(2)}_{n,\psi,\kappa\opt(\psi)}\left(\bTheta_{p}^c(\eps)\right) \le \Delta \right)\right] \le 0.
\end{equation}
This is basically saying that $\xi_{n,\psi,\hat{\kappa}_n}\left(\bTheta_{p}^c(\eps)\right)$ is stochastically larger than 
$\xi^{(2)}_{n,\psi,\kappa\opt(\psi)}\left(\bTheta_{p}^c(\eps)\right)$ (up to a factor of $2$ in Eq.~\eqref{eqn:step-one-wball}).
The constant $2$ is from the Gordon's comparison inequality (cf. Lemma~\ref{lemma:xi-0-xi-1}).
\item In the second step, we prove for any $\eps > 0$,  
\begin{equation}
\label{eqn:step-two-wball}
\lim_{n\to \infty, p/n \to \psi}\P\left(\xi^{(2)}_{n,\psi,\kappa\opt(\psi)}\left(\bTheta_{p}^c(\eps)\right) \ge
	\xi^{(2)}_{n,\psi,\kappa\opt(\psi)}\left(\ball_p^c(\eps/2)\right)\right) = 1,
\end{equation}
where (as a reminder, for any $\psi, \kappa > 0$, $\hat{\btheta}_{n, \psi, \kappa}^{(2)}$ denotes the minimum of the 
optimization problem defining $\xi_{n, \psi, \kappa}^{(2)}$ (see Eq.~\eqref{eqn:xi-2-expansion} for detail))
\begin{equation*}
\begin{split}
\ball_p^c(\eps') &\defeq
	\left\{\btheta\in \R^p: \ltwo{\btheta} \le 1, \ltwobig{\btheta-\hat{\btheta}_{n, \psi, \kappa^*(\psi)}^{(2)}} > \eps'\right\}.
\end{split}
\end{equation*}
The intuition why Eq.~\eqref{eqn:step-two-wball} holds is that $\bTheta_{p}^c(\eps) 
\subseteq \ball^c(\eps/2)$ for large $n, p$ when $p/n = \psi$ (we will see this from the proof) .
\item In the third step, we prove for any $\eps > 0$, there exists $\eta(\psi; \eps) > 0$ (independent of $n$), such that 
\begin{equation}
\label{eqn:step-three-wball}
\lim_{n\to \infty, p/n \to \psi} \P \left(\xi^{(2)}_{n,\psi,\kappa\opt(\psi)}\left(\ball_{p}^c(\eps)\right) \ge \eta(\psi; \eps)\right) = 1.
\end{equation}
\end{enumerate}
It is straightforward to see the above facts imply the desired Eq.~\eqref{eqn:goal-one-asprediction}, thus giving 
the desired Proposition~\ref{proposition:ECD-maxmargin}. We prove the above three facts in the three paragraphs below. 

\paragraph{Proof of Step 1 (Eq.~\eqref{eqn:step-one-wball})}
First, Lemma~\ref{lemma:xi-0-xi-1} and Lemma~\ref{lemma:xi-1-xi-2} relate $\xi_{n, \psi, \kappa^*(\psi)}\left(\bTheta_{p}^c(\eps)\right)$ 
and $\xi^{(2)}_{n, \psi, \kappa^*(\psi)}\left(\bTheta_{p}^c(\eps)\right)$---showing that for any $\eps, \Delta > 0$, 
\begin{equation}
\limsup_{n\to \infty, p/n \to \psi} \left[\P\left(\xi_{n,\psi, \kappa^*(\psi)}\left(\bTheta_{p}^c(\eps)\right) \le \Delta\right) - 2 \P\left(
	\xi^{(2)}_{n, \psi, \kappa^*(\psi)}\left(\bTheta_{p}^c(\eps)\right) \le \Delta \right)\right] \le 0.
\end{equation}
Thus, it suffices to prove for any $\eps, \Delta > 0$, 
\begin{equation*}
\limsup_{n\to \infty, p/n \to \psi} \left[\P\left(\xi_{n,\psi, \hat{\kappa}_n}\left(\bTheta_{p}^c(\eps)\right) \le \Delta\right) - 
	\P\left(\xi_{n,\psi, \kappa^*(\psi)}\left(\bTheta_{p}^c(\eps)\right) \le \Delta\right)\right] \le 0.
\end{equation*}
We show a slightly stronger result---we prove the convergence: as $n\to \infty, p/n \to \psi$, 
\begin{equation*}
\xi_{n,\psi,\hat{\kappa}_n}\left(\bTheta_{p}^c(\eps)\right) - \xi_{n,\psi,\kappa\opt(\psi)}\left(\bTheta_{p}^c(\eps)\right) \pto 0
\end{equation*}
To see the above convergence in probability result, we note the below two facts:
\begin{enumerate}[(a)]
	\item we have the convergence $\hat{\kappa}_n \pto \kappa\opt(\psi)$ by comparing
		Eq.~\eqref{eqn:nu-kappa-connection}, Eq.~\eqref{eqn:summary-xi-pos} and Eq.~\eqref{eqn:summary-xi-neg}. 
	\item we have the mapping $\kappa \to \xi_{n, \psi, \kappa}(\bTheta^c_p(\eps))$ is $(\psi)^{-1/2}$-Lipschitz uniformly in 
		$n, \psi, \eps$. Indeed, this is true since we have the variational characterization 
		$\xi_{n, \psi, \kappa}(\bTheta_{p}^c(\eps)) = \min_{\btheta \in \bTheta_p^c(\eps)}
			\frac{1}{\sqrt{p}}	\ltwo{(\kappa \ones - y \odot X\theta)_+}$
		and for each $\btheta$, the mapping $\kappa\to \frac{1}{\sqrt{p}}\ltwo{(\kappa \ones - y \odot X\theta)_+}$ is $(\psi)^{-1/2}$ 
		Lipschitz, again uniformly in $n, \psi, \eps$.
\end{enumerate}

\paragraph{Proof of Step 2 (Eq.~\eqref{eqn:step-two-wball})}
%Define for any $\eps > 0$, $n \in \N$ the event $\bLambda_n(\eps) = \left\{\bTheta_{p}^c(\eps) \subseteq \ball_p^c(\eps/2) \right\}$.
By definition, it suffices to prove for any $\eps > 0$, $\bTheta_{p}^c(\eps) \subseteq \ball_p^c(\eps/2)$
happens eventually as $n \to \infty, p/n\to \psi$. To see this, the two key observations are: 

\begin{enumerate}[(a)]
\item $W_2(\LL_p(\btheta_1), \LL_p(\btheta_2)) \le \ltwo{\btheta_1-\btheta_2}$ for any $\btheta_1, \btheta_2 \in \R^p$.
\item $W_2(\LL_p(\hat{\btheta}_{n, \psi, \kappa^*(\psi)}^{(2)}), \LL_{\psi, \kappa^*(\psi)}) \asto 0$ as $n \to \infty, p/n \to \psi$ 
by Proposition~\ref{proposition:xi-n-2-xi-infty}.(b). (as a kind note: we apply Proposition~\ref{proposition:xi-n-2-xi-infty}.(b) to 
$\kappa = \kappa\opt(\psi)$, and the notation here $\LL_p(\hat{\btheta}_{n, \psi, \kappa^*(\psi)}^{(2)})$ has the same meaning 
as the notation $\LL_{n, \psi, \kappa\opt(\psi)}^{(2)}$ in Proposition~\ref{proposition:xi-n-2-xi-infty})
\end{enumerate}
Now, pick any $\btheta \in \bTheta_p^c(\eps)$ so $W_2(\LL_p(\btheta), \LL_{\psi, \kappa^*(\psi)}) \ge \eps$. Now triangle inequality and 
point $(a)$ above indicate 
\begin{equation}
\label{eqn:two-balls-relation}
\begin{split}
\ltwobig{\btheta - \hat{\btheta}_{n, \psi, \kappa^*(\psi)}^{(2)}}
	&\ge W_2(\LL_p(\btheta), \LL_p(\hat{\btheta}_{n, \psi, \kappa^*(\psi)}^{(2)})) \\
	&\ge W_2(\LL_p(\btheta), \LL_{\psi, \kappa^*(\psi)}) 
		- W_2(\LL_p(\hat{\btheta}_{n, \psi, \kappa^*(\psi)}^{(2)}), \LL_{\psi, \kappa^*(\psi)}) \\  
	&\ge \eps - W_2(\LL_p(\hat{\btheta}_{n, \psi, \kappa^*(\psi)}^{(2)}), \LL_{\psi, \kappa^*(\psi)}). 
\end{split}
\end{equation}
Now Point $(b)$ and Eq.~\eqref{eqn:two-balls-relation} imply that, as  $n \to \infty, p/n \to \psi$, we have for any $\btheta \in \bTheta_p^c(\eps)$,
\begin{equation*}
\ltwobig{\btheta - \hat{\btheta}_{n, \psi, \kappa^*(\psi)}^{(2)}} \ge \eps/2.
\end{equation*}
This means $\bTheta_p^c(\eps) \subseteq \ball_p^c(\eps/2)$ as $n \to \infty, p/n\to \psi$, as desired. 

\paragraph{Proof of Step 3 (Eq.~\eqref{eqn:step-three-wball})}
We define for any $\delta > 0, \kappa > 0$ the gap function: 
(we remind the reader that $x_{\max}, x_{\min} > 0$ are constants independent of $n\in \N$ such that 
	$x_{\max} \ge \lambda_{\max}(\bLambda) \ge \lambda_{\min}(\bLambda) \ge x_{\min}$ where
	$\bLambda = \bLambda_n$ for all $n \in\N$)
\begin{equation}
\label{eqn:gap-function-indpn}
\gap_{\kappa}(\delta) = \min_{\substack{
	\max\left\{c_{1, 0}^2 + c_{2, 0}^2,~c_{1, 1}^2 + c_{2, 1}^2\right\} \le x_{\max}\\
	{c_{1, 1/2} = (c_{1, 0}+c_{1, 1})/2,~ c_3 \le (c_{2, 0}+c_{2, 1})/2} \\
	x_{\min}^{-1} \cdot \left(8 x_{\max}^{1/2} \left(\half\left(c_{2, 0} + c_{2, 1}\right) - c_3\right) + (c_{2, 0} - c_{2, 1})^2
	 +  (c_{1, 0} - c_{1, 1})^2\right) \ge \delta^2}} 
	\left\{ \half \Big(F_{\kappa}(c_{1, 0}, c_{2, 0}) + 
		F_{\kappa}(c_{1, 1}, c_{2, 1}) \Big) - F_{\kappa}(c_{1, 1/2}, c_3) \right\}.
\end{equation}
Our proof of Step 3 is based on the following two facts: 
\begin{enumerate}
\item $\eta(\psi; \eps) = \psi^{-1} \cdot \gap_{\kappa^*(\psi)}(\eps/2) > 0$ and 
\item  Eq.~\eqref{eqn:step-three-wball} holds for $\eta(\psi; \eps) = \psi^{-1} \cdot \gap_{\kappa^*(\psi)}(\eps/2)$. 
\end{enumerate}
Below we give their proofs. 

The proof of fact $(i)$ is simple---$\gap_{\kappa}(\delta) > 0$ for any $\delta > 0$ since $(c_1, c_2) \to F_{\kappa}(c_1, c_2)$ is 
continuous, strictly convex w.r.t $(c_1, c_2)$, and strictly increasing w.r.t $c_2$ for any given $c_1$, thanks to Lemma~\ref{lemma:F-convex-increasing}.
Thus, we have $\eta(\psi; \eps) = \psi^{-1} \cdot \gap_{\kappa}(\eps/2) > 0$. 

The proof of fact $(ii)$ is trickier. Introduce the function $f_{n, \psi, \kappa}$ for all $n, \psi, \kappa$:
\begin{equation*}
f_{n, \psi, \kappa}(\btheta) = \psi^{-1/2} \cdot F_{\kappa} \left(\langle \btheta, \bLambda^{1/2}\bw\rangle, 
	\ltwo{\proj_{\bw^{\perp}} \bLambda^{1/2} \btheta}\right) + \frac{1}{\sqrt{p}} \bg^T \proj_{\bw^\perp} \bLambda^{1/2}\btheta. 
\end{equation*}
The main technical tool for showing fact $(ii)$ is Lemma~\ref{lemma:strictly-convex-ind-n} below. To avoid interruption of the flow, we defer 
its proof into Section~\ref{sec:proof-lemma:strictly-convex-ind-n}.

\begin{lemma}
\label{lemma:strictly-convex-ind-n}
For any $\psi, \kappa > 0$, and any $\btheta_0, \btheta_1$ satisfying $\ltwo{\btheta_0} \le 1, \ltwo{\btheta_1} \le 1$,  
\begin{equation}
\label{eqn:strictly-convex-ind-n}
\half \left(f_{n, \psi, \kappa}(\btheta_0) + f_{n, \psi, \kappa}(\btheta_1)\right) - 
	f_{n, \psi, \kappa}\left(\frac{1}{2}(\btheta_0 + \btheta_1)\right) \ge \psi^{-1/2} \cdot \gap_{\kappa} \left(\ltwo{\btheta_0 - \btheta_1}\right).
\end{equation}
\end{lemma}

Now we show how Lemma~\ref{lemma:strictly-convex-ind-n} implies that Eq.~\eqref{eqn:step-three-wball} holds for 
$\eta(\psi; \eps) = \psi^{-1} \cdot \gap_{\kappa^*(\psi)}(\eps/2)$. To see this, we first prove 
\begin{equation}
\label{eqn:almost-sure-converge-xi-2-nu-2-pre}
\xi^{(2)}_{n,\psi, \kappa^*(\psi)}\left(\ball_p^c(\eps)\right) \ge 
	\xi_{n, \psi, \kappa^*(\psi)}^{(2)} +2\eta(\psi; \eps).
\end{equation}

Pick $\btheta \in \ball_p^c(\eps)$. By substituting $\kappa = \kappa\opt(\psi)$, $\btheta_0 = \hat{\btheta}_{n, \psi, \kappa^*(\psi)}^{(2)}$, 
$\btheta_1 = \btheta$ into Eq.~\eqref{eqn:strictly-convex-ind-n}, we obtain
\begin{equation}
\label{eqn:immediate-from-strict-convex-pre}
\half \left(f_{n, \psi, \kappa^*(\psi)}(\hat{\btheta}_{n, \psi, \kappa^*(\psi)}^{(2)}) + f_{n, \psi, \kappa^*(\psi)}(\btheta)\right) - 
	f_{n, \psi, \kappa^*(\psi)}\left(\hat{\btheta}_{n, \psi, 1/2, *}^{(2)}\right) \ge \psi^{-1/2} \cdot \gap_{\kappa^*(\psi)} \left(\ltwobig{\hat{\btheta}_{n, \psi, \kappa^*(\psi)}^{(2)} - \btheta}\right)
		%\ge \psi^{-1/2}\cdot \gap_{\kappa}(\eps/2).
\end{equation}
for the $\hat{\btheta}_{n, \psi, 1/2, *}^{(2)} = \half \left(\hat{\btheta}_{n, \psi, \kappa^*(\psi)}^{(2)} + \btheta\right)$. Notice that 
\begin{itemize}
\item $\gap_{\kappa^*(\psi)} \left(\ltwobig{\hat{\btheta}_{n, \psi, \kappa^*(\psi)}^{(2)} - \btheta}\right) \ge \gap_{\kappa^*(\psi)}(\eps)$ for any $\btheta \in \ball_p^c(\eps)$. 
\item $f_{n, \psi, \kappa^*(\psi)}\left(\hat{\btheta}_{n, \psi, 1/2, *}^{(2)}\right) \ge f_{n, \psi, \kappa^*(\psi)}(\hat{\btheta}_{n, \psi, \kappa^*(\psi)}^{(2)})
	 = \xi_{n, \psi, \kappa^*(\psi)}^{(2)}$ since $\hat{\btheta}_{n, \psi, \kappa^*(\psi)}^{(2)}$ minimizes $f_{n, \psi, \kappa^*(\psi)}(\btheta)$.
\end{itemize}
Thus Eq.~\eqref{eqn:immediate-from-strict-convex-pre} implies for any $\btheta \in \ball_p^c(\eps)$
\begin{equation}
\label{eqn:immediate-from-strict-convex-pre-two}
f_{n, \psi, \kappa^*(\psi)}(\btheta) \ge \xi_{n, \psi, \kappa^*(\psi)}^{(2)} + 2\psi^{-1} \cdot \gap_{\kappa^*(\psi)}(\eps)
	= \xi_{n, \psi, \kappa^*(\psi)}^{(2)} + 2\eta(\psi; \eps).
\end{equation}
Now we use the crucial observation 
\begin{equation*}
\xi^{(2)}_{n,\psi, \kappa^*(\psi)}\left(\ball_p^c(\eps)\right) = \min_{\btheta \in \ball_p^c(\eps)} f_{n, \psi, \kappa^*(\psi)}(\btheta).
\end{equation*} 
Thus, it suffices to prove for any $\btheta \in \ball_p^c(\eps)$, 
\begin{equation}
f_{n, \psi, \kappa^*(\psi)}(\btheta) \ge \xi_{n, \psi, \kappa^*(\psi)}^{(2)} + 2\eta(\psi; \eps).
\end{equation}
We see that Eq.~\eqref{eqn:almost-sure-converge-xi-2-nu-2-pre} follows after taking infimum over $\btheta \in \ball_p^c(\eps)$ 
on both sides of Eq.~\eqref{eqn:immediate-from-strict-convex-pre-two}. 

Following the lower bound in Eq.~\eqref{eqn:almost-sure-converge-xi-2-nu-2-pre}, and using the fact that almost surely (see Proposition~\ref{proposition:xi-n-2-xi-infty})
\begin{equation}
\label{eqn:almost-sure-converge-xi-2-nu-2}
\lim_{n\to \infty, p/n\to \psi} \xi^{(2)}_{n, \psi, \kappa^*(\psi)} = \xi^{(2)}_{n, \psi, \kappa\opt(\psi)} = T(\psi, \kappa\opt(\psi)) = 0,
\end{equation}
one can easily see that Eq.~\eqref{eqn:step-three-wball} holds for $\eta(\psi, \eps)> 0$, i.e.,  
\begin{equation*}
\lim_{n\to \infty, p/n \to \psi} \P \left(\xi^{(2)}_{n,\psi,\kappa\opt(\psi)}\left(\ball_{p}^c(\eps)\right) \ge \eta(\psi; \eps)\right) = 1.
\end{equation*}
This concludes the proof. 

\subsection{Proof of Lemma~\ref{lemma:strictly-convex-ind-n}}
\label{sec:proof-lemma:strictly-convex-ind-n}
Let $\btheta_0 \in \R^p$, $\btheta_1 \in \R^p$ be such that $\ltwo{\btheta_0} \le 1, \ltwo{\btheta_1} \le 1$.
Denote $\btheta_{1/2} = \half \left(\btheta_0 + \btheta_1\right)$. Set
\begin{equation*}
\begin{split}
c_{1, 0} = \langle \bLambda^{1/2}\bw, \btheta_0\rangle,~~~~~~~~~~~~~~~~~~~~~~~~~~~&~ 
c_{1, 1} = \langle \bLambda^{1/2}\bw, \btheta_1\rangle,~\\
c_{2, 0} = \ltwobig{\proj_{\bw^{\perp}}\bLambda^{1/2} \btheta_0},~~~~~~~~~~~~~~~~~~~~~~&~
c_{2, 1} = \ltwobig{\proj_{\bw^{\perp}}\bLambda^{1/2} \btheta_1}, \\
c_{1, 1/2} = \langle \bLambda^{1/2}\bw, \btheta_{1/2}\rangle = \half(c_{1, 0} + c_{2, 0}),~&~
c_3 =  \ltwobig{\proj_{\bw^{\perp}}\bLambda^{1/2} \btheta_{1/2}}.
\end{split}
\end{equation*} 
With these definitions, we have 
\begin{equation*}
\begin{split}
&\half \left(f_{n, \psi, \kappa}(\btheta_0) + f_{n, \psi, \kappa}(\btheta_1)\right) - 
	f_{n, \psi, \kappa}\left(\frac{1}{2}(\btheta_0 + \btheta_1)\right) \\
&= \psi^{-1} \cdot \left(\half F_{\kappa}(c_{1, 0}, c_{2, 0}) + F_{\kappa}(c_{1, 1}, c_{2, 1}) - 
	F_{\kappa}(c_{1, 1/2}, c_{3}) \right)
\end{split}
\end{equation*}
Recall the definition of $\gap$ function at Eq.~\eqref{eqn:gap-function-indpn}. The desired Lemma 
\ref{lemma:strictly-convex-ind-n} follows if we can show: 
\begin{enumerate}[(a)]
\item $c_3 \le \half \left(c_{1, 2} + c_{2, 2}\right)$. 
\item $\max \left\{c_{1, 1}^2 + c_{1, 2}^2\right\} \le x_{\max}$. 
	$c_{2, 1}^2 + c_{2, 2}^2 = \ltwobig{\bLambda^{1/2} \btheta_0}^2 \le x_{\max}$.
\item $x_{\min}^{-1} \cdot \left(8 x_{\max}^{1/2} \left(\half\left(c_{2, 0} + c_{2, 1}\right) - c_3\right) + (c_{2, 0} - c_{2, 1})^2
	 +  (c_{1, 0} - c_{1, 1})^2\right) \ge \ltwo{\btheta_0 - \btheta_1}^2$.
\end{enumerate}
We show the above three elementary bounds as follows (which implies the desired Lemma~\ref{lemma:strictly-convex-ind-n} as discussed): 
\begin{enumerate}
\item Point $(a)$ is true due to triangle inequality. 
\item Point $(b)$ is true because 
	\begin{equation*}
	\begin{split}
		c_{1, 1}^2 + c_{1, 2}^2 &= \ltwobig{\bLambda^{1/2} \btheta_1}^2 \le \opnorm{\bLambda} \ltwo{\btheta_1} \le x_{\max} \\
		c_{2, 1}^2 + c_{2, 2}^2 &= \ltwobig{\bLambda^{1/2} \btheta_2}^2 \le \opnorm{\bLambda} \ltwo{\btheta_2} \le x_{\max}.
	\end{split}
	\end{equation*}
\item Point $(c)$ is a little bit involved. Let us denote 
\begin{equation*}
\begin{split}
\btheta_{0, w} = \langle\bLambda^{1/2}\bw, \btheta_0\rangle,  ~~\btheta_{0, w^\perp} = \proj_{\bw^{\perp}}\bLambda^{1/2} \btheta_0, \\
\btheta_{1, w} = \langle\bLambda^{1/2}\bw, \btheta_1\rangle,  ~~\btheta_{1, w^\perp} = \proj_{\bw^{\perp}}\bLambda^{1/2} \btheta_1.
\end{split}
\end{equation*}
By elementary calculation, we first have for any vector $\bz_0, \bz_1 \in \R^p$, 
\begin{equation*}
\half \left(\ltwo{\bz_0} + \ltwo{\bz_1}\right) - \ltwobig{\half (\bz_0 + \bz_1)} = 
	\frac{\ltwo{\bz_0}\ltwo{\bz_1} - \langle \bz_0, \bz_1\rangle}{\ltwo{\bz_0} + \ltwo{\bz_1} + \ltwobig{(\bz_0 + \bz_1)}}
\end{equation*}
Substitute $\bz_0= \btheta_{0, w^\perp}$ and $\bz_1= \btheta_{1, w^\perp}$ into above, we get that 
\begin{equation*}
\half \left(c_{2, 0} + c_{2, 1}\right) - c_3 = \frac{c_{2, 0} c_{2, 1} - \langle \btheta_{0, w^\perp}, \btheta_{1, w^\perp} \rangle}{c_{2, 0} + c_{2, 1} + 2c_3}
\end{equation*}
Now that $0 \le c_{2, 0}, c_{2, 1}, c_3 \le x_{\max}^{1/2}$ by Point $(a)$ and $(b)$. This proves that 
\begin{equation}
\label{eqn:elem-bound-one}
4 x_{\max}^{1/2}\left( \half \left(c_{2, 0} + c_{2, 1}\right) - c_3\right)  \ge c_{2, 0} c_{2, 1} - \langle \btheta_{0, w^\perp}, \btheta_{1, w^\perp} \rangle.
\end{equation}
Note further the identities: 
\begin{equation}
\label{eqn:elem-bound-two}
\begin{split}
\half (c_{2, 0} - c_{2, 1})^2 &=  \half \ltwo{\btheta_{0, w^\perp}}^2 + \half \ltwo{\btheta_{1, w^\perp}}^2 - c_{2, 0}c_{2, 1} \\
\half (c_{1, 0} - c_{1, 1})^2 &= \half \left(\btheta_{0, w} - \btheta_{1, w}\right)^2. 
\end{split}
\end{equation}
By summing up all the equations in Eq.~\eqref{eqn:elem-bound-one} and Eq.~\eqref{eqn:elem-bound-two}, we get that 
\begin{equation*}
\begin{split}
&4 x_{\max}^{1/2}\left( \half \left(c_{2, 0} + c_{2, 1}\right) - c_3\right) + \half (c_{2, 0} - c_{2, 1})^2
	 + \half (c_{1, 0} - c_{1, 1})^2  \\
&= \half \left(\btheta_{0, w} - \btheta_{1, w}\right)^2 + \half \ltwo{\btheta_{0, w^\perp} - \btheta_{1, w^\perp}}^2
= \half \ltwo{\bLambda^{1/2}(\btheta_0 - \btheta_1)}^2 \ge \half x_{\min} \ltwo{\btheta_0 - \btheta_1}^2. 
\end{split}
\end{equation*}
This proves point $(c)$. 

\end{enumerate}

\section{Data Separability} 
\label{sec:data-separability}
This section proves Lemma~\ref{lemma:proportional-asymptotics-of-Xi}. The proof contains two steps. 

\subsection{A comparison result} 
In light of the Gaussian comparison inequality (Theorem~\ref{thm:Gordon-improve}), we have 
\begin{equation*}
	\P(\Xi_{n, \psi} \le t) \le 2\P(\Xi_{n, \psi}^{(1)} \le t)~~\text{for all $t \in \R$}.
\end{equation*}
Above, $\Xi_{n, \psi}^{(1)}$ is defined by 
(below $\bg \in \normal(0, I_p), \bh \in \normal(0, I_n), \bu \in \normal(0, I_n)$ are independent vectors) 
\begin{equation*}
	\Xi_{n, \psi}^{(1)} = \min_{\btheta: \norm{\btheta}_{\bLambda} = 1}
 	\max_{\ltwo{\blambda} \le 1, \blambda \odot \by \geq 0}~
		\frac{1}{\sqrt{p}}  \left(\blambda^{\sT} (- \langle \bLambda^{1/2} \bw, \btheta \rangle \bu - 
			\ltwobig{\proj_{\bw^{\perp}} \bLambda^{1/2} \btheta}\bh) + \ltwo{\blambda} \bg^{\sT}\proj_{\bw^{\perp}}\bLambda^{1/2} \btheta\right).
\end{equation*}  
Below is a chain of inequalities on $\Xi_{n, \psi}^{(1)}$ that holds almost surely: 
\begin{equation*}
\begin{split}
	\Xi_{n, \psi}^{(1)}  &=
		\min_{\btheta: \norm{\btheta}_{\bLambda} = 1}
		\frac{1}{\sqrt{p}} \left(\ltwo{\left(- \langle \bLambda^{1/2} \bw, \btheta \rangle (\by \odot \bu) - 
			\ltwobig{\proj_{\bw^{\perp}} \bLambda^{1/2} \btheta} (\by \odot  \bh)\right)_{+}}
		+ \bg^{\sT}\proj_{\bw^{\perp}}\bLambda^{1/2} \btheta\right)_+ \\
		&\ge
		\min_{\btheta: \norm{\btheta}_{\bLambda} = 1}
		\frac{1}{\sqrt{p}} \left(\ltwo{\left(- \langle \bLambda^{1/2} \bw, \btheta \rangle (\by \odot \bu) - 
			\ltwobig{\proj_{\bw^{\perp}} \bLambda^{1/2} \btheta} (\by \odot  \bh)\right)_{+}}
		-\ltwo{\proj_{\bw^{\perp}}\bLambda^{1/2} \btheta}\ltwo{\bg}\right)_+ \\
		&\ge 
		\min_{c_1, c_2: c_1^2 + c_2^2 = 1}
			\frac{1}{\sqrt{p}} \left(\ltwo{ \left(-c_1 (\by \odot \bu) -c_2 (\by \odot  \bh)\right)_+} -c_2 \ltwo{\bg}\right)_+.
\end{split} 
\end{equation*}
Motivated by the above chain of inequality, we now introduce $\Xi_{n, \psi}^{(2)}$ to be 
\begin{equation*}
	\Xi_{n, \psi}^{(2)} =\frac{1}{\sqrt{p}} \left(\ltwo{ \left(-c_1 (\by \odot \bu) -
			c_2 (\by \odot  \bh)\right)_+} -c_2 \ltwo{\bg}\right)_+.
\end{equation*} 
Then all the previous results yield that 
\begin{equation*}
	\P(\Xi_{n, \psi} \le t) \le 2 \P(\Xi_{n, \psi}^{(2)} \le t)~~\text{for all $t \in \R$}.
\end{equation*}

\subsection{Typical value is positive} 
Below we show that, assuming $\psi < \psi^*(0)$, then the below convergence in probability holds
\begin{equation*}
	\Xi_{n, \psi}^{(2)} \pto \Xi^*(\psi)
\end{equation*}
for some $\Xi^*(\psi) > 0$. This allows us to conclude the desired Lemma~\ref{lemma:proportional-asymptotics-of-Xi}.

Indeed, using concentration results, (e.g., Lemma~\ref{lemma:xi-1-xi-2}), it is straightforward to see that, 
\begin{equation*}
	\Xi^{(2)} \pto \Xi^*(\psi)~~\text{where}~~\Xi^*(\psi)  =  \min_{(c_1, c_2): c_1^2 + c_2^2 = 1}
			\left(\frac{1}{\sqrt{\psi}} \cdot F_0(c_1, c_2) - c_2\right)_{+}.
\end{equation*}
To show that $\Xi^*(\psi) > 0$, it suffices to prove that
\begin{equation*}
	\frac{1}{\sqrt{\psi}} \cdot F_0(c_1, c_2) - c_2 > 0~~\text{holds for all $(c_1, c_2): c_1^2 + c_2^2 = 1$.}
\end{equation*} 
This follows from the following facts: (i) the mapping $(c_1, c_2) \mapsto F_0(c_1, c_2)$ is homogeneous (ii) 
the strict inequality $\psi < \psi^*(0) = \min_{c} F_0^2(c, 1)$ and (iii) $F_0(c_1, c_2) > 0$ on 
the set $(c_1, c_2): c_1^2 + c_2^2 = 1$.

\section{Proof of Theorem~\ref{thm:benign}}
\label{sec:proof:benign}
This section proves Theorem~\ref{thm:benign}, assuming Theorem~\ref{theorem:main} holds. 
Recall that Theorem~\ref{theorem:main}-$(e)$ shows that, under
Assumptions \ref{assumption:Lambdas}-\ref{assumption:non-degenerate-f}, the limit of 
$\Pred_n(\by,\bX)$ is $\Pred\opt(\mu, \psi)$. 

As a result, all that remains is to prove the following Proposition~\ref{prop:benign}.

\begin{proposition}\label{prop:benign}
    Assume Assumptions \ref{assumption:Lambdas}-\ref{assumption:non-degenerate-f}. Additionally, assume that the link function
$f$ is almost everywhere differentiable, monotonically increasing with $f(0) = 1/2$ and $f^\prime(0)>0$.
\begin{itemize}
\item (Lower bound) There exists a constant $c>0$ depending only on $\rho$, $\lam$ and the link function $f$, where $\rho$ (resp. $\lam$) are defined in \eqref{eq:def:rho} (resp. \eqref{eq:lambda:max}), such that 
	\begin{equation*}
		\error^\star(\mu,\psi)-\Bayes \geq c \cdot \frac{1}{\psi}\,.
	\end{equation*}
	Thus, in the proportional regime, the $\error^\star(\mu,\psi)$ cannot equal the Bayes error. Rather, $\error^\star(\mu,\psi)$ can only achieve near Bayes error when $\psi$ is large enough.
\item (Upper bound) Let $(X, W) \sim \mu$ where $\mu$ is defined in Assumption \ref{assumption:converge}. There exist constants $C > 0$ depending only on $\rho$, $\lam$ and the link function $f$ such that the following holds: for any $\lambda>0$
% \begin{equation}\label{eq:error:upper:bound:condition}
% \begin{split}
%   &\psi\cdot \E\big[X\ind(X\leq \lambda) + \ind(X \ge \lambda) \big] + \frac{1}{\psi}\cdot\frac{\E\big[X^2\ind(X\leq \lambda)\big]}{\big(\E\big[X\ind(X\leq \lambda)\big]\big)^2}
%   	 \leq c\, ,\\
% \end{split}
% \end{equation}
we have the error bound
\begin{equation*}
\error^\star(\mu,\psi)-\Bayes\leq C\cdot(\mathcal{B}+\mathcal{V})\,
\end{equation*}
where $\mathcal{B}\equiv \mathcal{B}(\la,\mu, \psi)$ and $\mathcal{V} \equiv \mathcal{V}(\la,\mu, \psi)$ are defined by
\begin{equation*}
\begin{split}
    &\mathcal{B}:= \Big(\psi\cdot \E\big[X\ind(X< \lambda)\big]\Big)^2\cdot\E\bigg[\bigg(\frac{W}{X}\bigg)^2\ind(X\geq \la)\bigg]+\E\Big[W^2\ind(X<\lambda)\Big]\, ,\\
    &\mathcal{V}:= \psi\cdot \E[ \ind(X \ge \lambda)]+\frac{1}{\psi}\cdot\frac{\E\big[X^2\ind(X<\lambda)\big]}{\big(\E\big[X\ind(X< \lambda)\big]\big)^2}\, .
\end{split}
\end{equation*}
We take the convention that the final fraction above is defined to be $1$ if $\E\big[X\ind(X< \lambda)\big]=0$.
\end{itemize}
\end{proposition}
Before diving into the proof of Proposition \ref{prop:benign}, we briefly recall a precise form of $\Pred\opt(\mu, \psi)$. Consider the 
following optimization problem indexed by $(\psi, \kappa)$: 
\begin{equation*}
\begin{split}
{\rm OPT}(\psi, \kappa):~~~~	\minimize_{h} ~~~~~ &
		\psi^{-1/2} \cdot F_{\kappa} \left( \langle X^{1/2} h, W \rangle_{\Q}, 
			\normbig{\proj_{W^{\perp}}(X^{1/2} h)}_{\Q} \right)  + \langle X^{1/2} \proj_{W^{\perp}}(G), h\rangle_{\Q} \\
	\subjectto~~~~~ &  \norm{h}_{\Q} \le 1.
\end{split}
\end{equation*}
The minimum value of the optimization is exactly $T(\psi, \kappa)$ for $\psi > \psi^{\low}(\kappa)$ (Appendix~\ref{sec:all-property-optimization}). Let 
\begin{equation*}
	\kappa^*(\psi) = \inf\{\kappa\ge 0: \psi^{\low}(\kappa) < \psi, T(\psi, \kappa) = 0\}.
\end{equation*}
Note that $\kappa^*(\psi)$ is well-defined (Proposition~\ref{proposition:system-of-Eq-T}). Furthermore,
if we denote $h_{\psi, \kappa\opt(\psi)}^*$ to be the minimizer of ${\rm OPT}(\psi, \kappa\opt(\psi))$,
and further introduce the quantities 
\begin{equation*}
	c_1\opt(\psi) =  \langle X^{1/2} h_{\psi, \kappa\opt(\psi)}^*, W \rangle_{\Q}
		~~\text{and}~~
	c_2\opt(\psi) = \normbig{\proj_{W^{\perp}}(X^{1/2} h_{\psi, \kappa\opt(\psi)}^*)}_{\Q}
\end{equation*} 
then Theorem~\ref{theorem:main} gives the explicit formula for the prediction error $\Pred\opt(\mu, \psi)$: 
\begin{equation*}
	\Pred\opt(\mu, \psi) = \Gamma \left(c_2\opt(\psi)/c_1\opt(\psi)\right)~~~~~~\text{where}~
		\Gamma(c) \defeq \P(\sign(c) YG + cZ \le 0).
\end{equation*}
Below we shall be interested to analyze the gap between the above prediction error and the 
Bayes error. Since the Bayes classifier $x\mapsto \one_{\langle \theta^*, x\rangle > 0}$ 
has test error: $\Bayes(\mu, \psi) \equiv \Gamma(0^+)$, this means that the gap is
\begin{equation*}
	{\rm GAP} = \Pred\opt(\mu, \psi) - \Bayes(\mu, \psi) = 
		\Gamma \left(c_2\opt(\psi)/c_1\opt(\psi)\right)- \Gamma(0^+).
\end{equation*}
To understand GAP, we need to understand the behavior of $\Gamma$ near $0^+$. This 
is studied in Lemma~\ref{lem:error:quadratic}.

\begin{lemma}\label{lem:error:quadratic}
Assume $f$ is differentiable a.e., and that $f$ is increasing with $f^\prime(0)>0$. Then 
(i) $\Gamma(-x) \ge 1/2 > \Gamma(x)$ for any $x > 0$
(ii) $\Gamma$ is monotonically increasing on $\reals_{> 0}$ with 
$\Gamma^\prime(0^+)=0$ and $\Gamma^{\prime\prime}(0^+)>0$. 
\end{lemma}
\begin{proof}
%That $\Gamma(x) = \Gamma(-x)$ follows from the fact that $Z\stackrel{d}{=}-Z$ when $Z\sim \normal(0,1)$. 
Below we evaluate $\Gamma'(x)$ and $\Gamma''(x)$ for $x \in \reals_{> 0}$. 
Let $g(x):=\P(YG\leq -x)$ so that $\Gamma(x) = \E[g(xZ)]$. 
Stein's identity yields $\Gamma^\prime(x)=\E[Z g^\prime(xZ)]=x \E g^{\prime\prime}(xZ)$.
As $\P(Y=1|G) = f(\rho Z)$, $g(x) = \E[(1-f(\rho Z) + f(-\rho Z)) \mathbf{1}(Z \ge x)]$.
Thus, $g^{\prime\prime}(x)=\rho(f^\prime(\rho x)+f^\prime(-\rho x))\phi(x)-xg^\prime(x)$
where $\phi(z) = \exp(-z^2/2)/\sqrt{2\pi}$ is the density of $\normal(0, 1)$. This leads to 
the expression for $x \in \R_{> 0}$
\begin{equation*}
    \Gamma^\prime(x)=\frac{\rho x}{1+x^2}\E[(f^\prime(\rho xZ)+f^\prime(-\rho xZ))\phi(xZ)].
\end{equation*}
Hence (i) $\Gamma'(0^+) = 0$, (ii) $\Gamma'(x) > 0$ for $x > 0$, (iii) $\Gamma(x) \le \Gamma(+\infty) = 1/2$ 
for $x > 0$ and (iv) $\Gamma^{\prime \prime}(0^+) =2\rho f^\prime(0)\phi(0)>0$. Similarly, 
one can show $\Gamma(-x) \ge \Gamma(-\infty) = 1/2$ for $x > 0$. 
\end{proof}

Noticeably, the function $\Gamma$ is completely determined by the size of $\rho$ and the link 
function $f$ (assumed to be fixed). As a consequence of Lemma~\ref{lem:error:quadratic}, 
the gap is completely characterized by 
\begin{equation}
	{{\rm GAP}} = \Pred\opt(\mu, \psi) - \Bayes(\mu, \psi) = 
		\Gamma \left(c_2\opt/c_1\opt\right) - \Gamma(0^+) \in 
			\left(\left(\left|c_2\opt/c_1\opt\right| + \one_{c_1\opt \le 0} \right) \wedge 1\right)^2 
				\cdot [a, b].
\end{equation}
where $a, b$ are constants depending only on $\rho$ and $f$. 
%
%
%where the notation $a \asymp_{\rho} b$ means that $c_1 a \le b \le c_2 b$ for some constants 
%$c_1, c_2 > 0$ depending only on $\rho$. In the proof, we also use $a \lrho b$, 
%or equivalently, $b \grho a$, if $a \le c b$ for some $c > 0$ depending only on $\rho$.
%
%
This reaches the below conclusion. 
\begin{equation*}
\text{
Understanding the surrogate gap ${\rm GAP}$ amounts to understanding the ratio 
$|c_2\opt/c_1\opt| + \one_{c_1\opt \le 0}$.
}
\end{equation*}
This leads to our strategy: in the subsequent sections, we shall derive upper and lower bounds 
on the GAP by deriving upper and lower bounds 
on the ratio $|c_2\opt/c_1\opt| + \one_{c_1\opt \le 0}$.

%\subsection{Relation Between Test Error and Ratio $|c_2\opt/c_1\opt| + \one_{c_1\opt \le 0}$} 
 
%Below is standard. 

\subsection{Lower bound in Proposition~\ref{prop:benign}}
\begin{proposition}\label{lem:lower:bound}
There exists $c_{\rho, f} > 0$ depending on $\rho, f$ such that %with probability at least $1-\delta$: 
%There exists $C_{\rho}^\prime$ depending only on $\rho$ such that for any $t \ge 0$, 
\begin{equation*}
	|c_2\opt/ c_1\opt| \ge c_{\rho, f} \cdot \psi^{-1/2}.
\end{equation*} 
%holds with probability at least $1-\exp\left(-\frac{nt^2}{4(t+\frac{p}{n})}\right)$.
\end{proposition}
\begin{proof}
As $T(\psi, \kappa\opt) = 0$, $\psi^{-1/2} F_{\kappa\opt}(c_1\opt, c_2\opt) + \langle X^{1/2} \proj_{W^{\perp}}(G), h\opt \rangle_{\Q} = 0$. By Cauchy Schwartz inequality, 
\begin{equation*}
	F_{\kappa\opt}(c_1\opt, c_2\opt)% = \psi^{1/2} 
		%\cdot |\langle X^{1/2} \proj_{W^{\perp}}(G), h\opt \rangle_{\Q}| 
			\le \psi^{1/2} \cdot \norm{G}_{\Q} \cdot \normbig{ \proj_{W^{\perp}}(X^{1/2}h\opt)}_{\Q}
			= \psi^{1/2} \cdot c_2\opt.
\end{equation*}
Since $(\kappa,c_1,c_2)\to F_{\kappa}(c_1,c_2)$ is homogenous---the LHS can be lower bounded by 
\begin{equation}\label{eq:lem:apriori:tech:2}
F_{\kappa\opt}(c_1\opt, c_2\opt)
	\geq (|\kappa\opt|^2+|c_1\opt|^2 + |c_2\opt|^2)^{1/2} \cdot \inf_{\substack{x^2+y^2+z^2=1\\ x,z\geq 0}} F_{x}(y,z).
\end{equation}
The result now follows from algebraic manipulation. 

%The result follows from concentration of $\norm{g}_2$ for $g \sim \normal(0, I_p)$ (e.g., 
%Example 2.12 of \cite{boucheron2013concentration}).
\end{proof}

\subsection{Upper bound in Proposition~\ref{prop:benign}}
The main goal of the section is to establish the following bound: for some
$C_{\rho, f} > 0$ depending only on $\rho, f$
\begin{equation}
\label{eq:prop:partial:two:F:upper:1}
 (\left|c_2\opt/c_1\opt\right| + \one_{c_1\opt \le 0}) \wedge 1 \le C \cdot
	\left(\E_{\Q}\left[ \left(\psi +\las^2 \frac{W^2}{X^2}\right) \mathbf{1}_{X > \lambda} + 
		\left(\frac{\psi X^2}{\las^2} + W^2\right) \mathbf{1}_{X \le \lambda}\right]
		\right)^{1/2}.
\end{equation}
Above, $\las := \psi \cdot \E_{\Q}[\mathbf{1}_{X \le \lambda}]$ is a notation 
shorthand for a constant repetitively appeared in the proof.

\paragraph{A Reduction Argument} 
Let $c_{\rho, f} > 0$ denote a constant---which we shall determine later in the proof---that depends only on $\rho, f$. 
Throughout the proof, we can W.L.O.G. assume that the following holds: 
\begin{equation}
\label{eqn:reduction-assumption}
	\text{Assumption (i)}:
		 \E_{\Q} \left[\psi \mathbf{1}_{X > \lambda} +  \frac{\psi X^2}{ \las^2} \mathbf{1}_{X \le \lambda}\right]
		 	\le c_{\rho, f}
	~~~\text{and}~~~\text{Assumption (ii)}: \las \le c_{\rho, f}. 
\end{equation}
Indeed, since $\E_{\Q}[W^2] = 1$, and $\supp(X) \subseteq [0, M]$, 
violation of either Assumption (i) or (ii) will lead to 
\begin{equation*}
\E_{\Q}\left[ \left(\psi +\las^2 \frac{W^2}{X^2}\right) \mathbf{1}_{X > \lambda} + 
		\left(\frac{\psi X^2}{\las^2} + W^2\right) \mathbf{1}_{X \le \lambda}\right]  \ge c_{\rho, f}'
\end{equation*}
for some constant $c_{\rho, f}' > 0$ depending only on $\rho$, so that we can choose $C_{\rho, f} > 0$ sufficiently 
large such that equation~\eqref{eq:prop:partial:two:F:upper:1} holds. 
 
As a consequence, the argument simply suggests the following reduction: throughout the proof, we can 
W.L.O.G. assume that the Assumptions in equation~\eqref{eqn:reduction-assumption} hold
for some $c_{\rho, f} > 0$ depending only on $\rho$, $f$.

%$a \asymp_{\rho} b$ 
%means that $c_1 a \le b \le c_2 b$ for some constants 
%$c_1, c_2 > 0$ depending only on $\rho$. In the proof, we also use $a \lrho b$, 
%or equivalently, $b \grho a$, if $a \le c b$ for some $c > 0$ depending only on $\rho$.

Hereafter, with abuse of notation, we abbreviate the arguments for the functions $F$, 
$\partial_{\kappa}F$, 
$\partial_1 F$, or $\partial_2 F$---which are certainly evaluating at 
the parameters $(\kappa\opt, c_1\opt, c_2\opt)$---unless 
stated otherwise.
\paragraph{Notation\& Convention} 
For simplicity, we omit the dependence of many quantities of $\psi, \mu$ (since $\psi$ 
and $\mu$ are both treated as fixed quantities throughout the proof). For instance, we write 
$\kappa\opt = \kappa\opt(\psi)$, $c_1\opt = c_1\opt(\psi)$, $c_2\opt = c_2\opt(\psi)$. We also use the simplified notation $h\opt = h\opt_{\psi, \kappa\opt(\psi)}$. Additionally, we use the notation $a \lrho b$, or equivalently, 
$b \grho a$, if the inequality $a \le c(\rho, f) \cdot b$ holds for some constant 
$c(\rho, f) > 0$ depending only on $\rho, f$, with the careful notice that the definition 
of $c(\rho, f)$ shall not depend on the specific choice of $c_{\rho, f}$ in the reduction 
Assumption~\eqref{eqn:reduction-assumption}. For instance, when we say 
``$a \lrho b$'', we mean that ``there exists a constant $C(\rho, f)$ depending only on $\rho$, 
such that $a \le C(\rho, f) \cdot b$ holds under the Assumption~\eqref{eqn:reduction-assumption}
regardless of how one specifies $c_{\rho, f}$ for the Assumption~\eqref{eqn:reduction-assumption}''; 
and when we say ``$a \lrho b$ for all sufficiently small $c_{\rho, f}$'', we mean that 
``there exist constants $C(\rho, f), c(\rho, f)$ depending only on $\rho, f$, 
such that, regardless how one specifies the constant $c_{\rho, f} \le c(\rho, f)$ in the 
Assumption~\eqref{eqn:reduction-assumption}, the inequality $a \le C(\rho, f) \cdot b$ just simply holds
under that assumption''. 

\paragraph{Proof Sketch and Organization}
The proof of equation~\eqref{eq:prop:partial:two:F:upper:1} is lengthy. For clarity, we 
decompose the proof into four pieces, which are presented in 
Section~\ref{sec:characterization-of-h-opt}---Section~\ref{sec:final-technical-piece} below.

%\rfcomment{TO BE FIXED}

\subsubsection{Characterization of $h\opt$}
\label{sec:characterization-of-h-opt}
The determination of $h\opt = h\opt_{\psi, \kappa\opt(\psi)}$ 
involves two parts. 
\begin{enumerate}[(I)]
\item Taking first-order variation of the convex optimization ${\rm OPT}(\psi, \kappa\opt(\psi))$ yields
KKT conditions. 
%	\begin{equation}
%	\label{eq:KKT1}
%	\begin{split}
%	&\partial_1 F_{\hk}(\hc_1,\hc_2) X^{1/2} W+ \hc_2^{-1} \partial_2 F_{\hk}(\hc_1,\hc_2)
%		\bLambda^{1/2}\pw \bLambda^{1/2}\wtilde{\btheta}
%			+\frac{1}{\sqrt{n}}\bLambda^{1/2}\pw \bg+s \wtilde{\btheta}=0\\
%	& s(\|\wtilde{\btheta}\|_2-1)=0,~~s\geq 0,~~ \|\wtilde{\btheta}\|_2\leq 1.
%	\end{split}
%	\end{equation}	
\item The fact that ${\rm OPT}(\psi, \kappa\opt(\psi)) = 0$ yields one more additional identity. 
\end{enumerate}
To implement the first part, we use the results in Appendix~\ref{sec:all-property-optimization}. 
According to the KKT condition (equation~\eqref{eqn:KKT-opt-general}), and 
noticing that $\psi^{\low}(\kappa\opt(\psi)) < \psi$,
there is the existence of a scalar $s\opt > 0$ such that 
\begin{equation}
\label{eqn:KKT-opt-general-bo}
\begin{split}
& X^{1/2}G 
	+ \psi^{-1/2} X^{1/2} \Big( \partial_1F \cdot W
		+ (c_2\opt)^{-1} \cdot \partial_2 F\cdot  \proj_{W^{\perp}}(X^{1/2}h\opt ) \Big) + 
			s\opt h\opt = 0.\\
& \norm{h\opt} = 1.
\end{split}
\end{equation}
Since $\psi^{\low}(\kappa\opt(\psi)) < \psi$, Lemma~\ref{lemma:technical-3} 
yields that $c_2\opt, s\opt > 0$, and
the following \emph{equivalent} system of equations
\begin{equation}
\label{eqn:KKT1-complicated}
		%\label{eqn:system-of-equations-P}
		\begin{split}
		c_1\opt &=  \E_Q \left[ \left(\bigg(c_1\opt - c_2\opt \frac{\partial_1 F}{\partial_2 F}\bigg)W 
			- \psi^{1/2} \frac{c_2\opt}{\partial_2 F} G \right) \cdot \frac{WX}{X+\lan}
			\right] \\
		%\bigg(c_1\opt - c_2\opt \frac{\partial_1 F}{\partial_2 F}\bigg)\cdot \E_{\Q} \left[
				%\frac{W^2 X}{X + \lan}\right] -  \psi^{1/2}  \frac{c_2\opt}{\partial_2 F} \cdot 
					%\E_{\Q} \left[\frac{GWX}{X+\lan}\right]  \\
		(c_1\opt)^2 + (c_2\opt)^2 &=  \E_{\Q} \left[\left(
			\bigg(c_1\opt - c_2\opt \frac{\partial_1 F}{\partial_2 F}\bigg)W
			+ \psi^{1/2} \frac{c_2\opt}{\partial_2 F} G \right)^2
			\cdot \frac{X^2}{(X+\lan)^2}\right] \\
		1 &= \E_{\Q} \left[\left(
			\bigg(c_1\opt - c_2\opt \frac{\partial_1 F}{\partial_2 F}\bigg)W
			+ \psi^{1/2} \frac{c_2\opt}{\partial_2 F} G \right)^2
			\cdot \frac{X}{(X+\lan)^2}\right]
		\end{split}
\end{equation}
where we introduce the quantity $ \lan:= \psi^{1/2} c_2\opt s\opt/\partial_2 F$.
% (the definition of $s\opt = s\opt(\psi)$ is given in equation~\eqref{eqn:system-of-equations-P}). 

To implement the second part, we use Proposition~\ref{proposition:system-of-Eq-T}, which allows us 
to conclude that 
\begin{equation}
\label{eqn:additional-equation-bo}
	0 = {\rm OPT}(\psi, \kappa\opt(\psi)) = T(\psi, \kappa\opt(\psi)) 
		= \psi^{-1/2} (F - c_1\opt \partial_1 F - c_2\opt \partial_2 F) - s\opt.
\end{equation}
Since the homogeneity of the mapping $(\kappa, c_1, c_2) \mapsto F_\kappa(c_1, c_2)$, 
we have $F - c_1\opt \partial_1 F - c_2\opt \partial_2 F = \kappa\opt \partial_\kappa F$. 
Hence, this yields the identity $s\opt = \psi^{-1/2} \kappa\opt \partial_\kappa F$. Consequentially, 
we obtain the following identity
\begin{equation}\label{eq:def:la:naught}
    \lan=\frac{c_2\opt s\opt}{\partial_2 F} = 
    	\frac{c_2\opt \kappa\opt}{\partial_2F} \cdot   \partial_{\kappa}F.
\end{equation}
Noticeably, $\lan$ characterizes the ``implicit regularization'' that's enforced in the model. % due to the randomness of the covariate $X$. 

\subsubsection{Controls on the implicit regularization term $\lan$}
%Proposition~\ref{proposition:la:naught:upperbound} gives a high probability upper bound on $\lan$.  
\begin{proposition}\label{proposition:la:naught:upperbound}
For all small enough $c_{\rho, f} > 0$,  the bound $\lan\lrho \las$ holds. 
\end{proposition}

\begin{proof}
The result is a consequence of the two key lemmas below. By equation~\eqref{eq:def:la:naught}, we have
\begin{equation*}
	 \lan=\frac{\kappa\opt \cdot \partial_{\kappa}F_{\kappa\opt}(c_1\opt, c_2\opt)}
	 	{(c_2\opt)^{-1} \cdot \partial_2F_{\kappa\opt}(c_1\opt, c_2\opt)}
\end{equation*}
The first lemma is deterministic, and characterizes the size of $\partial_2 F_{\kappa}(c_1, c_2)$ for all 
$\kappa \ge 0, c_1 \in \R, c_2 \ge 0$.
\begin{lemma}
\label{lem:F:partial:two:estimate}
There exists $c, C > 0$ depending only on $\rho, f$ such that 
for all $\kappa\ge 0, c_2 \ge 0, c_1 \in \R$,
\begin{equation*}
	c \cdot \frac{c_2}{\kappa \vee |c_1| \vee c_2} \le 
	\partial_2 F_{\kappa}(c_1,c_2) \le C \cdot \frac{c_2}{\kappa \vee |c_1| \vee c_2}.
\end{equation*}
\end{lemma}
\begin{proof}
The result is largely a consequence of the homogeneity of $(\kappa, c_1, c_2) \mapsto F_{\kappa}(c_1, c_2)$,
namely,  $F_{\kappa t} (c_1 t, c_2 t) = t F_{\kappa} (c_1, c_2)$ holds for all $t > 0$. 

Indeed, write $c \defeq \kappa \vee |c_1| \vee c_2$. The homogeneity implies that 
$\partial_2 F_{\kappa t} (c_1 t, c_2 t) = \partial_2 F_{\kappa} (c_1, c_2)$ for $t > 0$. In particular, 
$\partial_2 F_{\kappa}(c_1,c_2) = \partial_2 F_{\kappa/c} (c_1/c, c_2/c)$. This indicates that 
\begin{equation*}
	\partial_2 F_{\kappa}(c_1,c_2)  \cdot \frac{\kappa \vee |\hc_1| \vee \hc_2}{c_2} 
		=\frac{\partial_2 F_{\kappa/c} (c_1/c, c_2/c)}{c_2/c}
		\in S \defeq \left\{\frac{ \partial_2 F_{\kappa}(c_1, c_2)}{c_2} \vert \max\{\kappa, |c_1|, c_2\} = 1\right\}.
\end{equation*}
As a result, to prove Lemma~\ref{lem:F:partial:two:estimate}, it suffices to show that $S \subseteq \R_{\ge 0}$ is 
bounded away from $0$ and $+\infty$. 

To see this, note first 
$\partial_2 F_{\kappa}(c_1, 0) = -\E\big[(\kappa-\hc_1 YG)_{+}Z\big]/\sqrt{\E\big[(\kappa-\hc_1 YG)_{+}^2\big]} = 0$
for all $\kappa \ge 0, c_1 \in \R$.
As a result of Taylor's intermediate theorem, we obtain the identity
\begin{equation*}
	\frac{\partial_2 F_{\kappa}(c_1, c_2)}{c_2} = \int_0^1 \partial_{22} F_{\kappa}(c_1, c_2 t) dt. 
\end{equation*}
The mapping $(\kappa, c_1, c_2) \mapsto \partial_{22} F_{\kappa}(c_1, c_2)$ is well-defined as long as 
$\kappa \vee |c_1| \vee c_2 \neq 0$, and is positive-valued and continuous on its domain. As a consequence,
the set $S$ belongs to $\R_{\ge 0}$, and is bounded, away from $0$. 
\end{proof}

As a consequence of Lemma~\ref{lem:F:partial:two:estimate}, and the fact that 
$0 \le \partial_{\kappa}F_{\wtilde{\kappa}}(\hc_1,\hc_2) \le 1$
(as the mapping $\kappa \mapsto F_{\kappa}(c_1,c_2)$ is monotonically 
increasing and $1$-Lipschitz), we immediately obtain that
\begin{equation}
\label{eqn:lan-deterministic-bound}
    \lan=\frac{\kappa\opt \cdot \partial_{\kappa}F_{\kappa\opt}(c_1\opt, c_2\opt)}
	 	{(c_2\opt)^{-1} \cdot \partial_2F_{\kappa\opt}(c_1\opt, c_2\opt)}
    	\lrho (\kappa\opt \vee |c_1\opt|\vee c_2\opt)^2.
\end{equation}
It remains to upper bound $\kappa\opt \vee |c_1\opt|\vee c_2\opt$. 
%The second lemma does this work assuming the ``good'' event $\event$.

\begin{lemma}\label{lem:apriori:est}
For all small enough $c_{\rho, f}$, the bound  
$|\kappa\opt| \vee |c_1\opt|\vee |c_2\opt| \lrho \las^{1/2}$ holds. 
\end{lemma}
\begin{proof}
Multiply $h\opt$ on both sides 
of equation~\eqref{eqn:KKT-opt-general-bo} and take expectation. This yields 
\begin{equation*}
	0 = \E_\Q[h\opt X^{1/2} G]+ \psi^{-1/2} \cdot (c_1\opt \partial_1 F + c_2\opt \partial_2 F) + s\opt
\end{equation*}
which after substituting $s\opt$ from equation~\eqref{eqn:additional-equation-bo}, gives the
following key identity
\begin{equation}
\label{eqn:starting-point-of-lemma-naught}
	F_{\kappa\opt}(c_1\opt, c_2\opt) = \psi^{1/2} \cdot \E_\Q \left[h\opt X^{1/2} G\right].
\end{equation}
Below we analyze LHS and RHS of equation~\eqref{eqn:starting-point-of-lemma-naught}. 
\begin{itemize}
\item LHS: $F_{\kappa\opt}(c_1\opt, c_2\opt) \asymp_\rho \left(|\kappa\opt|^2+|c_1\opt|^2+|c_2\opt|^2\right)^{1/2}$
	 as $(\kappa, c_1, c_2) \mapsto F_{\kappa}(c_1,c_2)$ is positive and homogeneous. 
\item RHS: Note that $\E_Q[|h\opt|^2 X] = (c_1\opt)^2 + (c_2\opt)^2$. Standard divide-and-conquer 
technique yields the bound: 
\begin{equation*}
\begin{split}
	\psi^{1/2} \cdot \E_\Q \left[h\opt X^{1/2} G\right] &= 
		 \psi^{1/2} \cdot\E_\Q \left[h\opt X^{1/2} G \mathbf{1}_{X >  \lambda} \right]
	 	+  \psi^{1/2} \cdot\E_\Q \left[h\opt X^{1/2} G \mathbf{1}_{X \le  \lambda} \right] \\
		& \le \norm{X^{1/2} h\opt}_{\mathcal{L}_2(\Q)} \cdot \left(\E_{\Q} 
			[\psi \cdot \mathbf{1}_{X > \lambda}]\right)^{1/2} 
			+ \norm{h\opt}_{\mathcal{L}_2(\Q)} \cdot \left(\E_{\Q}[ \psi \cdot X \mathbf{1}_{X\le \lambda}]\right)^{1/2} \\
		& \le c_{\rho, f}^{1/2} \cdot (|c_1\opt|^2 + |c_2\opt|^2)^{1/2} + \las^{1/2}.
\end{split}
\end{equation*}
where the first inequality uses Cauchy-Schwartz, and the last inequality uses the reduction hypothesis. 
\end{itemize}
Substituting the above bounds into equation~\eqref{eqn:starting-point-of-lemma-naught} yields the result: 
\begin{equation*}
	\left(|\kappa\opt|^2+|c_1\opt|^2+|c_2\opt|^2\right)^{1/2} 
	\lrho c_{\rho, f}^{1/2} \cdot 
		\left(|c_1\opt|^2+|c_2\opt|^2\right)^{1/2} + \las^{1/2}.
\end{equation*}
As a consequence, 
$(|\kappa\opt| \vee |c_1\opt|\vee |c_2\opt|)^2 \lrho \las$ holds 
for all small enough $c_{\rho, f} > 0$. 
\end{proof}
Now that equation~\eqref{eqn:lan-deterministic-bound} and Lemma~\ref{lem:apriori:est}
yield that $\lan \lrho  \las$ holds for all sufficiently small $c_{\rho, f} > 0$.
%This completes the proof of 
%Proposition~\ref{proposition:la:naught:upperbound}. 

\end{proof}

\subsubsection{Order equivalence between the ratio $|c_2\opt/c_1\opt|$ and the derivative 
$\partial_2 F$}
\label{sec:right-order}
\begin{proposition}
\label{proposition:equivalence-ratio-partial-derivative}
For small enough $c_{\rho, f} > 0$,  (i) $c_1\opt > 0$ and (ii) 
$|c_2\opt/c_1\opt|  \asymp_\rho \partial_2 F_{\kappa\opt}(c_1\opt, c_2\opt)$ hold.
\end{proposition}

\begin{proof}
Specializing Lemma~\ref{lem:F:partial:two:estimate} to $c_1= c_1\opt$, $c_2 = c_2\opt$, 
$\kappa = \kappa\opt$, we obtain the equivalence
\begin{equation*}
\partial_2 F_{\kappa\opt}(c_1\opt, c_2\opt) \asymp_\rho 
	\frac{|c_2\opt|}{|c_1\opt| \vee c_2\opt \vee \kappa\opt}.
\end{equation*}
It remains to show $\kappa\opt \vee |c_1\opt| \vee c_2\opt \asymp_\rho |c_1\opt|$ and $c_1\opt > 0$ 
for all small enough $c_{\rho, f}$. The proof contains four steps.

\paragraph{Step I.}
We start by proving $c_1\opt > 0$ holds when $c_{\rho, f}$ is small. 
Suppose $c_1\opt \le 0$. Note (i) $c_1 \to \partial_1 F_{\kappa}(c_1,c_2)$ is increasing 
by convexity of $F_{\kappa}$ and (ii) $(\kappa, c_1, c_2) \mapsto F_\kappa (c_1, c_2)$ is homogeneous, 
we obtain that $\partial_1 F_{\kappa\opt}(c_1\opt,c_2\opt)\leq \partial_1 F_{\kappa\opt}(0,c_2\opt)
=\partial_1F_{\kappa\opt/c_2\opt}(0,1)$. 
Since $\E[YG] = \E[\big(2f(\rho G)-1\big) G] = 
	2\rho\E[f^\prime(\rho\cdot G)]>0$, 
\begin{equation}\label{eq:lem:c1:largest:tech:5}
   \partial_1 F_{\kappa\opt}(c_1\opt,c_2\opt) \leq \sup_{\kappa \ge 0}\partial_1 F_{\kappa}(0,1)=
    	- \E\big[YG\big] \cdot \inf_{\kappa\ge0}\left\{\frac{\E\big[(\kappa-Z)_{+}\big]}{\E\big[(\kappa-Z)_{+}^2\big]^{1/2}}\right\}=:C_0(\rho)<0. 
\end{equation}
As a result, $|\partial_1 F_{\kappa\opt}(c_1\opt, c_2\opt)| \ge |C_0(\rho)|$ if $c_1\opt \le 0$.

Hence, to prove $c_1\opt > 0$, it suffices to show that $|\partial_1 F_{\kappa\opt}(c_1\opt, c_2\opt)|  <  |C_0(\rho)|$
for all small enough $c_{\rho, f}$. %\paragraph{Step II.}
To see this, rearranging the first equation in~\eqref{eqn:KKT1-complicated} gives the identity: 
\begin{equation}
\label{eqn:rearrange-one}
\partial_1 F = -\left(\E_{\Q} \left[\frac{XW^2}{X+\lan}\right]\right)^{-1}\cdot 
	\frac{c_1\opt \partial_2 F}{c_2\opt} \cdot
		\E_{\Q}\left[\frac{ \lan  W^2}{X+\lan}\right]
	%	+ \psi^{1/2} \cdot \E_{\Q}\left[\frac{XWG}{X+\lan}\right]\right).
\end{equation}
Now we give upper bound on the RHS. Note the results: 
\begin{enumerate}[(i)]
\item Recall $\rho^2 = \E_{\Q}[W^2/X]$. Hence we have the lower bound
\begin{equation}
\label{eqn:bound-G-one}
\E_{\Q} \left[\frac{XW^2}{X+\lan}\right]=1-\lan \cdot 
	\E_{\Q} \left[\frac{W^2}{X+\lan}\right]\ge 1-\lan \cdot \rho^2. 
\end{equation} 
Note $\lan \lrho \las \le c_{\rho, f}$ by Lemma \ref{proposition:la:naught:upperbound} and reduction hypothesis. 
This proves that $\E_{\Q} \left[\frac{XW^2}{X+\lan}\right] \ge 1/2$ 
for all small enough $c_{\rho, f} > 0$. 
\item Recall $(c_2\opt)^{-1} \partial_2 F \lan = \kappa\opt \partial_{\kappa}F$ by equation
\eqref{eq:def:la:naught}, and note $\partial_{\kappa} F \in [0, 1]$ since $\kappa \mapsto F_\kappa$
is $1$-Lipschitz.
This proves that, for all sufficiently small $c_{\rho, f}$, we have the bound 
\begin{equation*}
	 \frac{c_1\opt \partial_2 F}{c_2\opt} \cdot
		\E_{\Q}\left[\frac{ \lan  W^2}{X+\lan}\right]
=|c_1\opt|\kappa\opt \partial_\kappa F \cdot \E_{\Q}\left[\frac{ \lan  W^2}{X+\lan}\right]
\leq |c_1\opt|\kappa\opt \lrho \las \leq c_{\rho, f}.
\end{equation*}
where we have used
$\kappa\opt \vee |c_1\opt| \lrho \las^{1/2}$ due to Lemma~\ref{lem:apriori:est}.
%and the definition of $\event$. 
%\item The standard divide-and-conquer technique yields the following bound: 
%\begin{equation}
%\label{eqn:bound-G-two}
%\begin{split}
%	\psi^{1/2} \cdot \E_{\Q}\left[\frac{XWG}{X+\lan}\right]
%		&= \psi^{1/2} \cdot \E_{\Q}\left[\frac{XWG}{X+\lan}\mathbf{1}_{X \le \lambda} \right]
%			+  \psi^{1/2} \cdot \E_{\Q}\left[\frac{XWG}{X+\lan}\mathbf{1}_{X > \lambda} \right]
%			\\
%		&\le  (\psi \cdot \E_{\Q}[X\mathbf{1}_{X\le \lambda}])^{1/2} \cdot 
%			\norm{\frac{X^{1/2}W}{X+\lan}}_{\mathcal{L}_2(\Q)}
%				+ 
%			(\psi \cdot \P(X > \lambda))^{1/2} \cdot 
%			 	\norm{\frac{XW}{X+\lan}}_{\mathcal{L}_2(\Q)}
%			 \\
%		&\le c_\rho^{1/2} \cdot \left[\norm{\frac{W}{X^{1/2}}}_{\mathcal{L}_2(\Q)} 
%			+ \norm{W}_{\mathcal{L}_2(\Q)} \right] \le c_\rho^{1/2} (1+\rho) \lrho c_\rho^{1/2}.
%\end{split}
%\end{equation}
%In above, the first inequality is due to Cauchy-Schwartz inequality, and 
%the second inequality uses the assumptions made
%in the reduction hypothesis. 
\end{enumerate}
Substitute all above bounds into the RHS of equation~\eqref{eqn:rearrange-one}. This 
proves that, for all small enough $c_{\rho, f}$
\begin{equation*}
|\partial_1 F_{\kappa\opt}(c_1\opt,c_2\opt)| \lrho c_{\rho, f}^{1/2} <  |C_0(\rho)|.
\end{equation*}
Hence, as mentioned before, this implies that $c_1\opt > 0$.

\paragraph{Step II.}
We prove $c_2\opt \lrho c_1\opt$. Recall 
$|\partial_1 F_{\kappa\opt}(c_1\opt,c_2\opt)| \lrho c_{\rho, f}^{1/2}$ from Step I. Using homogeneity of 
$(\kappa, c_1, c_2) \mapsto F_{\kappa}(c_1, c_2)$, we obtain the lower bound: 
\begin{equation}\label{eq:lem:c1:largest:tech:6}
|\partial_1F_{\kappa\opt}(c_1\opt, c_2\opt)| =  \left|\partial_1F_{\kappa\opt/c_2\opt}\big(c_2\opt/c_1\opt,1\big)\right|
	\geq\inf_{\kappa>0}\left|\partial_1F_{\kappa}\Big(c_2\opt/c_1\opt,1\Big)\right|=:f_{\rho}\big(c_2\opt/c_1\opt\big).
\end{equation}
Note that $f_\rho$ is continuous with 
$f_{\rho}(0)=\inf_{\kappa>0}\Big|\partial_1F_{\kappa}(0,1)\Big| = C_0(\rho)$ where $C_0(\rho)$
is defined in equation~\eqref{eq:lem:c1:largest:tech:5}. Hence, by taking $c_{\rho, f}$ small 
enough, we can make $c_2\opt/c_1\opt$ bounded away from $0$, or
$c_2\opt \lrho c_1\opt$.

\paragraph{Step III.}
We prove $\kappa\opt \lrho c_1\opt$. Recall 
$|\partial_1 F_{\kappa\opt}(c_1\opt,c_2\opt)| \lrho c_{\rho, f}^{1/2}$ from Step I. 
Our result in Step II implies that $c_2\opt/c_1\opt \le C_1(\rho)$. 
Using homogeneity of 
$(\kappa, c_1, c_2) \mapsto F_{\kappa}(c_1, c_2)$, we obtain the lower bound: 
\begin{equation}\label{eq:lem:c1:largest:tech:8}
  |\partial_1F_{\kappa\opt}(c_1\opt, c_2\opt)| =
   \left|\partial_1F_{\kappa\opt/c_1\opt}\big(1, c_2\opt/c_1\opt\big)\right|\geq \inf_{c_2\in [0,C_1(\rho)]}
   	\big|\partial_1F_{\kappa\opt/c_1\opt}(1, c_2)\big|=:g_{\rho}(\kappa\opt/c_1\opt).
\end{equation}
As $g_\rho$ is continuous with $\lim_{\kappa \to \infty} g_{\rho}(\kappa)= \E[YG] > 0$, 
%Hence, %by taking $C$ small 
%enough, we can make 
$\kappa\opt/c_1\opt$ is bounded away from $\infty$, or equivalently, we have
derived the relation $\kappa\opt \lrho c_1\opt$.

\paragraph{Summary} Arguments in Step I---III complete the proof of 
Proposition~\ref{proposition:equivalence-ratio-partial-derivative}.

%As mentioned, Proposition~\ref{proposition:equivalence-ratio-partial-derivative} 
%follows from Lemma~\ref{lem:F:partial:two:estimate} and~\ref{lem:c1:largest}. 
\end{proof}

%\rfcomment{THE LAST PIECE TO FINISH}

\subsubsection{Control on the derivative $\partial_2 F$}
\label{sec:final-technical-piece}
\begin{proposition}
\label{proposition:partial:two:F:upper}
For small enough $c_{\rho, f} > 0$, the bound below holds %on the event $\event$
\begin{equation}\label{eq:prop:partial:two:F:upper}
    \partial_2 F_{\kappa\opt}(c_1\opt, c_2\opt)\lrho \left(
    	\E_{\Q}\left[ \left(\psi +\las^2 \frac{W^2}{X^2}\right) \mathbf{1}_{X > \lambda} + 
		\left(\frac{\psi X^2}{\las^2} + W^2\right) \mathbf{1}_{X \le \lambda}\right]
		\right)^{1/2}
\end{equation}
\end{proposition}

\begin{proof}
 For simplification of arguments, we introduce 
the following notation. 
\begin{equation*}
\begin{split}
	\alpha \equiv \E_{\Q}\left[\frac{W^2X}{X+\lan}\right],~~~~
	%\beta \equiv \psi^{1/2} \cdot \E_{\Q} \left[\frac{GWX}{X+\lan}\right],~~~~
	\gamma \equiv \E_{\Q}\left[\frac{W^2X^2}{(X+\lan)^2}\right],~~~~
	%\zeta \equiv \psi^{1/2} \cdot \E_{\Q} \left[\frac{X^2 WG}{(X+\lan)^2}\right],~~~~
	\nu \equiv  \psi \cdot \E_{\Q} \left[\frac{X^2} {(X+\lan)^2}\right],~~~~
	\Delta \equiv \frac{c_1\opt}{c_2\opt}\cdot \partial_2 F_{\kappa\opt}(c_1\opt, c_2\opt).
\end{split}
\end{equation*}
Let $\Lambda$ denote the quantity on the RHS of equation~\eqref{eq:prop:partial:two:F:upper}.

Below we start the main proof. 
According to the \emph{first} equation of the KKT conditions~\eqref{eqn:KKT1-complicated}, we have 
\begin{equation*}%\label{eq:KKT:first:alternate}
    c_1\opt- c_2\opt \frac{\partial_1 F}{\partial_2 F}
    	= \alpha^{-1} c_1\opt. % \cdot \left(1+ \frac{\beta}{\Delta}\right).
\end{equation*}
Substitute it into the \emph{second} equation of  the KKT conditions~\eqref{eqn:KKT1-complicated}. 
We derive
\begin{equation*}
	(c_1\opt)^2 + (c_2\opt)^2 = (c_1\opt)^2 \cdot  \E_{\Q}\left[ \frac{X^2}{(X+\lan)^2}
		\left(\alpha^{-1} W
		+ \psi^{1/2} \cdot \frac{1}{\Delta} \cdot G\right)^2\right]. 
\end{equation*}
%Write $\Delta \equiv \Delta(\hc_1, \hc_2) =  \hc_1 \partial_2 F(\hc_1, \hc_2)/\hc_2$. 
%Note then $\Delta \asymp_\rho 1$ since Lemma \ref{lem:F:partial:two:estimate}
%and \ref{lem:c1:largest}. 
Note that $(c_2\opt)^2 = (c_1\opt)^2 (\partial_2 F)^2/\Delta^2$.  Substituting it and canceling $(c_1\opt)^2$ 
on both sides yields the identity
\begin{equation}
\label{eqn: good-expression-of-partial-two}
%\begin{split}
\begin{split}
	(\partial_2 F)^2  &= \E_{\Q}\left[\frac{X^2}{(X+\lan)^2}
		\left(\alpha^{-1}
			\Delta W
		+ \psi^{1/2} \cdot G\right)^2\right] - \Delta^2 
		= \frac{\Delta^2}{\alpha^{2}}  \gamma + \nu - \Delta^2.\\
\end{split}
\end{equation}
Below we will derive several bounds that relate the greek letters 
$\alpha, \gamma, \nu$ to the desired error term $\Lambda$. Define $\Lambda$ the error on the RHS of 
equation~\eqref{eq:prop:partial:two:F:upper}.

\begin{enumerate}
\item Note first: $\Delta \asymp_\rho 1$ since Proposition~\ref{proposition:equivalence-ratio-partial-derivative}. 
\item From the proof of Proposition~\ref{proposition:equivalence-ratio-partial-derivative}: 
	$\alpha \grho 1$ (equation~\eqref{eqn:bound-G-one}) for small enough $c_{\rho, f}$.
\item We show $|\gamma -\alpha^2| \lrho \Lambda$. Note $\E_{\Q}[W^2] = 1$. Lagrange's identity yields 
	\begin{equation*}
		0 \le \gamma -\alpha^2 =\E_{\Q}\left[\frac{X^2 W^2}{(X+\lan)^2}\right]
			\E_{\Q}[W^2] -\left(\E_{\Q} \left[\frac{XW^2}{X+\lan}\right]\right)^2
			=
			\half \cdot \E_{\Q}\left[ \left|\frac{X}{X+\lan} - \frac{X'}{X'+ \lan}\right|^2 W^2 (W')^2\right]
			% \sum_{j, k}  
	%	\left(\frac{\lambda_j}{\lambda_j + \lan} - \frac{\lambda_k} {\lambda_k + \lan}\right)^2 w_i^2 w_j^2. 
	\end{equation*}
	where $(X', W')$ is an independent copy of $(X, W)$. Note the elementary inequality: 
	\begin{equation*}
		\left|\frac{X}{X+\lan} - \frac{X'}{X'+ \lan}\right| \le \lan^2 |X-X'|^2 \cdot 
			\left(\frac{1}{(X+\lan)^2} + \frac{1}{(X'+\lan)^2}\right)
	\end{equation*}
	Also, recall that $\supp(X) \in [0, M]$. This yields the bound 
	\begin{equation*}
		\gamma -\alpha^2 \le  M^2 \cdot \E_{\Q} \left[\frac{\lan^2 W^2}{(X+\lan)^2}\right]
		\lrho  \E_{\Q} \left[W^2 \mathbf{1}_{X\le \lambda} + 
			\frac{\lan^2 W^2}{X^2} \mathbf{1}_{X > \lambda}\right].
		%\lan^2\sum_{i \le k} \frac{ w_i^2}{\lambda_i^2} + \sum_{i > k} w_i^2
		%= \lan^2 \sum_{i \le k}\frac{(\theta^\star_i)^2}{\la_i} +\sum_{i > k} \la_i(\theta^\star_i)^2.
	\end{equation*}
	Note $\lan \lrho \las$ since Lemma~\ref{proposition:la:naught:upperbound}. % and the definition of $\event$. 
	This proves $|\gamma - \alpha^2| \lrho \Lambda^2$ as desired. 

\item Finally, $\nu \le \Lambda$. The proof is technical; see Lemma~\ref{lemma:technical-bound-on-nu}
below for the proof. 
\end{enumerate}
The above bounds collectively yield Proposition~\ref{proposition:partial:two:F:upper}:  
$
	|\partial_2 F|^2= |(\gamma - \alpha^2) \cdot \frac{\Delta^2}{\alpha^2} + \nu| \lrho \Lambda^2.
$
\end{proof}
\begin{lemma}
\label{lemma:technical-bound-on-nu}
The following bound holds for all sufficiently small $c_{\rho, f} > 0$: 
\begin{equation}
	\E_{\Q} \left[\frac{X^2}{(X+\lan)^2}\right] \le 
		\E_{\Q} \left[\frac{X^2}{\las^2} \mathbf{1}_{X \le \lambda} + \mathbf{1}_{X > \lambda}\right].
\end{equation} 
\end{lemma}

\begin{proof}
The key to the proof is to establish the following bound that holds for all small enough $c_{\rho, f} > 0$: 
\begin{equation}
\label{eqn:key-to-technical-bound-on-nu}
	\E_{\Q} \left[\frac{X^2} {(X+\lan)^2}\right] \lrho
		\E_{\Q} \left[\frac{X^2} {(X+\las)^2}\right]
\end{equation}
Note that, once equation~\eqref{eqn:key-to-technical-bound-on-nu} has been established, 
then we immediately reach the conclusion since: 
\begin{equation*}
	\E_{\Q} \left[\frac{X^2} {(X+\las)^2}\right]
		= \E_{\Q} \left[\frac{X^2} {(X+\las)^2}\mathbf{1}_{X \le \lambda}\right]
			+ \E_{\Q} \left[\frac{X^2} {(X+\las)^2}\mathbf{1}_{X > \lambda}\right]
		\le \E_{\Q} \left[\frac{X^2}{\las^2} \mathbf{1}_{X \le \lambda} + \mathbf{1}_{X > \lambda}\right].
\end{equation*}
Thus it suffices to prove equation~\eqref{eqn:key-to-technical-bound-on-nu}. Below we prove 
this key technical result. 
\begin{enumerate}
\item 
If $\lan \ge \las$, the result is obvious.
\item 
If $\lan \le \las$,  then consider the following function 
\begin{equation*}
h(x):= \E_{\Q}\left[\frac{X^2}{(X+x)^2}\right] / 
	\left(\E_{\Q}\left[\frac{X}{X+x}\right]\right)^2.
\end{equation*} 
Note then $x \mapsto h(x)$ is increasing when $x \ge 0$.
Hence $h(\lan) \le h(\las)$, and thereby
\begin{equation*}
	\E_{\Q}\left[\frac{X^2}{(X+\lan)^2}\right] 
	\le 
	\E_{\Q}\left[\frac{X^2}{(X+\las)^2}\right] 
	\cdot \left(\E_{\Q}\left[\frac{X}{X+\lan}\right] /
		\E_{\Q}\left[\frac{X}{X+\las}\right]\right)^2
	%\lrho \frac{1}{n}\sum_{i=1}^{p}\frac{\la_i^2g_i^2}{(\la_i+\las)^2}. 
\end{equation*}
It remains to prove that for small enough $c_{\rho, f} > 0$:  %, the following holds on the event $\event$
\begin{equation}
\label{eqn:technical-inequality-las-lan}
	\E_{\Q}\left[\frac{X}{X+\las}\right]  \grho \frac{1}{\psi} \grho
	\E_{\Q}\left[\frac{X}{X+\lan}\right] . 
\end{equation}
The first inequality in~\eqref{eqn:technical-inequality-las-lan} 
is simple to show. By Cauchy-Schwartz inequality, we obtain 
\begin{equation*}
	\E_{\Q}\left[\frac{X}{X+\las}\right] 
		\ge \E_{\Q}\left[\frac{X}{X+\las} \mathbf{1}_{X\le \lambda}\right] 
		\ge \frac{ \E_{\Q}[X \mathbf{1}_{X\le \lambda}]^2}{\E_{\Q}[X(X+\las) \mathbf{1}_{X\le \lambda}]}
		=  \frac{1}{\psi} \cdot \frac{\las^2/\psi}{ \las^2/\psi + \E_{\Q}[X^2 \mathbf{1}_{X\le \lambda}]}.
\end{equation*}
Note by reduction hypothesis, $\E_{\Q}[X^2 \mathbf{1}_{X\le \lambda}] \lrho \las^2/\psi$. This 
proves the first inequality of~\eqref{eqn:technical-inequality-las-lan}. 

To show the second inequality in~\eqref{eqn:technical-inequality-las-lan}, we start from the identity: 
\begin{equation*}
	\E_{\Q}\left[\frac{X}{X+\lan}\right]=
		\E_{\Q}\left[\frac{X^2}{(X+\lan)^2}\right]+\E_{\Q}\left[\frac{\lan X}{(X+\lan)^2}\right] %\equiv {\rm I} +  {\rm II}
\end{equation*}
Below we upper bound the two pieces on the RHS. %$C_\rho$. 

For the first piece, recall equation~\eqref{eqn: good-expression-of-partial-two}. We obtain 
the expression: 
\begin{equation*}
	\E_{\Q}\left[\frac{X^2}{(X+\lan)^2}\right] = \frac{1}{\psi} \cdot 
	\nu = \frac{1}{\psi} \cdot \left[(\partial_2 F)^2 + \Delta^2 -
		 \frac{\Delta^2}{\alpha^{2}} \gamma\right].
\end{equation*}
Recall the bound: (i) $\alpha \grho 1$, (ii) $|\Delta| \asymp_\rho 1$, (iii) $\gamma \le 1$. Furthermore, 
$0 \le  \partial_2 F \le 1$ since  $\partial_2 F_{\kappa}(c_1,c_2)\leq 
\liminf_{c_2\to\infty} c_2^{-1} F_{\kappa}(c_1,c_2)=(\E[(-Z)_{+}^2])^{1/2} \le 1$ as $F_{\kappa}$ is convex. 
Putting all pieces together, we have
\begin{equation*}
	\E_{\Q}\left[\frac{X^2}{(X+\lan)^2}\right]\lrho \frac{1}{\psi}.
\end{equation*}

For the second piece, using the third equation in the KKT conditions~\eqref{eqn:KKT1-complicated}, 
we obtain 
\begin{equation*}
	1 = (c_1\opt)^2 \cdot \E_{\Q}\left[\frac{1}{\alpha^2} \cdot \frac{W^2 X}{(X+\lan)^2} 
		+ \psi \cdot \frac{1}{\Delta^2} \cdot \frac{X}{(X+\lan)^2}\right].
\end{equation*}
This gives the following bound 
\begin{equation*}
	\E_{\Q}\left[\frac{\lan X}{(X+\lan)^2}\right] \le \frac{1}{\psi} \cdot \frac{\lan \Delta^2} {(c_1\opt)^2}.
\end{equation*}
Recall that $\lan = c_2\opt \kappa\opt \partial_{\kappa} F/\partial_2 F$. Applying 
Lemma~\ref{lem:F:partial:two:estimate} and Proposition~\ref{proposition:equivalence-ratio-partial-derivative}
and using the fact that $\partial_\kappa F \le 1$,
we obtain that $\lan \lrho (c_1\opt)^2 \cdot \partial_\kappa F \le (c_1\opt)^2$.
%and (iv) $\lan \lrho 1$. 
Since $|\Delta| \lrho 1$, this yields the desired bound  
\begin{equation*}
	\E_{\Q}\left[\frac{\lan X}{(X+\lan)^2}\right] \lrho \frac{1}{\psi}.
\end{equation*}
\end{enumerate}
We have shown equation~\eqref{eqn:key-to-technical-bound-on-nu} holds when
$c_{\rho, f}$ is small enough. This completes the proof of Lemma~\ref{lemma:technical-bound-on-nu}.
\end{proof}

Therefore, the proof of equation \eqref{eq:prop:partial:two:F:upper:1} follows from Propositions \ref{proposition:equivalence-ratio-partial-derivative} and \ref{proposition:partial:two:F:upper}.

\section{Some technical results}

\subsection{Proof of KKT conditions  \eqref{eqn:KKT-opt-general}}
\label{sec:proof-KKT-condition}
In this appendix we prove the KKT conditions \eqref{eqn:KKT-opt-general} for the optimization problem \eqref{eqn:def-L-star}.
The argument is quite standard and we present it mainly for completeness.

We begin by a simple variant of the Hahn-Banach theorem, which deals with the case of two closed convex sets
$\cC_1$, $\cC_2$ of which \emph{only one} is bounded. While this is a minor technical difference from the standard 
setting, we do not know a good reference for this result.
\begin{lemma}
	\label{lemma:Hahn-Banach-Separation}
	Let $\mathcal{C}_1$ and $\mathcal{C}_2$ be closed convex sets in a Hilbert space $\mathcal{H}$. Assume
	$\mathcal{C}_1$ is bounded, and $\mathcal{C}_1 \cap \mathcal{C}_2 = \emptyset$. 
	Then, there exists some $h \in \mathcal{H}$ such that
	\begin{equation}
	\inf_{g_1 \in \mathcal{C}_1} \langle g_1, h \rangle > \sup_{g_2 \in \mathcal{C}_2} \langle g_2, h\rangle. 
	\end{equation}
\end{lemma}
\begin{proof}
Define the distance $\dist(\mathcal{C}_1, \mathcal{C}_2)$ by   
\begin{equation}
\label{eqn:distance-function}
	\dist(\mathcal{C}_1, \mathcal{C}_2) = \inf_{g_1 \in \mathcal{C}_1, g_2 \in \mathcal{C}_2} \norm{g_1 - g_2}.
\end{equation}
We claim that $\dist(\mathcal{C}_1, \mathcal{C}_2) = \norm{g_1\opt - g_2\opt}$ for some $g_1\opt \in \mathcal{C}_1$
and $g_2 \in \mathcal{C}_2$. To show this, let $g_{1, n} \in \mathcal{C}_1$ and $g_{2, n} \in \mathcal{C}_1$ be such that 
$\norm{g_{1, n} - g_{2, n}} \to \dist(\mathcal{C}_1, \mathcal{C}_2)$. Since $\cC_1$ is bounded, we know that $\{g_{1, n}\}_{n\in \N}$ is bounded
and therefore  $\{g_{2, n}\}_{n\in \N}$ is also bounded. Hence, for some $M > 0$, 
we have $g_{2, n} \in \mathcal{C}_2 \cap \Ball(M)$ for all $n \in \N$ ($\Ball(r)$ denotes the closed ball with 
radius $r$ in $\mathcal{H}$). Note that:  $(i)$~by Banach-Alaoglu Theorem, since $\mathcal{C}_1$
and $\mathcal{C}_2 \cap \Ball(M)$ are bounded and closed, they are compact they are compact with respect to the weak-$*$ topology; $(ii)$~the mapping 
$h \to \norm{h}$ is lower semicontinuous with respect to the weak-$*$ topology. 
By $(i)$, we can choose a weak limit point $(g_1\opt,g_2\opt)\in\cC_1\times\cC_2$ of  the sequence
$\{(g_{1, n},g_{2,n})\}_{n\in \N}$. By $(ii)$, we have $\dist(\mathcal{C}_1, \mathcal{C}_2)  \ge  \norm{g_1\opt - g_2\opt}$, and therefore 
$\dist(\mathcal{C}_1, \mathcal{C}_2)  = \norm{g_1\opt - g_2\opt} > 0$. 
	
Now, we denote for each $\delta > 0$ the set 
\begin{equation*}
	\mathcal{C}_1^{\delta} = \left\{x: \norm{x- x_1} \le \delta~~\text{for some $x_1\in \mathcal{C}$}\right\}.
\end{equation*}
Then $\mathcal{C}_1^{\delta}$ is bounded, closed and convex. Moreover, if we let $\delta = \dist(\mathcal{C}_1, \mathcal{C}_2)/2$, 
then $\mathcal{C}_1^{\delta} \cap \mathcal{C}_2 = \emptyset$. By the  Hahn-Banach theorem shows 
there exists $h \in \mathcal{H}\setminus\{0\}$ such that 
\begin{equation}
\label{eqn:hahn-banach-old}
	\inf_{g_1 \in \mathcal{C}_1^{\delta}} \langle g_1, h \rangle \ge \sup_{g_2 \in \mathcal{C}_2} \langle g_2, h\rangle. 
\end{equation}
Notice that, since $h \neq 0$, we have 
\begin{equation}
\label{eqn:hahn-banach-C-1-C-1-delta}
	\inf_{g_1\in \mathcal{C}_1} \langle g_1, h\rangle > \inf_{g_1 \in \mathcal{C}_1^{\delta}} \langle g_1, h \rangle.
\end{equation}
The desired result follows by Eq.~\eqref{eqn:hahn-banach-old} and Eq.~\eqref{eqn:hahn-banach-C-1-C-1-delta}.
\end{proof}

\begin{lemma}
A point $h\in\cL^2(\P)$ is a minimizer of the optimization problem \eqref{eqn:def-L-star} if and only if it satisfies the KKT conditions 
\eqref{eqn:KKT-opt-general}. Further, by Lemma \ref{lemma:technical-1}, such a  minimizer is unique.
\end{lemma}
\begin{proof}
The probability measure $\P$ will be fixed throughout the proof, and hence we will drop it from the subscripts in
$\asL_{\psi,\kappa,\P}(h)$ and $\|h\|_\P$, $\<h_1,h_2\>_{\P}$.

We define $\partial \asL_{\psi, \kappa}(h)$ to be the following subset of $\cL^2(\P)$: 
\begin{equation}
\begin{split}
\partial \asL_{\psi, \kappa}(h) &\defeq \Bigg\{X^{1/2}\proj_{W^{\perp}}(G) + \psi^{-1/2} \cdot X^{1/2}
	\Big(
	 \partial_1 F_{\kappa}
		\left(\langle h, X^{1/2}W\rangle, \normbig{\proj_{W^{\perp}} (X^{1/2}h)}\right)W
			+  \\
	&~~~~~~~~~~~~~~~~~~~~~~
		\partial_2 F_{\kappa}\left(\langle h, X^{1/2}W\rangle, \normbig{\proj_{W^{\perp}} (X^{1/2}h)}\right)
			\proj_{W^{\perp}}(Z)
	\Big): Z \in \set_h\Bigg\} \\
\set_h &= \begin{cases}
\left\{\normbig{\proj_{W^{\perp}} (X^{1/2}h)}^{-1} \cdot \proj_{W^{\perp}} (X^{1/2}h)\right\}
	&~~\text{if $\normbig{\proj_{W^{\perp}}(X^{1/2}h)} \neq 0$} \\
\left\{Z^\prime: \norm{Z^{\prime}} \le 1\right\}
	&~~\text{if $\normbig{\proj_{W^{\perp}}(X^{1/2}h)} = 0$}.
\end{cases}
\end{split}
\end{equation}
A standard calculation yields the following  (see~\cite[Chapter VI]{HiriartJeLeCl13}): 
\begin{itemize}
\item For any $\Delta_h \in \cL^2(\P)$ and any  $t \in \R$, we have 
\begin{equation}
\label{eqn:sufficiency}
\asL_{\psi, \kappa}(h + t\Delta_h) \ge \asL_{\psi, \kappa}(h)
	+ t \cdot \sup_{g \in \partial \asL_{\psi, \kappa}(h)} \langle g, \Delta_h\rangle.
\end{equation}
\item For any $\Delta_h \in \cL^2(\P)$, we have 
\begin{equation}
\label{eqn:necessity}
\asL_{\psi, \kappa}(h + t \Delta_h) \le  \asL_{\psi, \kappa}(h)
	+ t\sup_{g \in \partial \asL_{\psi, \kappa}(h)} \langle g, \Delta_h\rangle + o(t)~~~(t \to 0).
\end{equation}
\end{itemize}
We will next show that the KKT conditions \eqref{eqn:KKT-opt-general}) are sufficient and necessary for $h$ to be optimal. 
\paragraph{Sufficiency}
Suppose the KKT conditions \eqref{eqn:KKT-opt-general}) hold for some primal variable $h \in \cL^2(\P)$ and dual 
variable $s$. This implies $-sh \in \partial L_{\psi, \kappa}(h)$ for some $s \ge 0$. Now we divide our discussion into two cases. 
\begin{itemize}
\item If $\norm{h} < 1$, then $s= 0$ by the KKT conditions \eqref{eqn:KKT-opt-general}). Thus,
	$0 \in \partial \asL_{\psi, \kappa}(h)$. Hence, Eq.~\eqref{eqn:sufficiency} immediately implies that $h$
	is a minimizer of the optimization problem.
\item If $\norm{h} = 1$, then for some $s \ge 0$, we have $-sh \in \partial \asL_{\psi, \kappa}(h)$.
 	Now, since any feasible direction $\Delta_h$ (i.e., $h + t\Delta_h \in \{h: \norm{h} \le 1\}$ 
	for some $t \ge 0$) must satisfy $\langle h, \Delta_h\rangle \le 0$, again by Eq.~\eqref{eqn:sufficiency} 
	implies that $h$ is a minimizer of the optimization problem. 
\end{itemize}
\paragraph{Necessity}
Let $h\opt$ be a minimizer. Then we know that $\asL_{\psi, \kappa}(h\opt+ t \Delta_h) \ge  \asL_{\psi, \kappa}(h\opt)$ for any 
$\Delta_h, t$ such that $h\opt + t\Delta_h \in \{h: \norm{h} \le 1\}$. Again, we divide our discussion into two cases. 
\begin{itemize}
\item Suppose the minimizer $h\opt$ also satisfies $\norm{h\opt} < 1$. By Eq.~\eqref{eqn:necessity}, we know that, for any 
	$\Delta_h \in \cL^2(\P)$,
	\begin{equation}
		 \sup_{g \in \partial \asL_{\psi, \kappa}(h\opt)} \langle g, \Delta_h\rangle \ge 0.
	\end{equation}
	Thus we must have $\partial \asL_{\psi, \kappa}(h\opt) = \{0\}$. Hence, the pair $(h, s) = (h\opt, 0)$ 
	satisfies the KKT conditions \eqref{eqn:KKT-opt-general}).
\item Suppose the minimum $h\opt$ also satisfies $\norm{h_{\min}} = 1$. By Eq.~\eqref{eqn:necessity}, we know for any 
	$\Delta_h \in \cL^2(\P)$ such that $\langle h\opt, \Delta_h \rangle < 0$ (such $\Delta_h$ is a feasible direction, 
	i.e, $h + t\Delta_h \in \{h: \norm{h} \le 1\}$ for  $t > 0$ small enough), we must have 
	\begin{equation}
	\label{eqn:g_h-Delta_h-pos}
		 \sup_{g \in \partial \asL_{\psi, \kappa}(h\opt)} \langle g, \Delta_h\rangle \ge 0.
	\end{equation}
	Since $\partial \asL_{\psi, \kappa}(h)$ is bounded, a standard perturbation argument implies that 
	Eq.~\eqref{eqn:g_h-Delta_h-pos} continues to hold for any $\Delta_h \in \cL^2(\P)$ 
	such that $\langle h\opt, \Delta_h \rangle \le 0$. Thus, if we define the convex set
	$\mathcal{C} = \{-sh\opt: s\ge 0\}$, we have
	\begin{equation}
	\label{eqn:condition-separate}
		\inf_{h^\prime \in \mathcal{C}} \langle h^\prime, \Delta_h\rangle \ge 0 
			~~\text{for some $\Delta_h \in \cL^2(\P)$}
			\Longrightarrow 
		 \sup_{g \in \partial \asL_{\psi, \kappa}(h)} \langle g, \Delta_h\rangle \ge 0.
	\end{equation}
	Assume $\mathcal{C} \cap \partial \asL_{\psi, \kappa}(h\opt) = \emptyset$. Since $\mathcal{C}, \partial \asL_{\psi, \kappa}(h\opt)$
	are closed and convex, and moreover $\partial \asL_{\psi, \kappa}(h\opt)$ is bounded, Lemma~\ref{lemma:Hahn-Banach-Separation}
	 implies the existence of $\Delta_h$ such that 
	 \begin{equation}
		\inf_{h^\prime \in \mathcal{C}} \langle h^\prime, \Delta_h\rangle >  
			\sup_{g \in \partial \asL_{\psi, \kappa}(h)} \langle g, \Delta_h\rangle.
	 \end{equation}
	 Since $\mathcal{C}$ is a cone, this implies $\inf_{h^\prime \in \mathcal{C}} \langle h^\prime, \Delta_h\rangle = 0$. This 
	 shows the existence of $\Delta_h$ such that 
	 \begin{equation}
	 \inf_{h^\prime \in \mathcal{C}} \langle h^\prime, \Delta_h\rangle = 0~~\text{and}~~
	  \sup_{g_h \in \partial L_{\psi, \kappa}(h)} \langle g_h, \Delta_h\rangle < 0, 
	 \end{equation}
	 which thus contradicts Eq.~\eqref{eqn:condition-separate}. Therefore, 
	 $\mathcal{C} \cap \partial \asL_{\psi, \kappa}(h\opt) \neq \emptyset$. This means that there exists $s$
	 such that the pair $(h\opt, s)$ satisfies the KKT conditions \eqref{eqn:KKT-opt-general}).
\end{itemize}
This concludes the proof.
\end{proof}

%\newpage

%\begin{center}
%{\bf\Large OLD STUFF }
%\end{center}

%\input{appendix-edited}
\section{Large width asymptotics for the random features model: Proof of Proposition \ref{prop:asymptoticRF}}
\label{sec:AsymptoticsRF}

In this section we study the random features model in the wide limit
$\psi_1\to\infty$, and prove Proposition \ref{prop:asymptoticRF}.
We begin by summarizing some notations and definitions for the
reader's convenience.

\subsection{Setup}

 The quantities $\psi, \psi_1, \psi_2$ satisfy $\psi =
  \psi_1/\psi_2$. We will fix $\psi_2$ throughout the proof so that
  taking the  limit $\psi \to \infty$ is the same as taking the limit
  $\psi_1 \to \infty$.  We will use the notation
  \begin{align}
    \kbar :=\frac{\kappa}{\sqrt{\psi}}\, ,
  \end{align}
    for the normalized margin.

  %To this end, we will often omit the dependence of $\psi_2$ on the
  %paramters of interest, such as
  %$\Pred\opt,\kappa\opt,c_1\opt,c_2\opt$ and denote them only as
  %functions of $\psi_1$. For example, with abuse
  %of notation, we denote $\Pred\opt(\psi_1)=\Pred\opt(\mu_{\psi_1},\psi)$ from now on. 

We denote by $\tilde{\mu}_{\psi_1}$ the Marchenko-Pastur's law of Eq.~\eqref{eq-MPlaw}.
  The notation $\mu_{\psi_1}$ stands for the joint 
	distribution of $(X, W)$ where 
\begin{align}
\label{eqn:RF:XWdist}
X = \gamma_1^2\tX+\gamma_*^2\, ,\;\;\;\;\;\;\;\;\;\;
W = \frac{\gamma_1\sqrt{\psi_1 \tX}\, \tilde{G}}{C_0 (\gamma_1^2\tX+\gamma_*^2)^{1/2}}\, ,\;\;\;\;\;\;\;\;\;\;
C_0 = \E\Big\{\frac{\gamma_1^2 \psi_1 \tX}{(\gamma_1^2\tX+\gamma_*^2)}\Big\} ^{1/2},
\end{align}
for independent $\tX \sim \tilde{\mu}_{\psi_1}$ and $\tilde{G} \sim \normal(0,1)$. 
The quantities $\gamma_1,\gamma_{*}>0$ (defined in Eq.~\eqref{def:psionepsitwo}) are absolute constants independent of $\psi_1$.

  We next recall the predictions for the asymptotic margin
  $\kappa\opt(\mu_{\psi_1},\psi)$ and prediction error
  $\Pred\opt(\mu_{\psi_1},\psi)$. We introduce the function
  $F_{\kappa,\tau}$: 
\begin{equation}
\label{eqn:def-F-kappa-tau}
F_{\kappa,\tau}(c_1, c_2) = \left(\E \left[(\kappa- c_1 Y_{\tau}G - c_2 Z)_+^2\right]\right)^{1/2}
~~\text{where}~
	\begin{cases}
		Z \perp (Y, G, G')\, ,\\
		\P(Y_{\tau} = +1 \mid G) = \E [h(\sqrt{1-\tau^{2}}G+ \sqrt{\tau}G')|G] \, ,\\
		\P(Y_{\tau} = -1 \mid G) = 1-\E [h(\sqrt{1-\tau^{2}}G+ \sqrt{\tau}G')|G] \, ,\\
                Z , G, G'\iid \normal(0, 1)\, .
              \end{cases}
      \end{equation}
      (Notice that this definition is the same as in  Eq.~\eqref{eq:FkDef}, with theonly difference that we keep track of the dependence on $\tau$.)
Further, let $\tau(\psi_1)$ be defined  as in Eq.~\eqref{eqn:def:tau}, namely
\begin{equation}
\label{eqn:def-tau-psi-1-RF}
\tau(\psi_1)^2 = 1-\psi_1\E\Big\{\frac{\gamma_1^2 \tX}{\gamma_1^2\tX+\gamma_*^2}\Big\}.
\end{equation}

For the purpose of deriving the $\psi_1\to\infty$ asymptotics, it is most convenient to
use the variational characterization of $\kappa\opt(\mu_{\psi_1},\psi)$ and $\Pred\opt(\mu_{\psi_1},\psi)$
which is obtained in  Appendix~\ref{sec:all-property-optimization} (we will slighly modify those notations 
for later convenience).
We introduce the optimization problem
\begin{equation} 
\label{eqn:optimization:RF}
\begin{split}
  \asL\opt_{\psi_1}(\kbar,\tau)&=  \min_{ \E_{\psi_1}[h^2] \leq 1}\asL_{\psi_1}(h;\kbar,\tau)\, ,\\
    \asL_{\psi_1}(h;\kbar,\tau)&=\psi^{-1/2}F_{\kbar\sqrt{\psi},\tau}\Big(\E_{\psi_1}[X^{1/2}Wh],\big(\E_{\psi_1}[Xh^2]-(\E_{\psi_1}[X^{1/2}W h])^{2}\big)^{1/2}\Big)+\E_{\psi_1}[X^{1/2}Gh],
\end{split}
\end{equation}
where all the expectation is taken with respect to the $\P_{\psi_1}$, a measure in $\R^3$ with coordinates $(g, x, w)$: 
\begin{equation*}
\P_{\psi_1}\equiv \normal(0,1)\otimes\mu_{\psi_1}\, .
\end{equation*}
We denote by $G(g, x, w) =g$, $X(g, x, w) = x$, $W(g, x, w) = w$ the corresponding random variables. Note that, by \eqref{eqn:RF:XWdist}, $\tX, \tG$ are functions of $(X,W)$, hence $\tX, \tG$ can also be regarded as elements in $\cL^2(\P_{\psi_1})$.

We observe that the quantity $T(\psi, \kappa)$ of Proposition  \ref{proposition:system-of-Eq-T} coincides
with $\asL\opt_{\psi_1}$, namely $T(\psi, \kappa) = \asL\opt_{\psi_1} (\kappa\psi^{-1/2},\tau)$. In particular,
the asymptotic prediction for the maximum rescaled margin is (adapting Definition \ref{def:KappaE})
\begin{equation}
\label{eqn:def-kappa-opt-psi}
\kbar\opt(\psi_1) = \inf\big\{\kbar \ge 0: \asL\opt_{\psi_1}(\kbar,\tau(\psi_1))= 0\big\}. 
\end{equation}
Using Lemma \ref{lemma:technical-3} and Corollary \ref{coro:PsiLarge}, the asymptotic prediction error
can also be expressed in terms of the variational principle \eqref{eqn:optimization:RF}
Denoting by $h\opt(\psi_1)$ the optimizer (which we know is unique by Lemma \ref{lemma:technical-3}), we define
$c_1^*(\psi_1)$, $c_2^*(\psi_1)$ via
\begin{equation}
\label{eqn:characterization:c-opt-psi}
c_1\opt(\psi_1)\equiv \E_{\psi_1}[X^{1/2}Wh\opt(\psi_1)],~~c_2\opt(\psi_1)\equiv \big(\E_{\psi_1}[Xh\opt(\psi_1)^2]-(\E_{\psi_1}[X^{1/2}W h\opt(\psi_1)])^{2}\big)^{1/2}.
\end{equation}
(These coincide with the quantities defined in Proposition~\ref{proposition:system-of-Eq-T}.)
In terms of these quantities, the prediction error is given by:
\begin{equation}
\label{eqn:error-RF-explicit}
\begin{split}
\error\opt(\mu_{\psi_1},\psi) &=  \P\left(\nu\opt(\psi_1) Y_{\tau(\psi_1)}G+ \sqrt{1-\nu\opt(\psi_1)^2} Z \le 0\right)\, ,\\
  \nu\opt(\psi_1) & \equiv \frac{c_1\opt(\psi_1)}{\sqrt{(c_1\opt(\psi_1))^2 + (c_2\opt(\psi_1))^2}}\, ,
\end{split}
\end{equation} 

\subsection{Proof of Proposition \ref{prop:asymptoticRF}}
\label{subsec:proof-of-proposition-asymptoticRF}

We begin by stating some properties of the variational problem in the statement of Proposition \ref{prop:asymptoticRF}.
The proof of the next lemma is deferred to Section \ref{subsec:RF:proofs}.
\begin{lemma}\label{lemma:unique-d1-d2:T-increasing}
  For $\psi_2,\gamma_1,\gamma_{*}>0$, define  $T_{\infty}(\,\cdot\,; \psi_2,\gamma_1,\gamma_{*})$
  as in the statement of Eq.~\eqref{eqn:def:T-infty:0}. Namely
    \begin{equation}
    \begin{split}
   \label{eqn:def:T-infty}
        R(\kbar,\tau,d_1,d_2)&\equiv F_{\kbar,\tau}\big(\sqrt{\psi_2}\gamma_1 d_1,\sqrt{\psi_2}\gamma_1 d_2\big)-\gamma_1 d_2-\gamma_{*}\sqrt{1-d_1^2-d_2^2}\, ,\\
        T_{\infty}(\kbar) &= \min_{\substack{d_1^2+d_2^2 \leq 1,\\d_2\geq 0}}  R(\kbar,\tau,d_1,d_2)\, .
    \end{split}
    \end{equation}

\begin{enumerate}
\item[$(a)$] 
    Then this optimization problem has a unique minimizer  $(d_1\opt,d_2\opt)=(d_1\opt(\kbar),d_2\opt(\kbar))$
    which lies in the open set $\mathscr{S}\equiv\{(d_1,d_2):d_1^2+d_2^2<1,d_2>0\}$.  Further
    $(d_1\opt,d_2\opt)$ is uniquely characterized by the first order stationarity conditions
    $\grad_{d_1,d_2}R(\kbar,\tau,d_1\opt,d_2\opt)=0$.
    \item[(b)] The function $\kbar \to T_{\infty}(\kbar ; \psi_2,\gamma_1,\gamma_{*})$ is strictly increasing and satisfies
    \begin{equation}
    \label{eqn:T:uniquezero}
        \lim_{\bar{\kappa} \downarrow 0} T_{\infty}(\bar{\kappa})<0<\lim_{\bar{\kappa} \uparrow \infty} T_{\infty}(\bar{\kappa}).
\end{equation}
Hence, $T_{\infty}(\,\cdot\, ; \psi_2,\gamma_1,\gamma_{*})$ has a unique zero, which coincides
with  $\kbar\owid (\psi_2,\gamma_1,\gamma_{*}) $.
\end{enumerate}
\end{lemma}

Note that, for $\alpha>0, c_1,c_2\in \R$ and $\tau\ge 0$,
\begin{equation}
\label{eqn:RF:F:scale}
    F_{\alpha\kappa,\tau}(\alpha c_1,\alpha c_2)= \alpha F_{\kappa,\tau}(c_1,c_2)\, .
  \end{equation}
  This implies
\begin{equation}
\label{eqn:asymptoticRF:betterexpression}
\asL_{\psi_1}(h;\kbar,\tau)=F_{\kbar, \tau(\psi_1)}\Big(\cbar_{1,\psi_1}(h),\cbar_{2,\psi_1}(h)\Big)+\cbar_{3,\psi_1}(h),
\end{equation}
where, for $h \in \cL^{2}(\P_{\psi_1})$, we define
\begin{equation}
\label{eqn:def:cbar-123}
\cbar_{1,\psi_1}(h)\equiv \psi^{-1/2} \E_{\psi_1}[X^{1/2}Wh],~~\cbar_{2,\psi_1}(h)\equiv\psi^{-1/2}\big(\E_{\psi_1}[Xh^2]-(\E_{\psi_1}[X^{1/2}W h])^{2}\big)^{1/2},~~\cbar_{3,\psi_1}(h)\equiv \E_{\psi_1}[X^{1/2}Gh].
\end{equation}
The main idea of the proof is to approximate $\cbar_{i,\psi_1}(h),i=1,2,3$ uniformly over $h\in \cL^{2}(\P_{\psi_1}), \E_{\psi_1}[h^2]\leq 1$, when $\psi_1 \to \infty$. Observe that if $\psi_1>1$,
\begin{align}
\label{eqn:conditionaldistn:Xtilde}
    \tX=0\quad\text{w.p.}\quad 1-\frac{1}{\psi_1}\quad\text{and}\quad\tX\mid \tX \neq 0 \stackrel{d}{=}\psi_1 V\quad\text{for}\quad V\sim \nu_{\frac{1}{\psi_1}}.
\end{align}
With \eqref{eqn:conditionaldistn:Xtilde} in mind, we compute $\cbar_{i,\psi_1}(h)$ by conditioning on to the events $\tX=0$ and $\tX\neq 0$. It turns out that $\cbar_{i,\psi_1}(h)$ can be approximated by the following quantities. For $h \in \cL^{2}(\P_{\psi_1})$, define
\begin{equation}
\label{eqn:def:dh}
\begin{split}
    &d_{1,\psi_1}(h) \equiv \E_{\psi_1}[\psi_1^{-1/2}h \tG\mid \tX \neq 0],\\
    &d_{2,\psi_1}(h) \equiv \Big(\E_{\psi_1}[\psi_1^{-1}h^2\mid \tX \neq 0]-(E_{\psi_1}[\psi_1^{-1/2}h\tG \mid\tX \neq 0])^{2}\Big)^{1/2},\\
    &d_{3,\psi_1}(h) \equiv \E_{\psi_1}[\psi_1^{-1/2}hG \mid \tX\neq 0],\\
    &d_{4,\psi_1}(h) \equiv \E_{\psi_1}[(1-\psi_1^{-1})^{1/2}hG \mid \tX=0].
\end{split}
\end{equation}
The lemma below shows that we can approximate $\cbar_{i,\psi_1}(h),i=1,2,3$ by linear functions of $d_{i,\psi_1}(h),i=1,2,3,4$, along with an estimate on $\tau(\psi_1)$. The proof of Lemma \ref{lemma:RF:approx:c} is deferred to \ref{subsec:RF:proofs} along with the proof of Lemma \ref{lem:approx:asymptoticRF} and Lemma \ref{lemma:asymptoticRF:betterexpression}, to be stated below.
\begin{lemma}
\label{lemma:RF:approx:c}
There exists a constant $C=C(\psi_2,\gamma_1,\gamma_{*})>0$ such that
\begin{align}
    \sup_{\E_{\psi_1}[h^2] \leq 1} |\cbar_i(h)-\sqrt{\psi_2}\gamma_1 d_{i,\psi_1}(h)| &\leq C\psi_1^{-1/4}\quad\text{for}\quad i=1,2\quad\text{and}\label{eqn:RF:approximate:c12}\\
    \sup_{\E_{\psi_1}[h^2] \leq 1} |\cbar_3(h)-\gamma_1 d_{3,\psi_1}(h)-\gamma_{*}d_{4,\psi_1}(h)| &\leq C\psi_1^{-1/4}\label{eqn:RF:approximate:c3}\\
    \tau(\psi_1)&\leq C\psi_1^{-1/2}\label{eqn:RF:approximate:tau}
\end{align}
\end{lemma}
Lemma \ref{lemma:RF:approx:c} suggests that for large $\psi_1$, the optimization problem for $T(\psi_1,\kappa)$ \eqref{eqn:asymptoticRF:betterexpression} can be approximated by the optimization below for $\tau=0$. For $\kbar\geq0,\tau \in [0,1]$, define 
    \begin{equation}
        \begin{split}
\label{eqn:def:asymptoticRF-optimization}
    \asG\opt(\bar{\kappa},\tau)&\equiv\min_{\E_{\psi_1}[h^2]\leq 1}\asG(h;\kbar,\tau)\, ,\\
    \asG(h;\kbar, \tau)&\equiv F_{\bar{\kappa},\tau}\Big(\psi_2^{1/2}\gamma_1 d_{1,\psi_1}(h), \psi_2^{1/2}\gamma_1d_{2,\psi_1}(h)\Big)+\gamma_1 d_{3,\psi_1}(h)+\gamma_{*}d_{4,\psi_1}(h),
\end{split}
\end{equation}
where $d_{i,\psi_1}(h)$ are defined in \eqref{eqn:def:dh}. 
We are most interested in the case  $\asG\opt(\bar{\kappa},\tau=0)$, but our proof makes use of the function $\asG\opt(\bar{\kappa},\tau)$ for general $\tau \in [0,1]$.

The next lemma shows that $\asG$ is a good approximation ot $\asL_{\psi_1}$, when $\psi_1$ is large.
\begin{lemma}
\label{lem:approx:asymptoticRF}
For $\psi_1,\kappa>0$ and $h \in \cL^2(\P_{\psi_1})$, define 
\begin{equation}
\label{eqn:RF:def:E}
    E_{\psi_1}(h;\kbar) \equiv\asL_{\psi_1}(h;\kbar,\tau(\psi_1))-\asG(h;\kbar,\tau(\psi_1)).
\end{equation}
Then, we have that
\begin{equation}
    \lim_{\psi_1\to\infty} \sup_{\kappa>0}\sup_{\E_{\psi_1}[h^2]\leq 1} |E_{\psi_1}(h;\kbar)| =0\, .
\end{equation}
\end{lemma}

The next lemma shows that$\asG\opt(\kbar,\tau)$ is in fact independent of $\psi_1$.
\begin{lemma}
\label{lemma:asymptoticRF:betterexpression}
    For $\kbar\geq0,\tau \in [0,1]$ and $d_1^2+d_2^2\leq 1$, define
    \begin{equation}
    \label{eqn:def:R}
        R(\kbar,\tau,d_1,d_2)\equiv F_{\bar{\kappa},\tau}\big(\sqrt{\psi_2}\gamma_1 d_1,\sqrt{\psi_2}\gamma_1 d_2\big)-\gamma_1 d_2-\gamma_{*}\sqrt{1-d_1^2-d_2^2}.
    \end{equation}
    Then, we have a simplified expression for $\asG\opt(\kbar,\tau)$:
    \begin{equation}
    \label{eqn:asymptoticRF-kbar-tau-betterexpression}
        \asG\opt(\bar{\kappa},\tau)= \min_{\substack{d_1^{2}+d_2^{2}\leq 1\\ d_2 \geq 0}}R(\kbar,\tau,d_1,d_2).
    \end{equation}
    In particular, recalling $T_{\infty}(\kbar)$ defined in \eqref{eqn:def:T-infty}, we have that
     $\asG\opt(\kbar,0)=T_{\infty}(\kbar)$
     for any $\kbar>0$.

     Finally, for any $\kbar> 0$, $\tau\in [0,1]$,
     \begin{align}
       \asG(h;\kbar,\tau)\ge R(\kbar,\tau,d_{1,\psi_1}(h),d_{2,\psi_1}(h))\, .\label{eq:LB_asG}
       \end{align}
\end{lemma}
We are now in position to prove Proposition \ref{prop:asymptoticRF}.

\paragraph{Proof of Proposition \ref{prop:asymptoticRF}, Eq.~\eqref{eq:Kwide} (margin)}
Recall the definition of $\kbar\opt(\psi)$ in Eq.~\ref{eqn:def-kappa-opt-psi}, and $\kbar\owid$
in the statement of Proposition \ref{prop:asymptoticRF}.$(a)$ (further using Lemma \ref{lemma:asymptoticRF:betterexpression}):
\begin{align}
  \kbar\opt(\psi_1) & = \inf\big\{\kbar \ge 0: \asL\opt_{\psi_1}(\kbar,\tau(\psi_1))= 0\big\}\, ,\\
  \kbar\owid & = \inf\big\{\kbar \ge 0: \asG\opt(\kbar,0)= 0\big\}\, .
\end{align}
Further notice that both functions $\kbar \mapsto  f_{\psi_1}(\kbar) \equiv \asL\opt_{\psi_1}(\kbar,\tau(\psi_1))$
and $\kbar\mapsto f_{\infty}(\kbar) \equiv \asG\opt(\kbar,0)$ are monotone increasing and continuous in $\kbar$,
with  the latter \emph{strictly} increasing by  Lemma \ref{lemma:unique-d1-d2:T-increasing}.
Finally $f_{\psi_1}(\kbar)\to f_{\infty}(\kbar)$ pointwise as $\psi_1\to\infty$
by Lemma \ref{lem:approx:asymptoticRF} and using the facts that $\tau(\psi_1)\to 0$ (by Lemma \ref{lemma:RF:approx:c}) and
that $(\kbar,\tau)\mapsto \asG\opt(\kbar,\tau)$ is continuous (by the expression in Lemma
\ref{lemma:asymptoticRF:betterexpression}). This implies that the zeros of  $f_{\psi_1}(\kbar)$ converge to the
(unique) zero of $f_{\infty}(\kbar)$, thus proving the claim.
\hfill\ensuremath{\blacksquare}

\paragraph{Proof of Proposition \ref{prop:asymptoticRF}, Eq.~\eqref{eq:Predwide} (prediction error)}
First of all, recall the definitions of $\Pred\opt(\psi_1)$ and $\Pred\owid(\psi_2,\gamma_1,\gamma_{*})$ in
Eqs.~\eqref{eqn:error-RF-explicit} and \eqref{eq:PredErrorWide}.
Since $\lim_{\psi_1\to\infty}\tau(\psi_1)=0$ by Eq.~\eqref{eqn:RF:approximate:tau}, it suffices to show that
\begin{equation}
\label{eqn:goal-RFproposition-b}
    \lim_{\psi_1 \to \infty} \frac{c_1\opt(\psi_1)}{c_2\opt(\psi_1)} =\frac{d_1\owid}{d_2\owid}\, ,
  \end{equation}
  where $d_i\owid = d_i\opt(\kbar\owid)$, and $(d_1\opt(\kbar),d_2\opt(\kbar))$ the unique minimizer
  of $R(\kbar,0,d_1,d_2)$ in $\{d_1^2+d_2^2 \leq 1,\, d_2\geq 0\}$ as per Lemma \ref{lemma:unique-d1-d2:T-increasing}.
  Recalling the characterization of $c_i\opt(\psi_1)$ in Eq.~\eqref{eqn:characterization:c-opt-psi}, we plug
  $h=h\opt(\psi_1)$ in \eqref{eqn:RF:approximate:c12}to get:
\begin{equation}
\label{eqn:approximate-c12-opt-psi}
\begin{split}
    &\mid \psi^{-1/2}c_1\opt(\psi_1)-\psi_2^{1/2}\gamma_1d_{1,\psi_1}\left(h\opt(\psi_1)\right) \mid \leq C\psi_1^{-1/4}\\
    &\mid \psi^{-1/2}c_2\opt(\psi_1)-\psi_2^{1/2}\gamma_1d_{2,\psi_1}\left(h\opt(\psi_1)\right) \mid \leq C\psi_1^{-1/4}.
\end{split}
\end{equation}
From now on, we denote $d_{i,\psi_1}\opt\equiv d_{i,\psi_1}\left(h\opt(\psi_1)\right),i=1,2$ for brevity. By \eqref{eqn:approximate-c12-opt-psi}, it suffices to show that 
\begin{equation}
\label{eqn:goal:mainprop:RF}
    \lim_{\psi_1\to\infty}d_{i,\psi_1}\opt=d_i\owid\quad\text{for}\quad i=1,2\, .
  \end{equation}
  
The rest of the proof establishes Eq.~\eqref{eqn:goal:mainprop:RF}. First recall the definition of $R(\kbar,\tau,d_1,d_2)$ in \eqref{eqn:def:R}. We state a few properties of $R(\kbar,\tau,d_1,d_2)$ which will play a crucial role:
\begin{itemize}
 %   \item For every $h\in \cL^{2}(\P_{\psi_1})$ such that $\E_{\psi_1}[h^2]\leq 1$, and any $\kbar>0,\tau\in [0,1]$, the following inequality holds.
  %  \begin{equation}
   % \label{eqn:Gprop-1}
    %    \asG(\kbar,\tau,h) \geq R(\kbar,\tau,d_{1,\psi_1}(h),d_{2,\psi_1}(h)).
   % \end{equation}
   % \eqref{eqn:Gprop-1} is proved in \eqref{eqn:proof-Gprop1}. See the proof of Lemma \ref{lemma:asymptoticRF:betterexpression} for exact details. 
    \item There exists a constant $C=C(\psi_2,\gamma_1,\gamma_{*})<\infty$ such that 
    \begin{equation}
    \label{eqn:Gprop-2}
        \mid R(\kbar_1,\tau,d_1,d_2)-R(\kbar_{2},\tau,d_1,d_2)\mid\leq C|\kbar_1-\kbar_2|,
    \end{equation}
 for every $\kbar_1,\kbar_2>0,\tau \in [0,1]$ and $d_1^2+d_2^2\leq 1,d_2\geq0$. Indeed, if we denote $\partial_{\kbar}F_{\kbar_{0},\tau}(d_1,d_2)$ to be the partial derivative of $\kbar \to F_{\kbar,\tau}(d_1,d_2)$ evaluated at $\kbar=\kbar_{0}$, 
    \begin{equation*}
    \begin{split}
         \mid R(\kbar_1,\tau,d_1,d_2)-R(\kbar_{2},\tau,d_1,d_2)\mid &\leq  \bigg[\sup_{\kbar_{0}>0}\sup_{\tau \in [0,1]}\sup_{\substack{d_1^2+d_2^2 \leq 1\\ d_2\geq 0}}\big\lvert \partial_{\kbar} F_{\kbar_{0},\tau}\big(\sqrt{\psi_2}\gamma_1 d_1,\sqrt{\psi_2}\gamma_1 d_2\big) \big\rvert\bigg]|\kbar_1-\kbar_2|\\
         & \eqqcolon C|\kbar_1-\kbar_2|.
    \end{split}
    \end{equation*}
    The constant $C$ defined above is finite, since a direct computation gives
    \begin{equation*}
\label{eqn:kappatoinfinity:partialF}
\begin{split}
    \limsup_{\kbar_{0}\to\infty}\sup_{\tau \in [0,1]}\sup_{\substack{d_1^2+d_2^2 \leq \psi_2\gamma_1^2\\ d_2\geq 0}}\big\lvert \partial_{\kbar} F_{\kbar_{0},\tau}\big(d_1,d_2\big) \big\rvert&=\limsup_{\kbar_0\to\infty}\sup_{\tau \in [0,1]}\sup_{\substack{d_1^2+d_2^2 \leq \psi_2\gamma_1^2\\ d_2\geq 0}}\frac{\E(\kbar_{0}-d_1 Y_{\tau}G-d_2 Z)_{+}}{(\E[(\kbar_{0}-d_1 Y_{\tau}G-d_2 Z)_{+}^{2}])^{1/2}}\\
    &=1.
\end{split}
\end{equation*}
    \item By dominated convergence, the following is true for every $\kbar>0$:
    \begin{equation}
    \label{eqn:Gprop-3}
        \lim_{\tau \to 0}\sup_{d_1^{2}+d_2^{2} \leq 1} \big\lvert R(\kbar,\tau,d_1,d_2)-R(\kbar,0,d_1,d_2)\big\rvert=0.
    \end{equation}
    \item By Lemma \ref{lemma:unique-d1-d2:T-increasing} and the definition of $\kbar^{**},d_1^{**},d_2^{**}$, there exists $c_0=c_0(\psi_2,\gamma_1,\gamma_{*})>0$ such that 
    \begin{equation}
    \label{eqn:Gprop-4}
    \begin{split}
        R(\kbar\owid,0,d_1,d_2)&\geq T(\kbar\owid)+c_0\big( (d_1-d_1\owid)^{2}+(d_2-d_2\owid)^{2}\big)\\
        &=c_0\big( (d_1-d_1\owid)^{2}+(d_2-d_2\owid)^{2}\big),
    \end{split}
    \end{equation}
    for every $d_1^2+d_2^2\leq 1, d_2 \geq 0$.
  \end{itemize}

  \begin{align}
    \asG\big(h\opt(\psi_1),\kbar\opt(\psi_1),\tau(\psi_1)\big)&
                                                             \stackrel{(a)}{\ge} R\big(\kbar\opt(\psi_1),\tau(\psi_1),d_{1,\psi_1}\opt,d_{2,\psi_1}\opt\big)\nonumber\\
                                                              & \stackrel{(b)}{\ge} R\big( \kbar\owid,\tau(\psi_1),d_{1,\psi_1}\opt,d_{2,\psi_1}\opt\big)-C|\kbar\opt(\psi_1)-\kbar\owid|\nonumber\\
                                                              &  \stackrel{(c)}{\ge} R\big( \kbar\owid,\tau(\psi_1),d_{1,\psi_1}\opt,d_{2,\psi_1}\opt\big)-C|\kbar\opt(\psi_1)-\kbar\owid|-\eps(\psi_1)\nonumber\\
                                                              &\stackrel{(d)}{\ge}  c_0\big( (d_{1,\psi_1}\opt-d_1\owid)^{2}+
                                                                (d_{2,\psi_1}\opt-d_2\owid)^{2}\big)-C|\kbar\opt(\psi_1)-\kbar\owid|-\eps(\psi_1)\nonumber\\
    &\stackrel{(e)}{\ge}  c_0\big( (d_{1,\psi_1}\opt-d_1\owid)^{2}+
                                                                (d_{2,\psi_1}\opt-d_2\owid)^{2}\big)-2\eps(\psi_1)\, .\label{eq:Boundd12}
  \end{align}
  Here $(a)$ follows from Eq.~\eqref{eq:LB_asG}; $(b)$~from Eq.~\eqref{eqn:Gprop-2}; $(c)$~holds for some function $\eps(\psi_1)$
  such that $\lim_{\psi_1\to\infty}\eps(\psi_1)=0$ by Eq.~\eqref{eqn:Gprop-3}, and Lemma \ref{lemma:RF:approx:c};
  $(d)$~by Eq.~{eqn:Gprop-4}; $(e)$~ follows by the convergence of the margin proved in the previous part (i.e. by Eq.~\eqref{eq:Kwide}),
  eventually redefining $\eps(\psi_1)$.
  
  Now recall that 
\begin{equation}
  \label{eqn:mainpropRF:intermediate-1}
  0 \stackrel{(a)}{=}\lim_{\psi_1 \to \infty}
  \asL_{\psi_1}\big(h\opt(\psi_1);\kbar\opt(\psi_1),\tau(\psi_1)\big) \stackrel{(b)}{=}
\lim_{\psi_1 \to \infty}\asG\big(\kbar(\psi_1),\tau(\psi_1),h\opt(\psi_1)\big)\, ,
\end{equation}
where  
$(a)$ follows from \eqref{eqn:def-kappa-opt-psi} and $(b)$  from Lemma \ref{lem:approx:asymptoticRF}.
Therefore, taking the limit $\psi_1\to\infty$ in Eq.~\eqref{eq:Boundd12}, we get
\begin{align}
  \lim_{\psi_1\to\infty}\big( (d_{1,\psi_1}\opt-d_1\owid)^{2}+ (d_{2,\psi_1}\opt-d_2\owid)^{2}\big) = 0\,,.
\end{align}
This concludes the proof of the claim \eqref{eqn:goal:mainprop:RF}, and therefore the proof of the proposition.
\hfill\ensuremath{\blacksquare}

\subsection{Proof of Lemma \ref{lemma:unique-d1-d2:T-increasing}, \ref{lemma:RF:approx:c},\ref{lem:approx:asymptoticRF} and \ref{lemma:asymptoticRF:betterexpression}}
\label{subsec:RF:proofs}
\paragraph{Proof of Lemma \ref{lemma:unique-d1-d2:T-increasing} (a)}
Recall that, by Eq.~\eqref{eqn:def:R}, we 
\begin{equation}
     R(\kbar,0,d_1,d_2) = F_{\kbar,0}(\sqrt{\psi_2}\gamma_1 d_1,\sqrt{\psi_2}\gamma_1d_2)-\gamma_1d_2-\gamma_{*}\sqrt{1-d_1^2-d_2^2}.
\end{equation}
We make the following two remarks about $R(\kbar,0,d_1,d_2)$.
\begin{enumerate}
    \item Observe that $(d_1,d_2) \mapsto -\sqrt{1-d_1^2-d_2^2}$ is strictly convex. Since $F_{\kbar,0}(\cdot)$ is also strictly convex by Lemma \ref{lemma:F-convex-increasing}, $(d_1,d_2)\mapsto R(\kbar,0,d_1,d_2)$ is strictly convex for every $\kbar>0$. Hence, there exists a unique minimizer $(d_1\opt,d_2\opt)=\big(d_1\opt(\kbar),d_2\opt(\kbar)\big)$ of $(d_1,d_2)\mapsto R(\kbar,0,d_1,d_2)$ in the set $\{(d_1,d_2):d_1^2+d_2^2\leq 1, d_2\geq 0\}$.
    \item By definition of $F_{\kbar,0}$, $F_{\kbar,0}(d_1,-d_2)=F_{\kbar,0}(d_1,d_2)$ for every $(d_1,d_2)\in \R^2$. Hence, for $d_2 \geq 0$
    \begin{equation}
    R(\kbar,0,d_1,-d_2)=R(\kbar,0,d_1,d_2)+2\gamma_1 d_2\geq R(\kbar,0,d_1,d_2).\label{eq:d2Flip}
    \end{equation}
    Therefore, in the definition of $T_{\infty}(\kbar)$ in \eqref{eqn:def:T-infty}, the constraint $d_2\geq 0$ can be removed: 
    \begin{equation}
    \label{eqn:T-infty-constraint-removed}
    T_{\infty}(\kbar)=\min_{d_1^2+d_2^2\leq 1} R(\kbar,0,d_1,d_2),
    \end{equation}
    and the unique minimizer of $R(\kbar,0,d_1,d_2)$ in the set $\{d_1^2+d_2^2\leq 1\}$ is given by $(d_1\opt,d_2\opt)\in \{(d_1,d_2):d_1^2+d_2^2\leq 1,d_2> 0\}$. (Note that, by Eq.~\eqref{eq:d2Flip}, $\partial_{d_2}R(\kbar,0,d_1,d_2=0)\le-2\gamma_1$.)
\end{enumerate}
We now aim to show that $(d_1\opt)^2+(d_2\opt)^2<1$. Assume, by contradiction, that $(d_1\opt)^2+(d_2\opt)^2=1$. Then, for every $0<\alpha<1$, we have
\begin{equation}
\label{eqn:R-minimizer-technical}
    R(\kbar,0,(1-\alpha)d_1\opt,(1-\alpha)d_2\opt)\geq R(\kbar,0,d_1\opt,d_2\opt),
\end{equation}
by definition of $(d_1\opt,d_2\opt)$. Expanding \eqref{eqn:R-minimizer-technical} gives
\begin{equation}
    F_{\kbar,0}\big(\sqrt{\psi_2}\gamma_1(1-\alpha)d_1\opt,\sqrt{\psi_2}\gamma_1(1-\alpha)d_2\opt\big)+\gamma_1\alpha d_2\opt-\gamma_{*}\sqrt{1-(1-\alpha)^2}\geq F_{\kbar,0}\big(\sqrt{\psi_1}\gamma_1d_1\opt,\sqrt{\psi_2}\gamma_1d_2\opt\big),
\end{equation}
since we are assuming $(d_1\opt)^2+(d_2\opt)^2=1$. Dividing by $\alpha>0$ on each side gives
\begin{equation}
    -\gamma_{*}\alpha^{-1}\sqrt{2\alpha-\alpha^2}\geq \alpha^{-1}\Big(F_{\kbar,0}\big(\sqrt{\psi_1}\gamma_1d_1\opt,\sqrt{\psi_2}\gamma_1d_2\opt\big)- F_{\kbar,0}\big(\sqrt{\psi_1}\gamma_1(1-\alpha)d_1\opt,\sqrt{\psi_2}\gamma_1(1-\alpha)d_2\opt\big)\Big)-\gamma_1d_2\opt.
\end{equation}
Observe that if we send to $\alpha \to 0$ on both sides, the LHS above tends to $-\infty$. However the right-hand side above has a finite limit as $\alpha \to 0$, by differentiability of $F_{\kbar,0}(\, \cdot\, )$ guaranteed by Lemma \ref{lemma:F-convex-increasing}. Therefore, it is a contradiction and we have proved  $(d_1\opt,d_2\opt) \in \{(d_1,d_2):d_1^2+d_2^2 <1\}$.

Finally, since $(d_1,d_2)\mapsto R(\kbar,0,d_1,d_2)$ is differentiable in the open set $\mathscr{S}$, and achieves its minimum at
$(d_1\owid,d_2\owid)$, its gradient $\grad_{d_1,d_2}R(\kbar,0,d_1,d_2)$ must vanish at $(d_1\owid,d_2\owid)$.

\paragraph{Proof of Lemma \ref{lemma:unique-d1-d2:T-increasing}.$(b)$}
Note that for every $d_1,d_2 \in \R$, $\bar{\kappa} \to F_{\bar{\kappa},0}(d_1,d_2)$ is strictly increasing. Hence, $\bar{\kappa} \to T_{\infty}(\bar{\kappa})$ in \eqref{eqn:def:T-infty} is obtained by minimizing strictly increasing function in $\kbar$ over a compact set, $\{(d_1,d_2):d_1^2+d_2^2\leq 1\,,
d_2\geq 0\}$. Therefore, $T_{\infty}(\cdot)$ is strictly increasing. Next, plugging in $d_1=d_2=0$ in \eqref{eqn:def:T-infty} gives
\begin{equation}
    \lim_{\bar{\kappa} \downarrow 0} T_{\infty}(\bar{\kappa}) \leq \lim_{\bar{\kappa} \downarrow 0} (\bar{\kappa}-\gamma_*)=-\gamma_*<0.
\end{equation}
Next, we compute $\lim_{\kbar \to \infty} T_{\infty}(\kbar)$. For any $d_1,d_2\in \R$,
\begin{equation}
\label{eqn:kappatoinfinity:F}
\begin{split}
    \left(F_{\kbar,0}(d_1,d_2)\right)^{2}&= \E[(\kbar-d_1 Y_{0}G-d_2 Z)^{2}\bfone(\kbar\geq d_1 Y_{0}G+d_2 Z)]\\
    &\geq \kbar^2\P(d_1 Y_{0}G+d_2 Z\leq \kbar) - 2\kbar\E[|d_1 Y_{0}G+d_2 Z|].
\end{split}
\end{equation}
For $d_1^2+d_2^2 \leq 1$, we can bound $d_1Y_0 G+d_2 Z\leq |d_1 Y_0G+d_2 Z| \leq |Y_0G|+|Z| = |G|+|Z|$, almost surely. Thus,
\begin{equation}
    \min_{d_1^2+d_2^2\leq 1} F_{\kbar,0}(d_1,d_2)\geq\big(\kbar^2 \P(|G|+|Z|\leq \kbar)-4\kbar\E|Z|\big)^{1/2}\,.
  \end{equation}
  Since  $\P(|G|+|Z|\leq \kbar)\to 1$ as $\kbar\to\infty$, the right-hand side
 tends to $\infty$ as $\kbar\to\infty$. Therefore
\begin{equation}
    \lim_{\kbar \uparrow \infty} T_{\infty}(\kbar) \geq  \lim_{\kbar \uparrow \infty}\left( \min_{d_1^2+d_2^2\leq 1} F_{\kbar,0}(d_1,d_2)-\max_{d_1^2+d_2^2 \leq 1} \left(\gamma_1 d_2+\gamma_{*}(1-d_1^2-d_2^2)^{1/2}\right)\right)=\infty>0,
\end{equation}
which finishes the proof of \eqref{eqn:T:uniquezero}.
\hfill\ensuremath{\blacksquare}

\paragraph{Proof of Lemma \ref{lemma:RF:approx:c}.}
Throughout, we make use of the big-$O$ notation in $\psi_1$:
\begin{equation}
\label{eqn:bigO:psione}
    f(\psi_1,\psi_2,\gamma_1,\gamma_*,h)=O(g(\psi_1)) \;\;\Leftrightarrow
    \;\; \big|f(\psi_1,\psi_2,\gamma_1,\gamma_*,h)\big| \leq Cg(\psi_1),
\end{equation}
for some $C=C(\psi_2,\gamma_1,\gamma_*)>0$. It is convenient to consider the conditional law of $(G,X,W,h(G,X,W)$ given $\tX=0$ and $\tX\neq 0$. Recall that we can express the conditional law of $(G,X,W)$ given $\tX=0$ and $\tX\neq0$ in terms of $(G,\tG,V)$, where $V\sim \nu_{\psi_1^{-1}}$ as in \eqref{eqn:conditionaldistn:Xtilde}. Explicitly, we denote $h_{=0}$ and $h_{\neq 0}$ to be
measurable functions of $(G,\tG,V)$ such that
\begin{equation}
\label{eqn:conditional-law-GXWh}
\begin{split}
    &{\sf Law}\big(G,X,W,h(G,X,W)\mid \tX =0\big) ={\sf Law}(G,\gamma_{*}^2, 0,(1-\psi_1^{-1})^{-1/2}h_{=0})\\
    &{\sf Law}\big(G,X,W,h(G,X,W)\mid \tX \neq 0\big) ={\sf Law}(G,\gamma_1^2\psi_1V+\gamma_{*}^2, \frac{\gamma_1\psi_1\sqrt{V}\, \tilde{G}}{C_0 (\gamma_1^2\psi_1 V+\gamma_*^2)^{1/2}},\psi_1^{1/2}h_{\neq 0}),
\end{split}
\end{equation}
where $G,\tG \iid \normal(0,1)$, independent of $V\sim \nu_{\psi_1^{-1}}$. Before proceeding, we make five remarks about \eqref{eqn:conditional-law-GXWh}, which will be crucial for the proof.
\begin{enumerate}
\item Recall that $\P_{\psi_1}(\tX=0)=1-\psi_1^{-1},\P_{\psi_1}(\tX\neq 0)=\psi_1^{-1}$ as stated \eqref{eqn:conditionaldistn:Xtilde}.
\item The normalization for $h_{=0}$ and $h_{\neq 0}$ was chosen so that for $\E_{\psi_1}[h^2]\leq 1$,
    \begin{equation}
    \label{eqn:h-equalzero-notequalzero-bounded}
    \begin{split}
        \E[h_{=0}^2]+\E[h_{\neq 0}^2]&=(1-\psi_1^{-1})\E_{\psi_1}[h^2\mid \tX =0]+\psi_1^{-1}\E_{\psi_1}[h^2\mid \tX\neq 0]\\
        &=\P_{\psi_1}(\tX=0)\E_{\psi_1}[h^2\mid \tX =0]+\P_{\psi_1}(\tX\neq0)\E_{\psi_1}[h^2\mid \tX\neq 0]\\
        &=\E_{\psi_1}[h^2]\leq 1.
        \end{split}
    \end{equation}
\item We can express $d_{i,\psi_1}(h), 1\leq i \leq 4$, in terms of $h_{=0}$ and $h_{\neq 0}$ as
\begin{equation}
\begin{split}
    &d_{1,\psi_1}(h) = \E[h_{\neq0}\tG],\quad\quad~ d_{2,\psi_1}(h)= \Big(\E[h_{\neq 0}^2]-(\E[h_{\neq 0}\tG])^{2}\Big)^{1/2},\\
    &d_{3,\psi_1}(h) = \E[h_{\neq 0}G],\quad~~ d_{4,\psi_1}(h) = \E[h_{=0}G].
\end{split}
\end{equation}
    \item The following bound, which is a direct consequence of Cauchy Schwarz, will be used throughout the proof.
    \begin{equation}
    \label{eqn:innerprod-hconditional-bounded}
        \E\big[|Gh_{=0}|\big]|,\, \E\big[|\tG h_{=0}|\big],\,\E\big[|Gh_{\neq0}|\big],\,\E\big[|\tG h_{\neq0}|\big] \leq 1.
    \end{equation}
    \item Note that $\nu_{\psi_1^{-1}}$ has support $[(1-\psi_1^{-1/2})^2,(1+\psi_1^{-1/2})^2]$ for $\psi_1>1$ by definition of Marchenko-Pastur's law in \eqref{eq-MPlaw}. Hence, $V\sim \nu_{\psi_1^{-1}}$ implies that
\begin{equation}
\label{eqn:approximate-V}
    |V -1|, |\sqrt{V}-1|\leq 3\psi_1^{-1/2},\quad\text{almost surely.}
\end{equation}
\end{enumerate}
With this setup, we first show \eqref{eqn:RF:approximate:tau}. By the definition of $\tau(\psi_1)$ in \eqref{eqn:def-tau-psi-1-RF},
\begin{equation}
\label{eqn:tau-psi-technical}
\begin{split}
    \tau(\psi_1)^{2}&=1-\gamma_1^2\E_{\psi_1}[\frac{\psi_1 \tX}{\gamma_1^2\tX+\gamma_{*}^2}]=1-\gamma_1^2\E_{\psi_1}[\frac{\psi_1 \tX}{\gamma_1^2\tX+\gamma_{*}^2}\mid \tX\neq 0]\P_{\psi_1}(\tX \neq 0)\\
    &=1-\gamma_1^2\E[\frac{\psi_1 V}{\gamma_1^2\psi_1 V +\gamma_{*}^2}]=\E[\frac{\gamma_{*}^2}{\gamma_1^2\psi_1 V +\gamma_{*}^2}]=O(\psi_1^{-1}),
\end{split}
\end{equation}
where in the last bound, we used the fact that $V$ is bounded away from $0$, as stated in \eqref{eqn:approximate-V}.  Similarly we compute $\cbar_{1,\psi_1}(h)$ by conditioning on to the events $\tX=0$ and $\tX\neq 0$. By \eqref{eqn:conditional-law-GXWh},
\begin{equation}
\begin{split}
    \cbar_{1,\psi_1}(h)&=\psi^{-1/2}\E_{\psi_1}[X^{1/2}Wh]=\psi^{-1/2}\E_{\psi_1}[X^{1/2}Wh\mid \tX \neq 0]\P_{\psi_1}(\tX\neq0)\\
    &=\psi^{-1/2}\E[(C_0)^{-1}\gamma_1\psi_1\sqrt{V}\tG\psi_1^{1/2}h_{\neq0}]\psi_1^{-1}=(C_0)^{-1} \psi_2^{1/2}\gamma_1 \E[\sqrt{V}\tG h_{\neq 0}].
\end{split}
\end{equation}
By the estimate $C_0^2=1-\tau(\psi_1)^2=1+O(\psi_1^{-1})$ and \eqref{eqn:approximate-V}, the right-hand side equals 
\begin{equation}
\label{eqn:asymptoticRF:technical-1}
\begin{split}
     \cbar_{1,\psi_1}(h) =(C_0)^{-1} \psi_2^{1/2}\gamma_1 \E[\sqrt{V}\tG h_{\neq 0}]&=\big(1+O(\psi_1^{-1/2})\big) \big(\psi_2^{1/2}\gamma_1\E[\tG h_{\neq0}]+O(\psi_1^{-1/2})\E[|\tG h_{\neq 0}|]\big)\\
    &=\psi_2^{1/2}\gamma_1 d_{1,\psi_1}(h)+O(\psi_1^{-1/2}),
\end{split}
\end{equation}
where the last bound is by \eqref{eqn:innerprod-hconditional-bounded}. Proceeding in the same fashion to compute $\psi^{-1}\E_{\psi_1}[Xh^2]$, which appears in the definition of $\cbar_{2,\psi_1}(h)$, we have
\begin{equation}
    \begin{split}
        \psi^{-1} \E_{\psi_1}[Xh^2]&=\psi^{-1}\E_{\psi_1}[Xh^2\mid\tX=0]\P_{\psi_1}(\tX=0)+\psi^{-1}\E_{\psi_1}[Xh^2\mid \tX\neq 0]\P_{\psi_1}(\tX\neq0)\\
        &=\psi^{-1}\gamma_{*}^2\E[h_{=0}^2]+\psi^{-1}\E[(\gamma_1^2 \psi_1 V+\gamma_{*}^2)h_{\neq0}^2]=\psi_2\gamma_1^2\E[h_{\neq0}^2]+O(\psi_1^{-1/2}),
    \end{split}
\end{equation}
where in the last bound is due to \eqref{eqn:h-equalzero-notequalzero-bounded} and \eqref{eqn:approximate-V}. Therefore,
\begin{equation}
\begin{split}
    \cbar_{2,\psi_1}(h)&=\big(\psi^{-1}\E_{\psi_1}[Xh^2]-(\psi^{-1/2}\E_{\psi_1}[X^{1/2}W h])^{2}\big)^{1/2}=\big(\psi^{-1}\E_{\psi_1}[Xh^2]-\cbar_{1,\psi_1}(h)^2\big)^{1/2}\\
    &=\Big(\psi_2\gamma_1^2\E[h_{\neq0}^2]+O(\psi_1^{-1/2})-\big(\psi_2^{1/2}\gamma_1\E[\tG h_{\neq0}]+O(\psi_1^{-1/2})\big)^{2}\Big)^{1/2}\\
    &=\big(\psi_2\gamma_1^2\E[h_{\neq0}^2]-(\psi_2^{1/2}\gamma_1\E[\tG h_{\neq0}])^{2}+O(\psi_1^{-1/2})\big)^{1/2},
\end{split}
\end{equation}
where the last bound is by \eqref{eqn:innerprod-hconditional-bounded}. Now, we use the inequality $|(x+y)^{1/2}-x^{1/2}|\leq |y|^{1/2}$ for every $x>0,y\in \R$ to bound
\begin{equation}
    \cbar_{2,\psi_1}(h)=\big(\psi_2\gamma_1^2\E[h_{\neq0}^2]-(\psi_2^{1/2}\gamma_1\E[\tG h_{\neq0}])^{2}\big)^{1/2}+O(\psi_1^{-1/4})=\psi_2^{1/2}\gamma_1 d_{2,\psi_1}(h)+O(\psi_1^{-1/4}).
\end{equation}
Finally we compute $\cbar_{3,\psi_1}(h)=\E_{\psi_1}[X^{1/2}Gh]$:
\begin{equation}
\begin{split}
    \cbar_{3,\psi_1}(h)&=\E_{\psi_1}[X^{1/2}Gh\mid \tX=0]\P_{\psi_1}(\tX=0)+\E_{\psi_1}[X^{1/2}Gh\mid \tX\neq0]\P_{\psi_1}(\tX\neq0)\\
    &=\gamma_{*}\E[G(1-\psi_1^{-1})^{-1/2}h_{=0}](1-\psi_1^{-1})+\E[(\gamma_1^2\psi_1 V+\gamma_{*}^2)^{1/2}G\psi_1^{1/2}h_{\neq 0}]\psi_1^{-1}\\
    &=\gamma_{*}\E[Gh_{=0}](1-\psi_1^{-1})^{1/2}+\E[(\gamma_1^2V+\psi_1^{-1}\gamma_{*}^2)^{1/2}Gh_{\neq0}].
\end{split}
\end{equation}
By \eqref{eqn:approximate-V}, 
$(\gamma_1^2V+\psi_1^{-1}\gamma_{*}^2)^{1/2}=\big(\gamma_1^2+O(\psi_1^{-1/2})\big)^{1/2}=\gamma_1+O(\psi_1^{-1/4})$ holds, almost surely. Therefore, \eqref{eqn:innerprod-hconditional-bounded} shows that the right-hand side above can be estimated as
\begin{equation}
    \cbar_{3,\psi_1}(h)=\gamma_{*}\E[Gh_{=0}]+\gamma_1\E[Gh_{\neq0}]+O(\psi_1^{-1/4})=\gamma_{1}d_{3,\psi_1}(h)+\gamma_{*}d_{4,\psi_1}(h)+O(\psi_1^{-1/4}).
\end{equation}
\hfill\ensuremath{\blacksquare}

\paragraph{Proof of Lemma \ref{lem:approx:asymptoticRF}}
We first show that first partial derivatives of $(c_1,c_2)\mapsto F_{\kappa,\tau}(c_1,c_2)$ are bounded. Observe that by convexity of $F_{\kappa,\tau}$ in Lemma \ref{lemma:F-convex-increasing}, $c_1 \mapsto \partial_1F_{\kappa,\tau}(c_1,c_2)$ is increasing. Thus, for every $c_1,c_2\in\R$ and $\kappa>0,\tau\in[0,1]$,
\begin{equation*}
\partial_1 F_{\kappa,\tau}(c_1,c_2)\leq \limsup_{c \to \infty}\partial_1 F_{\kappa,\tau}(c,c_2)= \limsup_{c \to\infty} \frac{F_{\kappa,\tau}(c,c_2)}{c}=\E[(-Y_{\tau}G)_{+}^{2}]\leq \E[|G|^2]=1. 
\end{equation*}
Analogously, we can lower bound
\begin{equation*}
\partial_1 F_{\kappa,\tau}(c_1,c_2)\geq \liminf_{c \to -\infty}\partial_1 F_{\kappa,\tau}(c,c_2)= \liminf_{c \to-\infty} \frac{F_{\kappa,\tau}(c,c_2)}{c}=-\E[(Y_{\tau}G)_{+}^{2}]\geq -\E[|G|^2]=-1. 
\end{equation*}
Proceeding in the same fashion, the analogous bound for $\partial_{2}F_{\kappa,\tau}(c_1,c_2)$ holds as well. Thus, for $i=1,2$
\begin{equation}
    \sup_{\kappa>0,\tau\in[0,1]}\sup_{c_1,c_2\in\R}|\partial_i F_{\kappa,\tau}(c_1,c_2)|\leq 1.
\end{equation}
Therefore, for every $\kappa>0,\tau \in [0,1]$ and $ c_1,c_2,c_1^\prime,c_2^\prime \in \R$, we have that
\begin{equation}
\label{eqn:bound-difference-F-kappa-tau}
    |F_{\kappa,\tau}(c_1,c_2)-F_{\kappa,\tau}(c_1^\prime,c_2^\prime)| \leq |c_1-c_1^\prime|+|c_2-c_2^\prime|.
\end{equation}
Now, we use \eqref{eqn:bound-difference-F-kappa-tau} to bound $|E_{\psi_1,\kappa}(h)|$. Recalling the expression of $\asL_{\psi_1,\kappa}(h)$ in \eqref{eqn:asymptoticRF:betterexpression} and the definition of $\asG(\kappa,\tau,h)$ in \eqref{eqn:def:asymptoticRF-optimization},
\begin{equation}
\label{eqn:upperbound-E-diff-cd}
\begin{split}
   |E_{\psi_1,\kappa}(h)|=|\asL_{\psi_1,\kappa}(h)&-\asG(\psi^{-1/2}\kappa,\tau(\psi_1),h)|\\
   =\bigg\lvert F_{\psi^{-1/2}\kappa, \tau(\psi_1)}&\Big(\cbar_{1,\psi_1}(h),\cbar_{2,\psi_1}(h)\Big)+\cbar_{3,\psi_1}(h)\\
   &-F_{\psi^{-1/2}\kappa,\tau(\psi_1)}\Big(\psi_2^{1/2}\gamma_1 d_{1,\psi_1}(h), \psi_2^{1/2}\gamma_1d_{2,\psi_1}(h)\Big)-\gamma_1 d_{3,\psi_1}(h)-\gamma_{*}d_{4,\psi_1}(h)\bigg\rvert\\
   \leq \sum_{i=1,2}&|\cbar_{i,\psi_1}(h)-\psi_2^{1/2}\gamma_1 d_{i,\psi_1}|+|\cbar_{3,\psi_1}(h)-\gamma_1 d_{3,\psi_1}(h)-\gamma_{*}d_{4,\psi_1}(h)|.
\end{split}
\end{equation}
Observe that there is no dependence on $\kappa$ on the RHS above. Also, by Lemma \ref{lemma:RF:approx:c}, 
\begin{equation}
\label{eqn:diffcd-zero}
    \lim_{\psi_1\to\infty}\sup_{\E_{\psi_1}[h^2]\leq1}\bigg(\sum_{i=1,2}|\cbar_{i,\psi_1}(h)-\psi_2^{1/2}\gamma_1 d_{i,\psi_1}|+|\cbar_{3,\psi_1}(h)-\gamma_1 d_{3,\psi_1}(h)-\gamma_{*}d_{4,\psi_1}(h)|\bigg)=0.
\end{equation}
Therefore, by \eqref{eqn:upperbound-E-diff-cd} and \eqref{eqn:diffcd-zero}, $\lim_{\psi_1\to\infty} \sup_{\kappa>0}\sup_{\E_{\psi_1}[h^2]\leq 1} |E_{\psi_1,\kappa}(h)| =0$.
\paragraph{Proof of Lemma \ref{lemma:asymptoticRF:betterexpression}}
We consider $\psi_1>1$ and fix $\kbar>0,\tau\in [0,1]$ throughout the proof. Define the valid set of $\{\big(d_{i,\psi_1}(h)\big)_{1\leq i \leq 4}:\E_{\psi_1}[h^2]\leq 1\}$ as
\begin{equation}
     \cT_{\psi_1} \equiv \{(d_i)_{1\leq i \leq 4}\in\R^4:\text{there exists $h\in\cL^2(\P_{\psi_1}),\E_{\psi_1}[h^2]\leq 1$ such that $d_i=d_{i,\psi_1}(h), 1\leq i \leq 4$}\}.
\end{equation}
We will show that $\cT_{\psi_1}$ is the same for all $\psi_1>1$ and equals
\begin{equation}
\label{eqn:asymptoticRF-same-valid-set}
    \cT_{\psi_1}=\cT\equiv \{(d_1,d_2,d_3,d_4)\in \R^4:|d_3|\leq d_2,\quad d_1^2+d_2^2+d_4^2 \leq 1\}.
\end{equation}
First, we show that $\cT_{\psi_1}\subset \cT$. Note that for every $\E_{\psi_1}[h^2]\leq 1$, 
\begin{equation}
\label{eqn:decompose-h-norm}
\begin{split}
    d_{1,\psi_1}(h)^2+d_{2,\psi_1}(h)^2+d_{4,\psi_1}(h)^2
    &=\P_{\psi_1}(\tX\neq 0)\E_{\psi_1}[h^{2}\mid \tX \neq 0]+\P_{\psi_1}(\tX=0)(\E_{\psi_1}[Gh\mid \tX=0])^{2}\\
    &\leq \P_{\psi_1}(\tX\neq 0)\E_{\psi_1}[h^{2}\mid \tX \neq 0]+\P_{\psi_1}(\tX=0)\E_{\psi_1}[h^{2}\mid \tX=0]\E_{\psi_1}[G^2\mid \tX=0]\\
    &=\E_{\psi_1}[h^{2}\bfone(\tX \neq 0)]+\E_{\psi_1}[h^2\bfone(\tX=0)]=\E_{\psi_1}[h^2]\leq 1,
\end{split}
\end{equation}
where the first inequality is by Cauchy-Schwarz. We now show that $|d_{3,\psi_1}(h)|\leq d_{2,\psi_1}(h)$. Indeed, if we consider any $\alpha \in \R$, 
\begin{equation}
\begin{split}
   \psi_1 \big(\alpha d_{1,\psi_1}(h)+d_{3,\psi_1}(h)\big)^{2}&=(\E_{\psi_1}[h(\alpha \tG+G)\mid \tX\neq0])^{2}\\
    &\leq \E_{\psi_1}[h^2\mid \tX\neq 0]\E_{\psi_1}[(\alpha \tG+G)^2\mid \tX\neq0]=\psi_1(d_{1,\psi_1}(h)^{2}+d_{2,\psi_1}(h)^2)(\alpha^2+1),
\end{split}
\end{equation}
by Cauchy-Schwarz. Expanding the above inequality gives
\begin{equation}
    d_{2,\psi_1}(h)^2\alpha^2-2 d_{1,\psi_1}(h)d_{3,\psi_1}(h)\alpha+d_{1,\psi_1}(h)^2+d_{2,\psi_1}(h)^2-d_{3,\psi_1}(h)^2 \geq 0,
\end{equation}
for any $\alpha \in \R$. Hence, the discriminant of the above quadratic form is non-positive:
\begin{equation}
\label{eqn:discriminant-d2-leq-d3}
\begin{split}
    0 &\geq d_{1,\psi_1}(h)^2d_{3,\psi_1}(h)^2-d_{2,\psi_1}(h)^2(d_{1,\psi_1}(h)^2+d_{2,\psi_1}(h)^2-d_{3,\psi_1}(h)^2)\\
    &=(d_{1,\psi}(h)^2+d_{2,\psi_1}(h)^2)(d_{3,\psi_1}(h)^2-d_{2,\psi_1}(h)^2).
\end{split}
\end{equation}
In the case of $0=d_{1,\psi}(h)^2+d_{2,\psi_1}(h)^2=\psi_1^{-1}\E_{\psi_1}[h^2\mid \tX\neq 0]$, it must be that $h\bfone(\tX=0)=0$ $\P_{\psi_1}$-a.s., thus $d_{2,\psi_1}(h)=d_{3,\psi_1}(h)=0$. Therefore, \eqref{eqn:discriminant-d2-leq-d3} implies that $|d_{3,\psi_1}(h)|\leq d_{2,\psi_1}(h)$ and together with \eqref{eqn:decompose-h-norm}, $\cT_{\psi_1}\subset\cT$.

Conversely for any $(d_i)_{1\leq i \leq 4} \in \cT$, if we let $h_0 \in \cL^{2}(\P_{\psi_1})$ to be
\begin{equation}
    h_0=(1-\psi_1^{-1})^{-1/2} d_4G\bfone(\tX=0)+\psi_1^{1/2}(d_1\tG+d_3 G+\sqrt{d_2^2-d_3^2}) \bfone(\tX \neq 0),
\end{equation}
then $d_{i,\psi_1}(h_0)=d_i,1\leq i \leq 4$ holds and $\E_{\psi_1}[h_0^2]=d_1^2+d_2^2+d_4^2 \leq 1$, so $(d_i)_{1\leq i \leq 4} \in \cT_{\psi_1}$. Therefore $\cT_{\psi_1}=\cT$ for any $\psi_1>1$ as claimed in \eqref{eqn:asymptoticRF-same-valid-set}.

Now, we can express $\asG\opt(\kbar,\tau)$ as an $\psi_1$-independent quantity:
\begin{equation}
\begin{split}
    \asG\opt(\kbar,\tau)&=\min_{(d_1,d_2,d_3,d_4)\in \cT_{\psi_1}}\big\{F_{\kbar,\tau}(\psi_2^{1/2}\gamma_1d_1,\psi_2^{1/2}\gamma_1d_2)+\gamma_1d_3+\gamma_{*}d_4\big\}\\
    &=\min_{(d_1,d_2,d_3,d_4)\in \cT}\big\{F_{\kbar,\tau}(\psi_2^{1/2}\gamma_1d_1,\psi_2^{1/2}\gamma_1d_2)+\gamma_1d_3+\gamma_{*}d_4\big\}.
\end{split}
\end{equation}
Observe that in the above expression, fixing $d_1,d_2$, the minimizer for $d_3$ and $d_4$ is given by $d_3\opt =-d_2$ and $d_4\opt=-\sqrt{1-d_1^{2}-d_2^2}$. This is because $(d_1,d_2,d_3,d_4)\in \cT$ implies that $|d_3|\leq d_2,|d_4|\leq \sqrt{1-d_1^{2}-d_2^2}$. Moreover, the set for $(d_1,d_2)$ which there exists $(d_3,d_4)\in \R^2$ such that $(d_1,d_2,d_3,d_4)\in\cT$ is clearly given by $\{(d_1,d_2):d_1^2+d_2^2\leq 1, d_2\geq 0\}$. Therefore, we have the simplified expression of $\asG\opt(\kbar,\tau)$:
\begin{equation}
\begin{split}
    \asG\opt(\kbar,\tau)&=\min_{\substack{d_1^2+d_2^2\leq 1\\d_2\geq 0}}\big\{ F_{\bar{\kappa},\tau}\big(\sqrt{\psi_2}\gamma_1 d_1,\sqrt{\psi_2}\gamma_1 d_2\big)-\gamma_1 d_2-\gamma_{*}\sqrt{1-d_1^2-d_2^2}\big\}\\
    &=\min_{\substack{d_1^2+d_2^2\leq 1\\d_2\geq 0}}R(\kbar,\tau,d_1,d_2),
\end{split}
\end{equation}
which finishes the proof of our goal \eqref{eqn:asymptoticRF-kbar-tau-betterexpression}.

We conclude by proving Eq.~\eqref{eq:LB_asG}. For any $\E_{\psi_1}[h^2]\leq 1$, we have shown that $\left(d_{i,\psi_1}(h)\right) \in \cT$, thus,
\begin{equation}
\label{eqn:proof-Gprop1}
\begin{split}
    \asG(h;\kbar,\tau)&= F_{\kbar,\tau}\big(\psi_2^{1/2}\gamma_1d_{1,\psi_1}(h),\psi_2^{1/2}\gamma_1d_{2,\psi_1}(h)\big)+\gamma_1d_{3,\psi_1}(h)+\gamma_{*}d_{4,\psi_1}(h)\\
    &\geq F_{\kbar,\tau}\big(\psi_2^{1/2}\gamma_1d_{1,\psi_1}(h),\psi_2^{1/2}\gamma_1d_{2,\psi_1}(h)\big)-\gamma_1 d_{2,\psi_1}(h)-\gamma_{*}\sqrt{1-d_{1,\psi_1}(h)^2-d_{2,\psi_1}(h)^2}\\
    &=R(\kbar,\tau,d_{1,\psi_1}(h),d_{2,\psi_1}(h)).
\end{split}
\end{equation}
\hfill\ensuremath{\blacksquare}

%\section{Three dimension plots}

%\begin{figure}[H]
%\phantom{A}\hspace{-0.5cm}
%\includegraphics[width=0.54\textwidth]{FIG/random_feature_beta=4_3d_kappa.pdf}\hspace{-0.5cm}
%\includegraphics[width=0.54\textwidth]{FIG/random_feature_beta=4_3d_error.pdf}
%
%\caption{Random features model, with ReLU activations. 
%Left: maximum margin. Right: test error. Labels are generated using the logistic function 
%$f(x) = (1+e^{-\beta x})^{-1}$ with $\beta = 4$.
%Here red circles stand for empirical values for $d = 400$, and results were averaged over $20$ samples.
%Blue surfaces are the predicted values, both within the Gaussian covariates model of Section 
%\ref{sec:GaussianModel}.}\label{fig:random-features}
%\end{figure}

%
%******************************************************
%
\section{Proof of Proposition \ref{propo:Soft-Margin}}
\label{sec:appendix:soft:margin}
Notice that the soft margin $\kappa_n^{\sSM}(\by,\bX)$ (which is the value of the optimization problem
\eqref{eq:SMClassifier}) can be equivalently defined by
\begin{align}
\mbox{maximize}&\;\;\;\;\;\;\;\;\min_{i\le n}
                \big[\gamma_1y_i\<\btheta,\bx_i\>+\gamma_*u_i\big]\,,\\
  \mbox{subj. to}&\;\;\;\;\;\;\;\;
                   \|\btheta\|_2^2+\frac{\|\bu\|_2^2}{d}\le  1\, .
\end{align}
Notice indeed that the constraint    $\|\btheta\|_2^2+\|\bu\|_2^2/d= 1$ can be replaced by the inequality
constraint because any optimizer of the above problem satisfies $\bu\ge 0$, and its
value is non-decreasing if we increase the norm of $\bu$ by replacing $\bu$ by $\bu+c\bfone$
for some $c>0$.
         
  We next define
  \begin{align}
    \omega_{n, \psi_2,\kappa}\equiv
    \frac{1}{\sqrt{d}}\min_{\btheta\in\reals^d,\, \bu\in\reals^n}\left\{
\big\|\big(\kappa\psi_2^{-1/2}\cdot\bfone_n-\gamma_1\by\odot
    \bX\btheta-\gamma_{*}\bu\big)_+\big\|_2\, 
    \|\btheta\|_2^2+\frac{\|\bu\|_2^2}{d}\le  1\right\}\, .
 \end{align}
 The relation between this quantity and the soft margin $\kappa_n^{\sSM}(\by,\bX)$ is straightforward:
 \begin{align}
   \omega_{n, \psi_2,\kappa}=0 \;\;\; \Leftrightarrow \;\;\; \kappa^{\sSM}_n(\by,\bX)\ge \kappa\psi_2^{-1/2}\, .
 \end{align}
 Notice that the optimization over $\bu$ at fixed $\|\bu\|_2$ can be performed explicitly.
 Indeed, it is easy to check that, for any vector $\bv\in\reals^n$,
 \begin{align}
   \min\big\{\|(\bv-\bu)_+\|_2:\; \|\bu\|_2\le r \big\} =
   (\|\bv_+\|_2-r)_+\, .\label{eq:ExplicitOpt}
 \end{align}
 Further, whenever  $\|\bv_+\|_2-r\ge 0$ the minimum is uniquely achieved for $\bu = r\bv/\|\bv_+\|_2$.
 We are therefore led to define
 \begin{align}
    \omega^{(1)}_{n, \psi_2,\kappa}\equiv
    \min_{\btheta\in\reals^d, \|\btheta\|_2\le 1}\left\{
\frac{1}{\sqrt{d}}\big\|\big(\kappa\psi_2^{-1/2}\cdot\bfone_n-\gamma_1\by\odot
    \bX\btheta\big)_+\big\|_2-\gamma_*\sqrt{1-\|\btheta\|_2^2}\,\right\}\, .\label{eq:Omega1def}
 \end{align}
 By Eq.~\eqref{eq:ExplicitOpt} we have $\omega_{n,\psi_2,\kappa}= (\omega^{(1)}_{n,\psi_2,\kappa})_+$,
 and therefore
 \begin{align}
   \omega^{(1)}_{n, \psi_2,\kappa}\le 0 \;\;\; \Leftrightarrow \;\;\; \kappa^{\sSM}_n(\by,\bX)\ge \kappa\psi_2^{-1/2}\, .
   \label{eq:OmegaKappa}
 \end{align}
 By a simple rescaling,
 \begin{align}
    \omega^{(1)}_{n, \psi_2,\kappa}&=\min_{s\in [0,1]}
   \left\{
\frac{1}{\sqrt{d}} \min_{\btheta\in\reals^d,\|\btheta\|_2=s}\big\|\big(\kappa\psi_2^{-1/2}\cdot\bfone_n-\gamma_1\by\odot
   \bX\btheta\big)_+\big\|_2-\gamma_*\sqrt{1-s^2}\,\right\}\\
   &=\min_{s\in [0,1]}
   \left\{
     \frac{\gamma_1s}{\sqrt{d}} \min_{\btheta\in\reals^d,\|\btheta\|_2=1}\big\|\big(\frac{\kappa}{\gamma_1s\sqrt{\psi_2}}
     \cdot\bfone_n-\by\odot
   \bX\btheta\big)_+\big\|_2-\gamma_*\sqrt{1-s^2}\,\right\}\, .\label{eq:Omega1}
 \end{align}
 By Lemma \ref{lemma:xi-0-xi-1}, \ref{lemma:xi-1-xi-2} and Proposition \ref{proposition:xi-n-2-xi-infty},
 the following limit holds in probability, for any fixed $\tilde\kappa\in [0,\infty]$,
 \begin{align}
\lim_{n\to\infty} \xi_{n,\psi_2^{-1},\tilde\kappa}& = T(\psi_2^{-1}, \tilde\kappa)\,, \label{eq:ConvergencePerKappa}\\
   \xi_{n,\psi_2^{-1},\tilde\kappa} & = \frac{1}{\sqrt{d}} \min_{\btheta\in\reals^d,\|\btheta\|_2=1}\big\|\tilde{\kappa}
     \cdot\bfone_n-\by\odot
     \bX\btheta\big)_+\big\|_2
 \end{align}
 (Notice that $\xi_{n,\psi_2^{-1},\tilde\kappa}$ is defined as in Section \ref{sec:ProofMain},
 with $p$ replaced by $d$, and $\psi$ by $\psi_2^{-1}$.)
 Further, using Corollary~\ref{coro:PsiLarge} and the definition \eqref{eqn:def-L-star}
 we get
 \begin{align}
   T(\psi_2^{-1}, \tilde\kappa) &= \inf_{\|h\|_{\P}\le 1}\asL_{\psi_2^{-1}, \tilde\kappa, \P}(h)\\
                                & = \min_{c\in [0,1]}\sqrt{\psi_2}\, F_{\tilde\kappa}(c,\sqrt{1-c^2})-\sqrt{1-c^2}\,.
 \end{align}
 The second equality follows by using the definition of $\asL_{\psi_2^{-1}, \tilde\kappa, \P}(h)$ in Eq.~\eqref{eq:asLDef}
 and the fact that in the present case $(G,X,W)\sim\normal(0,1)\otimes \delta_0\otimes\normal(0,1)$.
 Notice that both the left and right-hand sides of Eq.~\eqref{eq:ConvergencePerKappa}
 are non-decreasing function of $\tilde\kappa$, and Lipchitz continuous in $\tilde\kappa$,
 with Lipschitz constant $\sqrt{\psi_2}$
 (the latter follows because they are minima of Lipchitz-continuous functions).
 Therefore the convergence of Eq.~\eqref{eq:ConvergencePerKappa} takes place uniformly over compacts.
 Namely, for any $K>0$ we have
 \begin{align}
 \lim_{n\to\infty} \sup_{\tilde\kappa\in[0,K]}
   \big|\xi_{n,\psi_2^{-1},\tilde\kappa}- T(\psi_2^{-1}, \tilde\kappa)\big| = 0\, .
 \end{align}
 Using this result in Eq.~\eqref{eq:Omega1}, we obtain that the following limit holds in probability
 \begin{align}
   \lim_{n\to\infty}    \omega^{(1)}_{n, \psi_2,\kappa} &=\min_{s\in [0,1]}
   \left\{
   \gamma_1s \, T\Big(\psi_2^{-1}, \frac{\kappa}{\gamma_1s\sqrt{\psi_2}}\Big)
                                                          -\gamma_*\sqrt{1-s^2}\,\right\}\\
   &=\min_{c,s\in [0,1]}
   \left\{
   \gamma_1s \, \sqrt{\psi_2}F_{\kappa/\gamma_1s\sqrt{\psi_2}}\Big(c,\sqrt{1-c^2}\Big)
 -\gamma_1s\sqrt{1-c^2}    -\gamma_*\sqrt{1-s^2}\,\right\}\\
   &=\min_{c,s\in [0,1]}
   \left\{
   \gamma_1s \, \sqrt{\psi_2}F_{\kappa/\gamma_1s\sqrt{\psi_2}}\Big(c,\sqrt{1-c^2}\Big)
     -\gamma_1s\sqrt{1-c^2}    -\gamma_*\sqrt{1-s^2}\,\right\}\\
   &=\min_{c,s\in [0,1]}
   \left\{
   F_{\kappa}\Big(\gamma_1\sqrt{\psi_2} sc,\gamma_1\sqrt{\psi_2} s\sqrt{1-c^2}\Big)
     -\gamma_1s\sqrt{1-c^2}    -\gamma_*\sqrt{1-s^2}\,\right\}\, .
 \end{align}
 Here in the last step we used the homogeneity property $aF_{\kappa}(c_1,c_2) = F_{a\kappa}(ac_1,ac_2)$,
 which holds for $a\ge 0$.
 Comparing the last expression with Eq.~\eqref{eqn:def:T-infty}, we get
 (identifying $d_1=sc$ and $d_2=s\sqrt{1-c^2}$):
 \begin{align}
   \lim_{n\to\infty}    \omega^{(1)}_{n, \psi_2,\kappa} &=
                                                          T_{\infty}(\kappa; \psi_2,\gamma_1,\gamma_{*})\, .
 \end{align}
 Hence the claim \eqref{eq:Kwide} follows from Eq.~\eqref{eq:OmegaKappa} by noticing
 that $\kappa\mapsto  T_{\infty}(\kappa; \psi_2,\gamma_1,\gamma_{*})$ is strictly monotone decreasing
 in $\kappa$ with a unique zero $\kbar\owid(\psi_2,\gamma_1,\gamma_{*}$
 (the last property follows by specializing Proposition \ref{proposition:system-of-Eq-T}).

 Finally \eqref{eq:Predwide} follows by keeping track of the minimizer \eqref{eq:Omega1def} in the above derivation.

\section{Special examples and numerical illustrations}
\label{sec:Special}

In this section we further illustrate our main results by
considering a few special cases, namely special sequences of the true parameter vector $\btheta_{*,n}$, and covariance matrix $\bSigma_n$. 
We also discuss some statistical insights that can be drawn from the analysis of these cases. 

\subsection{Isotropic well specified model}
\label{sec:Isotropic}

\begin{figure}[!ht]
\phantom{A}\hspace{-0.5cm}\includegraphics[width = 0.54\linewidth]{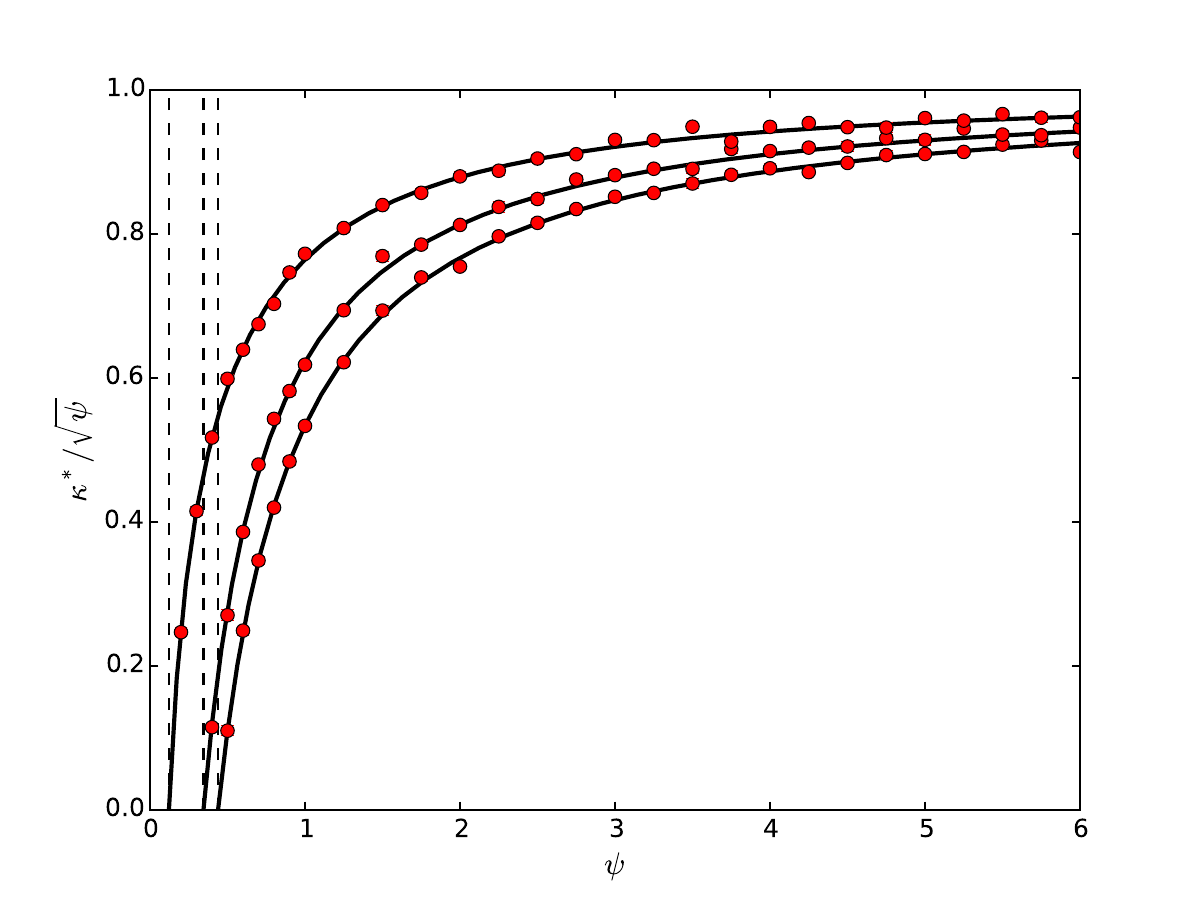}\hspace{-0.5cm}
\includegraphics[width = 0.54\linewidth]{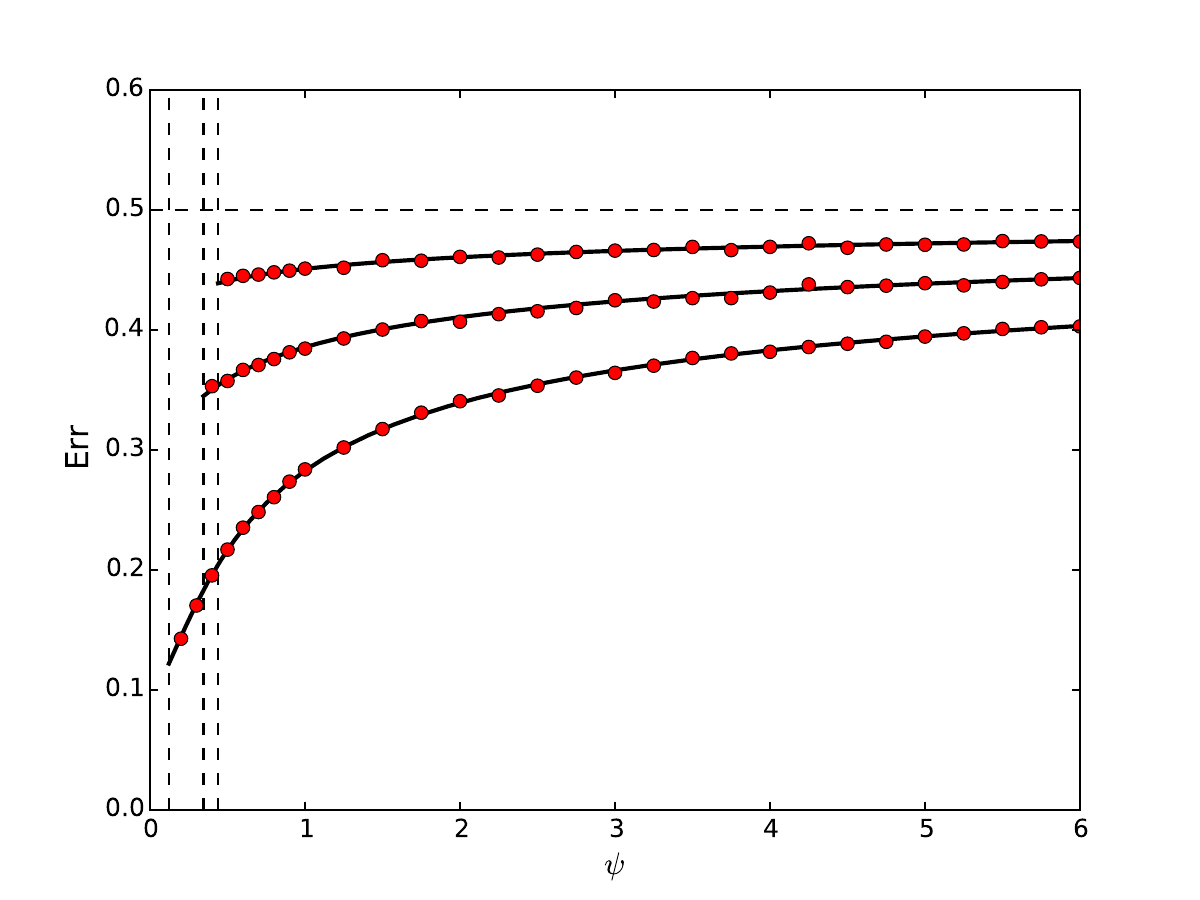}
\caption{Max-margin linear classification for isotropic well specified
  data. Left: maximum margin (scaled by $\sqrt{\psi} = \sqrt{p/n}$). Right: test error. Labels are generated using  the logistic function 
$f(x) = (1+e^{-\beta x})^{-1}$ with $\beta =1, 2, 8$ (from bottom to top on the left, and from top to bottom on the right). 
Vertical dashed lines correspond to the interpolation threshold $\psi\opt$,
and continuous lines to the analytical predictions of Corollary \ref{coro:Isotropic}. Symbols are empirical results for the max-margin 
(left) and prediction error (right).
Here $p=800$ and we vary $n=p/\psi$, averaging results over $20$ instances. Error bars (barely visible) report standard errors on the empirical means of $20$ instances.}
\label{fig:Isotropic}
\end{figure}
We begin by considering the simplest case, namely isotropic covariates
$\bx_i\sim \normal(\bzero, \id_p)$, (i.e $\bSigma_n=\id_p$). In this case, by  rotational invariance, the margin and prediction error do not depend on the vector 
$\btheta_{*,n}$ which has unit norm.
Figure \ref{fig:Isotropic} report the results of numerical experiments with $p=800$ and various values of $n$. We observe that the classification
error decreases as $n$ increase, i.e. as $\psi$ decreases, until it crosses a threshold below which the data is no longer separable.

 In order to state our characterization of the maximum margin and prediction error,
we introduce the function $F_{\kappa}: \R \times \R_+ \to \R_+$ (for $\kappa\in\reals$):
\begin{equation}\label{eq:FkDef}
F_{\kappa}(c_1, c_2) = \left(\E \left[(\kappa - c_1 YG - c_2 Z)_+^2\right]\right)^{1/2}
~~\text{where}~
	\begin{cases}
		Z \perp (Y, G) \, ,\\
		Z \sim \normal(0, 1), G\sim \normal(0, 1)\, , \\
		\P(Y = +1 \mid G) = f(G) \, ,\\
		\P(Y = -1 \mid G) = 1-f(G)\, .
	\end{cases}
\end{equation}
The next corollary is an immediate consequence of our main result, Theorem \ref{theorem:main}.
\begin{corollary}\label{coro:Isotropic}
  Consider the isotropic model, and let $f$ satisfy Assumption \ref{assumption:non-degenerate-f}.
  % Define  $L(c;\kappa,\psi) \equiv F_{\kappa}(c,\sqrt{1-c^2})-\sqrt{\psi(1-c^2)}$.
  Then the following hold:
\begin{enumerate}
\item[$(a)$] The maximum margin $\kappa_{n}(\by,\bX)$ converges almost surely to a strictly positive limit if and only if
$\psi>\psi\opt_{\siso}$, where the interpolation threshold is given by
\begin{align}
\psi\opt_{\siso} = \min_{c\ge 0} F_0(c,1)^2\, .
\end{align}
\item[$(b)$] For any $\psi>\psi\opt_{\siso}$ the asymptotic maximum margin is given by $\lim_{n\to\infty}\kappa\opt_{n}(\by,\bX)\to \kappa\opt_{\siso}(\psi)$, 
where 
\begin{align}
\kappa\opt_{\siso}(\psi) =  \inf\Big\{\, \kappa\ge 0 :\;\; F_{\kappa}(c,\sqrt{1-c^2})-\sqrt{\psi(1-c^2)}>0 \;\;\;\;\forall c\in [0,1]\Big\}\, .
\end{align}
\item[$(c)$] The asymptotic prediction error is given by  $\lim_{n\to\infty}\Pred_n(\by,\bX) = \Pred^*_{\siso}(\psi)$, where
\begin{align}
\Pred\opt_{\siso}(\psi)& = \P\Big(c\opt_{\siso}(\psi)YG+\sqrt{1-c\opt_{\siso}(\psi)^2}\,  Z\le 0\Big)\, ,\\
c\opt_{\siso}(\psi) &\equiv \arg\min_{c\in [0,1]}\, \Big\{F_{\kappa=\kappa\opt_{\siso}}(c,\sqrt{1-c^2})-\sqrt{\psi(1-c^2)}\Big\}\, .
\end{align}
(Here expectation is taken with respect to the random variables $(G,Y,Z)$ with joint distribution defined in Eq.~\eqref{eq:FkDef}.)
\end{enumerate}
\end{corollary}
In Figure \ref{fig:IsoBound} we compare the theoretical prediction for the maximum margin and 
test error given in the last corollary with the numerical results: the agreement is excellent. 
Our analytical predictions confirm the observation made above: the error is monotone increasing in  $\psi$, for $\psi>\psi\opt$.
It is possible to show that $\Pred^*_{\siso}(\psi)\to 1/2$ as $\psi\uparrow \infty$, while $\Pred^*_{\siso}(\psi)\to\Pred^*_{\siso}(\psi\opt)$
as $\psi\downarrow\psi\opt$, where (in general) $\Pred^*_{\siso}(\psi\opt)\in (0,1/2)$.

Notice that this behavior is different from the one observed for min-norm least squares \cite{belkin2019two,hastie2022surprises}, under 
isotropic covariates for a similarly well-specified model. In that case, at small signal-to-noise ratio (SNR), the error is monotone decreasing
for $\psi>\psi\opt$, while at high SNR it is monotone decreasing in an interval $\psi\in (\psi\opt,\psi_{\min})$, and increasing for $\psi>\psi_{\min}$. 
This different behavior can be explained, at least in part, by the observation that the square loss is unbounded and diverges (for min-norm
least squares) at the interpolation threshold. Hence it is necessarily decreasing right above that threshold. 
In contrast, since the classification error is bounded,  it can be monotone increasing with the overparametrization ratio $\psi = p/n$.

\begin{figure}[!ht]
\centering
\includegraphics[width = 0.6\linewidth]{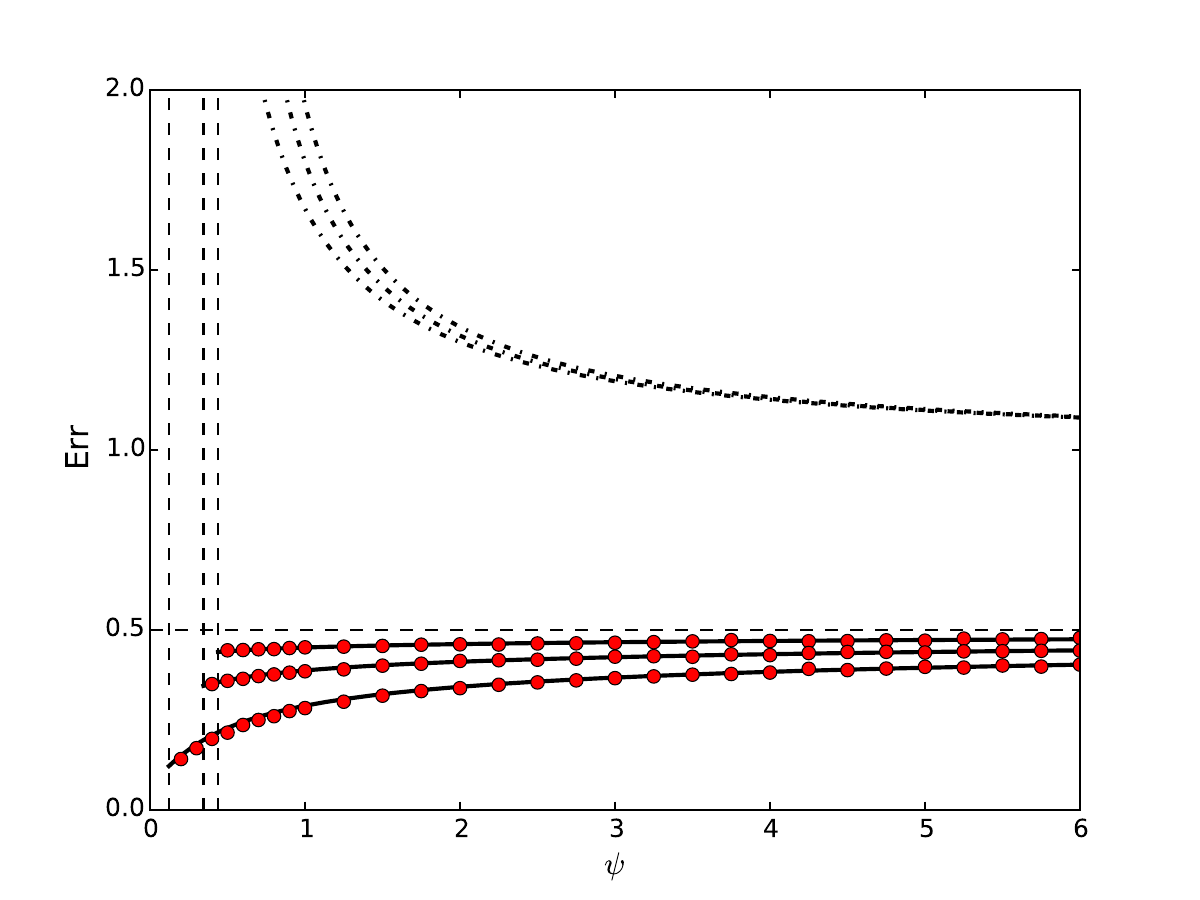}
\caption{Isotropic well-specified model: comparing margin-based bounds and actual test error. 
In the lower part of the plot, analytical predictions (continuous lines) and numerical simulations (circles) for the test error: same data 
as in Figure \ref{fig:Isotropic}. Dot-dashed lines are margin-based estimates  of the test error $\sqrt{\psi}/\kappa\opt(\psi)$ (see text).}
\label{fig:IsoBound}
\end{figure}
The maximum margin $\kappa\opt_{\siso}(\psi)$ is also monotone increasing with $\psi$, with $\kappa\opt_{\siso}(\psi)\downarrow 0$
as $\psi\downarrow \psi\opt$,  and $\kappa\opt_{\siso}(\psi)/\sqrt{\psi}\to 1$ as $\psi\to\infty$. 
Notice that the relation between margin and classification error is somewhat counterintuitive. 
On the basis of classical margin theory \cite{shalev2014understanding}, one would expect that the test error increases when the margin decreases.
The opposite happens in Fig.~\ref{fig:Isotropic}: as $\psi$ decreases both the error and the margin decrease. The explanation to 
this counterintuitive phenomenon is related to the fact that, in the present regime, the margin does not concentrate near its population value 
(and is not separable in this case).

In order to further clarify this phenomenon, in Figure \ref{fig:IsoBound} we compare the actual test error (both numerical simulations, and the predictions of
 Corollary \ref{coro:Isotropic}), with a margin-based bound from \cite[Theorem 26.14]{shalev2014understanding}. The latter implies, with our notations,
\begin{align}
\Pred_n(\by,\bX) \le \frac{4r\sqrt{\psi}}{\kappa\opt(\psi)} +o_n(1)\, ,\;\;\;\;\; r^2\equiv \int x\, \mu(\de x,\de w)\, . \label{eq:MarginBound:1}
\end{align}
Here $r$ is the typical (normalized) radius of the feature vectors, namely the asymptotic value of  $r_n^2 = \E\|\bx_1\|^2/p$
 (in the present case, $r=1$).
Even discarding the factor $4$ (which we do in Figure \ref{fig:IsoBound}), this upper bound has the wrong qualitative dependence
on $\psi$ and is never non-trivial in the present setting (never smaller than 1).

We conclude that the  isotropic well-specified data distribution does not capture the benefits of overparametrization
discussed in the introduction. This is not unexpected: as $p/n$ increases in this setting we are increasing the complexity of the model, but also the complexity of the target function.

\subsection{Isotropic misspecified models}

Assuming the model to be well specified can be unrealistic. In this section we keep considering isotropic covariates, but
introduce a simple misspecification structure to capture the approximation benefits of adding more covariates.

We assume that label $y_i$ depend on a potentially infinitely dimensional feature vector $\bz_i\in\reals^{\infty}$, $\bz_i\sim\normal(0,\id_{\infty})$,
via
\begin{align}
\P\big(y_i=+1\big|\bz_i\big) = f_0(\<\bbeta_{*},\bz_i\>)\, . \label{eq:Misspecified}
\end{align}
Note that this makes mathematical sense as long as $\bbeta_{*}\in
\ell_2$.  Without loss of generality, we can assume $\|\bbeta_{*}\|_2=1$.
We  learn a max-margin classifier over the first $p$ features. Namely,
we write $\bz_i = (\bx_i, \tbz_i)$ where $\bx_i\in\reals^p$ contains the first $p$ coordinates of $\bz_i$, and $\tbz_i$ contains the other coordinates. We then apply max-margin
classification to data $\{(y_i,\bx_i)\}_{i\le n}$.   Given a vector $\bv\in\reals^{\infty}$, we write $\proj_{\le \ell}\bv$ for the $\ell$-dimensional vector formed by the
 first $\ell$ entries of $\bv$.

\begin{figure}[!ht]
\phantom{A}\hspace{-0.5cm}\includegraphics[width = 0.54\linewidth]{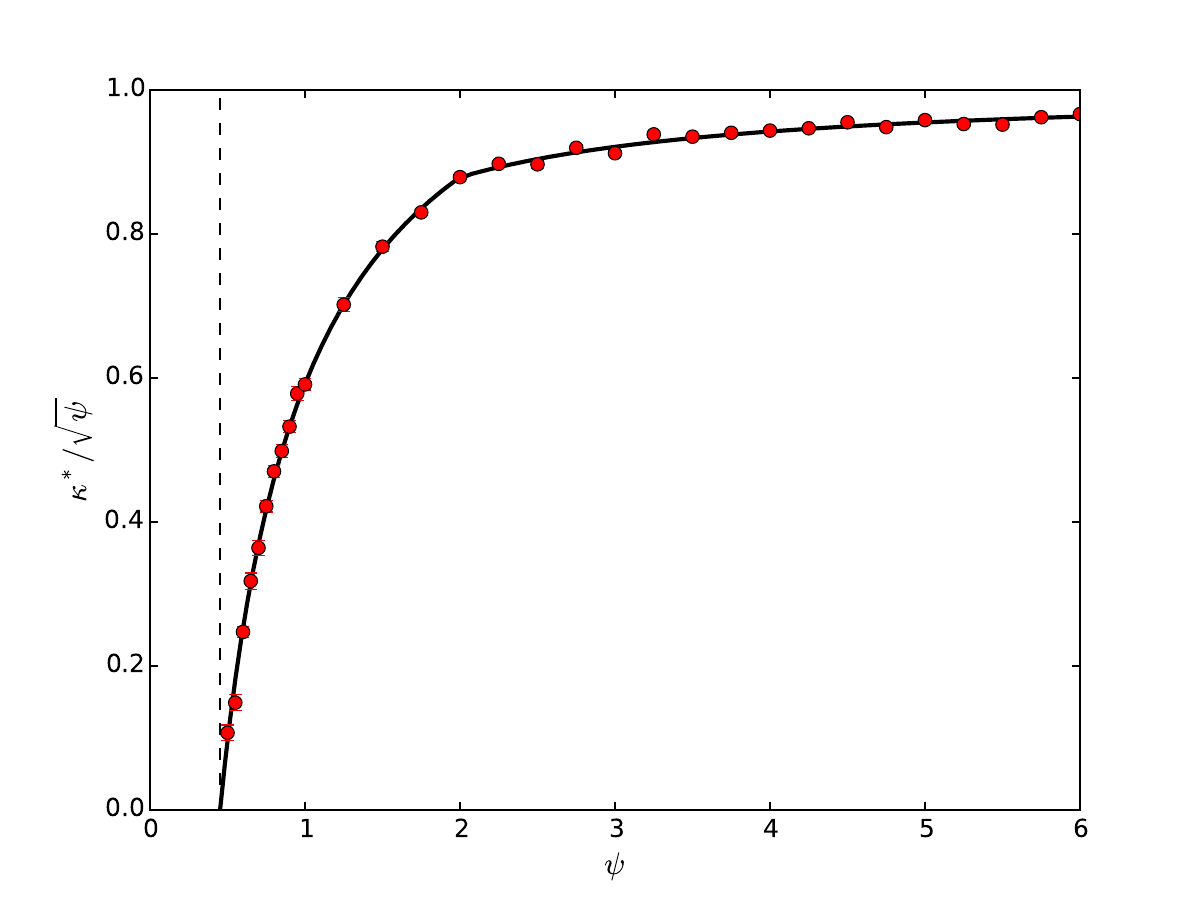}\hspace{-0.5cm}
\includegraphics[width = 0.54\linewidth]{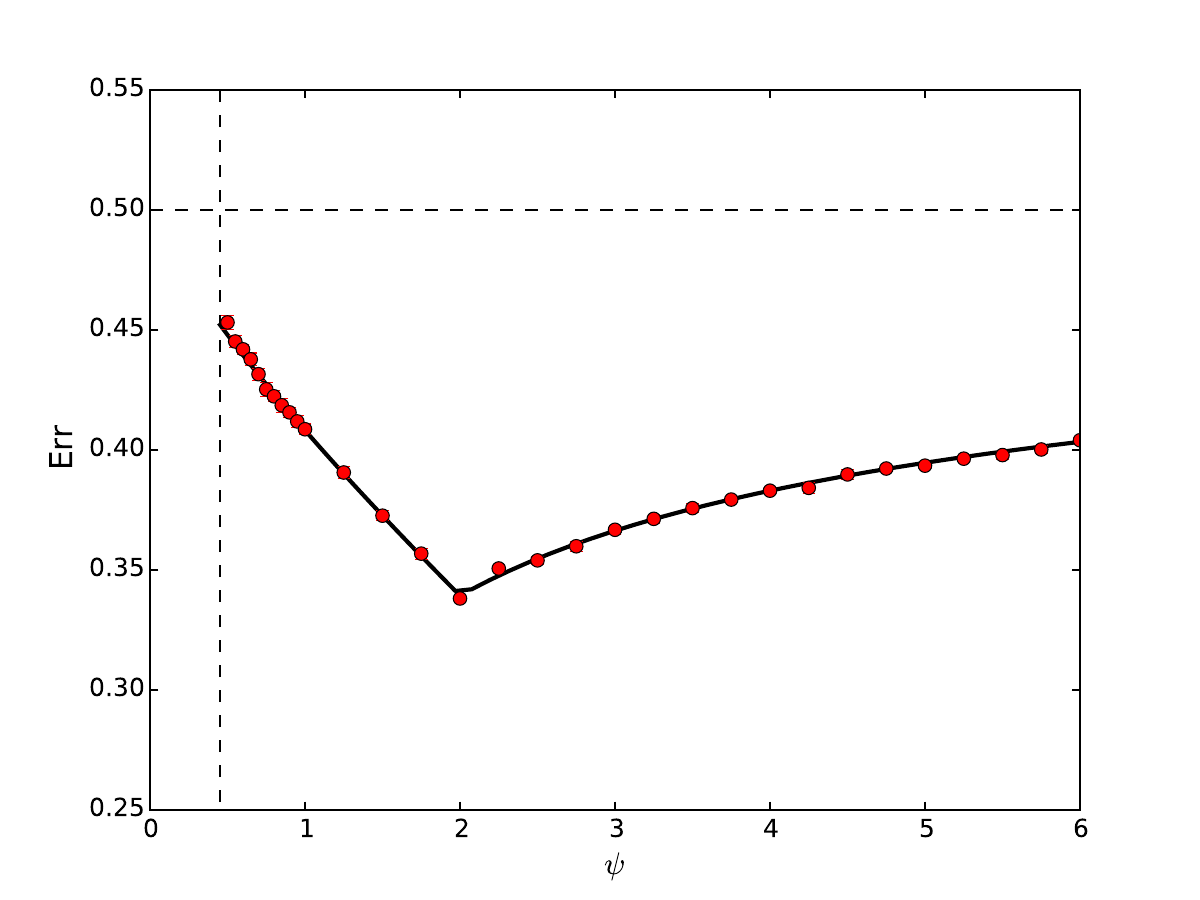}
\caption{Isotropic miss-specified model. Left: maximum margin (scaled by $\sqrt{\psi} = \sqrt{p/n}$). Right: test error. 
Labels are generated according to Eq.~\eqref{eq:Misspecified}, with $f_0(x) = (1+e^{-\beta x})^{-1}$, with $\beta=8$.
The vertical dashed line corresponds to the interpolation threshold $\psi\opt_{\smiss}$, and the continuous line to the analytical prediction from Corollary 
\ref{coro:Misspecified}.  Symbols are empirical results for the prediction error,
for $n=400$, $p_0=800$ and varying $p=n\psi$, averaged over $20$ instances. Error bars (barely visible) report standard errors on the empirical means of $20$ instances.}\label{fig:mis}
\end{figure}
This setting can be reduced to the one in the previous section (isotropic well-specified model), whereby  labels are assigned according
to Eq.~\eqref{eq:LabelProbability}, with
$f(x) = \E\{f_0(\sqrt{\gamma_n} x+ \sqrt{1-\gamma_n}G')\}$ with $G'\sim\normal(0,1)$ and $\gamma_n = \|\proj_{\le p}\bbeta_{*,n}\|_2$.
Further $\btheta_{*,n} = \proj_{\le p}\bbeta_{*,n}/\gamma_n$.
 In words, those
features that are not included in the model, and that correspond to non-zero entries in $\bbeta_{*,n}$, act as additional noise in the labels.
As more and more features are added to the model, the corresponding noise variance $(1-\gamma_n)$ decreases.

In order to  state the asymptotic characterization of the max margin classifier, we modify the function  of Eq.~\eqref{eq:FkDef}
as follows
\begin{equation}\label{eq:FkDefMiss}
F_{\kappa,\gamma}(c_1, c_2) = \left(\E \left[(\kappa - c_1 Y_{\gamma}G - c_2 Z)_+^2\right]\right)^{1/2}
~~\text{where}~
	\begin{cases}
		Z \perp (Y, G)\, ,\\
		\P(Y_{\gamma} = +1 \mid G) = \E [f_0(\sqrt{\gamma}G+\sqrt{1-\gamma} G')|G] \, ,\\
		\P(Y_{\gamma} = -1 \mid G) = 1-\E [f_0(\sqrt{\gamma}G+\sqrt{1-\gamma} G')|G] \, ,\\
                Z , G, G'\sim_{iid} \normal(0, 1)\, .
	\end{cases}
\end{equation}
\begin{corollary}\label{coro:Misspecified}
Consider the misspecified isotropic model, and let $f_0$ satisfy Assumption \ref{assumption:non-degenerate-f}. 
Further assume $\bbeta_{*,n}\in\reals^{\infty}$ to be such that $\|\bbeta_{*,n}\|_2=1$ and $\|\proj_{\le p(n)}\bbeta_{*,n}\|^2_2\to \gamma(\psi)$
as $n\to\infty$.
For  any $\psi>0$,  define
\begin{align}
\kappa\opt_{\smiss}(\psi) =  \inf\Big\{\, \kappa\ge 0 :\;\; F_{\kappa,\gamma(\psi)}(c,\sqrt{1-c^2})-\sqrt{\psi(1-c^2)}>0 \;\;\;\; \forall c\in [0,1]\Big\}\, .
\end{align}
\begin{enumerate}
\item[$(a)$] The maximum margin $\kappa\opt_{n}(\by,\bX)$ converges almost surely to a strictly positive limit if and only if
$\psi>\psi\opt_{\smiss}\equiv \inf\{\psi>0:\; \kappa\opt_{\smiss}(\psi) >0\}$.
\item[$(b)$] For $\psi>\psi\opt_{\smiss}$,
the asymptotic max-margin is given by $\lim_{n\to\infty}\kappa\opt_{n}(\by,\bX)\to \kappa\opt_{\smiss}(\psi)$.
\item[$(c)$] The asymptotic prediction error is given by  $\lim_{n\to\infty}\Pred_n(\by,\bX) = \Pred\opt_{\smiss}(\psi)$, where
\begin{align}
\Pred\opt_{\smiss}(\psi)& = \P\Big(c\opt_{\smiss}(\psi)Y_{\gamma(\psi)}G+\sqrt{1-c\opt_{\smiss}(\psi)^2}\,  Z\le 0\Big)\, ,\\
c\opt_{\smiss}(\psi) &\equiv \arg\min_{c\in [0,1]}\, \Big\{F_{\kappa=\kappa\opt_{\smiss}(\psi),\gamma(\psi)}(c,\sqrt{1-c^2})-\sqrt{\psi(1-c^2)}\Big\}\, .
\end{align}
\end{enumerate}
\end{corollary}

Note that the misspecified model has in important conceptual advantage over the well specified one:
the data distribution \eqref{eq:Misspecified} is independent of the number of features 

In Figure \ref{fig:mis} we consider a misspecified problem in which $\bbeta_{*,n}$ puts equal asymptotically 
weight over the first $p_0$ features, where $p_0/n\to \psi_0\in(0,\infty)$. Explicitly, we assume
$\beta_{*,i}\in \{+1/\sqrt{p_0},-1/\sqrt{p_0}\}$ for $i\le p_0$, and $\beta_{*,i}=0$ for $i>p_0$.
This results in  $\gamma(\psi) = \psi/\psi_0$ if $\psi\le \psi_0$ and $\gamma(\psi) = 1$ for $\psi>\psi_0$. 
Note that the same limiting function $\gamma(\psi)$ is obtained for other choices of the vector $\bbeta_*,n$. For instance, if $\bbeta_*$ is a uniformly random vector drawn
 independent for each $n$, with unit norm and support on $\{1,\dots,p_0\}$, the assumptions of Corollary \ref{coro:Misspecified} are satisfied again,
with $\gamma(\psi) = \min(\psi/\psi_0,1)$ as before.

In this example the test error decreases in the overparametrized regime for $\psi\opt_{\smiss}<\psi<\psi_0$, and then increases again for $\psi_0<\psi$.
As explained above, adding more features reduces the approximation error, and hence results in smaller test error.
A similar behavior was observed in \cite{hastie2022surprises} for the case of min-norm least squares regression. 
Notice that the maximum margin is monotone increasing in $\psi$, with $\kappa\opt(\psi)/\sqrt{\psi}<1$.
Further, since $\kappa\opt(\psi)/\sqrt{\psi}<1$,
the classical margin-based bound of Eq.~\eqref{eq:MarginBound:1} is always larger than one.
As for the well-specified model, the margin  does not seem to capture the behavior of the actual test error. 

While the present misspecified data distribution is richer than the well specified distribution of the previous section,
it seems too simplistic to capture the benefits of overparametrization in modern machine learning. In particular,
the optimum overparametrization ratio $\psi$ is bounded.

 \bibliographystyle{amsalpha}
\bibliography{all-bibliography}

%\section{Analysis of Eq (\ref{eqn:target-up})}
%\input{upper-bound}

\end{document}